%% file: PeculiarModules.tex
\documentclass{lms}

\usepackage{lmodern}
\usepackage[T1]{fontenc}
\usepackage[utf8]{inputenc}
\usepackage{cite}

\usepackage{amsmath}
\usepackage[UKenglish]{babel}
\usepackage{amsfonts}
\usepackage{amssymb}
\usepackage{units}
\usepackage{kbordermatrix}
\usepackage{minibox}

\usepackage{wrapfig}
\newcommand{\myfixwrapfig}{~\vspace*{-\baselineskip}\vspace*{-10pt}}
\newcommand{\myfixwrapfigdouble}{~\vspace*{-\baselineskip}\vspace*{-20pt}}

\usepackage{mathtools}
\mathtoolsset{mathic=true}

\usepackage{xcolor}
\colorlet{lightred}{red!20!white}
\colorlet{lightblue}{blue!30!white}
\colorlet{darkgreen}{green!70!black}
\colorlet{lightgreen}{green!50!white}
\colorlet{gold}{yellow!90!black!70!red}
\newcommand{\darkgreen}[1]{\textcolor{darkgreen}{#1}}
\newcommand{\gold}[1]{\textcolor{gold}{#1}}

\usepackage{rotating}
\usepackage{calc}

\usepackage{pstricks,pstricks-add,pst-pdf}

\usepackage{cancel}
\usepackage[bookmarks=true,pdfencoding=auto,colorlinks=true,citecolor=black,urlcolor=black,linkcolor=black]{hyperref}
\usepackage{bookmark}
\usepackage{ifpdf}
\ifpdf
\usepackage{tikz-cd}
\pgfdeclarelayer{background}
\pgfsetlayers{background,main}
\fi

\usepackage{graphicx}

\usepackage[textfont=it,labelfont=sc]{caption}
\usepackage[textfont=it,labelfont=up,margin=5pt]{subcaption}
\usepackage{fancybox}
\usepackage{multirow}

\newtheorem{theorem}{Theorem}[section] 
\newtheorem{lemma}[theorem]{Lemma}     
\newtheorem{corollary}[theorem]{Corollary}
\newtheorem{proposition}[theorem]{Proposition}
\newtheorem{fact}[theorem]{Fact}

\newnumbered{conjecture}[theorem]{Conjecture}  
\newnumbered{definition}[theorem]{Definition}
\newnumbered{remark}[theorem]{Remark}
\newnumbered{observation}[theorem]{Observation}
\newnumbered{question}[theorem]{Question}
\newnumbered{questions}[theorem]{Questions}
\newnumbered{example}[theorem]{Example}

\renewcommand{\descriptionlabel}[1]{\hspace\labelsep \upshape\bfseries #1}
\def\co{\colon\thinspace}

\newcommand{\qp}[2]{(\textcolor{red}{p_{#2}}\vert{\blue q_{#1}})}
\newcommand{\pq}[2]{(\textcolor{red}{q_{#2}}\vert{\blue p_{#1}})}

\newcommand{\pn}[1]{(\textcolor{red}{-}\vert{\blue p_{#1}})}
\newcommand{\qn}[1]{(\textcolor{red}{-}\vert{\blue q_{#1}})}
\newcommand{\np}[1]{(\textcolor{red}{p_{#1}}\vert{\blue -})}
\newcommand{\nq}[1]{(\textcolor{red}{q_{#1}}\vert{\blue -})}

\newcommand{\ob}{\operatorname{ob}}
\DeclareMathOperator{\Mor}{Mor}
\DeclareMathOperator{\Mat}{Mat} 
\DeclareMathOperator{\Cx}{Cx} 
\DeclareMathOperator{\Cxpre}{Cx^\text{\normalfont pre}}
\DeclareMathOperator{\pqMod}{pqMod}
\DeclareMathOperator{\gpqMod}{gpqMod}

\DeclareMathOperator{\Com}{Com}

\DeclareMathOperator{\loops}{\normalfont C}
\DeclareMathOperator{\CC}{CC}

\DeclareMathOperator{\CFTd}{CFT^\partial}
\DeclareMathOperator{\CFT}{\widehat{CFT}}
\DeclareMathOperator{\HFT}{\widehat{HFT}}
\newcommand{\CFTminus}{\operatorname{CFT}^-}

\DeclareMathOperator{\m}{m}  
\DeclareMathOperator{\rr}{r}
\DeclareMathOperator{\mr}{mr}

\DeclareMathOperator{\HFL}{\widehat{HFL}}
\DeclareMathOperator{\CFL}{\widehat{CFL}}
\newcommand{\CFLminus}{\operatorname{CFL}^-}

\DeclareMathOperator{\BSD}{\widehat{BSD}}

\DeclareMathOperator{\BSAA}{\widehat{BSAA}}
\DeclareMathOperator{\BSAD}{\widehat{BSAD}}
\DeclareMathOperator{\BSDA}{\widehat{BSDA}}
\DeclareMathOperator{\BSDD}{\widehat{BSDD}}

\newcommand{\typeA}[2]{\relax\!\!~_{#1}#2}

\DeclareMathOperator{\HF}{\widehat{HF}}
\DeclareMathOperator{\LagrangianFH}{HF}
\DeclareMathOperator{\LagrangianFC}{CF}

\DeclareMathOperator{\SFC}{SFC}

\newcommand{\Ad}{\operatorname{\mathcal{A}}^\partial}
\newcommand{\Anull}{\operatorname{\mathcal{A}}^{\partial}_{31}}
\newcommand{\Apre}{\operatorname{\mathcal{A}}^\text{\normalfont pre}}
\newcommand{\Rpren}{\operatorname{R}^\text{pre}_n}
\newcommand{\Rn}{\operatorname{R}_n}
\newcommand{\Aminus}{\operatorname{\mathcal{A}}^-_n}
\newcommand{\Id}{\operatorname{\mathcal{I}}^\partial}
\newcommand{\Atw}{\operatorname{\mathcal{A}}^{\infty}}

\DeclareMathOperator{\id}{id} 

\DeclareMathOperator{\Sym}{Sym}

\DeclareMathOperator{\GL}{GL}
\DeclareMathOperator{\im}{im}

\newcommand{\A}{\boldsymbol{\alpha}}
\newcommand{\B}{\boldsymbol{\beta}}
\newcommand{\Ac}{\boldsymbol{\alpha}^c}
\newcommand{\Aa}{\operatorname{\boldsymbol{\alpha}}^a}
\newcommand{\x}{\boldsymbol{x}}
\newcommand{\y}{\boldsymbol{y}}
\newcommand{\z}{\boldsymbol{z}}

\makeatletter
\def\Mydot{\pst@object{Mydot}}
\def\Mydot@i(#1,#2,#3,#4,#5){%
	\begin@ClosedObj%
	\rput(#2,#3){
		\psdot[linecolor=#4, dotsize=5pt](0,0)
		\uput{0.3}[#1](0,0){\footnotesize #5}
	}
	\end@ClosedObj%
}
\makeatother
\newlength{\stringwidth}
\newlength{\stringwhite}
\newcommand{\qedsymbol}{$\square$}
\newcommand{\Red}[1]{{\textcolor{red}{#1}}} 
\newcommand{\nmathphantom}[1]{\settowidth{\dimen0}{$#1$}\hspace*{-\dimen0}}
\newcommand{\nphantom}[1]{\settowidth{\dimen0}{#1}\hspace*{-\dimen0}}

\title{Peculiar modules for 4-ended tangles}

\author{Claudius Zibrowius}

\classno{57M27 (primary), 57R58 (secondary)}

\simpleequations 

\makeatletter
\gdef\@journal{}
\makeatother

\begin{document}
\psset{arrowsize=3pt 2}
\maketitle

\begin{abstract}
	With a 4-ended tangle $T$, we associate a Heegaard Floer invariant $\CFTd(T)$, the peculiar module of $T$. Based on Zarev's bordered sutured Heegaard Floer theory~\cite{ZarevThesis}, we prove a glueing formula for this invariant which recovers link Floer homology $\HFL$. 
	Moreover, we classify peculiar modules in terms of immersed curves on the 4-punctured sphere. In fact, based on an algorithm of Hanselman, Rasmussen and Watson~\cite{HRW}, we prove general classification results for the category of curved complexes over a marked surface with arc system. This allows us to reinterpret the glueing formula for peculiar modules in terms of Lagrangian intersection Floer theory on the 4-punctured sphere.
	
	We then study some applications: firstly, we show that peculiar modules detect rational tangles. Secondly, we give short proofs of various skein exact triangles. Thirdly, we compute the peculiar modules of the 2-stranded pretzel tangles $T_{2n,-(2m+1)}$ for $n,m>0$ using nice diagrams. We then observe that these peculiar modules enjoy certain symmetries which imply that mutation of the tangles $T_{2n,-(2m+1)}$ preserves $\delta$-graded, and for some orientations even bigraded link Floer homology. 
\end{abstract}

%
%

\input{sections/Intro}

\input{sections/AlgebraicStructures.tex}
\input{sections/CFTd}
\input{sections/Pairing}

\input{sections/Classification}

\input{sections/PeculiarModulesAsCurves}
\input{sections/Applications}

\newcommand*{\arxiv}[1]{(\href{http://arxiv.org/abs/#1}{arXiv:#1})}
\bibliographystyle{abbrv}
\bibliography{PeculiarModules}

\affiliationone{
  Claudius Zibrowius\\
  Department of Mathematics\\
  The University of British Columbia\\
  Room 121, 1984 Mathematics Road\\
  Vancouver, BC\\
  Canada V6T 1Z2
  \email{claudius.zibrowius\textcolor{white}{.antispam}\nphantom{.antispam}@\textcolor{white}{antispam-}\nphantom{antispam-}posteo.net}}
\end{document}

%% file: sections/Intro.tex
\section*{Introduction}\label{sec:intro}\addcontentsline{toc}{section}{Introduction}

\subsection{Peculiar modules}
Let $L$ be an oriented link in the 3-sphere~$S^3$. Consider an embedded closed 3-ball~$B^3\subset S^3$ whose boundary intersects~$L$ transversely. Then, modulo a parametrization of the boundary $\partial B^3$, the embedding $L\cap B^3\hookrightarrow B^3$ is essentially what we call an oriented tangle. In \cite{HDsForTangles}, I introduced a set of Alexander polynomials $\nabla_T^s$ and a Heegaard Floer theory $\HFT(T)$ for such tangles~$T$. They should be regarded as generalisations of the classical multivariate Alexander polynomial \cite{Alexander} and Ozsváth and Szabó's and J.\,Rasmussen's knot and link Floer homology \cite{OSHFK,Jake,OSHFL}, respectively. Indeed, both tangle invariants have similar properties to their corresponding link invariants. In particular, the graded Euler characteristic of $\HFT(T)$ recovers the polynomial invariants $\nabla_T^s$. Moreover, the polynomials $\nabla_T^s$ satisfy a simple glueing theorem which allows one to prove results about the classical multivariate Alexander polynomial of links, such as invariance under Conway mutation~\cite[Corollary~3.6]{HDsForTangles}. 
Unfortunately, we do not have a similar glueing theorem for the categorified invariants $\HFT(T)$.

The main objective of this paper is to resolve this problem in the case of 4-ended tangles, ie tangles with four ends on $\partial B^3$, see Figure~\ref{fig:INTRO2m3pt}. For this purpose, we upgrade the tangle Floer homology $\HFT(T)$ of 4-ended tangles $T$ to an invariant which we call the peculiar module of $T$ and denote by $\CFTd(T)$. This is done by adding some more structure maps. 
In fact, we construct an even more general invariant $\CFTminus(T):=\CFTminus(T,M)$, a generalised peculiar module, for tangles $T$ in homology 3-balls $M$ with spherical boundary.
As algebraic objects, both generalized and ordinary peculiar modules are curved type~D structures over certain algebras, the generalized and ordinary peculiar algebras $\Aminus$ and $\Ad$, respectively.
The algebra~$\Ad$ is a quotient of $\Aminus$, obtained by setting certain variables equal to 0. Similarly, $\CFTd(T)$ can be recovered from $\CFTminus(T)$, just as the hat version $\CFL$ of link Floer homology can be recovered from its $-$-version $\CFLminus$.

Both generalised and ordinary peculiar modules are Heegaard Floer type invariants and, as such, rely on some choice of Heegaard diagram. So the first goal is to prove that the invariants are independent of this choice. However, this follows essentially from multi-pointed Heegaard Floer theory. 

\begin{theorem}[(\ref{thm:PecMod})]\label{thm:INTROPecMod}
	Given a 4-ended oriented tangle \(T\) in a homology 3-ball with spherical boundary, the bigraded chain homotopy types of \(\CFTminus(T)\) and \(\CFTd(T)\) are invariants of~\(T\).
\end{theorem}

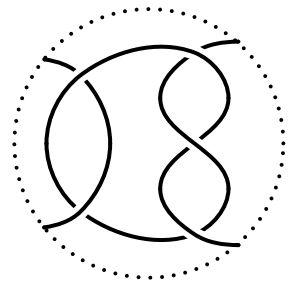
\begin{wrapfigure}{r}{0.3333\textwidth}
	\centering
	\psset{unit=0.6}
	\begin{pspicture}[showgrid=false](-4.2,-3.1)(2.2,3.1)
	\psset{linewidth=\stringwidth}
	\psecurve{c-c}(-2.5,1.5)(0,2)(0.75,1)(-0.75,-1)(0,-2)(0.97,-2.24)(2,-2)
	\psecurve{c-c}(2,2)(0.97,2.24)(0,2)(-0.75,1)(0.75,-1)(0,-2)(-2.5,-1.5)(-3.25,0)(-2.5,1.5)(0,2)(0.75,1)
	\psecurve{c-c}(-6,1.5)(-3.3,1.85)(-2.5,1.5)(-1.85,0)(-2.5,-1.5)(-3.3,-1.85)(-6,-1.5)
	
	\psecurve[linecolor=white,linewidth=\stringwhite](0.75,-1)(0,-2)(-2.5,-1.5)(-3.25,0)(-2.5,1.5)
	\psecurve{c-c}(0.75,-1)(0,-2)(-2.5,-1.5)(-3.25,0)(-2.5,1.5)
	
	\psecurve[linecolor=white,linewidth=\stringwhite](0.75,1)(-0.75,-1)(0,-2)(0.97,-2.24)(2,-2)
	\psecurve[linecolor=white,linewidth=\stringwhite](0,2)(-0.75,1)(0.75,-1)(0,-2)
	\psecurve[linecolor=white,linewidth=\stringwhite](-2.5,-1.5)(-3.25,0)(-2.5,1.5)(0,2)(0.75,1)(-0.75,-1)
	
	\psecurve{c-c}(0.75,1)(-0.75,-1)(0,-2)(0.97,-2.24)(2,-2)
	\psecurve{c-c}(0,2)(-0.75,1)(0.75,-1)(0,-2)
	\psecurve{c-c}(-2.5,-1.5)(-3.25,0)(-2.5,1.5)(0,2)(0.75,1)(-0.75,-1)
	
	\psecurve[linecolor=white,linewidth=\stringwhite](-2.5,1.5)(-1.85,0)(-2.5,-1.5)(-3.3,-1.85)(-6,-1.5)
	\psecurve{c-c}(-2.5,1.5)(-1.85,0)(-2.5,-1.5)(-3.3,-1.85)(-6,-1.5)
	
	\pscircle[linestyle=dotted](-1,0){3}
	
	\end{pspicture}
	\caption{A diagram of a 4-ended tangle; in this case, the $(2,-3)$-pretzel tangle $T_{2,-3}$
		.}\label{fig:INTRO2m3pt}
	\vspace{-35pt}
\end{wrapfigure}

We do not have a glueing theorem for the generalised peculiar modules \(\CFTminus(T)\), except that one can recover $\CFLminus$ of certain closures of the tangle $T$ from \(\CFTminus(T)\) (see Remark~\ref{rem:LazyClosing}). Nonetheless, we do have a glueing formula for peculiar modules \(\CFTd(T)\). Its proof relies on Zarev's Glueing Theorem for his bordered sutured invariants~\cite{ZarevThesis} and an identification of some structure maps of certain bordered sutured invariants for tangles and peculiar modules. The precise statement of our Glueing Theorem uses the $\boxtimes$-tensor product between type~A and type~D structures familiar from bordered Heegaard Floer homology; for details, see Definition~\ref{def:PairingTypeDandA}.

\begin{wrapfigure}{r}{0.3333\textwidth}
	\centering
	\psset{unit=0.2}
	\vspace*{20pt}
	\begin{pspicture}[showgrid=false](-11,-7)(11,7)
	\rput{-45}(0,0){
		\pscircle[linecolor=lightgray](4,4){2.5}
		\pscircle[linecolor=lightgray](-4,-4){2.5}
		
		\psline[linecolor=white,linewidth=\stringwhite](1,-4)(-2,-4)
		\psline[linewidth=\stringwidth,linecolor=gray](1,-4)(-2,-4)
		\psline[linecolor=white,linewidth=\stringwhite](-4,1)(-4,-2)
		\psline[linewidth=\stringwidth,linecolor=gray](-4,1)(-4,-2)
		\psline[linecolor=white,linewidth=\stringwhite](-6,-4)(-9,-4)
		\psline[linewidth=\stringwidth,linecolor=gray](-6,-4)(-9,-4)
		\psline[linecolor=white,linewidth=\stringwhite](-4,-6)(-4,-9)
		\psline[linewidth=\stringwidth,linecolor=gray](-4,-6)(-4,-9)
		
		\psline[linecolor=white,linewidth=\stringwhite](-1,4)(2,4)
		\psline[linewidth=\stringwidth,linecolor=gray](-1,4)(2,4)
		\psline[linecolor=white,linewidth=\stringwhite](4,-1)(4,2)
		\psline[linewidth=\stringwidth,linecolor=gray](4,-1)(4,2)
		\psline[linecolor=white,linewidth=\stringwhite](6,4)(9,4)
		\psline[linewidth=\stringwidth,linecolor=gray](6,4)(9,4)
		\psline[linecolor=white,linewidth=\stringwhite](4,6)(4,9)
		\psline[linewidth=\stringwidth,linecolor=gray](4,6)(4,9)
		
		\pscircle[linestyle=dotted,linewidth=\stringwidth](4,4){5}
		\pscircle[linestyle=dotted,linewidth=\stringwidth](-4,-4){5}
		
		\psecurve[linewidth=\stringwidth]{C-C}(8,-3)(-1,4)(-10,-3)(-9,-4)(-8,-3)
		\psecurve[linewidth=\stringwidth]{C-C}(-3,8)(4,-1)(-3,-10)(-4,-9)(-3,-8)
		
		\psecurve[linewidth=\stringwhite,linecolor=white](-8,3)(1,-4)(10,3)(9,4)(8,3)
		\psecurve[linewidth=\stringwhite,linecolor=white](3,-8)(-4,1)(3,10)(4,9)(3,8)
		
		\psecurve[linewidth=\stringwidth]{C-C}(-8,3)(1,-4)(10,3)(9,4)(8,3)
		\psecurve[linewidth=\stringwidth]{C-C}(3,-8)(-4,1)(3,10)(4,9)(3,8)
		
		\rput{45}(-4,-4){$T_1$}
		\rput{45}(4,4){$T_2$}
		
%
	}
	\end{pspicture}
	\caption{A link obtained from two 4-ended tangles.}\label{fig:INTROglueing2tangles}
	\vspace*{10pt}
\end{wrapfigure}
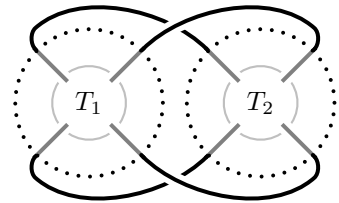
\myfixwrapfig

\begin{theorem}[(\ref{thm:CFTdGeneralGlueing}, \ref{prop:reversedmirror})]\label{thm:INTROCFTdGeneralGlueing}
	Let \(T_1\) and \(T_2\) be two 4-ended tangles and \(L\) the link obtained by glueing them together as illustrated in Figure~\ref{fig:INTROglueing2tangles}. Then the link Floer homology \(\HFL(L)\) can be computed from \(\CFTd(T_1)\) and \(\CFTd(T_2)\). 
	More precisely, there exists a bounded, strictly unital type~AA structure \(\mathcal{P}\) such that
	\[\CFL(L)\otimes V^{i}\cong\rr(\CFTd(T_1))\boxtimes\,\mathcal{P}\boxtimes\,\CFTd(T_2),\]
where \(V\) is some 2-dimensional vector space, \(i\in\{0,1\}\) and $\rr(\cdot)$ is the operation on peculiar modules which reverses Alexander gradings (see Definition~\ref{def:reversedmirror}). 
\end{theorem}	

\subsection{Classification of peculiar modules}

In sections~\ref{sec:classification} and~\ref{sec:glueingrevisited}, we classify peculiar modules in terms of immersed curves on the 4-punctured sphere. 

\begin{definition}[(\ref{def:curves})]\label{def:INTROcurves}
	A \textbf{loop} on the 4-punctured sphere $S=S^2\smallsetminus 4D^2$ is a pair $(\gamma, X)$, where $\gamma$ is an immersion of an oriented circle into $S$ representing a non-trivial primitive element of $\pi_1(S)$ and $X\in \GL_n(\mathbb{F}_2)$ for some integer $n$. For such loops $(\gamma, X)$, we call $X$ the \textbf{local system} of the loop. A \textbf{collection of loops} is a set of loops \(\{(\gamma_i,X_i)\}_{i\in I}\) such that the immersed curves $\gamma_i$ are pairwise non-homotopic.
\end{definition}

\begin{theorem}[(\ref{exa:pqModSpecialCaseofCC}, \ref{cor:SplittingCatsForEquivalenceComplexes}, \ref{thm:PairingMorLagrangianFH}, \ref{thm:CompleteClassification}, \ref{thm:ImmersedCurveInvariants})]\label{thm:classificationPecMod}
		With every peculiar module \((C,\partial)\), we can associate a collection of loops \(L(C,\partial)=\{(\gamma_i,X_i)\}_{i\in I}\) such that if \((C',\partial')\) is another peculiar module with \(L(C',\partial')=\{(\gamma'_{i'},X'_{i'})\}_{i'\in I'}\), \((C,\partial)\) and \((C',\partial')\) are homotopic iff there is a bijection \(\iota\co I\rightarrow I'\) such that \(\gamma_i\)~is homotopic to~\(\gamma'_{\iota(i)}\) and \(X_i\)~is similar to~\(X'_{\iota(i)}\) for all $i\in I$.
		Moreover, the homology of the space of morphisms between two peculiar modules is chain homotopic to the Lagrangian intersection Floer theory between their associated collections of immersed curves:
		\[H_\ast(\Mor((C,\partial),(C',\partial')))\cong\LagrangianFH(L(C,\partial),L(C',\partial')).\]
\end{theorem}

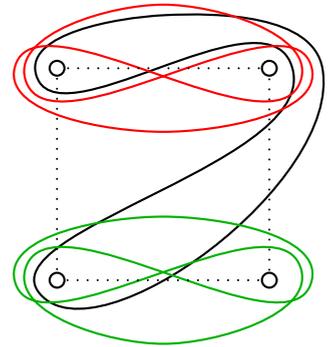
\begin{wrapfigure}{r}{0.3333\textwidth}\centering
	\psset{unit=1.4}
	\begin{pspicture}(-1.6,-1.7)(1.6,1.7)

	\psecurve(-1.2,0.95)(0,1.5)(1.5,1)(-1.16,-1.13)(1.2,1.05)(-1.2,0.95)(0,1.5)(1.5,1)

	\psecurve[linecolor=red](1.3,0.9)(0,1.6)(-1.3,0.9)(1.4,1)(0,0.4)(-1.4,1)(1.3,0.9)(0,1.6)(-1.3,0.9)
	
	\psrotate(0,0){180}{
		\psecurve[linecolor=darkgreen](1.3,0.9)(0,1.6)(-1.3,0.9)(1.4,1)(0,0.4)(-1.4,1)(1.3,0.9)(0,1.6)(-1.3,0.9)
	}
	
	\psline[linestyle=dotted](1,1)(1,-1)
	\psline[linestyle=dotted](1,-1)(-1,-1)
	\psline[linestyle=dotted](-1,-1)(-1,1)
	\psline[linestyle=dotted](-1,1)(1,1)
	
	\pscircle[fillstyle=solid, fillcolor=white](1,1){0.08}
	\pscircle[fillstyle=solid, fillcolor=white](-1,1){0.08}
	\pscircle[fillstyle=solid, fillcolor=white](1,-1){0.08}
	\pscircle[fillstyle=solid, fillcolor=white](-1,-1){0.08}
	
	\end{pspicture}
	\caption{The three loops of $L_{T_{2,-3}}$ (with the unique 1-dimensional local systems) on the 4-punctured sphere for the $(2,-3)$-pretzel tangle $T_{2,-3}$ from Figure~\ref{fig:INTRO2m3pt}. }\label{fig:INTROmutationexamplefinalresult}
	\vspace*{-30pt}
\end{wrapfigure}
\myfixwrapfig

\begin{definition}[(\ref{def:ImmersedCurveInvariantspqMod})]
	In particular, with any 4-ended tangle~\(T\) in a homology 3-ball $M$ with spherical boundary, we can associate a collection of loops, denoted by $L_T:=L_{T,M}$, which is a tangle invariant up to homotopy of the underlying curves and similarity of the local systems.
\end{definition}

\begin{question}\label{que:GeometricMeaningOfTheNumberOfCurves}
	The number of curves in $L_T$ is obviously a tangle invariant. What is its geometric meaning?
\end{question}

The proof of Theorem~\ref{thm:classificationPecMod} is based on an algorithm due to Hanselman, Rasmussen and Watson~\cite{HRW} which they use to classify bordered Heegaard Floer invariants for 3-manifolds with torus boundary. There are striking similarities between their bordered Heegaard Floer invariants and peculiar modules, and it would be interesting to see if there exists a closer connection between them apart from their formal properties.

Theorem~\ref{thm:classificationPecMod} allows us to restate the Glueing Theorem as follows.

\begin{corollary}[(\ref{thm:CFTdGlueingAsMorphism})]\label{cor:INTROCFTdGlueingAsMorphism}
	With the same notation as in Theorems~\ref{thm:INTROCFTdGeneralGlueing} and~\ref{thm:classificationPecMod}, 
	$$\HFL(L)\otimes V^{i}\cong H_\ast(\Mor(\CFTd(\mr(T_1)),\CFTd(T_2)))\cong\LagrangianFH(L_{\mr(T_1)},L_{T_2}),$$
	where \(\mr(T_1)\) denotes the reversed mirror of \(T_1\) (see Definition~\ref{def:reversedmirror}). 
\end{corollary}

Section~\ref{sec:glueingrevisited} contains a variety of examples illustrating this version of the Glueing Theorem. 

We actually prove Theorem~\ref{thm:classificationPecMod} for all categories of curved complexes over arbitrary marked surfaces, see section~\ref{sec:classification}. As a consequence, we can show that by passing to certain quotients of the peculiar algebra $\Ad$, we do not lose information. This rather abstract observation has the following very practical consequence.

\begin{theorem}[(\ref{cor:PeculiarModulesFromNiceDiagrams})]\label{thm:INTROPeculiarModulesFromNiceDiagrams}
	Peculiar modules for 4-ended tangles can be computed combinatorially.
\end{theorem}

We illustrate this result in section~\ref{subsec:pretzels} by computing $\CFTd(T)$ for an infinite family of 2-stranded pretzel tangles. Moreover, the general algorithm is implemented in the Mathematica package~\cite{PQM.m}, which can be used to compute the peculiar module and the corresponding immersed curves of any 4-ended tangle.

\begin{definition}\label{def:INTROlooptype}
	Inspired by \cite[Definition~3.2]{HanselmanWatson}, we call a tangle $T$ \textbf{loop-type} if all local systems in $L_T$ are similar to permutation matrices. 
\end{definition}

All tangles for which I have so far computed the invariants are loop-type. This provokes the following question, whose corresponding counterpart for 3-manifolds with torus boundary is open as well~\cite[section~2.4]{HRW}.

\begin{question}\label{que:LocalSystemsForTangles}
	Are all tangles loop-type?
\end{question}

\subsection{Applications}
Peculiar modules detect rational tangles. In fact, this is true for any invariant of 4-ended tangles for which there exists a glueing theorem recovering link Floer homology. This is an easy consequence of unlink detection of link Floer homology and should be regarded as its corresponding analogue for invariants of 4-ended tangles. However, the description of peculiar modules of rational tangles is particularly simple, and hence so is rational tangle detection for our invariants:
\begin{theorem}[(\ref{thm:CFTdDetectsRatTan})]\label{thm:INTROCFTdDetectsRatTan}
	A 4-ended tangle \(T\) in the 3-ball is rational iff \(L_T\) is a single \textit{embedded} loop with the unique 1-dimensional local system.
\end{theorem}

As another application, we reprove a result originally due to Manolescu~\cite{Manolescu}, namely the existence of an unoriented skein exact sequence, see Theorem~\ref{thm:ResolutionExactTriangle}. Similarly, we obtain the following slight generalisation of Ozsváth and Szabó's oriented skein exact sequence \cite{OSHFK}.

\begin{wrapfigure}{r}{0.3333\textwidth}
	\centering
	\vspace*{10pt}
	\psset{unit=0.35}
	\begin{subfigure}[b]{0.116\textwidth}\centering
		$n\left\{\raisebox{-1cm}{
			\begin{pspicture}(-1.1,-3.1)(1.1,3.1)
			\psset{linewidth=\stringwidth}
			\psecurve(1,5)(-1,3)(1,1)(-1,-1)
			\psecurve[linecolor=white,linewidth=\stringwhite](-1,5)(1,3)(-1,1)(1,-1)
			\psecurve(-1,5)(1,3)(-1,1)(1,-1)
			\psline[linestyle=dotted,dotsep=0.4](0,-0.6)(0,0.6)
			\psecurve(-1,-5)(1,-3)(-1,-1)(1,1)
			\psecurve[linecolor=white,linewidth=\stringwhite](1,-5)(-1,-3)(1,-1)(-1,1)
			\psecurve(1,-5)(-1,-3)(1,-1)(-1,1)
			\end{pspicture}}\right.
		$
		\caption{$T_{n}$}\label{fig:INTROOrientedSkeinRelationTn}
	\end{subfigure}
	\begin{subfigure}[b]{0.116\textwidth}\centering
		$n\left\{\raisebox{-1cm}{
			\begin{pspicture}(-1.1,-3.1)(1.1,3.1)
			\psset{linewidth=\stringwidth}
			\psecurve(-1,5)(1,3)(-1,1)(1,-1)
			\psecurve[linecolor=white,linewidth=\stringwhite](1,5)(-1,3)(1,1)(-1,-1)
			\psecurve(1,5)(-1,3)(1,1)(-1,-1)
			\psline[linestyle=dotted,dotsep=0.4](0,-0.6)(0,0.6)
			\psecurve(1,-5)(-1,-3)(1,-1)(-1,1)
			\psecurve[linecolor=white,linewidth=\stringwhite](-1,-5)(1,-3)(-1,-1)(1,1)
			\psecurve(-1,-5)(1,-3)(-1,-1)(1,1)
			\end{pspicture}}\right.
		$
		\caption{$T_{-n}$}\label{fig:INTROOrientedSkeinRelationTmn}
	\end{subfigure}
	\begin{subfigure}[b]{0.085\textwidth}\centering
		$\left.\raisebox{-1cm}{
			\begin{pspicture}(-1.1,-3.1)(1.1,3.1)
			\psset{linewidth=\stringwidth}
			\psecurve(-2,6)(-1,3)(-1,-3)(-2,-6)
			\psecurve(2,6)(1,3)(1,-3)(2,-6)
			\end{pspicture}}\right.
		$
		\caption{$T_0$}\label{fig:INTROOrientedSkeinRelationT0}
	\end{subfigure}
	\vspace*{-20pt}
	\caption{Basic tangles.}\label{fig:INTROOrientedSkeinRelation}
	\vspace*{-50pt}
\end{wrapfigure}
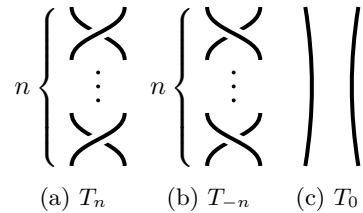
\myfixwrapfig

\begin{theorem}[(\ref{thm:nTwistSkeinRelation}, see also \ref{rem:nTwistSkeinRelation})]\label{thm:INTROnTwistSkeinRelation}
	Let \(T_n\) be the positive \(n\)-twist tangle, \(T_{-n}\) the negative \(n\)-twist tangle and \(T_0\) the trivial tangle, see Figure~\ref{fig:INTROOrientedSkeinRelation}. Then there is an exact triangle
	$$\begin{tikzcd}[row sep=0.9cm, column sep=-0.5cm]
	\CFTd(T_{-n})
	\arrow{rr}
	& &
	\CFTd(T_{n})
	\arrow{dl}
	\\
	&
	\CFTd(T_{0})\otimes V
	\arrow{lu}
	\end{tikzcd}$$
	where \(V\) is a 2-dimensional vector space. If the tangles are oriented and coloured consistently, one obtains (bi)graded versions of this triangle. 
	Furthermore, it gives rise to an exact triangle relating the (appropriately stabilised) link Floer homologies of links that differ in these three tangles.
\end{theorem}

\begin{wrapfigure}{r}{0.3333\textwidth}
	\centering\vspace*{-5pt}
	{\psset{unit=0.2}
		\begin{pspicture}(-12,-5.5)(12,5.5)
		\rput{-45}(-1,0){
			\pscircle[linecolor=gray](-4,-4){2.5}
			
			\psline[linecolor=white,linewidth=\stringwhite](1,-4)(-2,-4)
			\psline[linewidth=\stringwidth,linecolor=black](1,-4)(-2,-4)
			\psline[linecolor=white,linewidth=\stringwhite](-4,1)(-4,-2)
			\psline[linewidth=\stringwidth,linecolor=black](-4,1)(-4,-2)
			\psline[linecolor=white,linewidth=\stringwhite](-6,-4)(-9,-4)
			\psline[linewidth=\stringwidth,linecolor=black](-6,-4)(-9,-4)
			\psline[linecolor=white,linewidth=\stringwhite](-4,-6)(-4,-9)
			\psline[linewidth=\stringwidth,linecolor=black](-4,-6)(-4,-9)
			
			\pscircle[linestyle=dotted,linewidth=\stringwidth](-4,-4){5}
					
			\rput{45}(-4,-4){$R$}
		}
		
		\rput(0,0){$\rightarrow$}	
		\rput{-45}(1,0){
			\pscircle[linecolor=gray](4,4){2.5}
			
			\psline[linecolor=white,linewidth=\stringwhite](-1,4)(2,4)
			\psline[linewidth=\stringwidth,linecolor=black](-1,4)(2,4)
			\psline[linecolor=white,linewidth=\stringwhite](4,-1)(4,2)
			\psline[linewidth=\stringwidth,linecolor=black](4,-1)(4,2)
			\psline[linecolor=white,linewidth=\stringwhite](6,4)(9,4)
			\psline[linewidth=\stringwidth,linecolor=black](6,4)(9,4)
			\psline[linecolor=white,linewidth=\stringwhite](4,6)(4,9)
			\psline[linewidth=\stringwidth,linecolor=black](4,6)(4,9)
			
			\pscircle[linestyle=dotted,linewidth=\stringwidth](4,4){5}
			
			\rput{-135}(4,4){$R$}
		}
		\end{pspicture}
		\caption{Conway mutation.}\label{fig:MutationIntro}}\bigskip
\end{wrapfigure}
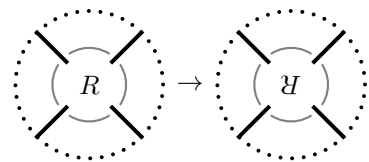
\myfixwrapfigdouble

\begin{definition}[(Conway mutation)]\label{def:INTROmutation}
	Given a link $L$ in $S^3$, let $L'$ be the link obtained by cutting out a tangle diagram~$R$ with four ends from a diagram of~$L$ and glueing it back in after a half-rotation, see Figure~\ref{fig:MutationIntro} for an illustration. We say $L'$ is a \textbf{Conway mutant} of~$L$ and we call~$R$ the \textbf{mutating tangle} in this mutation. If $L$ is oriented, we choose an orientation of $L'$ that agrees with the one for~$L$ outside of~$R$. If this means that we need to reverse the orientation of the two open components of~$R$, then we also reverse the orientation of all other components of~$R$ during the mutation; otherwise we do not change any orientation. 
\end{definition}

Ozsváth and Szabó showed in~\cite{OSmutation} that knot and link Floer homology is not invariant under mutation in general. But in~\cite[Conjecture~1.5]{BaldwinLevine}, Baldwin and Levine conjectured the following.

\begin{conjecture}\label{conj:MutInvHFL}
Let \(L\) be a link and let \(L'\) be obtained from \(L\) by Conway mutation. Then \(\HFL(L)\) and \(\HFL(L')\) agree after collapsing the bigrading to a single \(\mathbb{Z}\)-grading, known as the \(\delta\)-grading. More colloquially, \(\delta\)-graded link Floer homology is mutation invariant.
\end{conjecture}

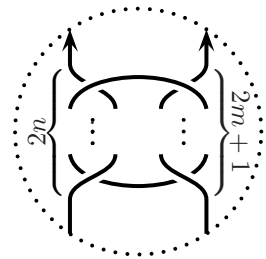
\begin{wrapfigure}{r}{0.3333\textwidth}
	\centering
	\psset{unit=0.3}
	\vspace*{-5pt}
	\begin{pspicture}(-5.5,-5.5)(5.5,5.5)
	\psset{linewidth=\stringwidth}
	\pscircle[linestyle=dotted](0,0){5.5}
	
	\rput(0,0){$\left\{\textcolor{white}{\rule[-0.8cm]{1.9cm}{1.8cm}}\right\}$}
	\rput{-90}(4.45,0){$2m+1$}
	\rput{90}(-4.45,0){$2n$}
	
	\psecurve(-1,1)(-3,-1)(0,-2.4)(3,-1)(1,1)
	
	\pscustom{
		\psline{<-}(-3,4.5)(-3,3)
		\psecurve(-1,5)(-3,3)(-1,1)(-3,-1)
	}
	\rput(-2,0.3){$\vdots$}
	\psline[linestyle=dotted,dotsep=0.4](-2,-0.6)(-2,0.6)
	
	\psecurve[linecolor=white,linewidth=\stringwhite](-1,-5)(-3,-3)(-1,-1)(-3,1)
	\pscustom{
		\psline(-3,-4.5)(-3,-3)
		\psecurve(-1,-5)(-3,-3)(-1,-1)(-3,1)
	}
	
	\pscustom{
		\psline{<-}(3,4.5)(3,3)
		\psecurve(1,5)(3,3)(1,1)(3,-1)
	}
	\psline[linestyle=dotted,dotsep=0.4](2,-0.6)(2,0.6)
	\psecurve[linecolor=white,linewidth=\stringwhite](1,-5)(3,-3)(1,-1)(3,1)
	\pscustom{
		\psline(3,-4.5)(3,-3)
		\psecurve(1,-5)(3,-3)(1,-1)(3,1)
	}
	
	\psecurve[linecolor=white,linewidth=\stringwhite](-1,-1)(-3,1)(0,2.4)(3,1)(1,-1)
	\psecurve(-1,-1)(-3,1)(0,2.4)(3,1)(1,-1)
	
%
	\end{pspicture}
	\caption{An infinite family of pretzel tangles for $n,m>0$.}\label{fig:INTROpretzeltangle2nm2mp1}
	\vspace*{-45pt}
\end{wrapfigure}
In this paper, we prove a slightly stronger version of this conjecture for an infinite family of mutating tangles.

\begin{theorem}[(\ref{cor:pretzeltangleMutation})]\label{thm:INTROpretzeltangleMutation}
	Mutation of \((2n,-(2m+1))\)-pretzel tangles for \(n,m>0\), oriented as in Figure~\ref{fig:INTROpretzeltangle2nm2mp1}, preserves bigraded link Floer homology, after identifying the Alexander gradings of the two open strands. If we reverse the orientation of one of the two strands, mutation of these tangles preserves \(\delta\)-graded link Floer homology.
\end{theorem}
This generalises an earlier result from my thesis~\cite{MyThesis} for the $(2,-3)$-pretzel tangle. The result again simply follows from an observation that the peculiar invariants for the mutating tangles have a certain symmetry. However, the calculation of the invariants for general \(n,m>0\) is more involved and relies on Theorem~\ref{thm:INTROPeculiarModulesFromNiceDiagrams}. 

\subsection{Similar work by other people}
It is interesting to compare the ideas described in this paper to the combinatorial tangle Floer theory by Petkova and Vértesi \cite{cHFT} and the algebraic tangle homology theory by Ozsváth and Szabó \cite{OSKauffmanStates1,OSKauffmanStates2} as well as their corresponding decategorifications in terms of the representation theory of $\mathcal{U}_q(\mathfrak{gl}(1\vert 1))$ \cite{DecatCTFH,ManionDecat}. In fact, the definition of our generalised peculiar modules $\CFTminus$ is primarily inspired by the invariants from~\cite{OSKauffmanStates2}, see Remark~\ref{rem:ComparisonOS}.
In~\cite{EftekharyAlishahiTangles}, Eftekhary and Alishahi define a Heegaard Floer theory for tangles using a suitable generalisation of sutured Floer homology~\cite{EftekharyAlishahiSFT}. They study cobordism maps between their tangle invariants, but they do not discuss glueing properties. It would be interesting to see how these two approaches can be merged.  
Finally, I also want to mention some impressive work of Lambert-Cole \cite{LambertCole1,LambertCole2}, where he confirms Conjecture~\ref{conj:MutInvHFL} for various families of mutant pairs (different from the one in Theorem~\ref{thm:INTROpretzeltangleMutation}), using entirely different techniques.

\begin{acknowledgements}\label{ackref}
	This paper has grown out of the final chapter of my PhD thesis~\cite{MyThesis}. I would therefore like to take the opportunity to thank my PhD supervisor Jake Rasmussen for his generous support. I consider myself very fortunate to have been his student.
	
	My PhD was funded by an EPSRC scholarship covering tuition fees and a DPMMS grant for maintenance, for which I thank the then Head of Department Martin Hyland. Some parts of this paper were written during my stay at the Isaac Newton Institute during the programme \textit{Homology Theories in Low Dimensions} (EPSRC grant number EP/K032208/1). The paper was completed during my time as a CIRGET postdoctoral fellow at the Université de Sherbrooke. 
	
	I thank my PhD examiners Ivan Smith and András Juhász for many valuable comments on and corrections to my thesis. I also thank Jonathan Hanselman, Robert Lipshitz, Andy Manion, Allison Moore, Ina Petkova, Vera Vértesi and Marcus Zibrowius for helpful conversations. I am immensely grateful to the anonymous referees who provided many valuable and detailed comments on and corrections to an earlier version of this paper. My special thanks go to Liam Watson for his continuing interest in my work.
\end{acknowledgements}

%% file: sections/AlgebraicStructures.tex
\section{Preliminaries: Algebraic structures from dg categories}\label{sec:AlgStructFromGDCats}
In this paper, we often work in categories of various algebraic structures, namely type~D and curved type~D structures, but also type~A structures and various bimodules. In all settings, we often want to simplify these structures by replacing them by homotopy equivalent ones. The main goal of this section is to develop some tools for dealing with this problem, namely the Cancellation Lemma (\ref{lem:AbstractCancellation}) and the Clean-Up Lemma (\ref{lem:AbstractCleanUp}). The former can be used to reduce the number of generators of an algebraic structure, essentially by doing Gaussian elimination as in~\cite[Lemma~3.2]{BarNatanBurgosSoto}, the latter for making the structure maps ``look nicer'', essentially by changing the basis. 

In the category of ordinary chain complexes, both tools will be familiar to the reader as easy exercises in linear algebra. So it might not be too surprising that they also work in quite general settings. We will spend the first part of this section explaining a general construction which turns any differential graded category into another such category in which the lemmas hold in some generality sufficient for our purposes, see Definitions~\ref{def:CatOfMatrices} and~\ref{def:CatOfComplexes}. Next, we show that the various different algebraic structures mentioned above arise naturally from this general construction. We also study its functoriality properties and interpret the $\boxtimes$-tensor product between type~A and type~D structures in this framework. Finally, we state and prove the Cancellation and Clean-Up Lemmas.

For simplicity, we only work over the field $\mathbb{F}_2=\mathbb{Z}/2$, so we do not need to keep track of signs. However, with the correct sign conventions, most (if not all) statements should also hold over fields of arbitrary characteristic.
\subsection{The general construction}

\begin{definition}
Let $\Com$ be the category of $\mathbb{Z}$-graded chain complexes over $\mathbb{F}_2$ and grading preserving chain maps between them. A \textbf{differential graded (dg) category}~$\mathcal{C}$ over~$\mathbb{F}_2$ is an enriched category over $\Com$. To spell this out more explicitly, the hom-objects are $\mathbb{Z}$-graded $\mathbb{F}_2$-vector spaces,
\[\Mor(A,B)=\bigoplus_{i\in\mathbb{Z}}\Mor_i(A,B)\]
endowed with differentials
\[\partial_i\co\Mor_i(A,B)\rightarrow \Mor_{i-1}(A,B),\]
ie vector space homomorphisms satisfying 
$\partial_{i-1}\partial_i=0$
and
\begin{equation}\label{eqn:CompatibleWithComposition}
\partial\circ m=m\circ(\partial\otimes \id+\id\otimes \partial),
\end{equation}
where
\[m\co \Mor_j(B,C)\otimes\Mor_i(A,B)\rightarrow\Mor_{i+j}(A,C)\]
denotes composition in $\mathcal{C}$, which is associative and unital. For more details on enriched categories, see for example \cite{cathtpy}. Note that the identity morphisms have degree zero and lie in the kernel of~$\partial$.
\end{definition}
\begin{definition}\label{def:UnderlyingOrdinaryCat}\cite[Definition~3.4.5]{cathtpy} 
Given an enriched category $\mathcal{C}$ over some monoidal category $\mathcal{V}$, the \textbf{underlying ordinary category} $\mathcal{C}_0$ of $\mathcal{C}$ has the same objects as $\mathcal{C}$ and its hom-sets are defined by 
\[\mathcal{C}_0(A,B):=\Mor_\mathcal{V}(1_\mathcal{V},\mathcal{C}(A,B)).\]
\end{definition}
\begin{example}\label{exa:UnderlyingOrdCat}
Let $\mathcal{C}$ be a dg category. The unit in $\Com$ is the complex $0\rightarrow\mathbb{F}_2\rightarrow0$, supported in homological degree 0, and the morphisms in $\Com$ are grading preserving. Hence, the hom-sets of $\mathcal{C}_0$ are those elements in the kernel of $\partial_0$.
Next, consider the enriched category $H_\ast(\mathcal{C})$ over the category of graded vector spaces and grading preserving morphisms between them, obtained from $\mathcal{C}$ by replacing the hom-objects by their homologies with respect to the differential $\partial$. By passing to the underlying ordinary category, we pick out the degree~0 morphisms in $H_\ast(\mathcal{C})$. Therefore, we denote this category by $H_0(\mathcal{C})$.
Since the hom-sets in $H_0(\mathcal{C})$ are just quotients of those in $\mathcal{C}_0$, we now get the usual notions of chain homotopies between morphisms and objects. The reason why we need to pass to the underlying category is that otherwise, two objects could be (chain) isomorphic through grading shifting morphisms. 
\end{example}
\begin{definition}\label{def:CatOfMatrices}
(cp.~\cite[section~6]{BarNatanKhT} and~\cite[section~I.3k]{Seidel})
Given a dg category~$\mathcal{C}$, we define another dg category $\Mat(\mathcal{C})$ as follows. Its objects are formal direct sums
\[\bigoplus_{i\in I}O_i[n_i],\]
where $I$ is some finite index set and $O_i[n_i]$ denotes the object $O_i\in\ob(\mathcal{C})$ with a formal grading shift by an integer $n_i$. Morphisms are given by
\[\Mor_n(\bigoplus_{i\in I}O_i[n_i],\bigoplus_{j\in J}O_j[n_j]):=\bigoplus_{(i,j)\in I\times J}\Mor_{n+n_i-n_j}(O_i,O_j).\]
Compositions and differentials in $\Mat(\mathcal{C})$ are induced by those in $\mathcal{C}$.
\end{definition}
\begin{definition}\label{def:CatOfComplexes}(cp.~\cite[section~I.3l]{Seidel})
Given a differential graded category $\mathcal{C}$, we define an auxiliary category $\Cxpre(\mathcal{C})$, \textbf{the category of pre-complexes}, which is an enriched category over the category of $\mathbb{Z}$-graded vector spaces and grading preserving morphisms between them. Its objects are pairs $(O,d_O)$, where $O\in\ob(\mathcal{C})$ and $d_O\in\Mor_{-1}(O,O)$. 
The hom-objects are the same as in $\mathcal{C}$,
\[\Mor((O,d_O),(O',d_{O'}))=\Mor(O,O'),\]
viewed as $\mathbb{Z}$-graded vector spaces. On these, we can define a map
\[D\co\Mor_{i}((O,d_O),(O',d_{O'}))\rightarrow \Mor_{i-1}((O,d_O),(O',d_{O'}))\]
by setting 
\[D(f):=d_{O'}\circ f+f \circ d_O+\partial(f).\]
We would like $D$ to be a differential in order to turn $\Cxpre(\mathcal{C})$ into a dg category. However, this only works in general if we restrict ourselves to a full subcategory of $\Cxpre(\mathcal{C})$. 
It is easy to check that $D$ is always compatible with multiplication in the sense of (\ref{eqn:CompatibleWithComposition}). So $D$ is a differential iff
\[D^2(f)=(d_{O'}^2+\partial(d_{O'}))\circ f
+f\circ (d_O^2+\partial(d_O))\]
vanishes. This is, of course, the case for the full subcategory $\Cx^{0}(\mathcal{C})$ of $\Cxpre(\mathcal{C})$ consisting of those objects $(O,d_O)$ for which
\begin{equation} 
d^2_O+\partial(d_O) \tag{$\ast$}\label{eqn:d2term}
\end{equation}
vanishes. However, in some situations, other conditions on (\ref{eqn:d2term}) also  work. For example, if we replace $\Com$ by the category of $\mathbb{Z}/2$-graded chain complexes, we can restrict to those objects $(O,d_O)$ for which (\ref{eqn:d2term}) is equal to the identity. Also, if the hom-objects are bimodules over an algebra~$\mathcal{A}$, we can ask (\ref{eqn:d2term}) to be equal to $a\cdot\id_O$ for a fixed central algebra element $a$ (of degree $-2$) which commutes with all morphisms~$f$. In both cases, $D$ will be a differential.
We call any such full subcategory a \textbf{category of complexes}, denoted by $\Cx^{\ast}(\mathcal{C})$, where $\ast\in\{0,1,a\}$ is the value of~(\ref{eqn:d2term}). By construction, $\Cx^{\ast}(\mathcal{C})$ is a dg category. 
\end{definition}

\begin{remark}\label{exa:GraphsForAlgebraicStructures}
As usual, we can associate a directed graph to a category, where objects correspond to vertices and arrows to  morphisms. In the same way, we can think of complexes in $\Cx^{\ast}(\Mat(\mathcal{C}))$ as graphs. 
\end{remark}
The point of the construction above is that we can interpret the categories of type~D, type~A, type~AA and curved type~D structures as instances of $\Cx^{\ast}(\Mat(\mathcal{C}))$ for suitable choices of relatively simple differential graded categories $\mathcal{C}$. But let us start with an even simpler example: ordinary chain complexes.\bigskip

\textbf{Note of warning.}
In the following examples, our definitions only coincide with the usual ones after passing to the underlying ordinary categories, see Example~\ref{exa:UnderlyingOrdCat}. The advantage of our point of view is that the conditions we usually impose on morphism and chain homotopies for various algebraic structures arise naturally by viewing those morphisms as elements of chain complexes.  

\begin{example}[(ordinary chain complexes over $\mathbb{F}_2$)] \label{exa:HighBrowDefChainCxs}
Let $\mathcal{C}$ be the category with a single object $\bullet$ in grading 0, $\Mor(\bullet,\bullet)=\mathbb{F}_2$ and vanishing differential. 
Then (the underlying ordinary category of) $\Cx^{0}(\Mat(\mathcal{C}))$ is isomorphic to the dg category of $\mathbb{Z}$-graded chain complexes over $\mathbb{F}_2$ with a preferred choice of basis via the identification of $\bullet$ with $\mathbb{F}_2$, considered as the 1-dimensional vector space over $\mathbb{F}_2$. Thus, $\Cx^{0}(\Mat(\mathcal{C}))$ is equivalent to $\Com$. 
\end{example}
\begin{example}[(type~D structures over dg $\mathbb{F}_2$-algebras)]\label{exa:HighBrowDefTypeDoverF2}
Let $\mathcal{A}$ be a differential graded algebra over $\mathbb{F}_2$. Let $\mathcal{C}$ be the category with a single object~$\bullet$ and morphisms being elements in $\mathcal{A}$. Composition is multiplication in $\mathcal{A}$ and the differential~$\partial$ is induced by the differential on $\mathcal{A}$. We define the category of type~D structures over $\mathcal{A}$ by $\Cx^{0}(\Mat(\mathcal{C}))$. (Again, note that we need to pass to the underlying ordinary category to obtain the definitions in \cite{Zarev} and \cite{LOT}.)
\end{example}

\begin{example}[(type~D structures over dg $\mathcal{I}$-algebras)]\label{exa:HighBrowDefTypeDoverI}
Let us assume that $\mathcal{A}$ is an algebra over some ring $\mathcal{I}\subseteq\mathcal{A}$ of idempotents and fix a basis $\{i_j\}_{j\in J}$ of idempotents of $\mathcal{I}$, where $J$ is some index set. Let $\mathcal{C}^D_\mathcal{A}$ be the category with one object for each basis element of $\mathcal{I}$, and for any two such elements $i_1$ and $i_2$, let $\Mor(i_1,i_2):=i_2.\mathcal{A}.i_1$. 
Again, composition is multiplication in $\mathcal{A}$ and the differential $\partial$ is induced by the differential on $\mathcal{A}$. We call $\Cx^{0}(\Mat(\mathcal{C}^D_\mathcal{A}))$ the category of (right) type~D structures over the dg $\mathcal{I}$-algebra $\mathcal{A}$. To obtain the category of \emph{left} type~D structures, we only need to change the morphism spaces to $\Mor(i_1,i_2):=i_1.\mathcal{A}.i_2$ with the obvious multiplication; however, we usually work with right type D structures in this paper, since we interpret algebra elements as functions and thus read them from right to left.
\end{example}
\begin{remark}\label{rem:HighBrowNoBasisForI}
\textit{A priori}, the definition in the previous example depends on a choice of basis for~$\mathcal{I}$. In all examples in this paper, there is a natural choice of such a basis, so this is not an issue. 
However, we can replace $\mathcal{C}^D_\mathcal{A}$ above by the enlarged category $\overline{\mathcal{C}^D_\mathcal{A}}$, where there is an object for \textit{every} element in $\mathcal{I}$. 	Then $\mathcal{C}^D_\mathcal{A}$ is a full subcategory of $\overline{\mathcal{C}^D_\mathcal{A}}$ and it is not hard to see that $\Mat(\mathcal{C}^D_\mathcal{A})$ and $\Mat(\overline{\mathcal{C}^D_\mathcal{A}})$ are equivalent. Now, the construction of the category of complexes is functorial (in the category of dg categories), so after all, the definition above does not depend on a basis for $\mathcal{I}$. 
\end{remark}
\begin{example}[(curved type~D structures over dg $\mathcal{I}$-algebras)]\label{exa:HighBrowDefcurvedTypeD}
We start with the same category $\mathcal{C}^D_\mathcal{A}$ as in the previous example, but we fix a central element $a_c\in Z(\mathcal{A})$, the curvature, and define the category of curved (right) type~D structures over $\mathcal{A}$ with curvature $a_c$ as $\Cx^{a_c}(\Mat(\mathcal{C}^D_\mathcal{A}))$. For a more explicit, but less concise definition, see Definition~\ref{def:curvedTypeDStructure}.
\end{example}
\begin{example}[(type~A structures over an $A_\infty$-algebra over $\mathcal{I}$)]\label{exa:HighBrowDefTypeAoverI}
Let $\mathcal{A}$ be an $A_\infty$-algebra over a ring of idempotents $\mathcal{I}$ over $\mathbb{F}_2$. As in Example~\ref{exa:HighBrowDefTypeDoverI}, fix a basis $\{i_k\}_{k\in I}$ of idempotents of $\mathcal{I}$, where $I$ is some index set. Let $\mathcal{C}^A_\mathcal{A}$ be the category with one object for each basis element in $\mathcal{I}$, just as for type~D structures. However, a morphism in a hom-object $\Mor(i_1,i_2)$ of $\mathcal{C}^A_\mathcal{A}$ is given by a sequence of vector space homomorphisms 
$$(f_i\co i_2.\mathcal{A}^{\otimes i}.i_1\rightarrow \mathbb{F}_2)_{i\geq0}$$
where composition is defined by 
$$(f\circ g)_i(a_i\otimes\cdots\otimes a_{1}):= \sum_{j+k=i}f_k(a_{i}\otimes\cdots\otimes a_{j+1}) \cdot g_j(a_{j}\otimes\cdots\otimes a_{1}).$$
For $i\in\{i_k\}_{k\in I}$, the identity morphism $\id_{i}=(\id_{i,l})_{l\geq0}\in\Mor(i,i)$ is given by $\id_{i,0}(i.1.i)=1$ and $\id_{i,l}=0$ for all $l>0$.
The differential $\partial$ is given by
$$(\partial(f))_i(a_i\otimes\cdots\otimes a_{1}):=\sum_{j+k=i+1}\sum_{l=0}^{i-k} f_j(a_i\otimes\dots\otimes \mu_k(a_{l+k}\otimes \dots \otimes a_{l+1})\otimes \dots \otimes a_{1}).$$
We call $\Cx^{0}(\Mat(\mathcal{C}^A_\mathcal{A}))$ the category of (left) type~A structures over $\mathcal{A}$. 
We define the category of strictly unital type~A structures by restricting to those objects $(O,d_O)$ with the identity action 
\begin{equation}\label{eqn:identityAction}
d_O(\cdot,1)=\id_O
\end{equation}
and with the property that
\[d_O(\cdot,a_i\otimes\dots\otimes a_1)=0\text{ if $i>1$ and $a_j=1$ for some $j=1,\dots,i$}\]
and restricting to those morphisms satisfying
\[f(\cdot,a_i\otimes\dots\otimes a_1)=0\text{ if $i>0$ and $a_j=1$ for some $j=1,\dots,i$}.\]
When we talk about type A structures in this paper, we will always implicitly assume them to be strictly unital. 

\emph{Right} type~A structures are defined in an analogous way; we only define the hom-objects of the underlying dg category to be given by sequences of vector space homomorphisms 
\[(f_i\co i_1.\mathcal{A}^{\otimes i}.i_2\rightarrow \mathbb{F}_2)_{i\geq0},\]
and adapt the multiplication maps accordingly. In this paper, we restrict ourselves to left type~A structures.
\end{example}

\begin{remark}\label{exa:GraphsForTypeAstructures}
When we describe type~A structures as directed graphs, it is useful to fix a basis of the algebra $\mathcal{A}$. Then, we label an arrow corresponding to a morphism $f$ by the formal sum of those tuples/tensor products of basis elements of the algebra $\mathcal{A}$ on which $f$ is non-zero. In particular, the identity morphisms are the length 0 (ie ``empty'') labels. In this language, composition of two morphisms $f$ and $g$ can be described as the sum of all concatenations of labels for $f$ and $g$ (modulo~2).

Now consider a morphism $f_a$ whose only label is a tuple of basic algebra elements $a=(a_i,\dots,a_1)$. Then the arrow of $\partial(f_a)$ is labelled by the formal sum (modulo 2) of all labels obtained from $a$ by replacing a single entry $a_i$ by sequences $(a'_j,\dots, a'_1)$ such that \[a_i=\mu_j(a'_j\otimes\cdots\otimes a'_1).\]

Note that in accordance with our conventions from the previous definition, we will always omit the arrows corresponding to the identity action~\eqref{eqn:identityAction}, ie arrows from each vertex to itself labelled by the length one sequence consisting of the identity of the algebra $\mathcal{A}$. 
\end{remark}

\begin{example}[(bimodules of various kinds)]
	We can form the categories of type~DD, type~DA, type~AD and type~AA bimodules as follows. We start with the dg category where objects correspond to idempotents as before, but where the hom-objects are defined as the corresponding products of the hom-objects of $\mathcal{C}^A_\mathcal{A}$ and $\mathcal{C}^D_\mathcal{A}$ in Examples~\ref{exa:HighBrowDefTypeDoverI} and~\ref{exa:HighBrowDefTypeAoverI}. Multiplication is defined as the product on the two factors and the differential is defined as usual by the Leibniz rule. Likewise, multi-modules can be defined, but we will not need those in the current paper. 
	
	We sometimes include the algebras over which the modules are defined in our notation, following the usual convention to use subscripts for type A and superscripts for type D sides. For example, the notation $\typeA{\mathcal{A}}{M}^\mathcal{B}$ means that $M$ is a type~AD $\mathcal{A}$-$\mathcal{B}$-bimodule where $\mathcal{A}$ acts on the left and $\mathcal{B}$ on the right.
\end{example}

\subsection{Functoriality and pairing for type A and type D structures}

In the proofs of the glueing formula in section~\ref{sec:Pairing}, we need to change the underlying algebras and rings of idempotents of the algebraic structures involved. 
\begin{definition}\label{def:InducedFunctors}
Suppose we have a subring $\mathcal{J}$ of the ring of idempotents~$\mathcal{I}$. Let us also fix an $\mathbb{F}_2$-basis $\{\iota_i\mid i\in J\}$ of $\mathcal{J}$ and extend it to a basis $\{\iota_i\mid i\in I\}$ of~$\mathcal{I}$. Let $\mathcal{A}$ be a dg $\mathcal{I}$-algebra and $\mathcal{B}$ be a dg $\mathcal{J}$-algebra. Note that via the inclusion $\mathcal{J}\hookrightarrow\mathcal{I}$, we can regard $\mathcal{A}$ also as a dg $\mathcal{J}$-algebra. Let $\pi\co\mathcal{B}\rightarrow\mathcal{A}$ be a dg $\mathcal{J}$-algebra homomorphism. 

Let us first consider the construction for (curved) type~D structures: consider the category $\mathcal{C}^D_{\mathcal{A}}$ corresponding to the $\mathcal{I}$-algebra $\mathcal{A}$ from Example~\ref{exa:HighBrowDefTypeDoverI}. Similarly, let $\mathcal{C}^D_{\mathcal{B}}$ be the category corresponding to the $\mathcal{J}$-algebra $\mathcal{B}$. The inclusion $\mathcal{J}\hookrightarrow\mathcal{I}$ and the $\mathcal{J}$-algebra homomorphism $\pi$ induce a functor 
$$\tilde{\mathcal{F}}^D_\pi\co\mathcal{C}^D_{\mathcal{B}}\rightarrow\mathcal{C}^D_{\mathcal{A}}$$
and hence also a functor
$$\mathcal{F}^D_\pi\co\Cx^{\ast}(\Mat(\mathcal{C}^D_{\mathcal{B}}))\rightarrow\Cx^{\ast}(\Mat(\mathcal{C}^D_{\mathcal{A}})),$$
both of which respect the differentials on both sides.

There is also a dual construction for type~A structures: consider the category $\mathcal{C}^A_{\mathcal{A}}$ corresponding to the $\mathcal{I}$-algebra $\mathcal{A}$ from Example~\ref{exa:HighBrowDefTypeAoverI}. Similarly, let $\mathcal{C}^A_{\mathcal{B}}$ be the category corresponding to the $\mathcal{J}$-algebra $\mathcal{B}$. However, to define an induced functor, we need to slightly modify $\mathcal{C}^A_{\mathcal{B}}$ by adding a zero object. Let us call this new category $\mathcal{C}^{A,0}_{\mathcal{B}}$. Then, we can define a functor 
$$\tilde{\mathcal{F}}^A_\pi\co\mathcal{C}^A_{\mathcal{A}}\rightarrow\mathcal{C}^{A,0}_{\mathcal{B}}$$
as follows: an object $\iota_i$ is sent to $\iota_i$ if $i\in J$ and to the zero-object otherwise. A basic morphism $f\in \Mor(\iota_1,\iota_2)$ of the form $\iota_2.\mathcal{A}^{\otimes i}.\iota_1\rightarrow \mathbb{F}_2$ is sent to 
$$\left(f\circ \iota_2.\pi^{\otimes i}.\iota_1\co\iota_2.\mathcal{B}^{\otimes i}.\iota_1\rightarrow \iota_2.\mathcal{A}^{\otimes i}.\iota_1\rightarrow \mathbb{F}_2\right)\in \Mor(\iota_1,\iota_2)$$
if $\iota_1$ and $\iota_2$ are in $J$ and to $0\in\Mor\left(\tilde{\mathcal{F}}^A_\pi(\iota_1),\tilde{\mathcal{F}}^A_\pi(\iota_2)\right)$ otherwise.
This functor is well-defined and respects the differential if $\mathcal{A}$ and $\mathcal{B}$ are $A_\infty$-algebras over $\mathcal{I}$ and $\mathcal{J}$, respectively, and $\pi$ is an $A_\infty$-algebra homomorphism. Thus, it induces a functor
$$\mathcal{F}^A_\pi\co\Cx^{0}(\Mat(\mathcal{C}^A_{\mathcal{A}}))\rightarrow\Cx^{0}(\Mat(\mathcal{C}^{A,0}_{\mathcal{B}}))$$
that likewise respects the differentials on both sides. Note that the category  $\Cx^{0}(\Mat(\mathcal{C}^{A,0}_{\mathcal{B}}))$ agrees with the category of type~A structures $\Cx^{0}(\Mat(\mathcal{C}^A_{\mathcal{B}}))$, since adding a zero-object in the underlying dg category does not change the result after applying $\Cx^{0}(\Mat(\cdot))$.

Similarly, we may define functors for bimodules of various types.  
\end{definition}

\begin{remark}\label{rem:InducedFunctors}
	In the previous definition, non-injective homomorphisms $\mathcal{J}\rightarrow\mathcal{I}$ also induce functors
	$$\mathcal{F}^D_\pi\co\Cx^{\ast}(\Mat(\mathcal{C}^D_{\mathcal{B}}))\rightarrow\Cx^{\ast}(\Mat(\mathcal{C}^D_{\mathcal{A}}))$$
	between the categories of (curved) type~D structures. All we need to change in its construction is to add a zero-object to $\mathcal{C}^D_{\mathcal{A}}$. 
	Note, however, that in all examples that we are concerned with in this paper, either $\mathcal{J}=\mathcal{I}$ and $\pi$ is a quotient map, or $\mathcal{B}=\mathcal{J}.\mathcal{A}.\mathcal{J}$ and $\pi$ is the inclusion.
\end{remark}

\begin{observation}\label{obs:InducedFunctorsForGraphs}
	Let \(\mathcal{J}\), \(\mathcal{I}\), \(\mathcal{A}\), \(	\mathcal{B}\) and \(\pi\co\mathcal{B}\rightarrow\mathcal{A}\) be as in Definition~\ref{def:InducedFunctors}. Choose a basis of the kernel of~$\pi$ and extend it to~$\mathcal{B}$. This induces a basis on the image of~$\pi$, which we can then extend it to a basis of $\mathcal{A}$. 
	Suppose we have an oriented labelled graph representing a type~D structure $N$ over $\mathcal{B}$ (with respect to the basis chosen above). Then the graph representing $\mathcal{F}^D_{\pi}(N)$ is obtained by replacing all algebra elements by their images under $\pi$. Similarly, if we have an oriented labelled graph representing a type~A structure $M$ over $\mathcal{A}$, $\mathcal{F}^A_{\pi}(M)$ has the effect of replacing each label $a=(a_n,\dots,a_1)$ by the formal sum of all labels whose images under $\pi^{\otimes n}$ are equal to $a$. 
\end{observation}

\begin{definition}[(pairing type~D and type~A structures)]\label{def:PairingTypeDandA}
Let $M$ be a (right) type~D structure and $N$ a (left) type~A structure over the same dg algebra $\mathcal{A}$ over a ring $\mathcal{I}$ of idempotents, together with a fixed basis of $\mathcal{I}$ and $\mathcal{A}$. We now reformulate the definition of the chain complex $(M\boxtimes N,\partial^\boxtimes)$ from \cite[Definition~7.4]{Zarev} and \cite[section~2.4]{LOT} in terms of the graphs associated with $M$ and $N$. The generators of $M\boxtimes N$ are defined by pairs of vertices in $M$ and $N$ labelled by the same idempotents. Given two such pairs $(v_1,w_1)$ and $(v_2,w_2)$, the $(v_2,w_2)$-component of $\partial^\boxtimes(v_1,w_1)$ is equal to the number of sequences $s$ of labels of consecutive arrows along a path from~$v_1$ to~$v_2$ in~$M$ such that $s$ agrees with a label on the arrow from $w_1$ to $w_2$ in~$N$, all modulo 2. 
Note that for the differential to be well-defined, we need to make sure that this number is finite. This is usually done by requiring that at least one of $M$ or $N$ is \textbf{bounded}: for type D structures, this means that there are no loops in its graph; for type A structures, this means that there are only finitely many labels.

If we start with a type~AD or type~DD bimodule~$M$ and a type~AA or type~AD bimodule $N$, the pairing is defined in the same way, except that we need to record the labels for the remaining type~A or type~D sides: the first component is equal to the product (ie algebra product or concatenation) of the first components of the arrows along the corresponding path from~$v_1$ to~$v_2$ in~$M$; the second component of a label on the arrow from $(v_1,w_1)$ to $(v_2,w_2)$ is equal to the second component of the corresponding label of the arrow from $w_1$ to $w_2$ in~$N$. Similarly, we can define a pairing between other types of bimodules, as long as we pair type~A sides with type~D sides and left structures with right structures. For details, see~\cite[Definition~2.3.9]{LOTBimodules}.
\end{definition}

\begin{remark}\label{rem:PairingIsFunctorial}
	The result of the pairing operation described above is a well-defined object in the corresponding category, ie a chain complex, type~A, type~D structure or a bimodule, see \cite[Proposition~2.3.10]{LOTBimodules}. 	
	Furthermore, pairing type~D and type~A structures is invariant under homotopy up to homotopy since it is functorial, see \cite[Lemma~2.3.13]{LOTBimodules}. 	
\end{remark}

\begin{question}
	Is it possible to interpret the pairing on the level of the categories $\mathcal{C}^D_\mathcal{A}$ and $\mathcal{C}^A_\mathcal{A}$?
\end{question}

\begin{theorem}[(Pairing Adjunction)]\label{thm:PairingAdjunction}
	Let \(\mathcal{J}\), \(\mathcal{I}\), \(\mathcal{A}\), \(
	\mathcal{B}\) and \(\pi\co\mathcal{B}\rightarrow\mathcal{A}\) be as in Definition~\ref{def:InducedFunctors}. Let \(M\) be a (right) type~D structure over~\(\mathcal{B}\) and \(N\) a (left) type~A structure over~\(\mathcal{A}\). Then we have an identification
	$$ \mathcal{F}^D_{\pi}(M)^\mathcal{A}\boxtimes \typeA{\mathcal{A}}{N}\cong  M^\mathcal{B}\boxtimes \typeA{\mathcal{B}}{\mathcal{F}^A_{\pi}(N)}.$$
	The same holds true for bimodules of various types.
\end{theorem}
\begin{proof}
	This follows almost tautologically from the interpretation of the induced functors in terms of oriented labelled graphs in Observation~\ref{obs:InducedFunctorsForGraphs}. 
	On the level of generators this identity is clear, since by definition, $M$, $\mathcal{F}^D_{\pi}(M)$ and $\mathcal{F}^A_{\pi}(N)$ only possess generators belonging to idempotents in $\mathcal{J}$, so only those generators of $N$ that belong to idempotents in $\mathcal{J}$ survive the pairing. Next, fix a label $a=(a_m,\dots,a_1)$ from a vertex $w_1$ to $w_2$ in $N$, along with two vertices $v_1$ and $v_2$ in $M$ such that $v_1$ and $w_1$ belong to the same idempotent in~$\mathcal{J}$ and so do $v_2$ and $w_2$. The contribution of $a$ to the differential from $(v_1, w_1)$ to $(v_2, w_2)$ on the left-hand side is equal to the number of all sequences $b=(b_n,\dots,b_1)$ of labels from $v_1$ to $v_2$ such that $\pi(b):=(\pi(b_n),\dots,\pi(b_1))=a$.
	The label $a$ in $N$ corresponds to labels  $b'=(b'_n,\dots,b'_1)$ from $w_1$ to $w_2$ in $\mathcal{F}^A_{\pi}(N)$ such that $\pi(b'):=(\pi(b'_n),\dots,\pi(b'_1))=a$.
	Those labels contribute 1 to the differential from $(v_1, w_1)$ to $(v_2, w_2)$ on the right-hand side iff they agree with some $b$ and do not contribute otherwise. So the contributions agree.
	
	For bimodules, note that the functors only act on one component of the morphism spaces.
\end{proof}


\subsection{Cancellation and Cleaning-up}

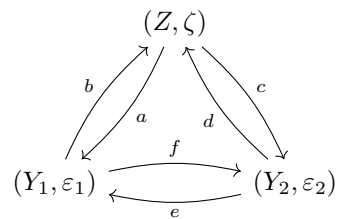
\begin{wrapfigure}{r}{0.3333\textwidth}
	\centering
	$\begin{tikzcd}[row sep=1.5cm, column sep=0.3cm]
	& (Z,\zeta) \arrow[bend left=12]{ld}{a} \arrow[bend left=12]{rd}{c} \\
	(Y_1,\varepsilon_1)\arrow[bend left=12]{rr}{f} \arrow[bend left=12]{ru}{b}
	& &
	(Y_2,\varepsilon_2) \arrow[bend left=12]{ll}{e}\arrow[bend left=12]{lu}{d}
	\end{tikzcd}$
	\caption{The object~$(X,\delta)$ for Lemma~\ref{lem:AbstractCancellation}.}\label{fig:AbstractCancellation}\vspace*{-35pt}
\end{wrapfigure}

We now state and prove the two central lemmas mentioned at the beginning of this section.
\begin{lemma}[(Cancellation Lemma)]\label{lem:AbstractCancellation}
Let \((X,\delta)\) be an object of \(\Cx^{\ast}(\Mat(\mathcal{C}))\) for some differential graded category \(\mathcal{C}\) and suppose it has the form shown in Figure~\ref{fig:AbstractCancellation}, 
where \((Y_1,\varepsilon_1)\), \((Y_2,\varepsilon_2)\), \((Z,\zeta)\in\ob(\Cxpre(\Mat(\mathcal{C})))\) and \(f\) is an isomorphism with inverse \(g\). Then \((X,\delta)\) is chain homotopic to \((Z,\zeta+bgc)\).
\end{lemma}

\begin{figure}[b]
	\centering
	\begin{subfigure}{0.48\textwidth}
		\centering
		$\begin{tikzcd}[row sep=1cm, column sep=0cm]
		(Z,\zeta+bgc)\arrow{rrr}{1}\arrow[bend right=12]{rrdd}{gc} &~~~~~~~ && (Z,\zeta) \arrow[bend left=12,pos=0.3]{ldd}{a} \arrow[bend left=12]{rrd}{c}\\
		& &&&~& (Y_2,\varepsilon_2) \arrow[bend left=12]{llld}{e}\arrow[bend left=12]{llu}{d}\\
		& &(Y_1,\varepsilon_1) \arrow[bend left=10,pos=0.7]{rrru}{f} \arrow[bend left=12]{ruu}{b}
		\end{tikzcd}$
		\caption{The chain map $F$.}\label{fig:ProofAbstractCancellationF}
	\end{subfigure}
	\begin{subfigure}{0.48\textwidth}
		\centering
		$\begin{tikzcd}[row sep=1cm, column sep=0cm]
		& (Z,\zeta) \arrow[bend left=12,pos=0.3]{ldd}{a} \arrow[bend left=12]{rrd}{c}\arrow{rrrr}{1}&&&~~~~~ &(Z,\zeta+bgc)\\
		&&~& (Y_2,\varepsilon_2) \arrow[bend left=12]{llld}{e}\arrow[bend left=12]{llu}{d}\arrow[bend right=12]{rru}{bg}\\
		(Y_1,\varepsilon_1) \arrow[bend left=10,pos=0.7]{rrru}{f} \arrow[bend left=12]{ruu}{b}
		\end{tikzcd}$
		\caption{The chain map $G$.}\label{fig:ProofAbstractCancellationG}
	\end{subfigure}
	\begin{subfigure}{0.95\textwidth}
		\centering
		$\begin{tikzcd}[row sep=1cm, column sep=0cm]
		& (Z,\zeta) \arrow[bend left=12,pos=0.3]{ldd}{a} \arrow[bend left=12]{rrd}{c}\arrow{rrrr}{1}&&&~~~~~ &(Z,\zeta+bgc)
		\arrow{rrr}{1}\arrow[bend right=12]{rrdd}{gc} &~~~~~~~ && (Z,\zeta) \arrow[bend left=12,pos=0.3]{ldd}{a} \arrow[bend left=12]{rrd}{c}
		\\
		&&~& (Y_2,\varepsilon_2) \arrow[bend left=12]{llld}{e}\arrow[bend left=12]{llu}{d}\arrow[bend right=12]{rru}{bg}\arrow[dashed,swap, bend right=12]{rrrrd}{g}&&
		& &&&~& (Y_2,\varepsilon_2) \arrow[bend left=12]{llld}{e}\arrow[bend left=12]{llu}{d}
		\\
		(Y_1,\varepsilon_1) \arrow[bend left=12,pos=0.7]{rrru}{f} \arrow[bend left=12]{ruu}{b}&&&&&
		& &(Y_1,\varepsilon_1) \arrow[bend left=12,pos=0.7]{rrru}{f} \arrow[bend left=12]{ruu}{b}
		\end{tikzcd}$
		\caption{The composition of $F$ and $G$ and the homotopy $H$ (dashed arrow).}\label{fig:ProofAbstractCancellationHomotopy}
	\end{subfigure}
	\caption{Maps for the proof of Lemma~\ref{lem:AbstractCancellation}.}\label{fig:ProofAbstractCancellation}
\end{figure}
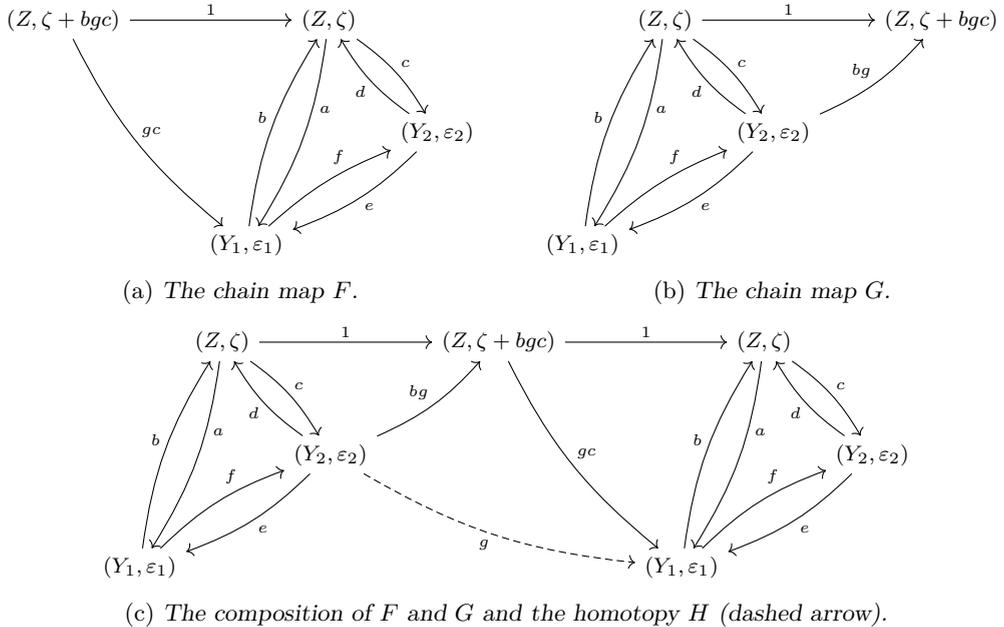

\begin{remark}
	We usually apply this lemma to the case where $(Y_1,\varepsilon_1)=(Y_2,\varepsilon_2)$ and $f$ is the identity map.
\end{remark}

\begin{proof}
First of all, let us check that $(Z,\zeta+bgc)$ is indeed an object of $\Cx^{\ast}(\Mat(\mathcal{C}))$:
\begin{align*}
(\zeta+bgc)^2+\partial(\zeta+bgc)=&~\zeta\zeta+\zeta bgc+bgc \zeta +bgcbgc+\partial(\zeta)+\partial(b)gc+b\partial(g)c+bg\partial(c)\\
=&~\zeta\zeta+(\zeta b +\partial(b))gc+bg(c \zeta +\partial(c)) +\partial(\zeta)+b\partial(g)c+bgcbgc\\
=&~\zeta\zeta+(b\varepsilon_1 +df)gc+bg(\varepsilon_2 c +fa) +\partial(\zeta)+b\partial(g)c+bgcbgc\\
=&~\zeta\zeta+dc+ba+\partial(\zeta)+\cancel{b\left(D(g)+gcbg\right)c}=\ast.
\end{align*}
For the last step, we observe that $$gcbg=gD(f)g=D(gfg)+D(g)fg+gfD(g)=D(g).$$
Next, we consider the two chain maps 
$$F\co(Z,\zeta+bgc)\rightarrow(X,\delta)\qquad\text{and}\qquad G\co(X,\delta)\rightarrow(Z,\zeta+bgc)$$
defined in Figure~\ref{fig:ProofAbstractCancellationF} and \ref{fig:ProofAbstractCancellationG}, respectively. One easily checks that $D(F)=0$ and $D(G)=0$. Indeed, the only non-trivial terms we need to compute are
\begin{align*}
gc(\zeta+bgc)+\varepsilon_1gc+\partial(gc)+a=& ~g(c\zeta+\partial(c))+(\varepsilon_1g+\partial(g))c+a+gcbgc\\
=&~g(fa+\varepsilon_2c)+(\varepsilon_1g+\partial(g))c+a+gcbgc\\
=&~\left(D(g)+gcbg\right)c=0
\end{align*}
for the first identity and similarly
\begin{align*}
(\zeta+bgc)bg+bg\varepsilon_2+\partial(bg)+d=& ~(\zeta b+\partial(b))g+b(g\varepsilon_2+\partial(g))+d+bgcbg\\
=&~(df+b\varepsilon_1)g+b(g\varepsilon_2+\partial(g))+d+bgcbg\\
=&~b\left(D(g)+gcbg\right)=0
\end{align*}
for the second identity. Now, $GF=\id_{Z}$ and conversely, it is not hard to check that 
$$FG=\id_{X}+D(H),$$
where $H$ is the homotopy given by the dashed line in Figure~\ref{fig:ProofAbstractCancellationHomotopy}.
\end{proof}

\begin{lemma}[(Clean-Up Lemma)]\label{lem:AbstractCleanUp}
Let \((O,d_O)\) be an object in \(\Cx^{\ast}(\mathcal{C})\) for some differential graded category \(\mathcal{C}\).
Then for any morphism \(h\in\Mor_0((O,d_O),(O,d_O))\) for which 
$$h^2, \quad hD(h) \quad\text{ and } \quad D(h)h$$
vanish, \((O,d_O)\) is chain isomorphic to \((O,d_O+D(h))\).
\end{lemma}
\begin{proof}
We can easily check that $(O,d_O+D(h))$ is an object in $\Cx^{\ast}(\Mat(\mathcal{C}))$:
\begin{align*}
(d_O+D(h))^2+\partial(d_O+D(h)) = &\left(d_O^2+\partial(d_O)\right)+\left(d_OD(h)+D(h)d_O^{\phantom{2}\!\!}+\partial(D(h))\right)\\
&+D(h)D(h).
\end{align*}
The first term on the right gives ($\ast$) and the last term vanishes, which can be seen by applying the differential $D$ to $h D(h)=0$. The middle term also vanishes, which can be seen by expanding $D(h)=d_Oh+hd_O+\partial(h)$ and using the fact that a term ($\ast$) commutes with any morphism.
The chain isomorphisms between the two objects are given by 
\begin{center}
$\begin{tikzcd}[row sep=1.5cm, column sep=0.6cm]
(O,d_O) \arrow{rr}{1+h} && (O,d_O+D(h))
\end{tikzcd}$
\quad
and
\quad
$\begin{tikzcd}[row sep=1.5cm, column sep=0.6cm]
(O,d_O+D(h)) \arrow{rr}{1+h} && (O,d_O).
\end{tikzcd}$
\end{center}
Indeed, these two morphisms lie in the kernel of $D$, since $hD(h)$ and $D(h)h$ vanish. Their composition is equal to $1+h^2=1$.
\end{proof}

%% file: sections/CFTd.tex

\section{\texorpdfstring{The invariant $\CFTd$}{The invariant CFTᵈ}}\label{sec:CFTd}

\subsection{Heegaard diagrams for tangles}

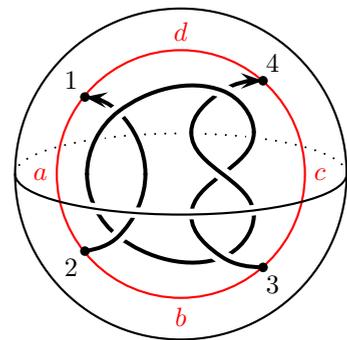
\begin{wrapfigure}{r}{0.3333\textwidth}
	\centering
	\psset{unit=0.55}
	\vspace*{-20pt}
	\begin{pspicture}[showgrid=false](-5.1,-4.1)(3.1,4.1)
	
	\psellipticarc[linestyle=dotted](-1,0)(4,1){0}{180}
	
	\psecurve[linecolor=white,linewidth=\stringwhite]{c-c}(-2.5,1.5)(0,2)(0.75,1)(-0.75,-1)(0,-2)(0.97,-2.24)(2,-2)
	\psecurve[linecolor=white,linewidth=\stringwhite]{c-c}(2,2)(0.97,2.24)(0,2)(-0.75,1)(0.75,-1)(0,-2)(-2.5,-1.5)(-3.25,0)(-2.5,1.5)(0,2)(0.75,1)
	{
		\psset{linewidth=\stringwidth}
		\psecurve{c-c}(-2.5,1.5)(0,2)(0.75,1)(-0.75,-1)(0,-2)(0.97,-2.24)(2,-2)
		\psecurve{<-c}(2,2)(0.97,2.24)(0,2)(-0.75,1)(0.75,-1)(0,-2)(-2.5,-1.5)(-3.25,0)(-2.5,1.5)(0,2)(0.75,1)
		\psecurve{<-c}(-6,1.5)(-3.3,1.85)(-2.5,1.5)(-1.85,0)(-2.5,-1.5)(-3.3,-1.85)(-6,-1.5)
		
		\psecurve[linecolor=white,linewidth=\stringwhite](0.75,-1)(0,-2)(-2.5,-1.5)(-3.25,0)(-2.5,1.5)
		\psecurve{c-c}(0.75,-1)(0,-2)(-2.5,-1.5)(-3.25,0)(-2.5,1.5)
		
		\psecurve[linecolor=white,linewidth=\stringwhite](0.75,1)(-0.75,-1)(0,-2)(0.97,-2.24)(2,-2)
		\psecurve[linecolor=white,linewidth=\stringwhite](0,2)(-0.75,1)(0.75,-1)(0,-2)
		\psecurve[linecolor=white,linewidth=\stringwhite](-2.5,-1.5)(-3.25,0)(-2.5,1.5)(0,2)(0.75,1)(-0.75,-1)
		
		\psecurve{c-c}(0.75,1)(-0.75,-1)(0,-2)(0.97,-2.24)(2,-2)
		\psecurve{c-c}(0,2)(-0.75,1)(0.75,-1)(0,-2)
		\psecurve{c-c}(-2.5,-1.5)(-3.25,0)(-2.5,1.5)(0,2)(0.75,1)(-0.75,-1)
		
		\psecurve[linecolor=white,linewidth=\stringwhite](-2.5,1.5)(-1.85,0)(-2.5,-1.5)(-3.3,-1.85)(-6,-1.5)
		\psecurve{c-c}(-2.5,1.5)(-1.85,0)(-2.5,-1.5)(-3.3,-1.85)(-6,-1.5)
	}
	
	\psellipticarc[linewidth=4pt,linecolor=white](-1,-0.09)(4,1){180}{0}
	
	\pscircle[linecolor=red](-1,0){3}
	
	\psdots(-3.3,1.85)(-3.3,-1.85)(0.97,-2.24)(0.97,2.24)
	
	\psellipticarc(-1,0)(4,1){180}{0}
	\pscircle(-1,0){4}
	
	\uput{3.2}[140](-1,0){1}
	\uput{3.2}[-140](-1,0){2}
	\uput{3.2}[-50](-1,0){3}
	\uput{3.2}[50](-1,0){4}
	
	\uput{3.2}[180](-1,0){\red$a$}
	\uput{3.2}[-90](-1,0){\red$b$}
	\uput{3.2}[0](-1,0){\red$c$}
	\uput{3.2}[90](-1,0){\red$d$}
	\end{pspicture}
	\caption{The $(2,-3)$-pretzel tangle in $B^3$.}\label{fig:2m3pt}
	\vspace*{-10pt}
\end{wrapfigure}

Let us start by recalling the basic definitions from~\cite[sections~1 and~4]{HDsForTangles}, adapted to 4-ended tangles. 

\begin{definition}
	A \textbf{4-ended tangle} $T$ in a homology 3-ball~$M$ with spherical boundary is an embedding
	\[T\co\left(I \amalg I\amalg \coprod S^1,\partial\right)\hookrightarrow \left(M,{\red S^1}\subset S^2=\partial M\right),\]
	such that the endpoints of the two intervals lie on a fixed oriented circle ${\red S^1}$ on the boundary of~$M$, together with a choice of distinguished tangle end. Starting at this distinguished ($=$first) tangle end and following the orientation of the fixed circle ${\red S^1}$, we number the tangle ends and label the arcs ${\red S^1}\smallsetminus \im(T)$ by $a$, $b$, $c$ and $d$, in that order. Sometimes, we will find it more convenient to use the labels $s_1$ for $a$, $s_2$ for $b$, $s_3$ for $c$ and $s_4$ for $d$ instead. We call a choice of a single arc a \textbf{site} of the tangle~$T$.
	\\\indent
	We consider tangles up to ambient isotopy which fixes the distinguished tangle end and the orientation of ${\red S^1}$ (and thus preserves the labelling of the tangle ends and arcs). The images of the two intervals are called the \textbf{open components}, the images of any circles are called the \textbf{closed components} of the tangle. We label these tangle components by variables $t_1$ and $t_2$ for the open components and $t'_1,t'_2,\dots$ for the closed components. We call those variables the \textbf{colours} of~$T$.
	An orientation of a tangle is a choice of orientation of the two intervals and the circles. 
	\\\indent
	Note that the orientation of ${\red S^1}$ enables us to distinguish between the two components of $\partial M\smallsetminus{\red S^1}$. The \textbf{back} component of $\partial M\smallsetminus{\red S^1}$ is the one whose boundary orientation agrees with the orientation of ${\red S^1}$, using the right-hand rule and a normal vector field pointing into $M$. We call the other one the \textbf{front} component. 
\end{definition}

\begin{definition}\label{def:HDsfortangles}
	A \textbf{Heegaard diagram $\mathcal{H}_T$ for a 4-ended tangle} $T$ with $n$ closed components in a $\mathbb{Z}$-homology 3-ball with spherical boundary $M$ is a tuple $(\Sigma_g,\A=\Ac\cup\Aa,\B)$, where
	\begin{itemize}
		\item $\Sigma_g$ is an oriented surface of genus $g$ with $2n+4$ boundary components, denoted by~$\Gamma$, which are partitioned into $(n+2)$ pairs,
		\item $\Ac$ is a set of $(g+n)$ pairwise disjoint circles $\alpha_1,\dots, \alpha_{g+n}$ on $\Sigma_g$,
		\item $\Aa$ is a set of $4$ pairwise disjoint arcs on $\Sigma_g$, labelled $a$, $b$, $c$, $d$, which are disjoint from~$\Ac$ and whose endpoints lie on~$\Gamma$, and
		\item $\B$ is a set of $(g+n+1)$ pairwise disjoint circles $\beta_1,\dots, \beta_{g+n+1}$ on~$\Sigma_g$.
	\end{itemize}
	We impose the following condition on the data above: the 3-manifold obtained by attaching 2-handles to $\Sigma_g\times [0,1]$ along $\Ac\times\{0\}$ and $\B\times\{1\}$ is equal to the tangle complement~$M\smallsetminus \nu(T)$ such that under this identification, 
	\begin{itemize}
		\item each pair of circles in $\Gamma$ is a pair of meridional circles for the same tangle component, and each tangle component belongs to exactly one such pair, and 
		\item $\Aa\times\{0\}$ is equal to ${\red S^1}\smallsetminus \nu(\partial T)\subset \partial M$.
	\end{itemize}
	If the tangle $T$ is oriented, we also orient the boundary components of $\Sigma_g$ as oriented meridians of the tangle components, using the right-hand rule.
\end{definition}

\begin{remark}\label{rem:conventions1}
	As in~\cite{HDsForTangles}, our \textbf{convention on the orientation of the Heegaard surface} is that its normal vector (determined using the right-hand rule) points in the direction of positive gradient of the defining Morse function, ie in the direction of the $\beta$-handlebody. However, we usually draw the Heegaard surfaces such that the normal vector points into the projection plane.
\end{remark}

\begin{definition}
Let $T$ be a 4-ended tangle. A \textbf{peculiar Heegaard diagram} for $T$ is obtained from a tangle Heegaard diagram for $T$ by a local modification around the punctures, as illustrated in Figure~\ref{fig:HD4ended}: we collapse the four boundary components of $\Sigma$ which meet the $\alpha$-arcs, thereby joining the four $\alpha$-arcs to a single $\alpha$-circle $\red S^1$. 
Then, for each tangle end, we add marked points on either side of $\red S^1$ and connect them by an arc which intersects $\red S^1$ exactly once and no other curve. As illustrated in Figure~\ref{fig:HD4ended}, 
we label the marked points on the front component of $\partial M\smallsetminus{\red S^1}$ by $p_i$ and those on the back by $q_i$. 
Furthermore, for each closed component, we contract the corresponding boundary components to points $z_j$ and $w_j$. Thus, we obtain a multi-pointed Heegaard diagram, whose underlying Heegaard surface is now closed and carries basepoints $p_i$, $q_i$, $z_j$ and $w_j$. If the tangle $T$ is oriented, we choose corresponding orientations for the basepoints. By convention, we choose the label $z_j$ for a basepoint corresponding to a closed component iff its orientation agrees with the orientation of the normal vector field of the Heegaard surface. This agrees with Ozsváth and Szabó's conventions in~\cite[Definitions~2.1]{OSHFLThurston} and in~\cite[section~3.5]{OSHFL} noting that their gradient flowlines flow upwards~\cite[section~3.1]{OSHFL}; for an illustration, see Figure~\ref{fig:HDBasepointszw}.  
As in~\cite{HDsForTangles}, we need to restrict ourselves to \textbf{admissible} diagrams, ie diagrams whose non-zero periodic domains avoiding all basepoints have both negative and positive multiplicities, see~\cite[Definition~4.16]{HDsForTangles}. 
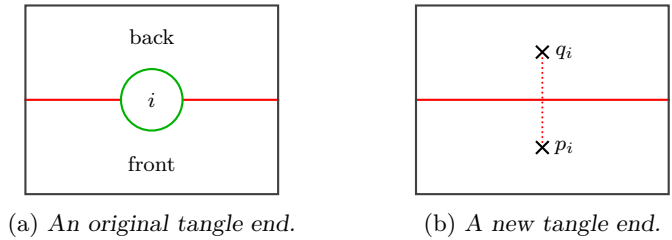
\begin{figure}[t]
\psset{unit=0.07}
\centering
{\psset{unit=0.6}
\begin{subfigure}{0.33\textwidth}\centering
\begin{pspicture}(-40,-30)(40,30)
\psline[linecolor=red](40,0)(-40,0)
\pscircle[fillstyle=solid,fillcolor=white,linecolor=darkgreen](0,0){10}
\rput(0,0){$i$}
\rput(0,20){back}
\rput(0,-20){front}
\psframe[linecolor=darkgray](-40,-30)(40,30)
\end{pspicture}
\caption{An original tangle end.}\label{fig:HD4endedoriginal}
\end{subfigure}
\quad
\begin{subfigure}{0.33\textwidth}\centering
\begin{pspicture}(-40,-30)(40,30)

\psline[linestyle=dotted, linecolor=red,dotsep=1pt](0,15)(0,-15)
\psline[linecolor=red](-40,0)(40,0)
\psline(-2,17)(2,13)
\psline(-2,13)(2,17)
\psline(-2,-17)(2,-13)
\psline(-2,-13)(2,-17)
\rput(7,15){$q_i$}
\rput(7,-15){$p_i$}
\psframe[linecolor=darkgray](-40,-30)(40,30)
\end{pspicture}
\caption{A new tangle end.}\label{fig:HD4endednew}
\end{subfigure}
}
\caption{The difference between peculiar and ordinary Heegaard diagrams for tangles. The former are obtained from the latter by local changes near the boundary of the Heegaard surface.}\label{fig:HD4ended}
\end{figure}
\end{definition}
\begin{remark}\label{rem:PeculiarHDmoves}
It is obvious that we can go from a peculiar Heegaard diagram back to an ordinary tangle Heegaard diagram. The only reason for introducing peculiar Heegaard diagrams is to avoid any bordered Heegaard Floer theory, so the proof of invariance of the algebraic structures we are about to define is a minor adaptation of the one for link Floer homology. In particular, note that the number of $\alpha$-circles and $\beta$-circles in a peculiar Heegaard diagram is the same.

Obviously, the Heegaard moves from \cite[Lemma~4.13]{HDsForTangles} are equivalent to the following moves for peculiar Heegaard diagrams:
\begin{itemize}
\item isotopies of the $\alpha$- and $\beta$-curves away from the marked points and the connecting arcs,
\item handleslides of $\beta$-curves over $\beta$-curves and handleslides of $\alpha$-curves over $\alpha$-curves other than $\red S^1$, and
\item stabilisation.
\end{itemize}
\end{remark}

\begin{remark}
The attribute ``peculiar'' should be considered as a homophone of ``$p$-$q$-lier'', a reference to the labels $p_i$ and $q_i$ we have chosen for the marked points. This choice of variables is loosely related to the notation used in \cite{Abouzaid} for describing the wrapped Fukaya category of the $n$-punctured sphere in terms of twisted complexes. From a Heegaard Floer perspective, we will find it more natural to use curved complexes instead, which are in some sense dual to twisted complexes; see also Remark~\ref{rem:comparisonToKontsevich} as well as~\cite[section~III.5]{MyThesis}. 
\end{remark}
Let us briefly recall the definition of the Heegaard Floer homology $\HFT(T)$ from~\cite[section~5]{HDsForTangles}, adapted to 4-ended tangles.
\begin{definition}\label{def:RecallCFT}
Given a peculiar Heegaard diagram $\mathcal{H}_T$ for a 4-ended tangle $T$, let $\mathbb{T}:=\mathbb{T}(\mathcal{H}_T)$ be the set of tuples of intersection points, also called generators, such that each $\alpha$- and each $\beta$-curve is occupied by exactly one point. In particular, each generator contains an intersection point that occupies the special $\alpha$-circle ${\red S^1}$, which replaced the four $\alpha$-arcs in the original Heegaard diagram from Definition~\ref{def:HDsfortangles}. So each generator $\x\in\mathbb{T}$ occupies exactly one of these four arcs. We denote the corresponding site by $s(\x)$ and write $\mathbb{T}_s$ for the set of all generators $\x$ with $s=s(\x)$. We obtain a partition
$$\mathbb{T}=\coprod_{i=1,2,3,4}\mathbb{T}_{s_i}.$$ 
Let $\Id$ be the ring of idempotents defined as the direct product of four copies of $\mathbb{F}_2$ and write $\iota_i$ for the $i^\text{th}$ element of the standard basis of $\Id$.
Then, we define a right $\Id$-module $\CFT(T)$ by letting $\CFT(T).\iota_i$ be the vector space over $\mathbb{F}_2$ freely generated by elements in $\mathbb{T}_{s_i}$. 

Furthermore, for any two generators $\x$ and $\y$ of $\mathbb{T}$, we can consider the space of domains from $\x$ to $\y$, denoted by $\pi_2(\x,\y)$. For each element $\phi\in\pi_2(\x,\y)$, we can define a moduli space $\mathcal{M}(\phi)$ of holomorphic curves in $\Sigma\times [0,1]\times\mathbb{R}$ representing $\phi$. (Alternatively, one may count holomorphic discs in $\Sym^{g+n+1}(\Sigma_g)$; however, by the main result of \cite{CylindricalReformulation}, this gives the same theory, see also \cite[section~5.2]{OSHFL}.) The dimension of this moduli space is given by the Maslov index $\mu(\phi)$. The differential in $\CFT(T)$ is defined as
\[\partial \x=\sum_{\y}\sum_{\substack{\phi\in\pi_2^0(\x,\y)\\ \mu(\phi)=1}}\#\left(\frac{\mathcal{M}(\phi)}{\mathbb{R}}\right)\y,\]
where $\pi_2^0(\x,\y)$ denotes the domains avoiding all basepoints and 
$\#\left(\tfrac{\mathcal{M}(\phi)}{\mathbb{R}}\right)$ denotes the number of points in the quotient of the 1-dimensional moduli spaces $\mathcal{M}(\phi)$ by a natural $\mathbb{R}$-action. The \textbf{non-glueable Heegaard Floer homology $\HFT(T)$ of the tangle $T$} is defined to be the homology of $\CFT(T)$.
\end{definition}

\begin{definition}\label{def:matching}
	A \textbf{matching} $P$ is a partition $\{\{i_1,o_1\},\{i_2,o_2\}\}$ of $\{1,2,3,4\}$ into pairs. An \textbf{ordered matching} is a matching in which the pairs are ordered. A 4-ended tangle $T$ gives rise to a matching $P_T$ as follows: the first pair consists of the two endpoints of the open component with colour $t_1$, the second consists of the two endpoints of the second open component of $T$, the one labelled $t_2$. Given an orientation of the two open components of $T$, we order each pair of points such that the inward pointing end comes first, the outward pointing end second. For an illustration, see Figure~\ref{fig:HDBasepointspq}.
\end{definition}
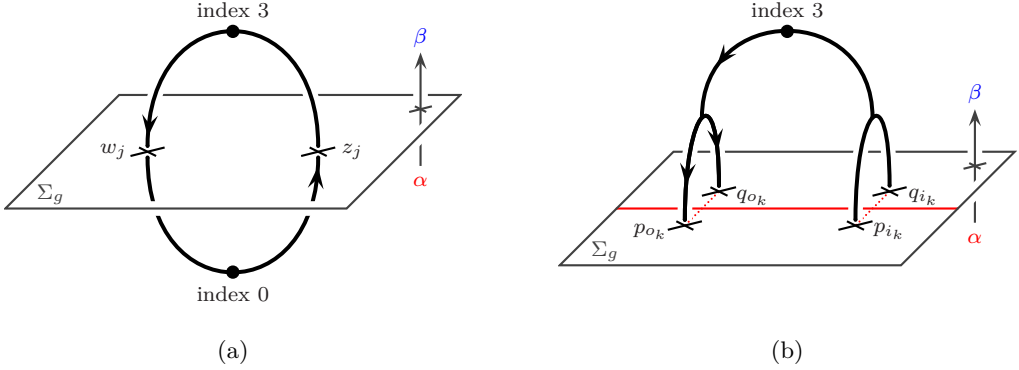
\begin{figure}[t]
	\centering
	{\psset{unit=0.075}
		\begin{subfigure}{0.48\textwidth}\centering
			\begin{pspicture}(-40,-30)(40,30)
			
			\psline[linecolor=darkgray](33,-2.5)(33,7.5)
			
			\rput(0,-1){
				\psbezier[linewidth=\stringwidth,ArrowInside=->,ArrowInsidePos=0.96]{c-c}(-15,0)(-15,-27)(15,-27)(15,0)
				\psdot[linewidth=\stringwidth](0,-20.2)
			}
			
			\psline[linecolor=white,linewidth=\stringwhite](0,-10)(20,-10)(40,10)(-20,10)(-40,-10)(0,-10)
			\psline[linecolor=darkgray](0,-10)(20,-10)(40,10)(-20,10)(-40,-10)(0,-10)
			
			\rput(0,1){
				\psbezier[linecolor=white,linewidth=\stringwhite](-15,0)(-15,27)(15,27)(15,0)
				\psbezier[linewidth=\stringwidth,ArrowInside=->,ArrowInsidePos=0.96]{c-c}(15,0)(15,27)(-15,27)(-15,0)
				\psdot[linewidth=\stringwidth](0,20.2)
			}
			
			\rput(-15,0){\psset{unit=0.75}
				\psline(2,-1)(-2,1)\psline(4,1)(-4,-1)}
			\rput(15,0){\psset{unit=0.75}
				\psline(2,-1)(-2,1)\psline(4,1)(-4,-1)}
			
			\rput(-32,-7.5){$\textcolor{darkgray}{\Sigma_g}$}
			\rput(-21,0){$w_j$}
			\rput(21,0){$z_j$}
			\rput(0,25){index $3$}
			\rput(0,-25){index $0$}
			
			\rput[t](33,-4.5){$\textcolor{red}{\alpha}$}
			\rput[b](33,18.5){$\textcolor{blue}{\beta}$}
			\psline[linecolor=white,linewidth=\stringwhite](33,7.5)(33,17.5)
			\psline[linecolor=darkgray]{->}(33,7.5)(33,17.5)
			\rput(33,7.5){\psset{unit=0.5}
				\psline[linecolor=darkgray](2,-1)(-2,1)\psline[linecolor=darkgray](4,1)(-4,-1)}
			\end{pspicture}
			\caption{}\label{fig:HDBasepointszw}
		\end{subfigure}
		\quad
		\begin{subfigure}{0.48\textwidth}\centering
			\begin{pspicture}(-40,-20)(40,40)
			
			\psline[linestyle=dotted, linecolor=red,dotsep=1pt](12,-3)(18,3)
			\psline[linestyle=dotted, linecolor=red,dotsep=1pt](-12,3)(-18,-3)
			\psline[linecolor=darkgray](33,-2.5)(33,7.5)
			
			\psline[linecolor=white,linewidth=\stringwhite](0,-10)(20,-10)(40,10)(-20,10)(-40,-10)(0,-10)
			\psline[linecolor=red]{c-c}(-30,0)(30,0)
			\psline[linecolor=darkgray](0,-10)(20,-10)(40,10)(-20,10)(-40,-10)(0,-10)
			
			\rput(-18,-3){\psset{unit=0.75}
				\psline(2,-1)(-2,1)\psline(4,1)(-4,-1)}
			\rput(-12,3){\psset{unit=0.75}
				\psline(2,-1)(-2,1)\psline(4,1)(-4,-1)}
			
			\rput(18,3){\psset{unit=0.75}
				\psline(2,-1)(-2,1)\psline(4,1)(-4,-1)}
			\rput(12,-3){\psset{unit=0.75}
				\psline(2,-1)(-2,1)\psline(4,1)(-4,-1)}
			
			\rput(0,1){
				\psbezier[linecolor=white,linewidth=\stringwhite](12,-3)(12,17)(18,23)(18,3)
				\psbezier[linewidth=\stringwidth]{c-c}(12,-3)(12,17)(18,23)(18,3)
				\psbezier[linecolor=white,linewidth=\stringwhite](-12,3)(-12,23)(-18,17)(-18,-3)
				\psbezier[linewidth=\stringwidth,ArrowInside=->,ArrowInsidePos=0.8]{c-c}(-12,3)(-12,23)(-18,17)(-18,-3)
				\psbezier[linewidth=\stringwidth,ArrowInside=->,ArrowInsidePos=0.8]{c-c}(-18,-3)(-18,17)(-12,23)(-12,3)
			}
			\rput(0,16){
				\psbezier[linewidth=\stringwidth,ArrowInside=->,ArrowInsidePos=0.8]{c-c}(15,0)(15,20.4)(-15,20.4)(-15,0)
				\psdot[linewidth=\stringwidth](0,15.2)
			}
			
			\rput(-32,-7.5){$\textcolor{darkgray}{\Sigma_g}$}
			\rput(-6,2){$q_{o_k}$}
			\rput(-24,-4){$p_{o_k}$}
			\rput(24,2){$q_{i_k}$}
			\rput(18,-4){$p_{i_k}$}
			\rput(0,35){index $3$}
			
			\rput[t](33,-4.5){$\textcolor{red}{\alpha}$}
			\rput[b](33,18.5){$\textcolor{blue}{\beta}$}
			\psline[linecolor=white,linewidth=\stringwhite](33,7.5)(33,17.5)
			\psline[linecolor=darkgray]{->}(33,7.5)(33,17.5)
			\rput(33,7.5){\psset{unit=0.5}
				\psline[linecolor=darkgray](2,-1)(-2,1)\psline[linecolor=darkgray](4,1)(-4,-1)}
			\end{pspicture}
			\caption{}\label{fig:HDBasepointspq}
		\end{subfigure}
	}
	\caption{A summary of our orientation conventions. The two schematic pictures show Heegaard diagrams near the basepoints corresponding to a closed (a) and an open (b) tangle component. If we think of the Heegaard diagram as a combinatorial description of a Morse function on the $\mathbb{Z}$-homology sphere $M$, the tangle components are the gradient flowlines connecting the basepoints $w_j$, $z_j$, $p_{o_k}/q_{o_k}$ and $p_{i_k}/q_{i_k}$ to the index~3 and index~0 critical points. The arrow in the top right corner of each picture shows a normal vector determining the orientation of $\Sigma_g$ via the right-hand rule; see Remark~\ref{rem:conventions1}.}\label{fig:HDBasepoints}
\end{figure}
\begin{definition}\label{def:RecallGradingsFromHDsForTangles}
	Given a domain $\phi\in\pi_2(\x,\y)$ between two generators $\x$ and $\y$ in $\mathbb{T}$ and a basepoint $x=p_i,q_i,z_j,w_j$, let $x(\phi)$ denote the multiplicity of $\phi$ at $x$. Then, we define three gradings on the generators of $\CFT(T)$: the \textbf{$\delta$-grading} $\delta$ is a relative $\frac{1}{2}\mathbb{Z}$-grading and defined by 
	\[\delta(\y)-\delta(\x)=\mu(\phi)-\sum_{i=1,2,3,4}\tfrac{1}{2}(p_i(\phi)+q_i(\phi))-\sum_{i=1}^n(z_i(\phi)+w_i(\phi)).\]
	When comparing this to~\cite[Definition~5.13]{HDsForTangles}, note that we now use peculiar Heegaard diagrams to compute the Maslov index $\mu$. 
	Furthermore, for every component $t$ of the tangle, there is a relative $\mathbb{Z}$-grading $A_t$, which is called the \textbf{Alexander grading}, see \cite[Definition~5.6]{HDsForTangles}. Given an ordered matching $P=\{\{i_1,o_1\},\{i_2,o_2\}\}$, we define $A_{t_k}$ for $k=1,2$ by 
	\[A_{t_k}(\y)-A_{t_k}(\x):=A_{t_k}(\phi):=p_{o_k}(\phi)+q_{o_k}(\phi)-p_{i_k}(\phi)-q_{i_k}(\phi).\]
	For $j=1,\dots,n$, we define $A_{t'_j}$ by
	\[A_{t'_j}(\y)-A_{t'_j}(\x):=A_{t'_j}(\phi):=2w_j(\phi)-2z_j(\phi).\]
	By taking the sum of all Alexander gradings, we obtain a relative $\mathbb{Z}$-grading, the \textbf{reduced Alexander grading} $\overline{A}$. 
	Finally, the \textbf{homological grading} $h$, a relative $\mathbb{Z}$-grading, is defined as
	\[h=\tfrac{1}{2}\overline{A}-\delta.\]
	We sometimes denote the Alexander grading on generators by a superscript list of integers (or half-integers, if eg, we want to achieve the same symmetry present in the decategorified invariants from~\cite{HDsForTangles}), like $a^\Red{+1}$ for the univariate or $a^{(\frac{3}{2},-\frac{1}{2})}$ for the multivariate grading. 
\end{definition}

\subsection{Peculiar algebras} 

\begin{definition}
	For $n\geq0$, let $\Rpren$ be the free polynomial ring generated by the variables $p_i$, $q_i$ and $U'_j$ and $V'_j$, where $i=1,2,3,4$ and $j=1,\dots,n$. Let $\Apre_n$ be the $\Id$-$\Id$-algebra whose underlying $\Id$-$\Id$-bimodule structure is given by  $\iota_{s'}\Apre_n.\iota_{s}:=\Rpren$ for pairs $(s,s')$ of sites
	and whose algebra multiplication is defined by the unique $\Id$-$\Id$-bimodule homomorphism $\Apre_n\otimes_{\Id}\Apre_n\rightarrow \Apre_n$ which, for all triples $(s,s',s'')$ of sites, restricts to the multiplication map in $\Rpren$:
	\[\underbrace{\iota_{s''}.\Apre_n.\iota_{s'}}_{\Rpren}\otimes_{\Id}\underbrace{\iota_{s'}.\Apre_n.\iota_{s}}_{\Rpren}\rightarrow \underbrace{\iota_{s''}.\Apre_n.\iota_{s}}_{\Rpren}.\]
	We define a $\tfrac{1}{2}\mathbb{Z}$-grading on $\Apre_n$, called the $\delta$-grading, by setting 
	\[\delta(\iota_i):=0,\quad \delta(p_i)=\delta(q_i):=\tfrac{1}{2}\quad\text{ and }\quad\delta(U'_j)=\delta(V'_j):=1,\]
	where $i=1,2,3,4$ and $j=1,\dots,n$, and then extending linearly to all of $\Apre_n$. Similarly, given an ordered matching $P=\{\{i_1,o_1\},\{i_2,o_2\}\}$, we define relative $\mathbb{Z}$-gradings $A_{t_k}$ for $k=1,2$, called Alexander gradings, by
	\[A_{t_k}(\iota_s):=0,\quad  A_{t_k}(p_{o_k})=A_{t_k}(q_{o_k}):=-1\quad\text{ and }\quad A_{t_k}(p_{i_k})=A_{t_k}(q_{i_k}):=1,\]
	and similarly Alexander gradings $A_{t'_j}$ for $j=1,\dots,n$ by
	\[A_{t'_j}(\iota_s):=0,\quad A_{t'_j}(U'_{j}):=-2\quad\text{ and }\quad A_{t'_j}(V'_{j}):=2,\]
	and then extend linearly to $\Apre_n$. These gradings give rise to a reduced Alexander grading and a homological grading as in Definition~\ref{def:RecallGradingsFromHDsForTangles}.
\end{definition}
\begin{definition}
	Let $\Rn$ be the free polynomial ring in the variables $U_i$, $U'_j$ and $V'_j$ for $i=1,2,3,4$ and $j=1,\dots,n$. Via the inclusion
	$$\Rn\hookrightarrow \Rpren,\quad U_i\mapsto p_iq_i,\quad U'_j\mapsto U'_j,\quad V'_j\mapsto V'_j,$$
	we can regard $\Apre_n$ as an $\Rn$-algebra. Let $\Aminus$ be the subalgebra of $\Apre_n$ generated as an $\Rn$-algebra by the idempotents in $\Id$ and 
	\[p_i:=\iota_{i-1}.p_i.\iota_i,\quad\text{ and }\quad q_i:=\iota_{i}.q_i.\iota_{i-1},\]
	where we take the indices $i=1,2,3,4$ modulo 4 with an offset of 1. Note that $p_iq_i=\iota_{i-1}.U_i$ and $q_ip_i=\iota_i.U_i$ as elements in $\Aminus$. Thus, any element in $\Aminus$ can be written uniquely as a sum of elements of the form 
	\[\iota_i.r, \quad p_ip_{i+1}\dots p_{k-1}p_{k}.r\quad\text{ and }\quad q_iq_{i-1}\dots q_{k+1}q_{k}.r,\]
	where $r\in \Rn$ is a monomial. This is the standard basis on $\Aminus$ as a vector space over $\mathbb{F}_2$. For convenience, we sometimes write the elements $p_ip_{i+1}\dots p_{k-1}p_{k}$ as $p_{i(i+1)\cdots (k-1)k}$ and $q_iq_{i-1}\dots q_{k+1}q_{k}$ as $q_{i(i-1)\cdots (k+1)k}$, where again, we take the indices modulo 4 with an offset of 1. Furthermore, to simplify notation, we set 
	\[
	p=p_1+p_2+p_3+p_4\in\Aminus
	\quad\text{ and }\quad
	q=q_1+q_2+q_3+q_4\in\Aminus,
	\]
	so we can write for example $p^4=p_{1234}+p_{2341}+p_{3412}+p_{4123}$. We call $\Aminus$ the \textbf{generalised peculiar algebra}. 
\end{definition}

\begin{wrapfigure}{r}{0.3333\textwidth}
	\centering
	{\vspace*{-5pt}
	$
	\begin{tikzcd}[row sep=0.5cm, column sep=0.5cm]
	&
	\overset{\iota_4}{\bullet}
	\arrow[bend left=10,leftarrow]{dr}[inner sep=1pt]{q_4}
	\arrow[bend right=10,swap]{dr}[inner sep=1pt]{p_4}
	&
	\\
	\!\!\!\raisebox{7pt}{$\underset{\iota_1}{~}$}\,\bullet
	\arrow[bend left=10,leftarrow]{ur}[inner sep=1pt]{q_1}
	\arrow[bend right=10,swap]{ur}[inner sep=1pt]{p_1}
	&
	&
	\bullet\,\raisebox{7pt}{$\underset{\iota_3}{~}$}\!\!\!
	\arrow[bend left=10,leftarrow]{dl}[inner sep=1pt]{q_3}
	\arrow[bend right=10,swap]{dl}[inner sep=1pt]{p_3}
	\\
	&
	\underset{\iota_2}{\bullet}
	\arrow[bend left=10,leftarrow]{ul}[inner sep=1pt]{q_2}
	\arrow[bend right=10,swap]{ul}[inner sep=1pt]{p_2}
	\end{tikzcd}
	\vspace{-5pt}
	$
	}
	\caption{The quiver for an alternative definition of~$\Ad$.}\label{fig:quiverForAd}\vspace*{-20pt}
\end{wrapfigure}
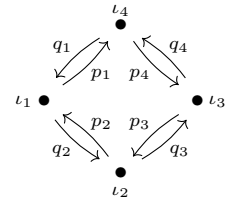
For most of this paper, we will only be concerned with a certain quotient of $\Aminus$, which was already defined in~\cite{MyThesis}.

\begin{definition}	
	Let $\Ad$ be the quotient of $\Aminus$ by the relations $U_i=0$, $U'_j=0$ and $V'_j=0$. This algebra can be interpreted as the path algebra of the quiver in Figure~\ref{fig:quiverForAd} 
	with relations $p_iq_i=0=q_ip_i$, see also Example~\ref{exa:pqModSpecialCaseofCC} and Figure~\ref{fig:nutshelll}. We call $\Ad$ the \textbf{peculiar algebra}. 
\end{definition}

\subsection{Peculiar modules}

For the differential in $\CFT$, we only consider holomorphic curves which stay away from the basepoints in our peculiar Heegaard diagrams. We claim that we also obtain a tangle invariant if we add those curves to our differential, recording their multiplicities at the basepoints by elements of the algebras $\Aminus$ or $\Ad$. However, the resulting complex does not satisfy the relation $\partial^2=0$. Instead, we obtain a slightly modified $\partial^2$-relation which enables us to promote $\CFT$ to more sophisticated homological invariants, namely certain \emph{curved} type D structures. As abstract algebraic structures, we defined curved type D structures in Example~\ref{exa:HighBrowDefTypeDoverI}. Let us recall this definition here in slightly more down-to-earth terms.

\begin{definition}\label{def:curvedTypeDStructure}
	Let $I$ be a ring of idempotents and $A$ a $\mathbb{Z}$-graded algebra over $I$. Also fix a central element $a_c\in Z(A)$ of degree $-2$. A \textbf{(right) curved type D structure} over $A$ is a $\mathbb{Z}$-graded $I$-module $M$ (with a preferred choice of basis) together with a (right) $I$-module homomorphism $\partial\co M\rightarrow M\otimes_I A$ of degree $-1$ satisfying 
	\[(1_M\otimes \mu)\circ(\partial\otimes 1_A)\circ\partial=1_M\otimes a_c,\]
	where $\mu$ denotes composition in $A$. We call $a_c$ the \textbf{curvature} of $M$. A morphism between two curved type D structures $(M,\partial_M)$ and $(N,\partial_N)$ is an $I$-module homomorphism \linebreak$M\rightarrow N\otimes_I A$. For two such morphisms $f$ and $g$, their composition is defined as 
	\[(g\circ f)=(1\otimes\mu)\circ(g\otimes 1_A)\circ f.\]
	We endow the space of morphisms $\Mor(M,N)$ with a differential~$D$ defined by
	\[D(f)= \partial_N\circ f+f\circ\partial_M.\]
	Then indeed $D^2=0$, since we have chosen $a_c$ to be central. This gives us an enriched category over $\Com$, the category of ordinary chain complexes over $\mathbb{F}_2$. The underlying ordinary category is obtained by restricting the morphism spaces to degree 0 elements in the kernel of $D$, giving us the usual notions of chain map and chain homotopy, see Definition~\ref{def:UnderlyingOrdinaryCat} and Example~\ref{exa:UnderlyingOrdCat}.
\end{definition}

\begin{remark}\label{rem:MatFakAsCurvedComplexes}
	It is interesting to compare curved type D structures to matrix factorisations as studied by Khovanov-Rozansky \cite{KhRoz}. Given an algebra~$A$ over some field $k$, a matrix factorisation of a potential $w\in A$ consists of two free $A$-modules $M_0$ and $M_1$ with two maps 
	\begin{equation}\label{eqn:matrixfac}
	\begin{tikzcd}[row sep=0.5cm, column sep=1cm]
	M_0
	\arrow[bend left=10]{r}[inner sep=1pt]{d_0}
	&
	M_1
	\arrow[bend left=10]{l}[inner sep=1pt]{d_1}
	\end{tikzcd}
	\end{equation}
	such that $d_1d_0=w.\id_{M_0}$ and $d_0d_1=w.\id_{M_1}$. If $\overline{M_0}$ and $\overline{M_1}$ denote the $k$-vector spaces generated by an $A$-basis of $M_0$ and $M_1$, respectively, we can regard $d_0$ and $d_1$ as maps 
	\[\overline{d_0}\co\overline{M_0}\rightarrow \overline{M_1}\otimes_IA\quad\text{ and }\quad\overline{d_1}\co\overline{M_1}\rightarrow \overline{M_0}\otimes_IA.\] 
	Then $(\overline{M_0}\oplus \overline{M_1},\overline{d_0}+\overline{d_1})$ defines a curved type D structure over the $k$-algebra $A$. 
	
	In general, we cannot go in the other direction. For example, curved complexes associated with manifolds with torus boundary do not, in general, admit a splitting of the form~\eqref{eqn:matrixfac}. This is for the simple reason that the total number of generators can be odd, see for example~\cite[Figure~5]{HRW}. However, for the curved type D structure invariants of tangles, such splittings exist, which is an easy corollary of the classification in terms of immersed curves on the 4-punctured sphere, see sections~\ref{sec:classification} and~\ref{sec:glueingrevisited}. 
\end{remark}

\begin{definition}
	Given an (ordered) matching $P=\{\{i_1,o_1\},\{i_2,o_2\}\}$, let $\pqMod:=\pqMod_P$ be the category of $\delta$-graded (and Alexander graded) curved complexes over $\Ad$ with curvature \[p^4+q^4.\] We call the objects of this category \textbf{peculiar modules}. Furthermore, let $\gpqMod_{P,n}$ be the category of $\delta$-graded (and Alexander graded) curved complexes over $\Aminus$ with curvature
	\begin{equation}\label{eqn:generalcurvature}
	p^4+q^4+U_{i_1}U_{o_1}+U_{i_2}U_{o_2}.
	\end{equation}
	We call the objects of this category \textbf{generalised peculiar modules}. 
\end{definition}

\begin{definition}\label{def:CFTd}
Given a 4-ended tangle $T$ with $n$ closed components in a $\mathbb{Z}$-homology 3-ball $M$ with spherical boundary and an (admissible) peculiar Heegaard diagram for $T$, let us define a generalised peculiar module $\CFTminus(T):=\CFTminus(T,M)$ in $\gpqMod_{P_T,n}$ whose underlying relatively bigraded right $\Id$-module agrees with $\CFT(T)$. However, the differential $\partial$ on $\CFTminus(T)$ is defined by 
\begin{equation}\label{eqn:differentialOnCFTd}
\partial \x=\sum_{y\in\mathbb{T}}\sum_{\substack{\phi\in\pi_2(\x,\y)\\ \mu(\phi)=1}}\#\left(\frac{\mathcal{M}(\phi)}{\mathbb{R}}\right) \cdot \y\otimes_{\Id} a(\phi),
\end{equation}
where for $\phi\in\pi_2(\x,\y)$, $a(\phi)$ is the preimage of 
\[\iota_{s(\y)}.\prod_{i=1,2,3,4} p_i^{p_i(\phi)}\cdot q_i^{q_i(\phi)}\cdot\prod_{j=1}^n (U'_j)^{w_j(\phi)}\cdot (V'_j)^{z_j(\phi)}.\iota_{s(\x)},\]
under the inclusion map $\Aminus\hookrightarrow\Apre_n$. We call $(\CFTminus(T),\partial)$ the \textbf{generalised peculiar module of~$T$}. Its image under the functor 
\[\gpqMod_{P_T,n}\rightarrow\pqMod\]
induced by the quotient map $\Aminus\rightarrow\Ad$ is denoted by $\CFTd(T)$, which we call the \textbf{peculiar module} of $T$.
\end{definition}
\begin{theorem}\label{thm:PecMod}
	\(\CFTminus(T)\) is indeed a well-defined generalised peculiar module. Furthermore, its relatively bigraded chain homotopy type is an invariant of the tangle~\(T\). Hence \(\CFTd(T)\) is a well-defined peculiar module, whose relatively bigraded chain homotopy type is also an invariant of the tangle~\(T\).
\end{theorem}

\begin{remark}\label{rmk:RelativeGradings}
	On link Floer homology, one can promote both Alexander and $\delta$-gradings to absolute gradings via symmetries and the spectral sequence to $\HF(S^3)$~\cite{OSHFL}. I expect that something similar can be done for our tangle invariants. Alternatively, one could simply fix absolute gradings on a specific test tangle, say a trivial tangle, and then define absolute gradings on all other tangles via the pairing with this test tangle (using Theorem~\ref{thm:CFTdGlueingAsMorphism}) and the absolute gradings on $\HFL$. However, in this paper, we are only working with the relative gradings on $\CFTd$ and $\CFTminus$ inherited from those on $\CFT$ which were defined in~\cite{HDsForTangles}. So, throughout this paper, all gradings on $\CFL$ should be regarded as relative, too. 
\end{remark}

\begin{remark}\label{rem:ComparisonOS}
	The generalized algebra $\Aminus$ and the generalised invariant $\CFTminus$ are inspired by Ozsv\'{a}th and Szab\'{o}'s algebra and tangle invariant from~\cite{OSKauffmanStates2}. Computations suggest that their invariant for one-sided 4-ended tangles is closely related to $\CFTminus$. Conceptually, it might also be interesting to set up their theory for an odd number of tangle strands and then compare $\CFTminus$ to their invariants of $(1,3)$-tangles. 
\end{remark}

\begin{lemma}
	For any \(\phi\in\pi_2(\x,\y)\), \(a(\phi)\) lies in the image of inclusion map \(\Aminus\hookrightarrow\Apre_n\).
\end{lemma}
\begin{proof}
	This follows from the observation that $\partial\phi$ intersected with the $\alpha$-circle ${\red S^1}$ is a path on ${\red S^1}$ connecting the two points of $\x$ and $\y$ on ${\red S^1}$. 
\end{proof}

\begin{lemma}\label{lem:CFTdGradings}
\(\partial\) increases the \(\delta\)-grading by 1 and preserves the Alexander grading. (As~usual, the grading on a tensor product is given by the sum of the gradings of the tensor factors.)
\end{lemma}
\begin{proof}
Both statements follows directly from the definitions of the gradings of generators and algebra elements. 
\end{proof}
\begin{lemma}
For each \(\x\in\CFTminus\), the sum on the right-hand side of  \eqref{eqn:differentialOnCFTd} is finite.
\end{lemma}
\begin{proof}
The proof is essentially the same as in link Floer homology, see \cite[Lemma~4.2]{OSHFL}. Since $\CFTminus$ is finitely generated, it is sufficient to show that the coefficient of each $\y\in\CFTminus$ is a finite sum. Note that by the previous lemma, the difference of the $\delta$-gradings of $\x$ and $\y$ determines the $\delta$-grading of $a(\phi)$. Thus, there are only finitely many choices for the coefficients $a(\phi)$. So let us also fix the multiplicities of $\phi$ at the basepoints. We can now argue as in the proof of \cite[Lemma~4.13]{OSHF3mfds}, using admissibility of the underlying Heegaard diagram.
\end{proof}
\pagebreak[3]
In the following, we need two analytical facts from \cite{OSHFL}.
\begin{fact}\label{fact:OSHFL_lemma_5_4}
\cite[Lemma~5.4]{OSHFL} Given a homology class \(\phi\) of an \(\alpha\)-injective boundary degeneration, write \(\phi\) as a linear combination of connected components of \(\Sigma\smallsetminus \A\). Then its Maslov index \(\mu(\phi)\) is equal to twice the sum of the coefficients. The same holds for \(\beta\)-injective boundary degenerations. 
\end{fact}
\begin{fact}\label{fact:OSHFL_lemma_5_5}
	\cite[Theorem~5.5]{OSHFL} Let \(\Sigma\) be a surface of genus \(g\), equipped with a set \(\A\) of \((g+r)\) attaching circles for a handlebody, where $r\geq1$.
	Let \(\phi\) be a homology class of boundary degenerations whose domain is non-negative and whose Maslov index is equal to 2. Then the domain of \(\phi\) is equal to one of the connected components of \(\Sigma\smallsetminus \A\).
	Moreover, the number of pseudo-holomorphic boundary degenerations (considered up to reparametrization) whose domains are equal to the same connected component of \(\Sigma\smallsetminus \A\) is odd. The same holds for \(\beta\)-injective boundary degenerations.
\end{fact}
\begin{proof}[of Theorem~\ref{thm:PecMod}]
Checking the $\partial^2$-identity is analogous to the link case; we can follow \cite[proof of Lemma~4.3]{OSHFL} and count ends of moduli spaces of Maslov index 2 curves. We fix two generators $\x$ and $\z$ and consider the disjoint union of moduli spaces $\mathcal{M}(\phi)$, where $\phi$ varies over those curves in $\pi_2(\x,\z)$ with $\mu(\phi)=2$ and $a(\phi)=a$ for some fixed $a\in \Aminus$. (In particular, this fixes the multiplicities of $\phi$ at the basepoints.)
If there are no boundary degenerations, there is an even number of ends, so the $\z\otimes a$-component of $\partial^2 \x$ vanishes. If there are boundary degenerations, then by Fact~\ref{fact:OSHFL_lemma_5_4} above, they contribute at least 2 to the Maslov index, so the remaining curve has to be constant, hence $\x=\z$. By Fact~\ref{fact:OSHFL_lemma_5_5}, we get a boundary degeneration for each component of $\Sigma\smallsetminus\A$ and $\Sigma\smallsetminus\B$. For closed tangle components, these boundary degenerations come in pairs which cancel each other. The remaining two $\alpha$-injective boundary degenerations contribute the first two terms of~\eqref{eqn:generalcurvature} and the remaining two $\beta$-injective boundary degenerations contribute the last two terms. All other ends appear in pairs again, so their contributions cancel. 

It remains to show that the peculiar module is an invariant of the tangle $T$. However, $\CFTminus(T)$ is essentially the chain complex associated with a multi-pointed Heegaard diagram in Heegaard Floer theory. Thus, we obtain invariance as a (curved) type D structure over the free polynomial ring~$\Rpren$. However, the same proof also works if we work over $\Aminus$, since we only allow handleslides of $\alpha$-curves over $\alpha$-curves other than the special $\alpha$-circle $\red S^1$. 
\end{proof}

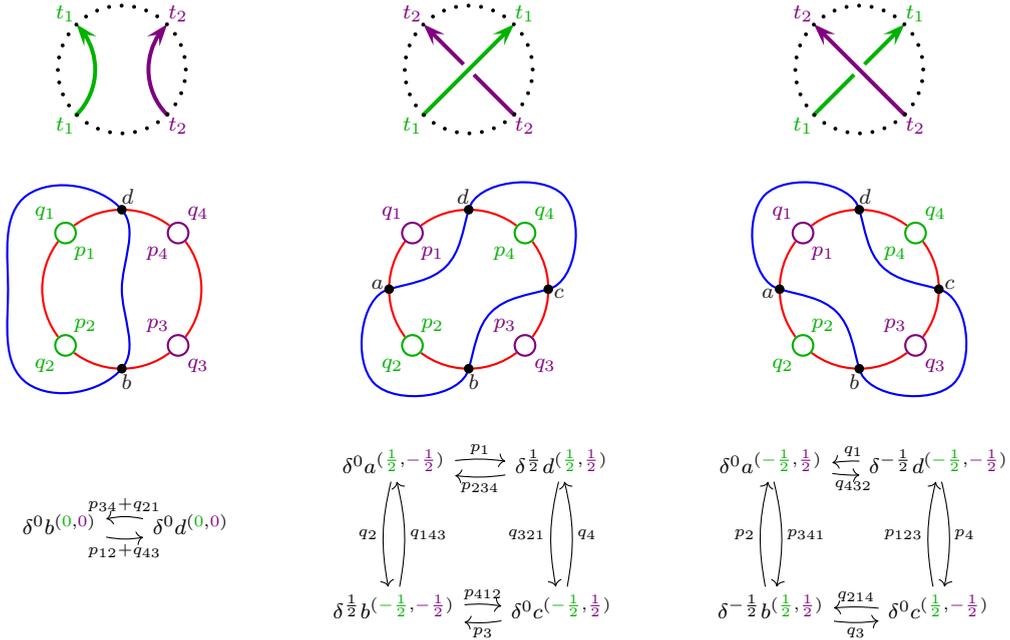
\begin{figure}
\centering
\psset{unit=0.15}
\begin{subfigure}[b]{0.26\textwidth}\centering
\psset{unit=4}
\begin{pspicture}(-1.5,-1.5)(1.5,1.5)
\psarc[linecolor=violet, linewidth=\stringwidth]{<-c}(2,0){1.41432}{-225}{-135}
\psarc[linecolor=darkgreen, linewidth=\stringwidth]{c->}(-2,0){1.41432}{-45}{45}
\pscircle[linestyle=dotted, linewidth=\stringwidth](0,0){1.45}
\rput[c](1.75;135){$\textcolor{darkgreen}{t_1}$}
\rput[c](1.75;-135){$\textcolor{darkgreen}{t_1}$}
\rput[c](1.75;45){$\textcolor{violet}{t_2}$}
\rput[c](1.75;-45){$\textcolor{violet}{t_2}$}
\end{pspicture}
\end{subfigure}
\quad
\begin{subfigure}[b]{0.32\textwidth}\centering
\psset{unit=4}
\begin{pspicture}(-1.5,-1.5)(1.5,1.5)
\psline[linecolor=violet, linewidth=\stringwidth]{c->}(1,-1)(-1,1)
\psline[linecolor=white, linewidth=\stringwhite](-1,-1)(1,1)
\psline[linecolor=darkgreen, linewidth=\stringwidth]{c->}(-1,-1)(1,1)
\pscircle[linestyle=dotted, linewidth=\stringwidth](0,0){1.45}
\rput[c](1.75;45){$\textcolor{darkgreen}{t_1}$}
\rput[c](1.75;-135){$\textcolor{darkgreen}{t_1}$}
\rput[c](1.75;135){$\textcolor{violet}{t_2}$}
\rput[c](1.75;-45){$\textcolor{violet}{t_2}$}
\end{pspicture}
\end{subfigure}
\quad
\begin{subfigure}[b]{0.34\textwidth}\centering
\psset{unit=4}
\begin{pspicture}(-1.5,-1.5)(1.5,1.5)
\psline[linecolor=darkgreen, linewidth=\stringwidth]{c->}(-1,-1)(1,1)
\psline[linecolor=white, linewidth=\stringwhite](1,-1)(-1,1)
\psline[linecolor=violet, linewidth=\stringwidth]{c->}(1,-1)(-1,1)
\pscircle[linestyle=dotted, linewidth=\stringwidth](0,0){1.45}
\rput[c](1.75;45){$\textcolor{darkgreen}{t_1}$}
\rput[c](1.75;-135){$\textcolor{darkgreen}{t_1}$}
\rput[c](1.75;135){$\textcolor{violet}{t_2}$}
\rput[c](1.75;-45){$\textcolor{violet}{t_2}$}
\end{pspicture}
\end{subfigure}
\bigskip\\
\begin{subfigure}[b]{0.26\textwidth}\centering
\begin{pspicture}(-10.3,-10.3)(10.3,10.3)
\psrotate(0,0){90}{
\psecurve[linecolor=blue](7,0)(0,0)(-7,0)(-8.3,8.3)(0,10)(8.3,8.3)(7,0)(0,0)(-7,0)
}

\psarc[linecolor=red](0,0){7}{0}{360}

\rput(7;45){\pscircle*[linecolor=white]{1}\pscircle[linecolor=violet]{1}}
\rput(7;135){\pscircle*[linecolor=white]{1}\pscircle[linecolor=darkgreen]{1}}
\rput(7;225){\pscircle*[linecolor=white]{1}\pscircle[linecolor=darkgreen]{1}}
\rput(7;315){\pscircle*[linecolor=white]{1}\pscircle[linecolor=violet]{1}}

\psdot(0,7)
\psdot(0,-7)

\rput[c](4.5;135){$\textcolor{darkgreen}{p_1}$}
\rput[c](9.5;135){$\textcolor{darkgreen}{q_1}$}
\rput[c](4.5;-135){$\textcolor{darkgreen}{p_2}$}
\rput[c](9.5;-135){$\textcolor{darkgreen}{q_2}$}

\rput[c](4.5;45){$\textcolor{violet}{p_4}$}
\rput[c](9.5;45){$\textcolor{violet}{q_4}$}
\rput[c](4.5;-45){$\textcolor{violet}{p_3}$}
\rput[c](9.5;-45){$\textcolor{violet}{q_3}$}

\rput[bl](0,7.5){$d$}
\rput[tl](0,-7.5){$b$}
\end{pspicture}
\end{subfigure}
\quad
\begin{subfigure}[b]{0.32\textwidth}\centering
\begin{pspicture}(-10.3,-10.3)(10.3,10.3)
\psrotate(0,0){90}{

\psecurve[linecolor=blue](8,-0.2)(7,0)(2,2)(0,7)(-0.2,8)
\psecurve[linecolor=blue](-4,4)(0,7)(-8.3,8.3)(-7,0)(-4,4)
\psecurve[linecolor=blue](-8,0.2)(-7,0)(-2,-2)(0,-7)(0.2,-8)
\psecurve[linecolor=blue](4,-4)(0,-7)(8.3,-8.3)(7,0)(4,-4)

\psarc[linecolor=red](0,0){7}{0}{360}

\rput(7;45){\pscircle*[linecolor=white]{1}\pscircle[linecolor=violet]{1}}
\rput(7;135){\pscircle*[linecolor=white]{1}\pscircle[linecolor=darkgreen]{1}}
\rput(7;225){\pscircle*[linecolor=white]{1}\pscircle[linecolor=violet]{1}}
\rput(7;315){\pscircle*[linecolor=white]{1}\pscircle[linecolor=darkgreen]{1}}

}
\psdot(0,7)
\psdot(7,0)
\psdot(0,-7)
\psdot(-7,0)

\rput[c](4.5;135){$\textcolor{violet}{p_1}$}
\rput[c](9.5;135){$\textcolor{violet}{q_1}$}
\rput[c](4.5;-135){$\textcolor{darkgreen}{p_2}$}
\rput[c](9.5;-135){$\textcolor{darkgreen}{q_2}$}

\rput[c](4.5;45){$\textcolor{darkgreen}{p_4}$}
\rput[c](9.5;45){$\textcolor{darkgreen}{q_4}$}
\rput[c](4.5;-45){$\textcolor{violet}{p_3}$}
\rput[c](9.5;-45){$\textcolor{violet}{q_3}$}

\rput[br](0,7.5){$d$}
\rput[tl](7.5,0){$c$}
\rput[tl](0,-7.5){$b$}
\rput[br](-7.5,0){$a$}
\end{pspicture}
\end{subfigure}
\quad
\begin{subfigure}[b]{0.34\textwidth}\centering
\begin{pspicture}(-10.3,-10.3)(10.3,10.3)
\psecurve[linecolor=blue](8,-0.2)(7,0)(2,2)(0,7)(-0.2,8)
\psecurve[linecolor=blue](-4,4)(0,7)(-8.3,8.3)(-7,0)(-4,4)
\psecurve[linecolor=blue](-8,0.2)(-7,0)(-2,-2)(0,-7)(0.2,-8)
\psecurve[linecolor=blue](4,-4)(0,-7)(8.3,-8.3)(7,0)(4,-4)

\psarc[linecolor=red](0,0){7}{0}{360}

\rput(7;45){\pscircle*[linecolor=white]{1}\pscircle[linecolor=darkgreen]{1}}
\rput(7;135){\pscircle*[linecolor=white]{1}\pscircle[linecolor=violet]{1}}
\rput(7;225){\pscircle*[linecolor=white]{1}\pscircle[linecolor=darkgreen]{1}}
\rput(7;315){\pscircle*[linecolor=white]{1}\pscircle[linecolor=violet]{1}}

\psdot(0,7)
\psdot(7,0)
\psdot(0,-7)
\psdot(-7,0)

\rput[c](4.5;135){$\textcolor{violet}{p_1}$}
\rput[c](9.5;135){$\textcolor{violet}{q_1}$}
\rput[c](4.5;-135){$\textcolor{darkgreen}{p_2}$}
\rput[c](9.5;-135){$\textcolor{darkgreen}{q_2}$}

\rput[c](4.5;45){$\textcolor{darkgreen}{p_4}$}
\rput[c](9.5;45){$\textcolor{darkgreen}{q_4}$}
\rput[c](4.5;-45){$\textcolor{violet}{p_3}$}
\rput[c](9.5;-45){$\textcolor{violet}{q_3}$}

\rput[bl](0,7.5){$d$}
\rput[bl](7.5,0){$c$}
\rput[tr](0,-7.5){$b$}
\rput[tr](-7.5,0){$a$}
\end{pspicture}
\end{subfigure}
\bigskip\\
\begin{subfigure}[b]{0.26\textwidth}\centering
$\begin{tikzcd}[row sep=1.4cm, column sep=0.5cm]
\delta^{0}b^{(\textcolor{darkgreen}{0},\textcolor{violet}{0})}
\arrow[leftarrow,bend left=12]{r}{p_{34}+q_{21}}
& 
\delta^{0}d^{(\textcolor{darkgreen}{0},\textcolor{violet}{0})}
\arrow[leftarrow,bend left=10]{l}{p_{12}+q_{43}}
\end{tikzcd}$
\end{subfigure}
\quad
\begin{subfigure}[b]{0.32\textwidth}\centering
$\begin{tikzcd}[row sep=1.4cm, column sep=0.5cm]
\delta^{0}a^{(\textcolor{darkgreen}{\frac{1}{2}},\textcolor{violet}{-\frac{1}{2}})}
\arrow[leftarrow,bend right=7,swap]{r}{p_{234}}
\arrow[leftarrow,bend left=12]{d}{q_{143}}
& 
\delta^{\frac{1}{2}}d^{(\textcolor{darkgreen}{\frac{1}{2}},\textcolor{violet}{\frac{1}{2}})}
\arrow[leftarrow,bend right=7,swap]{l}{p_1}
\arrow[leftarrow,bend left=12]{d}{q_4}
\\
\delta^{\frac{1}{2}}b^{(\textcolor{darkgreen}{-\frac{1}{2}},\textcolor{violet}{-\frac{1}{2}})}
\arrow[leftarrow,bend right=7,swap]{r}{p_3}
\arrow[leftarrow,bend left=12]{u}{q_2}
&
\delta^{0}c^{(\textcolor{darkgreen}{-\frac{1}{2}},\textcolor{violet}{\frac{1}{2}})}
\arrow[leftarrow,bend right=7,swap]{l}{p_{412}}
\arrow[leftarrow,bend left=12]{u}{q_{321}}
\end{tikzcd}$
\end{subfigure}
\quad
\begin{subfigure}[b]{0.34\textwidth}\centering
$\begin{tikzcd}[row sep=1.4cm, column sep=0.35cm]
\delta^{0}a^{(\textcolor{darkgreen}{-\frac{1}{2}},\textcolor{violet}{\frac{1}{2}})}
\arrow[leftarrow,bend left=5,pos=0.8]{r}{q_1}
\arrow[leftarrow,bend right=12,swap]{d}{p_2}
& 
\delta^{-\frac{1}{2}}d^{(\textcolor{darkgreen}{-\frac{1}{2}},\textcolor{violet}{-\frac{1}{2}})}
\arrow[leftarrow,bend left=5,pos=0.2]{l}{q_{432}}
\arrow[leftarrow,bend right=12,swap]{d}{p_{123}}
\\
\delta^{-\frac{1}{2}}b^{(\textcolor{darkgreen}{\frac{1}{2}},\textcolor{violet}{\frac{1}{2}})}
\arrow[leftarrow,bend left=5]{r}{q_{214}}
\arrow[leftarrow,bend right=12,swap]{u}{p_{341}}
&
\delta^{0}c^{(\textcolor{darkgreen}{\frac{1}{2}},\textcolor{violet}{-\frac{1}{2}})}
\arrow[leftarrow,bend left=7]{l}{q_3}
\arrow[leftarrow,bend right=12,swap]{u}{p_4}
\end{tikzcd}$
\end{subfigure}
\caption{Basic rational tangles, their Heegaard diagrams and peculiar modules. The superscripts of the generators specify the Alexander grading. Compare this to \protect\cite[Figure~17]{HDsForTangles}. }\label{fig:CFTdForSomeRatTangles}
\end{figure}

\begin{sidewaysfigure}[p]
\vspace*{435pt}
\centering
\begin{subfigure}[b]{0.35\textwidth}\centering
\psset{unit=0.8, linewidth=1.1pt}
{
\begin{pspicture}[showgrid=false](-5.2,-3.1)(3.2,3.1)
	\psset{linewidth=\stringwidth}
	\psecurve[linecolor=violet]{c-c}(-2.5,1.5)(0,2)(0.75,1)(-0.75,-1)(0,-2)(0.97,-2.24)(2,-2)
	\psecurve[linecolor=violet]{<-c}(2,2)(0.97,2.24)(0,2)(-0.75,1)(0.75,-1)(0,-2)(-2.5,-1.5)(-3.25,0)(-2.5,1.5)(0,2)(0.75,1)
	\psecurve[linecolor=darkgreen]{<-c}(-6,1.5)(-3.3,1.85)(-2.5,1.5)(-1.85,0)(-2.5,-1.5)(-3.3,-1.85)(-6,-1.5)
	
	\psecurve[linecolor=white,linewidth=\stringwhite](0.75,-1)(0,-2)(-2.5,-1.5)(-3.25,0)(-2.5,1.5)
	\psecurve[linecolor=violet]{c-c}(0.75,-1)(0,-2)(-2.5,-1.5)(-3.25,0)(-2.5,1.5)
	
	\psecurve[linecolor=white,linewidth=\stringwhite](0.75,1)(-0.75,-1)(0,-2)(0.97,-2.24)(2,-2)
	\psecurve[linecolor=white,linewidth=\stringwhite](0,2)(-0.75,1)(0.75,-1)(0,-2)
	\psecurve[linecolor=white,linewidth=\stringwhite](-2.5,-1.5)(-3.25,0)(-2.5,1.5)(0,2)(0.75,1)(-0.75,-1)
	
	\psecurve[linecolor=violet]{c-c}(0.75,1)(-0.75,-1)(0,-2)(0.97,-2.24)(2,-2)
	\psecurve[linecolor=violet]{c-c}(0,2)(-0.75,1)(0.75,-1)(0,-2)
	\psecurve[linecolor=violet]{c-c}(-2.5,-1.5)(-3.25,0)(-2.5,1.5)(0,2)(0.75,1)(-0.75,-1)
	
	\psecurve[linecolor=white,linewidth=\stringwhite](-2.5,1.5)(-1.85,0)(-2.5,-1.5)(-3.3,-1.85)(-6,-1.5)
	\psecurve[linecolor=darkgreen]{c-c}(-2.5,1.5)(-1.85,0)(-2.5,-1.5)(-3.3,-1.85)(-6,-1.5)

%
%
%
%

\psline[linecolor=violet]{<-}(-3.25,-0.1)(-3.25,0.1)
\pscircle[linestyle=dotted](-1,0){3.05}

\uput{0.2}[45](0.97,2.24){$\textcolor{violet}{t_2}$}
\uput{0.2}[-45](0.97,-2.24){$\textcolor{violet}{t_2}$}
\uput{0.2}[135](-3.3,1.85){$\textcolor{darkgreen}{t_1}$}
\uput{0.2}[-135](-3.3,-1.85){$\textcolor{darkgreen}{t_1}$}

\uput{2.5}[180](-1,0){$\textcolor{red}{a}$}
\uput{2.3}[-90](-1,0){$\textcolor{blue}{b}$}
\uput{2.1}[0](-1,0){$\textcolor{darkgreen}{c}$}
\uput{2.3}[90](-1,0){$\textcolor{gold}{d}$}

\end{pspicture}
}
\caption{A tangle diagram for the $(2,-3)$-pretzel tangle.}\label{fig:HFTdmutationpretzeltangleT}
\end{subfigure}
\quad
\begin{subfigure}[b]{0.6\textwidth}\centering
\psset{unit=0.25}
{
\begin{pspicture}(-24,-11)(24,11)
\rput(-12,0){\psrotate(0,0){-90}{

\psarc[linecolor=red](0,0){8}{0}{360}

\SpecialCoor
\rput(8;45){\pscircle*[linecolor=white]{1}\pscircle[linecolor=violet]{1}}
\rput(8;135){\pscircle*[linecolor=white]{1}\pscircle[linecolor=violet]{1}}
\rput(8;225){\pscircle*[linecolor=white]{1}\pscircle[linecolor=darkgreen]{1}}
\rput(8;315){\pscircle*[linecolor=white]{1}\pscircle[linecolor=darkgreen]{1}}


\psecurve[linecolor=blue]%
(10;130)(10;140)(8;155)%
(8;-95)(11;-45)%
(10.5;45)(8;70)%
(8;-10)%
(10;-45)(8;-70)%
(8;120)(10;130)(10;140)(8;155)%

\psdot(8;155)%
\psdot(8;120)%
\psdot(8;70)%
\psdot(8;-10)%
\psdot(8;-70)%
\psdot(8;-95)%
}
\rput[b](8.7;65){$d$}
\rput[l](8;30){~$x_1$}
\rput[l](8;-20){~$x_2$}
\rput[b](7.3;-100){$b$}
\rput[l](8;-160){~$a_2$}
\rput[l](8;-185){~$a_1$}

\footnotesize
\rput[t](7.1;110){$\textcolor{darkgreen}{p_1}$}
\rput[t](9.8;110){$\textcolor{darkgreen}{q_1}$}
\rput[c](8.8;-123){$\textcolor{darkgreen}{q_2}$}
\rput[l](2;-90){$\textcolor{darkgreen}{p_2}$}
\rput[c](3.5;90){\texttt{4}}
\rput[c](7;-60){\texttt{2}}
\rput[c](9;-80){\texttt{1}}
\rput[c](9.5;42){\texttt{3}}
}

\rput(12,0){\psrotate(0,0){-90}{

\psarc[linecolor=red](0,0){8}{0}{360}

\SpecialCoor
\rput(8;45){\pscircle*[linecolor=white]{1}\pscircle[linecolor=violet]{1}}
\rput(8;135){\pscircle*[linecolor=white]{1}\pscircle[linecolor=violet]{1}}
\rput(8;225){\pscircle*[linecolor=white]{1}\pscircle[linecolor=violet]{1}}
\rput(8;315){\pscircle*[linecolor=white]{1}\pscircle[linecolor=violet]{1}}


\psecurve[linecolor=blue]%
(10;-130)(10;-140)(8;-155)%
(8;95)(11;45)%
(10.5;-45)(8;-80)%
(8;60)
(9.5;45)(8;20)
(8;-60)(9.5;-45)(10;45)
(8;75)%
(5.5;100)%
(8;-120)(10;-130)(10;-140)(8;-155)%

\psdot(8;-155)%
\psdot(8;-120)%
\psdot(8;-80)%
\psdot(8;-60)%
\psdot(8;20)%
\psdot(8;60)%
\psdot(8;75)%
\psdot(8;95)%
}
\rput[b](8.7;-245){$d'$}
\rput[r](8;-210){$y_1$~}
\rput[r](8;-170){$y_2$~}
\rput[b](7;-154){$y_3$}
\rput[b](7.3;-70){$b'$}
\rput[r](8;-29){$c_3$~}
\rput[r](8;-15){$c_2$~}
\rput[l](8;5){~$c_1$}

\footnotesize
\rput[t](7.5;70){$\textcolor{violet}{p_4}$}
\rput[t](9.8;70){$\textcolor{violet}{q_4}$}
\rput[t](3;-90){$\textcolor{violet}{p_3}$}
\rput[l](9.5,-6){~$\textcolor{violet}{q_3}$}
\rput[l](9.5,-7.5){~\texttt{1}}
\rput[l](9.5,-9){~\texttt{5}}
\psline[linestyle=dotted,dotsep=1pt](9.1;-45)(9.5,-6)
\psline[linestyle=dotted,dotsep=1pt](8.7;-80)(9.5,-7.5)
\psline[linestyle=dotted,dotsep=1pt](9.9;-80)(9.5,-9)
\rput[c](4;90){\texttt{4}}
\rput[c](1;90){\texttt{6}}
\rput[c](7;-120){\texttt{2}}
\rput[c](9.5;138){\texttt{3}}
}

\psline[linestyle=dashed](-6.25,5.65685424949236)(6.25,5.65685424949236)

\psline[linestyle=dashed](-6.25,-5.65685424949236)(6.25,-5.65685424949236)

\end{pspicture}
}
\caption{A Heegaard diagram for the tangle on the left.}
\label{fig:HFTdmutationpretzeltangleHD}
\end{subfigure}
\\
\begin{subfigure}[b]{\textwidth}
	\centering
\includegraphics[height=200pt]{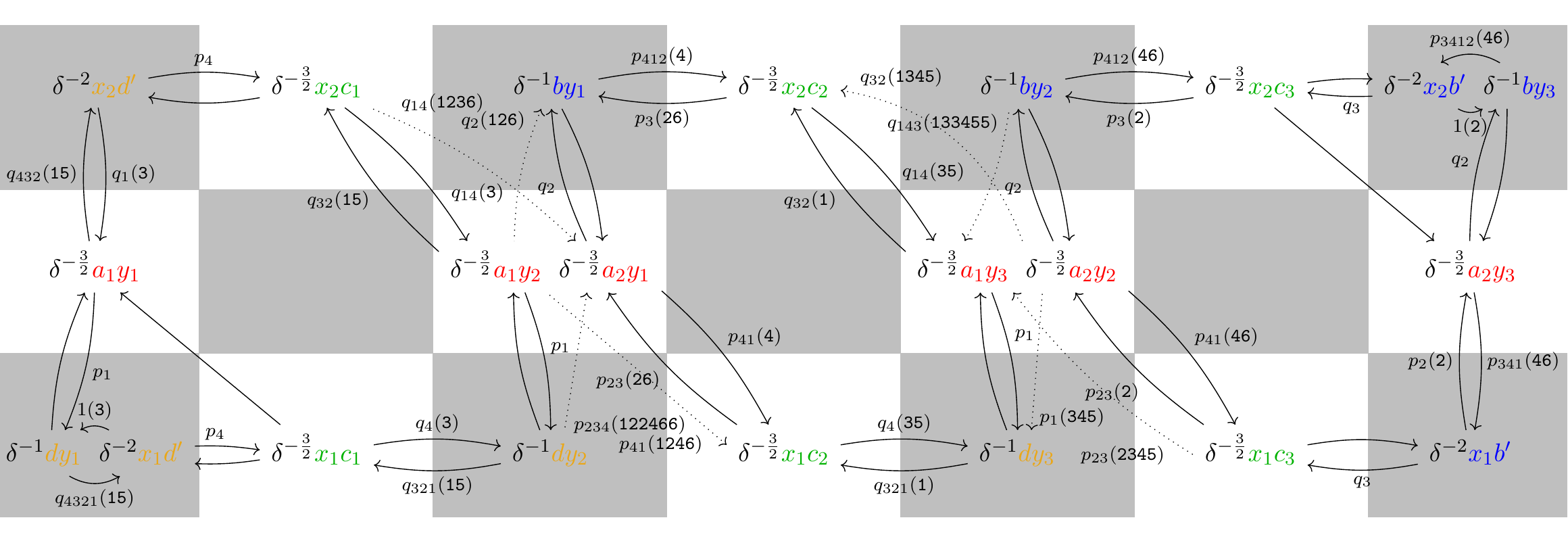}
\caption{The peculiar module for the pretzel tangle above.}\label{fig:firstcomplex}
\end{subfigure}
\caption{The computation of a peculiar module for a non-rational tangle, see example~\ref{exa:HFTdpretzeltangle}.}\label{fig:HFTdmutationexample}
\end{sidewaysfigure}

\begin{example}[(rational tangles)]\label{exa:CFTdRatTang}
	Figure~\ref{fig:CFTdForSomeRatTangles} shows the peculiar modules of some very simple 4-ended tangles. As shown in \cite[Example~4.3]{HDsForTangles}, every rational tangle $T$ has a tangle Heegaard diagram with just a single $\beta$-curve. Thus, we only count bigons in the differential of the peculiar invariant, and only those that do not occupy both~$p_i$ and~$q_j$. By tightening the $\beta$-curve, we can assume that there are no honest differentials in $\CFTd(T)$, ie that every bigon covers some $p_i$ or $q_i$. Then $\CFTd(T)$ can be read off from this single $\beta$-curve as follows: the vertices of the graph of $\CFTd(T)$ correspond to intersection points of this $\beta$-curve with the $\alpha$-arcs. Its arrows come in pairs labelled by powers of $p$ or $q$. More precisely, for each component of the $\beta$-curve minus the $\alpha$-arcs, we obtain an arrow pair connecting the vertices corresponding to the ends of this component; this arrow pair is labelled by powers of $p$ if the component goes via the front component of the 4-punctured sphere minus the $\alpha$-arcs and by powers of $q$ otherwise. Conversely, the $\beta$-curve can be read off from $\CFTd(T)$. So in this case, we might actually view the $\beta$-curve as the invariant associated with the tangle. 
\end{example}

\begin{figure}[t]\centering
\begin{subfigure}[b]{\textwidth}\centering
\psset{unit=1.5}
{
\begin{pspicture}(-3.6,-1.6)(3.6,1.6)

{\psset{linecolor=lightgray}

\psframe*(-2.5,0.5)(-3.5,1.5)
\psframe*(-0.5,0.5)(-1.5,1.5)
\psframe*(0.5,0.5)(1.5,1.5)
\psframe*(2.5,0.5)(3.5,1.5)

\psframe*(-1.5,-0.5)(-2.5,0.5)
\psframe*(-0.5,-0.5)(0.5,0.5)
\psframe*(1.5,-0.5)(2.5,0.5)

\psframe*(-2.5,-0.5)(-3.5,-1.5)
\psframe*(-0.5,-0.5)(-1.5,-1.5)
\psframe*(0.5,-0.5)(1.5,-1.5)
\psframe*(2.5,-0.5)(3.5,-1.5)
}

\psline(3,0)(3,-1)
\psline(2,1)(1,1)
\psline(0,1)(-1,1)
\psline(-2,1)(-3,1)
\psline(2,-1)(1.25,0)
\psline(1,-1)(0.75,0)
\psline(0,-1)(-0.75,0)
\psline(-1,-1)(-1.25,0)
\psline(-2,-1)(-3,0)

{\psset{linestyle=dashed,dash=4pt 2pt}
\psline(-3,0)(-3,1)
\psline(-2,-1)(-1,-1)
\psline(0,-1)(1,-1)
\psline(2,-1)(3,-1)
\psline(-2,1)(-1.25,0)
\psline(-1,1)(-0.75,0)
\psline(0,1)(0.75,0)
\psline(1,1)(1.25,0)
\psline(2,1)(3,0)
}

\psdots[linecolor=red](-3,0)(-1.25,0)(-0.75,0)(0.75,0)(1.25,0)(3,0)
\psdots[linecolor=blue](-1,1)(1,1)(3,-1)
\psdots[linecolor=gold](-1,-1)(1,-1)(-3,1)
\psdots[linecolor=darkgreen](-2,1)(-2,-1)(0,1)(0,-1)(2,1)(2,-1)

{
\uput{0.1}[180](-3,0){$-\tfrac{3}{2}$}
\uput{0.1}[205](-1.25,0){$-\tfrac{3}{2}$}
\uput{0.1}[25](-0.75,0){$-\tfrac{3}{2}$}
\uput{0.1}[205](0.75,0){$-\tfrac{3}{2}$}
\uput{0.1}[25](1.25,0){$-\tfrac{3}{2}$}
\uput{0.1}[0](3,0){$-\tfrac{3}{2}$}

\uput{0.1}[65](-2,1){$-\tfrac{3}{2}$}
\uput{0.1}[-115](-2,-1){$-\tfrac{3}{2}$}
\uput{0.1}[65](0,1){$-\tfrac{3}{2}$}
\uput{0.1}[-115](0,-1){$-\tfrac{3}{2}$}
\uput{0.1}[65](2,1){$-\tfrac{3}{2}$}
\uput{0.1}[-115](2,-1){$-\tfrac{3}{2}$}

\uput{0.1}[90](-1,1){$-1$}
\uput{0.1}[90](1,1){$-1$}
\uput{0.1}[-45](3,-1){$-2$}

\uput{0.1}[-90](-1,-1){$-1$}
\uput{0.1}[-90](1,-1){$-1$}
\uput{0.1}[135](-3,1){$-2$}
}
\end{pspicture}
}
\caption{A schematic picture of the result. Generators correspond to vertices, arranged according to their Alexander grading and labelled by their $\delta$-grading. The dotted edges correspond to pairs of arrows labelled by powers of $q$, the solid ones to pairs of arrows labelled by powers of $p$.}\label{fig:examplesimplifiedgraph}
\end{subfigure}
\\
\psset{unit=1.3}
\begin{subfigure}[b]{0.3\textwidth}\centering
{
\begin{pspicture}(-1.5,-1.5)(1.5,1.5)
\psecurve(1.2,1)(0,1.4)(-1.2,1)(1.4,1)(0,0.6)(-1.4,1)(1.2,1)(0,1.4)(-1.2,1)

{
\psframe*[linecolor=white](-2,-2)(-1,2)
\psframe*[linecolor=white](1,-2)(2,2)
\psframe*[linecolor=white](-2,1)(2,2)
\psframe*[linecolor=white](-2,-2)(2,-1)

\psline[linestyle=dotted](1,1)(1,-1)
\psline[linestyle=dotted](1,-1)(-1,-1)
\psline[linestyle=dotted](-1,-1)(-1,1)
\psline[linestyle=dotted](-1,1)(1,1)
}

\psecurve[linestyle=dashed,dash=4pt 2pt](1.2,1)(0,1.4)(-1.2,1)(1.4,1)(0,0.6)(-1.4,1)(1.2,1)(0,1.4)(-1.2,1)

\psset{dotsize=5pt}

\psdot[linecolor=red](-1,0.8)
\psdot[linecolor=red](-1,0.663)

\psdot[linecolor=gold](0.125,1)
\psdot[linecolor=gold](-0.125,1)

\psdot[linecolor=darkgreen](1,0.8)
\psdot[linecolor=darkgreen](1,0.663)

\pscircle[fillstyle=solid, fillcolor=white](1,1){0.08}
\pscircle[fillstyle=solid, fillcolor=white](-1,1){0.08}
\pscircle[fillstyle=solid, fillcolor=white](1,-1){0.08}
\pscircle[fillstyle=solid, fillcolor=white](-1,-1){0.08}

\end{pspicture}
}
\caption{}\label{fig:loopbottom}
\end{subfigure}
\quad
\begin{subfigure}[b]{0.3\textwidth}\centering
\begin{pspicture}(-1.5,-1.5)(1.5,1.5)

\psecurve(-0.8,-1.2)(-1.2,-1.2)(1.2,1.1)(-1.2,0.9)(0,1.4)(1.4,1.2)(-0.8,-1.2)(-1.2,-1.2)(1.2,1.1)

{
	\psframe*[linecolor=white](-2,-2)(-1,2)
	\psframe*[linecolor=white](1,-2)(2,2)
	\psframe*[linecolor=white](-2,1)(2,2)
	\psframe*[linecolor=white](-2,-2)(2,-1)
	
	\psline[linestyle=dotted](1,1)(1,-1)
	\psline[linestyle=dotted](1,-1)(-1,-1)
	\psline[linestyle=dotted](-1,-1)(-1,1)
	\psline[linestyle=dotted](-1,1)(1,1)
}

\psecurve[linestyle=dashed,dash=4pt 2pt](-0.8,-1.2)(-1.2,-1.2)(1.2,1.1)(-1.2,0.9)(0,1.4)(1.4,1.2)(-0.8,-1.2)(-1.2,-1.2)(1.2,1.1)

\psset{dotsize=5pt}

\psdot[linecolor=red](-1,0.724)
\psdot[linecolor=red](-1,-0.635)

\psdot[linecolor=gold](-0.027,1)

\psdot[linecolor=darkgreen](1,0.51)
\psdot[linecolor=darkgreen](1,-0.03)

\psdot[linecolor=blue](-0.38,-1)

\pscircle[fillstyle=solid, fillcolor=white](1,1){0.08}
\pscircle[fillstyle=solid, fillcolor=white](-1,1){0.08}
\pscircle[fillstyle=solid, fillcolor=white](1,-1){0.08}
\pscircle[fillstyle=solid, fillcolor=white](-1,-1){0.08}

\end{pspicture}
\caption{}\label{fig:loopmiddle}
\end{subfigure}
\quad
\begin{subfigure}[b]{0.3\textwidth}\centering
\begin{pspicture}(-1.5,-1.5)(1.5,1.5)

\psrotate(0,0){180}{
\psecurve(1.2,1)(0,1.4)(-1.2,1)(1.4,1)(0,0.6)(-1.4,1)(1.2,1)(0,1.4)(-1.2,1)

{
	\psframe*[linecolor=white](-2,-2)(-1,2)
	\psframe*[linecolor=white](1,-2)(2,2)
	\psframe*[linecolor=white](-2,1)(2,2)
	\psframe*[linecolor=white](-2,-2)(2,-1)
	
	\psline[linestyle=dotted](1,1)(1,-1)
	\psline[linestyle=dotted](1,-1)(-1,-1)
	\psline[linestyle=dotted](-1,-1)(-1,1)
	\psline[linestyle=dotted](-1,1)(1,1)
}

\psecurve[linestyle=dashed,dash=4pt 2pt](1.2,1)(0,1.4)(-1.2,1)(1.4,1)(0,0.6)(-1.4,1)(1.2,1)(0,1.4)(-1.2,1)

\psset{dotsize=5pt}

\psdot[linecolor=darkgreen](-1,0.8)
\psdot[linecolor=darkgreen](-1,0.663)

\psdot[linecolor=blue](0.125,1)
\psdot[linecolor=blue](-0.125,1)

\psdot[linecolor=red](1,0.8)
\psdot[linecolor=red](1,0.663)

\pscircle[fillstyle=solid, fillcolor=white](1,1){0.08}
\pscircle[fillstyle=solid, fillcolor=white](-1,1){0.08}
\pscircle[fillstyle=solid, fillcolor=white](1,-1){0.08}
\pscircle[fillstyle=solid, fillcolor=white](-1,-1){0.08}
}
\end{pspicture}
\caption{}\label{fig:looptop}
\end{subfigure}
\caption{The final result of the computation from Example~\ref{exa:HFTdpretzeltangle} and Figure~\ref{fig:HFTdmutationexample}. The Subfigures (b)--(d) show the three loops from (a) separately on 4-punctured spheres.}\label{fig:mutationexamplefinalresult}
\end{figure}

\begin{example}[(the $(2,-3)$-pretzel tangle)]\label{exa:HFTdpretzeltangle}
	Figure~\ref{fig:HFTdmutationexample} shows the computation of $\CFTd(T)$ for the $(2,-3)$ pretzel tangle from Figure~\ref{fig:2m3pt}. In subsection~\ref{subsec:pretzels}, we will compute the peculiar modules for more general pretzel tangles, using some more advanced methods, which we develop in section~\ref{sec:glueingrevisited} as a corollary of the general classification of peculiar modules; here, we compute everything from the definition, which works surprisingly well.
	
	First, we compute the generators of the complex. They are ordered according to their Alexander grading on an infinite chessboard, where the generators in each of its fields have the same Alexander grading and where moving one field down, respectively to the right, increases the Alexander grading corresponding to the colour $\textcolor{darkgreen}{t_1}$, respectively $\textcolor{violet}{t_2}$, by 1. Next, we compute bigons and squares. Those correspond to the labelled arrows in Figure~\ref{fig:firstcomplex}. (The numbers in parentheses indicate which of the regions appear in the domain with which multiplicity. For example, the label $q_{432}(\texttt{15})$ of the arrow $\textcolor{red}{a_1y_1}\rightarrow\textcolor{gold}{x_2d'}$ means that the corresponding domain is given by the regions in Figure~\ref{fig:HFTdmutationpretzeltangleHD} labelled $q_4$, $q_3$, $q_2$, $\texttt{1}$ and $\texttt{5}$, each with multiplicity 1.) But there are also other contributing domains. Grading constraints tell us that we can only get additional morphisms between those generators which are connected by the other arrows. In principle, those could point in both directions. However, in each case, the connecting domains in one direction either have negative multiplicities or occupy both $p_i$s and $q_i$s, so we can only get arrows in one direction. From this and the $\partial^2$-relation, we can deduce that all solid arrows contribute. There are only eight remaining arrows (the dotted ones) and they can only appear in pairs. But it is easy to see that we can homotope those dotted arrows away (using the Clean-Up Lemma for curved type D structures), so in any case, the complex is homotopic to the invariant consisting of the solid arrows only. 
	We can then apply the Cancellation Lemma. We obtain a complex in which every arrow is paired with another one going in the opposite direction and every generator is connected along the arrows to exactly two other generators -- just as for rational tangles! A schematic picture of this complex is shown in Figure~\ref{fig:examplesimplifiedgraph}, where these arrow pairs have been replaced by single unoriented edges, such that we obtain a collection of loops.
	
	In Figures~\ref{fig:mutationexamplefinalresult}b-d, the loops have been transferred onto separate 4-punctured spheres in such a way that the vertices lie on the four arcs that connect the punctures and the unoriented edges lie on the front or back of the spheres depending on whether they correspond to arrow pairs labelled by powers of $p$ or $q$.
	
	The meaning of both of these representations as loops will be discussed in section~\ref{sec:glueingrevisited}. For the moment, they are just a convenient way to see certain symmetries.
\end{example}

%

\begin{figure}
	\centering
	\begin{subfigure}[b]{0.34\textwidth}\centering
		\psset{unit=0.3}
		\bigskip
		\begin{pspicture}(-8.01,-3.01)(8.01,3.01)
		\pscircle[linestyle=dotted, linewidth=\stringwidth](0,0){3}
		\pscircle[fillcolor=white,fillstyle=solid,linecolor=lightgray](0,0){1.5}
		
		\rput(0,0){$T$}
		\rput(2.25;180){$a$}
		\rput(2.25;0){$c$}
		
		\psecurve[linecolor=white,linewidth=\stringwhite](0,0)(1.1;-135)(3;-140)(4;-150)(-4.5,0)(4;150)(3;140)(1.1;135)(0,0)
		
		\psecurve[linecolor=gray, linewidth=\stringwidth](0,0)(1.1;-135)(3;-140)(4;-150)
		\psecurve[linecolor=gray, linewidth=\stringwidth](4;150)(3;140)(1.1;135)(0,0)

		\psecurve[linecolor=white,linewidth=\stringwhite](0,0)(1.1;45)(3;40)(4;30)(4.5,0)(4;-30)(3;-40)(1.1;-45)(0,0)
		
		\psecurve[linecolor=gray, linewidth=\stringwidth](0,0)(1.1;45)(3;40)(4;30)
		\psecurve[linecolor=gray, linewidth=\stringwidth](4;-30)(3;-40)(1.1;-45)(0,0)
		
		\psecurve[linewidth=\stringwidth]{C-C}(1.1;-135)(3;-140)(4;-150)(-4.5,0)(4;150)(3;140)(1.1;135)
		\psecurve[linewidth=\stringwidth]{C-C}(1.1;45)(3;40)(4;30)(4.5,0)(4;-30)(3;-40)(1.1;-45)
		\end{pspicture}
		\caption{A link obtained as the closure of a tangle $T$.}\label{fig:ClosureDiagram}
	\end{subfigure}
	\quad
	\begin{subfigure}[b]{0.6\textwidth}\centering
		{\vspace*{-5pt}
			$
			\begin{tikzcd}[row sep=0.5cm, column sep=0.7cm]
			&
			\textcolor{gold}{d}\phantom{\delta^{\frac{1}{2}}}\nmathphantom{\delta^{\frac{1}{2}}}
			\arrow[bend left=10,leftarrow]{dr}[inner sep=1pt]{q_4}
			\arrow[bend right=10,swap]{dr}[inner sep=1pt]{p_4}
			\\
			\textcolor{red}{a}\phantom{t^{A(p_{12})}}\nmathphantom{t^{A(p_{12})}}
			\arrow[bend left=10,leftarrow]{ur}[inner sep=1pt]{q_1}
			\arrow[bend right=10,swap]{ur}[inner sep=1pt]{p_1}
			&
			&
			\textcolor{darkgreen}{c}\phantom{t^{A(p_{12})}}\nmathphantom{t^{A(p_{12})}}
			\arrow[bend left=10,leftarrow]{dl}[inner sep=1pt]{q_3}
			\arrow[bend right=10,swap]{dl}[inner sep=1pt]{p_3}
			\\
			&
			\textcolor{blue}{b}\phantom{\delta^{\frac{1}{2}}}\nmathphantom{\delta^{\frac{1}{2}}}
			\arrow[bend left=10,leftarrow]{ul}[inner sep=1pt]{q_2}
			\arrow[bend right=10,swap]{ul}[inner sep=1pt]{p_2}
			\end{tikzcd}
			\quad\raisebox{-3pt}{$\longrightarrow$}\quad
			\begin{tikzcd}[row sep=0.5cm, column sep=-0.6cm]
			&
			t^{(0,0)}\delta^{1}\mathbb{F}_2
			\\
			t^{A(p_{2})}\delta^{\frac{1}{2}}\mathbb{F}_2
			\arrow{ur}{1}
			&&
			t^{A(q_{3})}\delta^{\frac{1}{2}}\mathbb{F}_2
			\arrow[leftarrow]{dl}{1}
			\arrow[swap]{ul}{1}
			\\
			&
			t^{(0,0)}\delta^{0}\mathbb{F}_2
			\arrow{ul}{1}
			\end{tikzcd}
			$
		}
		\caption{The functor $\Omega'$.}\label{fig:ClosureFunctor}
	\end{subfigure}
	\caption{Illustration of Proposition~\ref{prop:LazyClosing}.}\label{fig:Closure}
\end{figure}
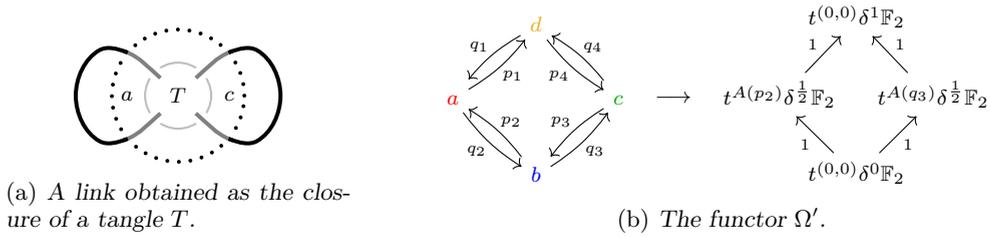

To state the next proposition, we first need to introduce some notation: 
let \(T\) be a 4-ended tangle and \(L\) the link obtained by closing \(T\) at the sites \(a\) and \(c\) as shown in Figure~\ref{fig:ClosureDiagram}. We distinguish two cases: either the two open components of \(T\) belong to distinct components of \(L\) (case 1) or they belong to the same component of \(L\) (case 2). Consider the quotient homomorphism
\begin{align*}
\omega\co\Ad\rightarrow\mathbb{F}_2,&\quad 
p_1,p_2,q_3,q_4\mapsto 1,\quad q_1,q_2,p_3,p_4\mapsto 0.
\end{align*}
We can promote this map to a functor $\Omega'$ from $\Ad$ to the category of bigraded vector spaces which on objects is defined by
\begin{align*}
\textcolor{red}{a}&\mapsto t^{A(p_2)}\delta^{\frac{1}{2}}\mathbb{F}_2&
\textcolor{blue}{b}&\mapsto t^{(0,0)}\delta^{0}\mathbb{F}_2&
\textcolor{darkgreen}{c}&\mapsto t^{A(q_3)}\delta^{\frac{1}{2}}\mathbb{F}_2&
\textcolor{gold}{d}&\mapsto t^{(0,0)}\delta^{1}\mathbb{F}_2,
\end{align*}
where 
$$t^{(a_1,a_2)}=
\begin{cases*}
t_1^{a_1}t_2^{a_2} & \text{(case 1)}\\
t^{a_1+a_2} & \text{(case 2)}
\end{cases*}
$$
For an illustration, see Figure~\ref{fig:ClosureFunctor}. 
Clearly, the functor $\Omega'$ preserves the $\delta$-grading of morphisms. It also preserves their Alexander grading, since in case~1, $$A(p_1)+A(p_2)=(0,0)=A(q_3)+A(q_4),$$
and in case 2, $t_1$ and $t_2$ are identified. 
So $\Omega'$ preserves the bigrading. Therefore, it induces a well-defined functor between the categories of peculiar modules and bigraded chain complexes, which we denote by $\Omega$.

\begin{proposition}\label{prop:LazyClosing}
With the notation as above, 
$$
\Omega(\CFTd(T))=
\begin{cases*}
\CFL(L)& \text{(case 1)}\\
\CFL(L)\otimes V & \text{(case 2)}\\
\end{cases*}
$$
where \(V\) is the 2-dimensional vector space over \(\mathbb{F}_2\) supported in Alexander gradings \(t\) and~\(t^{-1}\) and identical \(\delta\)-gradings.
Similar formulas hold for cyclic permutations of the indices.
\end{proposition}

\begin{proof}
If we delete those basepoints in a peculiar Heegaard diagram of $T$ that correspond to the variables that $\omega$ sends to 1, we obtain a Heegaard diagram for $L$. In $\CFL(L)$, we only count those holomorphic curves that stay away from the remaining basepoints, so we need to set those algebra elements equal to 0. On the relative gradings of the generators, this has the same effect as the functor $\Omega$. The additional tensor factor $V$ in case 2 comes from the fact that the new closed component has four instead of the usual two basepoints. 
\end{proof}

\begin{remark}\label{rem:LazyClosing}
One can obtain the minus version of link Floer homology from the generalised peculiar modules from the same idea as in the proof of the previous proposition. 
\end{remark}

\begin{definition}\label{def:reversedmirror}
	Let $T$ be an oriented 4-ended tangle in a $\mathbb{Z}$-homology 3-ball. Let $\m(T)$ be the tangle obtained by reversing the orientation of $M$, while preserving the labelling and orientation of $T$. We call $\m(T)$ the \textbf{mirror} of $T$. Note that the front and back component of $ \partial M\smallsetminus\textcolor{red}{S^1}$ are swapped. If $T$ is a tangle in $B^3$, a diagram of $\m(T)$ is obtained from one of $T$ by switching all crossings. 
	Furthermore, let $\rr(T)$ be the tangle obtained by reversing the orientation of all components of $T$. We write $\mr(T)$ for $\m(\rr(T))=\rr(\m(T))$ and call it the \textbf{reversed mirror} of $T$.
	
	If $X$ is a peculiar module, let $\m(X)$ be the peculiar module obtained from $X$ by reversing the direction of all arrows, inverting the gradings of all generators and switching the labels $p_i$ and $q_i$ for each $i=1,2,3,4$. Furthermore, if $X$ is bigraded, let $\rr(X)$ be the bigraded peculiar module obtained from $X$ by reversing the Alexander gradings of the generators, switching algebra elements $U'_j$ and $V'_j$ and reversing the Alexander gradings $A_{t_1}$ and $A_{t_2}$ of the underlying algebra. We write $\mr(X)$ for $\m(\rr(X))=\rr(\m(X))$.
\end{definition}

The following proposition should be compared to the analogous results \cite[Propositions~6.5 and~6.6, Proposition~2.3 and Corollary~2.7]{HDsForTangles} for the non-glueable Heegaard Floer theory~$\HFT$ and the polynomial invariant~$\nabla^s_T$.

\begin{proposition}\label{prop:reversedmirror}
	For any 4-ended tangle \(T\), 
	$$\CFTminus(\m(T))=\m(\CFTminus(T))
	\text{ and hence }
	\CFTd(\m(T))=\m(\CFTd(T)).$$
	Similarly,
	$$\CFTminus(\rr(T))=\rr(\CFTminus(T))
	\text{ and hence }
	\CFTd(\rr(T))=\rr(\CFTd(T)).$$
\end{proposition}
\begin{proof}
	The first part follows from the usual arguments by changing the orientation of the Heegaard surface. If we embed the Heegaard surface in $\mathbb{R}^3$, we may think of this operation as a reflection along a plane. Since front and back component of $ \partial M\smallsetminus\textcolor{red}{S^1}$ are swapped, so are the labels $p_i$ and $q_i$. If $\x$ and $\y$ are two generators, the moduli space $\mathcal{M}(\x,\y)$ for the original Heegaard diagram becomes $\mathcal{M}(\y,\x)$ for the new diagram. This has the effect of changing the direction of the arrows in the complexes computed from these diagrams. Naturally, this reverses the relative $\delta$-grading on generators and also the Alexander gradings, since the orientation of the components of $T$ is preserved.
	
	The second part follows from the fact that orientation reversal of the tangle components simply reverses the orientation of the basepoints in a Heegaard diagram of $T$. Orientation reversal of the open components of $T$ changes the ordered matching associated with $T$, which amounts to reversing the corresponding Alexander gradings in the algebra. Orientation reversal of the closed components of $T$ also involves switching labels $w_j$ and $z_j$ for closed components of $T$, so we need to switch the algebra elements $U'_j$ and $V'_j$. 
\end{proof}

We end this section with some simple observations about $\CFTd$.

\begin{observation}
By definition, the Alexander grading corresponding to a closed tangle component vanishes on $\Ad$. Also, the differential of a peculiar module preserves the Alexander grading by Lemma~\ref{lem:CFTdGradings}. Thus, $\CFTd(T)$ decomposes into the direct sum over the Alexander gradings of closed tangle components. 
\end{observation}
\begin{observation}
In \cite[Theorem~1.3]{OSHFKalt} and \cite[Theorem~1.3]{OSHFL}, Ozsv\'{a}th and Szab\'{o} showed that the link Floer homology $\HFL$ of alternating links is completely determined by the Alexander polynomial (and, to be precise, the signature, but this is only needed to fix the absolute homological and $\delta$-grading). The proof generalises immediately to $\HFT$, using the generalised clock Theorem~\cite[Theorem~1.13]{MyEssay}. 
\end{observation}
\begin{question}
Given an alternating tangle \(T\), is the bigraded chain homotopy type of \(\CFTd(T)\) determined by \(\nabla_T^s\)?
\end{question}

%% file: sections/Pairing.tex

\section{Pairing 4-ended tangles}\label{sec:Pairing}
In this section, we prove a glueing formula for $\CFTd$: given the peculiar modules of two 4-ended tangles $T_1$ and $T_2$, we compute the Heegaard Floer homology of the link $L$ obtained by glueing $T_1$ to $T_2$, up to at most one stabilisation. The proof is essentially a calculation of a bordered sutured type~AA bimodule. So let us start by recalling Zarev's Heegaard Floer theory for bordered sutured manifolds. 

\subsection{Review of bordered sutured Heegaard Floer theory}

We will assume some familiarity with \cite{Zarev,ZarevJoining,ZarevThesis} and only give a short review of the basic geometric objects involved.

\begin{definition}
	An \textbf{arc diagram} $\mathcal{Z}$ is a triple $(Z, \mathbf{a},M)$, where $Z$ is a (possibly empty) set of oriented line segments, $\mathbf{a}$ an even number of points on $Z$ and $M$ a matching of points in $\mathbf{a}$. The \textbf{graph $G(\mathcal{Z})$ of an arc diagram} $\mathcal{Z}$ is the graph obtained from the line segments $Z$ by adding an edge between matched points in $\mathbf{a}$. Given an arc diagram $\mathcal{Z}$, $-\mathcal{Z}$ is defined to be the same arc diagram, except that the orientation of the line segments is reversed. 
\end{definition}
\begin{definition}
	A \textbf{bordered sutured manifold with $\alpha$- and $\beta$-arcs} is a tuple 
	\[(Y,\Gamma, \mathcal{Z}_\alpha,\phi_\alpha, \mathcal{Z}_\beta,\phi_\beta),\]
	where
	\begin{itemize}
		\item $Y$ is a sutured manifold with sutures $\Gamma$; more precisely, $Y$ is an oriented manifold and $\Gamma\subset\partial Y$ are embedded oriented simple closed curves, dividing $\partial Y$ into two oriented open surfaces-with-boundary $R_-$ and $R_+$ such that $\partial R_-$ is equal to $\Gamma$ as embedded oriented 1-manifolds;
		\item $\mathcal{Z}_\alpha=(Z_\alpha, \mathbf{a}_\alpha,M_\alpha)$  and $\mathcal{Z}_\beta=(Z_\beta, \mathbf{a}_\beta,M_\beta)$ are arc diagrams;
		\item $\phi_\alpha$ is an embedding of $G(\mathcal{Z}_\alpha)$ into the closure of $R_-$ such that $\phi_\alpha(Z_\alpha)\subset\Gamma$;
		\item $\phi_\beta$ is an embedding of $G(\mathcal{Z}_\beta)$ into the closure of $R_+$ such that $\phi_\beta(Z_\beta)\subset\Gamma$ and $\phi_\alpha(\mathcal{Z}_\alpha)\cap\phi_\beta(\mathcal{Z}_\beta)=\emptyset$;
	\end{itemize}
	such that the map
	\begin{equation}\label{eqn:homlinind}
	\pi_0(\Gamma\smallsetminus (\phi_\alpha(Z_\alpha)\cup \phi_\beta(Z_\beta)))\rightarrow \pi_0(\partial Y\smallsetminus (\im(\phi_\alpha)\cup \im(\phi_\beta)))
	\end{equation}
	is surjective. 
\end{definition}
\begin{remark}
	Condition (\ref{eqn:homlinind}) is called \textbf{homological linear independence}. If we drop this condition, Zarev's invariants fail to be well-defined in general. Note that unlike Zarev, we allow the sutured surfaces of bordered sutured manifolds to be degenerate in the sense that surgery along all edges between matched points may contain closed components. This allows us to consider more general bordered sutured manifolds. If we restrict to non-degenerate sutured surfaces, homological linear independence is automatically satisfied, see \cite[Proposition~3.6]{Zarev}. 
\end{remark}

\begin{definition}\label{def:HDforBorderedSuturedMfdls}
	A \textbf{Heegaard diagram of a bordered sutured manifold} is obtained from a Heegaard diagram of the underlying sutured manifold by adding the graphs of the arc diagrams to it. To be more precise, consider a Heegaard diagram of the underlying sutured manifold. Then we can embed the graph $G(\mathcal{Z}_\alpha)$ into $R_-$ in such a way that it misses the 2-handles $D^1\times D^2$ corresponding to the $\alpha$-curves, simply by sliding them off $S^0\times D^2\subset R_-$. This gives us an embedding of $G(\mathcal{Z}_\alpha)$ into the Heegaard surface such that its image does not intersect the $\alpha$-curves. We view the images of the edges connecting points in $\mathbf{a}$ as $\alpha$-arcs. We proceed similarly for $G(\mathcal{Z}_\beta)$ in $R_+$ and the $\beta$-curves. 
	
	The image of $Z$ lies on the boundary of the Heegaard diagram, 
	ie the sutures, which we usually draw in \textcolor{darkgreen}{green}. We put a marked point, a \textbf{basepoint}, in every open component of the boundary minus the image of $Z$.
\end{definition}

Given an arc diagram $\mathcal{Z}$, Zarev defines a moving strands algebra $\mathcal{A}(\mathcal{Z})$, and given a bordered sutured Heegaard diagram, Zarev defines various bimodules over the strands algebras corresponding to its arc diagrams. Each arc diagram can either play the role of a type~D or type~A side of the bimodule; in fact, this is true for each connected component of an arc diagram, in which case, one obtains multimodules. For a discussion of type~A and~D structures as algebraic objects, we refer the reader to section~\ref{sec:AlgStructFromGDCats}, in particular Examples~\ref{exa:HighBrowDefTypeDoverI} and~\ref{exa:HighBrowDefTypeAoverI}.

A central result of Zarev's work is a glueing theorem, which is phrased in terms of the $\boxtimes$-tensor product (see Definition~\ref{def:PairingTypeDandA}). The following theorem summarizes his theory for bimodules. However, it is true for general multimodules. 

\begin{theorem}[{\cite[Theorem~3.10]{Zarev} or \cite[Theorem~12.3.2]{ZarevThesis}}]\label{thm:GlueingZarev}
Let \(Y\) be a bordered sutured manifold, bordered by two arc diagrams \(-\mathcal{Z}_1\) and \(\mathcal{Z}_2\). Then there are bimodules, well defined
up to homotopy equivalence:
\begin{align*}
_{\mathcal{A}(\mathcal{Z}_1)}
\BSAA(Y)
_{\mathcal{A}(\mathcal{Z}_2)} 
&&&
^{\mathcal{A}(\mathcal{Z}_1)}
\BSDA(Y)
_{\mathcal{A}(\mathcal{Z}_2)}\\
_{\mathcal{A}(\mathcal{Z}_1)}
\BSAD(Y)
^{\mathcal{A}(\mathcal{Z}_2)} 
&&&
^{\mathcal{A}(\mathcal{Z}_1)}
\BSDD(Y)
^{\mathcal{A}(\mathcal{Z}_2)}
\end{align*}
Let \(Y_1\) and \(Y_2\) be two such manifolds, bordered by \(-\mathcal{Z}_1\) and \(\mathcal{Z}_2\), and \(-\mathcal{Z}_2\) and \(\mathcal{Z}_3\),
respectively, and let \(Y_1\cup_{\mathcal{Z}_2} Y_2\) be the 3-manifold obtained by glueing \(Y_1\) and \(Y_2\) together along tubular neighbourhoods of the images of \(\mathcal{Z}_2\) on \(\partial Y_1\), respectively \(-\mathcal{Z}_2\) on \(\partial Y_2\). Then there are homotopy equivalences
\begin{align*}
_{\mathcal{A}(\mathcal{Z}_1)}
\BSAA(Y_1\cup_{\mathcal{Z}_2} Y_2)
_{\mathcal{A}(\mathcal{Z}_3)} 
& \cong
_{\mathcal{A}(\mathcal{Z}_1)}
\BSAA(Y)
_{\mathcal{A}(\mathcal{Z}_2)}
\boxtimes
^{\mathcal{A}(\mathcal{Z}_2)}
\BSDA(Y)
_{\mathcal{A}(\mathcal{Z}_3)}
,\\
~^{\mathcal{A}(\mathcal{Z}_1)}
\BSDA(Y_1\cup_{\mathcal{Z}_2} Y_2)
_{\mathcal{A}(\mathcal{Z}_3)} 
& \cong
~^{\mathcal{A}(\mathcal{Z}_1)}
\BSDD(Y)
^{\mathcal{A}(\mathcal{Z}_2)}
\boxtimes
_{\mathcal{A}(\mathcal{Z}_2)}
\BSAA(Y)
_{\mathcal{A}(\mathcal{Z}_3)},
\end{align*}
etc. Any combination of bimodules for \(Y_1\) and \(Y_2\) can be used, where one is type~A for
\(\mathcal{A}(\mathcal{Z}_2)\), and the other is type~D for \(\mathcal{A}(\mathcal{Z}_2)\).
\end{theorem}

\subsection{A first glueing formula}

\begin{definition}\label{def:tanglepairing}
	Given two 4-ended tangles $T_1$ and $T_2$, let $L(T_1,T_2)$ be the link obtained by glueing $T_1$ to $T_2$ as shown in Figure~\ref{fig:glueing2tangles1}. Equivalently, $L(T_1,T_2)$ is obtained by glueing  \reflectbox{$T_1$} to $T_2$ as shown in Figure~\ref{fig:glueing2tangles2}, where \reflectbox{$T_1$} is the tangle obtained from $T_1$ by rotating it along a vertical axis (along with the parametrization of the boundary). Note that $L(T_1,T_2)=L(T_2,T_1)$, which can be seen by rotating the link in Figure~\ref{fig:glueing2tangles2} along the vertical axis by $\pi$.
\end{definition}

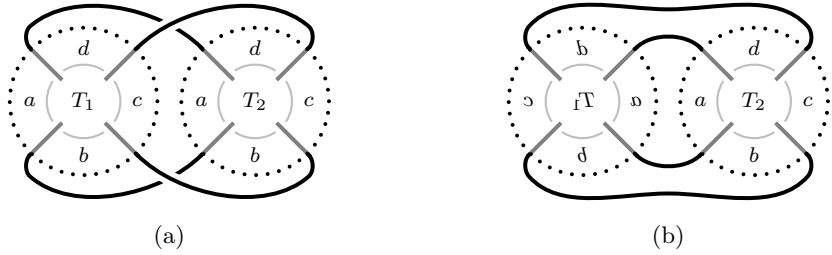
\begin{figure}[ht]
	\centering
	\psset{unit=0.2}
	\begin{subfigure}[b]{0.45\textwidth}\centering
		\begin{pspicture}[showgrid=false](-11,-7)(11,7)
		\rput{-45}(0,0){
			
			\pscircle[linecolor=lightgray](4,4){2.5}
			\pscircle[linecolor=lightgray](-4,-4){2.5}
			
			\psline[linecolor=white,linewidth=\stringwhite](1,-4)(-2,-4)
			\psline[linewidth=\stringwidth,linecolor=gray](1,-4)(-2,-4)
			\psline[linecolor=white,linewidth=\stringwhite](-4,1)(-4,-2)
			\psline[linewidth=\stringwidth,linecolor=gray](-4,1)(-4,-2)
			\psline[linecolor=white,linewidth=\stringwhite](-6,-4)(-9,-4)
			\psline[linewidth=\stringwidth,linecolor=gray](-6,-4)(-9,-4)
			\psline[linecolor=white,linewidth=\stringwhite](-4,-6)(-4,-9)
			\psline[linewidth=\stringwidth,linecolor=gray](-4,-6)(-4,-9)
			
			\psline[linecolor=white,linewidth=\stringwhite](-1,4)(2,4)
			\psline[linewidth=\stringwidth,linecolor=gray](-1,4)(2,4)
			\psline[linecolor=white,linewidth=\stringwhite](4,-1)(4,2)
			\psline[linewidth=\stringwidth,linecolor=gray](4,-1)(4,2)
			\psline[linecolor=white,linewidth=\stringwhite](6,4)(9,4)
			\psline[linewidth=\stringwidth,linecolor=gray](6,4)(9,4)
			\psline[linecolor=white,linewidth=\stringwhite](4,6)(4,9)
			\psline[linewidth=\stringwidth,linecolor=gray](4,6)(4,9)
			
			\pscircle[linestyle=dotted,linewidth=\stringwidth](4,4){5}
			\pscircle[linestyle=dotted,linewidth=\stringwidth](-4,-4){5}
			
			\psecurve[linewidth=\stringwidth]{C-C}(8,-3)(-1,4)(-10,-3)(-9,-4)(-8,-3)
			\psecurve[linewidth=\stringwidth]{C-C}(-3,8)(4,-1)(-3,-10)(-4,-9)(-3,-8)
			
			\psecurve[linewidth=\stringwhite,linecolor=white](-8,3)(1,-4)(10,3)(9,4)(8,3)
			\psecurve[linewidth=\stringwhite,linecolor=white](3,-8)(-4,1)(3,10)(4,9)(3,8)
			
			\psecurve[linewidth=\stringwidth]{C-C}(-8,3)(1,-4)(10,3)(9,4)(8,3)
			\psecurve[linewidth=\stringwidth]{C-C}(3,-8)(-4,1)(3,10)(4,9)(3,8)
			
			\rput{45}(-4,-4){$T_1$}
			\rput{45}(4,4){$T_2$}
			
			\rput{45}(1.5,1.5){$a$}
			\rput{45}(6.5,1.5){$b$}
			\rput{45}(6.5,6.5){$c$}
			\rput{45}(1.5,6.5){$d$}
			
			\rput{45}(-6.5,-6.5){$a$}
			\rput{45}(-1.5,-6.5){$b$}
			\rput{45}(-1.5,-1.5){$c$}
			\rput{45}(-6.5,-1.5){$d$}
		}
		\end{pspicture}
		\caption{}\label{fig:glueing2tangles1}
	\end{subfigure}
	\begin{subfigure}[b]{0.45\textwidth}\centering
		\begin{pspicture}[showgrid=false](-11,-7)(11,7)
		\rput{-45}(0,0){
			
			\pscircle[linecolor=lightgray](4,4){2.5}
			\pscircle[linecolor=lightgray](-4,-4){2.5}
			
			\psline[linecolor=white,linewidth=\stringwhite](1,-4)(-2,-4)
			\psline[linewidth=\stringwidth,linecolor=gray](1,-4)(-2,-4)
			\psline[linecolor=white,linewidth=\stringwhite](-4,1)(-4,-2)
			\psline[linewidth=\stringwidth,linecolor=gray](-4,1)(-4,-2)
			\psline[linecolor=white,linewidth=\stringwhite](-6,-4)(-9,-4)
			\psline[linewidth=\stringwidth,linecolor=gray](-6,-4)(-9,-4)
			\psline[linecolor=white,linewidth=\stringwhite](-4,-6)(-4,-9)
			\psline[linewidth=\stringwidth,linecolor=gray](-4,-6)(-4,-9)
			
			\psline[linecolor=white,linewidth=\stringwhite](-1,4)(2,4)
			\psline[linewidth=\stringwidth,linecolor=gray](-1,4)(2,4)
			\psline[linecolor=white,linewidth=\stringwhite](4,-1)(4,2)
			\psline[linewidth=\stringwidth,linecolor=gray](4,-1)(4,2)
			\psline[linecolor=white,linewidth=\stringwhite](6,4)(9,4)
			\psline[linewidth=\stringwidth,linecolor=gray](6,4)(9,4)
			\psline[linecolor=white,linewidth=\stringwhite](4,6)(4,9)
			\psline[linewidth=\stringwidth,linecolor=gray](4,6)(4,9)
			
			\pscircle[linestyle=dotted,linewidth=\stringwidth](4,4){5}
			\pscircle[linestyle=dotted,linewidth=\stringwidth](-4,-4){5}
			
			\psecurve[linewidth=\stringwidth]{C-C}(-2,-1)(1,-4)(4,-1)(1,2)
			\psecurve[linewidth=\stringwidth]{C-C}(2,1)(-1,4)(-4,1)(-1,-2)
			
			\psecurve[linewidth=\stringwidth]{C-C}(3,8)(4,9)(3,10)(-4.3,4.3)(-10,-3)(-9,-4)(-8,-3)
			\psecurve[linewidth=\stringwidth]{C-C}(-3,-8)(-4,-9)(-3,-10)(4.3,-4.3)(10,3)(9,4)(8,3)
			
			\rput{45}(4,4){$T_2$}
			\rput{45}(-4,-4){\reflectbox{$T_1$}}
			
			\rput{45}(-1.5,-1.5){\reflectbox{$a$}}
			\rput{45}(-6.5,-1.5){\reflectbox{$d$}}
			\rput{45}(-6.5,-6.5){\reflectbox{$c$}}
			\rput{45}(-1.5,-6.5){\reflectbox{$b$}}
			
			\rput{45}(6.5,6.5){$c$}
			\rput{45}(1.5,6.5){$d$}
			\rput{45}(1.5,1.5){$a$}
			\rput{45}(6.5,1.5){$b$}
		}
		\end{pspicture}
		\caption{}\label{fig:glueing2tangles2}
	\end{subfigure}
	\caption{The link obtained by glueing two 4-ended tangles together along matching sites.}\label{fig:glueing2tangles}
\end{figure}

\begin{theorem}[(Glueing Theorem, version 1)]\label{thm:CFTdGeneralGlueing}
	Let \(T_1\) and \(T_2\) be two 4-ended tangles and \(L=L(T_1,T_2)\). Let \(\mathcal{P}\) be the strictly unital left-left type~AA structure over \((\Ad,\Ad)\) defined in Figure~\ref{fig:GlueingTypeAAstructure}, where the Alexander grading of \(\Ad\) is given by the ordered matching for \(T_2\). Then
	\[\CFL(L)\otimes V^{i}\cong\CFTd(\rr(T_1))\boxtimes\,\mathcal{P}\boxtimes\,\CFTd(T_2)\]
	where \(V\) is the 2-dimensional vector space supported in Alexander gradings \(t\) and~\(t^{-1}\) and identical \(\delta\)-gradings and where \(i=\vert T_1\vert+\vert T_2\vert-\vert L\vert-2\in\{0,1\}\). 
\end{theorem}

\begin{remark}
	Let us discuss the type AA structure $\mathcal{P}$ in Figure~\ref{fig:GlueingTypeAAstructure} in more detail. First of all, the identity action is implicit. Secondly, the idempotent of a generator ${\red x}{\blue y}$ is $\iota_x$ in the first component and $\iota_y$ in the second. Likewise, the algebra elements in the first and second components are coloured {\red red} and {\blue blue}, respectively. This is the same convention that we use in the proof of Theorem~\ref{thm:CFTdGeneralGlueing} below, where we identify $\mathcal{P}$ with the bordered sutured type~AA structure $\mathcal{P}'$ illustrated in Figure~\ref{fig:GlueingDomains}. This colour convention forestalls Theorem~\ref{thm:CFTdGlueingAsMorphism}, where we will interpret the peculiar module of $\rr(T_1)$ (or more precisely $\mr(T_1)$) as an ${\red\alpha}$-curve, the one for $T_2$ as a ${\blue\beta}$-curve and the link Floer homology of $L(T_1,T_2)$ in terms of the Lagrangian intersection Floer homology of those two curves. 
	
	The $\delta$- and Alexander gradings of each generator of $\mathcal{P}$ are specified by the exponents of~$\delta$ and $t$, where we write
	$$t^{(a_1,a_2)}=
	\begin{cases*}
	t_1^{a_1}t_2^{a_2} & \text{if $\vert T_1\vert+\vert T_2\vert-\vert L\vert-2=0$}\\
	t^{a_1+a_2} & \text{if $\vert T_1\vert+\vert T_2\vert-\vert L\vert-2=1$}
	\end{cases*}
	$$
	just as in the definition of the functor $\Omega$ for Proposition~\ref{prop:LazyClosing}. 
	The gradings on $\mathcal{P}$ are defined in such a way that pairing with any two peculiar modules gives a chain complex which preserves Alexander grading and increases $\delta$-grading by 1. 
\end{remark}

\begin{example}
	The functor $\Omega$, interpreted as a type A structure, can be seen as a special case of this glueing formula by taking $T_1$ to be the trivial tangle from Figure~\ref{fig:CFTdForSomeRatTangles}. 
\end{example}

\begin{figure}[b]
	\centering
	\[
	\begin{tikzcd}[row sep=1.9cm, column sep=2.5cm]
	&
	t^{A(q_2)}\delta^{\frac{1}{2}}\textcolor{red}{a}\blue b
	\arrow[leftarrow,bend right=20]{ld}[description]{\qn{2}}
	&
	t^{-A(q_4)}\delta^{\frac{1}{2}}\textcolor{red}{d}\blue c
	\arrow[leftarrow,bend right=20]{l}[description]{\qp{3}{1}}
	\arrow[leftarrow,bend right=10]{lld}[description]{\qp{32}{1}}
	\arrow[leftarrow,pos=0.35,bend right=6]{lldd}[description]{\qp{3}{12}}
	\\
	t^{(0,0)}\delta^0\textcolor{red}{a}\blue a
	&&&
	t^{(0,0)}\delta^1\textcolor{red}{d}\blue d
	\arrow[leftarrow,bend right=20]{lu}[description]{\qn{4}}
	\arrow[leftarrow,bend right=10]{llu}[description]{\qp{43}{1}}
	\arrow[leftarrow,bend right=3]{lll}[description]{\pq{1}{432}+\qp{432}{1}}
	\arrow[leftarrow,bend right=7]{llld}[description]{\pq{12}{43}+\qp{43}{12}}
	\arrow[leftarrow,pos=0.65,bend left=6]{lldd}[description]{\pq{1}{43}}
	\\
	t^{(0,0)}\delta^0\textcolor{red}{b}\blue b
	\arrow[pos=0.55,bend left=11.5]{ruu}[description]{\np{2}}
	\arrow[bend right=10]{rrd}[description]{\pq{12}{3}}
	&&&
	t^{(0,0)}\delta^{1}\textcolor{red}{c}\blue c
	\arrow[leftarrow,pos=0.55,bend right=11.5]{luu}[description]{\np{4}}
	\arrow[leftarrow,pos=0.65,bend right=6]{lluu}[description]{\qp{3}{41}}
	\arrow[leftarrow,bend left=7]{lllu}[description]{\pq{41}{32}+\qp{32}{41}}
	\arrow[leftarrow,bend left=3]{lll}[description]{\pq{412}{3}+\qp{3}{412}}
	\arrow[leftarrow,bend left=10]{lld}[description]{\pq{41}{3}}
	\arrow[leftarrow,bend left=20]{ld}[description]{\pn{4}}
	\\
	&
	t^{A(p_2)}\delta^{\frac{1}{2}}\textcolor{red}{b}\blue a
	\arrow[leftarrow,pos=0.41,bend left=11.5]{luu}[description]{\nq{2}}
	\arrow[leftarrow,bend left=20]{lu}[description]{\pn{2}}
	&
	t^{-A(p_4)}\delta^{\frac{1}{2}}\textcolor{red}{c}\blue d
	\arrow[leftarrow,pos=0.35,bend left=6]{lluu}[description]{\pq{1}{32}}
	\arrow[leftarrow,bend left=20]{l}[description]{\pq{1}{3}}
	\arrow[pos=0.41,bend right=11.5]{ruu}[description]{\nq{4}}	
	\end{tikzcd}\]
	\caption{The type~AA structure $\mathcal{P}$ for Theorem~\ref{thm:CFTdGeneralGlueing}.
	}\label{fig:GlueingTypeAAstructure}
\end{figure}
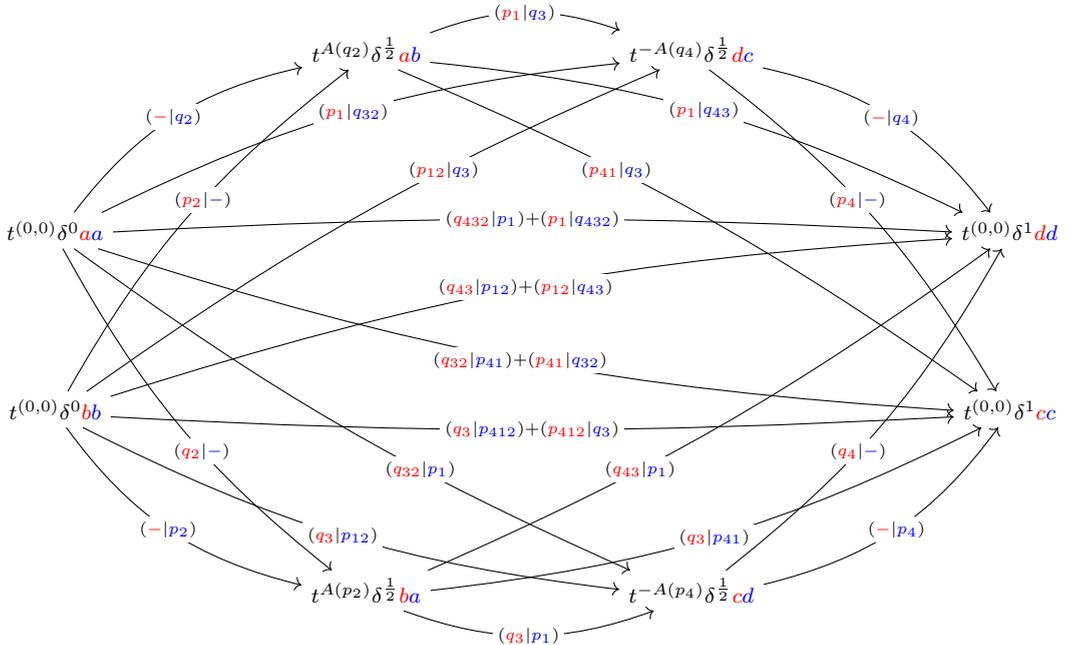

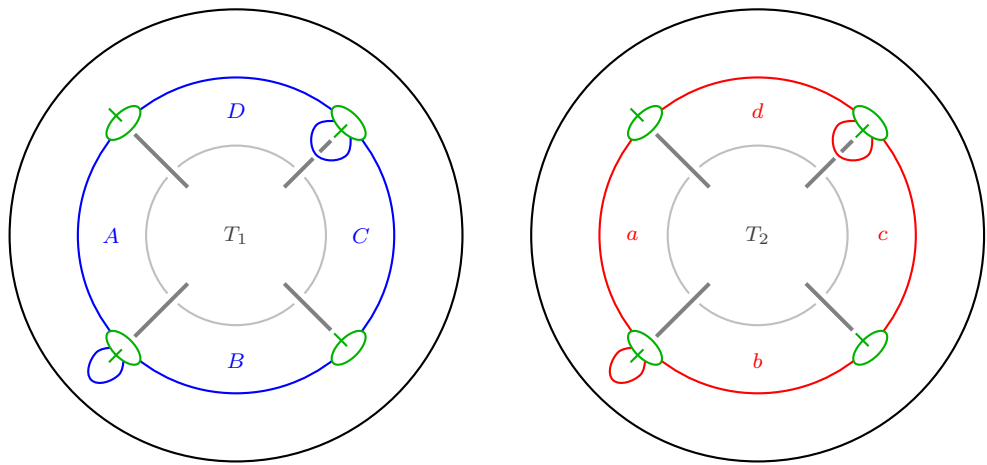
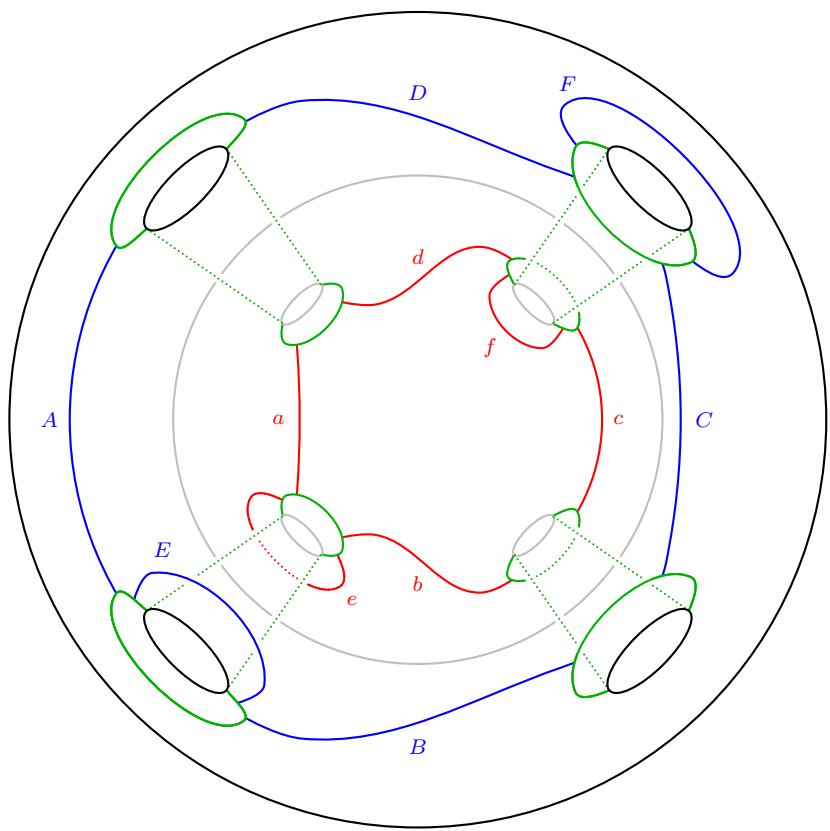
\begin{figure}[hp!]
	\centering
	\psset{unit=0.3}
	\bigskip
	\bigskip
	\begin{subfigure}[b]{0.45\textwidth}\centering
		\begin{pspicture}(-10,-10)(10,10)
		\pscircle(0,0){10}
		\pscircle[linecolor=lightgray](0,0){4}
		
		\psline[linecolor=white,linewidth=\stringwhite](3;45)(7;45)
		\psline[linecolor=white,linewidth=\stringwhite](3;135)(7;135)
		\psline[linecolor=white,linewidth=\stringwhite](3;-135)(7;-135)
		\psline[linecolor=white,linewidth=\stringwhite](3;-45)(7;-45)
		
		\psline[linecolor=gray,linewidth=\stringwidth](3;45)(4.8;45)
		\psline[linecolor=gray,linewidth=\stringwidth](5.2;45)(5.9;45)
		\psline[linecolor=gray,linewidth=\stringwidth](3;135)(6.3;135)
		\psline[linecolor=gray,linewidth=\stringwidth](3;-135)(6.3;-135)
		\psline[linecolor=gray,linewidth=\stringwidth](3;-45)(5.9;-45)
		
		\pscircle[linecolor=blue](0,0){7}
		
		\pscurve[linecolor=blue](7;-135)(8;-130)(9;-135)(8;-140)(7;-135)
		\pscurve[linecolor=blue](7;45)(6;36)(5;45)(6;54)(7;45)

		{\psset{fillstyle=solid,fillcolor=white}
			\rput{45}(7;45){\psellipse[linecolor=darkgreen](0,0)(0.5,1)}
			\rput{135}(7;135){\psellipse[linecolor=darkgreen](0,0)(0.5,1)}
			\rput{-135}(7;-135){\psellipse[linecolor=darkgreen](0,0)(0.5,1)}
			\rput{-45}(7;-45){\psellipse[linecolor=darkgreen](0,0)(0.5,1)}
		}
		
		\psline[linecolor=darkgreen](6.9;45)(6.1;45)
		\psline[linecolor=darkgreen](7.1;-135)(7.9;-135)
		\psline[linecolor=darkgreen](7.1;135)(7.9;135)
		\psline[linecolor=darkgreen](6.9;-45)(6.1;-45)
		
		\rput(-5.5,0){\blue $A$}
		\rput(0,-5.5){\blue $B$}
		\rput(5.5,0){\blue $C$}
		\rput(0,5.5){\blue $D$}
		\rput(0,0){\textcolor{darkgray}{$T_1$}}
		
		\end{pspicture}
		\caption{The parametrization on $\partial X_{T_1}$.}\label{fig:GlueingTangleT1}
	\end{subfigure}
	\quad
	\begin{subfigure}[b]{0.45\textwidth}\centering
		\begin{pspicture}(-10,-10)(10,10)
		\pscircle(0,0){10}
		\pscircle[linecolor=lightgray](0,0){4}
		
		\psline[linecolor=white,linewidth=\stringwhite](3;45)(7;45)
		\psline[linecolor=white,linewidth=\stringwhite](3;135)(7;135)
		\psline[linecolor=white,linewidth=\stringwhite](3;-135)(7;-135)
		\psline[linecolor=white,linewidth=\stringwhite](3;-45)(7;-45)
		
		\psline[linecolor=gray,linewidth=\stringwidth](3;45)(4.8;45)
		\psline[linecolor=gray,linewidth=\stringwidth](5.2;45)(5.9;45)
		\psline[linecolor=gray,linewidth=\stringwidth](3;135)(6.3;135)
		\psline[linecolor=gray,linewidth=\stringwidth](3;-135)(6.3;-135)
		\psline[linecolor=gray,linewidth=\stringwidth](3;-45)(5.9;-45)

		\pscircle[linecolor=red](0,0){7}
		
		\pscurve[linecolor=red](7;-135)(8;-130)(9;-135)(8;-140)(7;-135)
		\pscurve[linecolor=red](7;45)(6;36)(5;45)(6;54)(7;45)

		{\psset{fillstyle=solid,fillcolor=white}
			\rput{45}(7;45){\psellipse[linecolor=darkgreen](0,0)(0.5,1)}
			\rput{135}(7;135){\psellipse[linecolor=darkgreen](0,0)(0.5,1)}
			\rput{-135}(7;-135){\psellipse[linecolor=darkgreen](0,0)(0.5,1)}
			\rput{-45}(7;-45){\psellipse[linecolor=darkgreen](0,0)(0.5,1)}
		}
		
		\psline[linecolor=darkgreen](6.9;45)(6.1;45)
		\psline[linecolor=darkgreen](7.1;-135)(7.9;-135)
		\psline[linecolor=darkgreen](7.1;135)(7.9;135)
		\psline[linecolor=darkgreen](6.9;-45)(6.1;-45)
		
		\rput(-5.5,0){\red $a$}
		\rput(0,-5.5){\red $b$}
		\rput(5.5,0){\red $c$}
		\rput(0,5.5){\red $d$}
		\rput(0,0){\textcolor{darkgray}{$T_2$}}
		
		\end{pspicture}
		\caption{The parametrization on $\partial X_{T_2}$.}\label{fig:GlueingTangleT2}
	\end{subfigure}
	\\
	\bigskip
	\bigskip
	\begin{subfigure}[b]{0.9\textwidth}\centering
		\psset{unit=1.8}
		\begin{pspicture}(-10.2,-10.2)(10.2,10.4)
		
		\pscircle(0,0){10}
		
		\pscurve[linecolor=red](3.5;-135)(4.5;-155)(4.5;-115)(3.5;-135)
		\pscustom*[linecolor=white,linewidth=0pt]{
			\psline(4.05;-125.6)(8.1;-125.2)
			\psline(8.1;-144.8)(4.05;-144.4)
		}
		\pscurve[linecolor=red,linestyle=dotted,dotsep=1pt](3.5;-135)(4.5;-155)(4.5;-115)(3.5;-135)
		\pscurve[linecolor=red](4.3;33)(3.5;30)(3.5;60)(4.3;57)
		
		\pscurve[linecolor=red](4;140)(3.4;150)(3.4;-150)(4;-140)
		\pscurve[linecolor=red](4;-130)(3;-110)(4.5;-70)(4.5;-45)
		\psarcn[linecolor=red](0,0){4.5}{45}{-45}
		\pscurve[linecolor=red](4;130)(3;110)(4.5;70)(4.5;45)
		
		\pscurve[linecolor=blue](7.5;45)(8.5;65)(8.5;25)(7.5;45)
		
		\psarc[linecolor=blue](0,0){8.5}{135}{-135}
		\pscurve[linecolor=blue](8.5;-130)(8.3;-110)(7;-60)(7.5;-45)
		\pscurve[linecolor=blue](7.5;45)(7;30)(7;-30)(7.5;-45)
		\pscurve[linecolor=blue](8.5;130)(8.3;110)(7;60)(7.5;45)
		
		\pscurve[linecolor=darkgreen,fillstyle=solid,fillcolor=white](8;45)(8.1;54.8)(7.7;60)(7.7;30)(8.1;35.2)(8;45)
		\psrotate(0,0){90}{\pscurve[linecolor=darkgreen,fillstyle=solid,fillcolor=white](8;45)(8.1;54.8)(8.5;60)(8.5;30)(8.1;35.2)(8;45)}
		\psrotate(0,0){180}{\pscurve[linecolor=darkgreen,fillstyle=solid,fillcolor=white](8;45)(8.1;54.8)(8.5;60)(8.5;30)(8.1;35.2)(8;45)}
		\psrotate(0,0){-90}{\pscurve[linecolor=darkgreen,fillstyle=solid,fillcolor=white](8;45)(8.1;54.8)(7.7;60)(7.7;30)(8.1;35.2)(8;45)}
		
		\pscurve[linecolor=darkgreen,fillstyle=solid,fillcolor=white](4;45)(4.05;54.4)(4.5;60)(4.5;30)(4.05;35.6)(4;45)
		\pscustom*[linecolor=white,linewidth=0pt]{
			\psline[linecolor=darkgreen,linestyle=dotted,dotsep=1pt](4.05;54.4)(8.1;54.8)
			\psline[linecolor=darkgreen,linestyle=dotted,dotsep=1pt](8.1;35.2)(4.05;35.6)
		}
		\pscurve[linecolor=darkgreen,linestyle=dotted,dotsep=1pt](4;45)(4.05;54.4)(4.5;60)(4.5;30)(4.05;35.6)(4;45)
		\psrotate(0,0){90}{
			\pscurve[linecolor=darkgreen,fillstyle=solid,fillcolor=white](4;45)(4.05;54.4)(3.7;60)(3.7;30)(4.05;35.6)(4;45)
		}
		\psrotate(0,0){180}{
			\pscurve[linecolor=darkgreen,fillstyle=solid,fillcolor=white](4;45)(4.05;54.4)(3.7;60)(3.7;30)(4.05;35.6)(4;45)
		}
		\psrotate(0,0){-90}{
			\pscurve[linecolor=darkgreen,fillstyle=solid,fillcolor=white](4;45)(4.05;54.4)(4.5;60)(4.5;30)(4.05;35.6)(4;45)
			\pscustom*[linecolor=white,linewidth=0pt]{
				\psline[linecolor=darkgreen,linestyle=dotted,dotsep=1pt](4.05;54.4)(8.1;54.8)
				\psline[linecolor=darkgreen,linestyle=dotted,dotsep=1pt](8.1;35.2)(4.05;35.6)
			}
			\pscurve[linecolor=darkgreen,linestyle=dotted,dotsep=1pt](4;45)(4.05;54.4)(4.5;60)(4.5;30)(4.05;35.6)(4;45)
			}
		
		\pscircle[linecolor=lightgray](0,0){6}
		
		\rput{45}(4;45){\psellipse[linecolor=lightgray,fillstyle=solid,fillcolor=white](0,0)(0.25,0.7)}
		\rput{-45}(4;135){\psellipse[linecolor=lightgray,fillstyle=solid,fillcolor=white](0,0)(0.25,0.7)}
		\rput{45}(4;-135){\psellipse[linecolor=lightgray,fillstyle=solid,fillcolor=white](0,0)(0.25,0.7)}
		\rput{-45}(4;-45){\psellipse[linecolor=lightgray,fillstyle=solid,fillcolor=white](0,0)(0.25,0.7)}

		\psline[linecolor=white,linewidth=4pt](4.5;54.6)(5;54.6)
		\psline[linecolor=white,linewidth=4pt](4.5;35.4)(5;35.4)
		\psrotate(0,0){-90}{
			\psline[linecolor=white,linewidth=4pt](4.5;54.6)(5;54.6)
			\psline[linecolor=white,linewidth=4pt](4.5;35.4)(5;35.4)
		}
		\psrotate(0,0){180}{
			\psline[linecolor=white,linewidth=4pt](4.5;54.6)(5;54.6)
			\psline[linecolor=white,linewidth=4pt](4.5;35.4)(5;35.4)
		}
		
		\psline[linecolor=white,linewidth=4pt](5.8;54.6)(6.2;54.6)
		\psline[linecolor=white,linewidth=4pt](5.8;35.4)(6.2;35.4)
		\psrotate(0,0){90}{
			\psline[linecolor=white,linewidth=4pt](5.8;54.6)(6.2;54.6)
			\psline[linecolor=white,linewidth=4pt](5.8;35.4)(6.2;35.4)
		}
		\psrotate(0,0){180}{
			\psline[linecolor=white,linewidth=4pt](5.8;54.6)(6.2;54.6)
			\psline[linecolor=white,linewidth=4pt](5.8;35.4)(6.2;35.4)
		}
		\psrotate(0,0){-90}{
			\psline[linecolor=white,linewidth=4pt](5.8;54.6)(6.2;54.6)
			\psline[linecolor=white,linewidth=4pt](5.8;35.4)(6.2;35.4)
		}
		
		\psline[linecolor=darkgreen,linestyle=dotted,dotsep=1pt](4.05;54.4)(8.1;54.8)
		\psline[linecolor=darkgreen,linestyle=dotted,dotsep=1pt](4.05;35.6)(8.1;35.2)
		\psrotate(0,0){90}{
			\psline[linecolor=darkgreen,linestyle=dotted,dotsep=1pt](4.05;54.4)(8.1;54.8)
			\psline[linecolor=darkgreen,linestyle=dotted,dotsep=1pt](4.05;35.6)(8.1;35.2)
			}
		\psrotate(0,0){180}{
			\psline[linecolor=darkgreen,linestyle=dotted,dotsep=1pt](4.05;54.4)(8.1;54.8)
			\psline[linecolor=darkgreen,linestyle=dotted,dotsep=1pt](4.05;35.6)(8.1;35.2)
		}
		\psrotate(0,0){-90}{
			\psline[linecolor=darkgreen,linestyle=dotted,dotsep=1pt](4.05;54.4)(8.1;54.8)
			\psline[linecolor=darkgreen,linestyle=dotted,dotsep=1pt](4.05;35.6)(8.1;35.2)
		}
		
		\psline[linecolor=white,linewidth=4pt](6.7;54.6)(7;54.6)
		\psline[linecolor=white,linewidth=4pt](6.7;35.4)(7;35.4)
		\psrotate(0,0){-90}{
			\psline[linecolor=white,linewidth=4pt](6.7;54.6)(7;54.6)
			\psline[linecolor=white,linewidth=4pt](6.7;35.4)(7;35.4)
		}
		\psrotate(0,0){180}{
			\psline[linecolor=white,linewidth=4pt](6.65;54.6)(6.95;54.6)
			\psline[linecolor=white,linewidth=4pt](6.65;35.4)(6.95;35.4)
		}
		\pscurve[linecolor=blue](8.3;-147)(7.5;-150)(7.5;-120)(8.3;-123)
		
		\psrotate(0,0){180}{\pscurve[linecolor=darkgreen,fillstyle=solid,fillcolor=white](8;45)(8.1;54.8)(8.5;60)(8.5;30)(8.1;35.2)(8;45)}
		
		\pscurve[linecolor=darkgreen](8;45)(8.1;54.8)(7.7;60)(7.7;30)(8.1;35.2)(8;45)
		\psrotate(0,0){90}{\pscurve[linecolor=darkgreen](8;45)(8.1;54.8)(8.5;60)(8.5;30)(8.1;35.2)(8;45)}
		\psrotate(0,0){180}{\pscurve[linecolor=darkgreen](8;45)(8.1;54.8)(8.5;60)(8.5;30)(8.1;35.2)(8;45)}
		\psrotate(0,0){-90}{\pscurve[linecolor=darkgreen](8;45)(8.1;54.8)(7.7;60)(7.7;30)(8.1;35.2)(8;45)}
		
		\rput{45}(8;45){\psellipse[fillstyle=solid,fillcolor=white](0,0)(0.5,1.4)}
		\rput{-45}(8;135){\psellipse[fillstyle=solid,fillcolor=white](0,0)(0.5,1.4)}
		\rput{45}(8;-135){\psellipse[fillstyle=solid,fillcolor=white](0,0)(0.5,1.4)}
		\rput{-45}(8;-45){\psellipse[fillstyle=solid,fillcolor=white](0,0)(0.5,1.4)}

		\rput(-3.4,0){$\red a$}
		\rput(0,-4){$\red b$}
		\rput(4.9,0){$\red c$}
		\rput(0,4){$\red d$}
		\rput(2.5;45){$\red f$}
		\rput(4.7;-110){$\red e$}
		
		\rput(-9,0){$\blue A$}
		\rput(0,-8){$\blue B$}
		\rput(7,0){$\blue C$}
		\rput(0,8){$\blue D$}
		\rput(9;66){$\blue F$}
		\rput(7;-153){$\blue E$}
		
		\end{pspicture}
		\caption{The parametrization on the boundary of the thickened 4-punctured sphere $X$.}\label{fig:GlueingAAP}
	\end{subfigure}
	\caption{A decomposition of the complement of the link from Figure~\ref{fig:glueing2tangles}.}\label{fig:GlueingMflds}
\end{figure}

\begin{figure}[p]
	\centering
	\begin{subfigure}[b]{\textwidth}
		\centering
		\psset{unit=0.13}
		\begin{pspicture}(-27.5,-27.5)(27.5,27.5)
		\psrotate(0,0){-45}{
			
			\pscustom*[linecolor=lightgreen]{%
				\pscurve[linecolor=blue](1,20)(1,18)(18,1)(20,1)
			}
			\psline[linecolor=lightgreen,linewidth=3pt](1,20)(20,1)
			\pscustom*[linecolor=lightgreen]{%
				\pscurve[linecolor=red](1,20)(5,27)(24,4)(20,1)
			}
			
			\pscustom*[linecolor=lightgreen]{%
				\pscurve[linecolor=blue](-1,-20)(-1,-18)(-18,-1)(-20,-1)
			}
			\psline[linecolor=lightgreen,linewidth=3pt](-1,-20)(-20,-1)
			\pscustom*[linecolor=lightgreen]{%
				\pscurve[linecolor=red](-1,-20)(-5,-27)(-24,-4)(-20,-1)
			}
			
			\pscustom*[linecolor=lightgreen]{%
				\pscurve[linecolor=blue](-1,20)(-1,18)(-24,4)(-20,1)
				\pscurve[linecolor=red](-20,1)(-18,1)(-13,25)(-5,27)(-1,20)
			}
			\psline*[linecolor=white](-15.15,10.9)(0,10)(0,30)(-20,30)
			
			\psrotate(0,0){180}{
				\pscustom*[linecolor=lightgreen]{%
					\pscurve[linecolor=blue](-1,20)(-1,18)(-24,4)(-20,1)
					\pscurve[linecolor=red](-20,1)(-18,1)(-13,25)(-5,27)(-1,20)
				}
				\psline*[linecolor=white](-15.15,10.9)(0,10)(0,30)(-20,30)
			}
			
			\rput(0,20){$~$%
				\pscustom*[linecolor=lightgreen]{%
					\pscurve(-2,0)(-4,-2)(-8,3)(8,3)(4,-2)(2,0)
					\pscurve(2,0)(4,2)(8,-3)(-8,-3)(-4,2)(-2,0)
				}
				\pscustom*[linecolor=white]{
					\pscurve(-2,0)(-4,-2)(-8,3)(8,3)(4,-2)(2,0)
					\pscurve(-2,0)(-4,2)(-8,-3)(8,-3)(4,2)(2,0)
				}
			}
			\psrotate(0,0){180}{
				\rput(0,20){$~$%
					\pscustom*[linecolor=lightgreen]{%
						\pscurve(-2,0)(-4,-2)(-8,3)(8,3)(4,-2)(2,0)
						\pscurve(2,0)(4,2)(8,-3)(-8,-3)(-4,2)(-2,0)
					}
					\pscustom*[linecolor=white]{
						\pscurve(-2,0)(-4,-2)(-8,3)(8,3)(4,-2)(2,0)
						\pscurve(-2,0)(-4,2)(-8,-3)(8,-3)(4,2)(2,0)
					}
				}
			}

			\pscurve[linecolor=blue](1,20)(1,18)(18,1)(20,1)
			\pscurve[linecolor=red](1,20)(5,27)(24,4)(20,1)
			\pscurve[linecolor=blue](-1,20)(-1,18)(-24,4)(-20,1)
			\pscurve[linecolor=red](-1,20)(-5,27)(-13,25)(-18,1)(-20,1)
			
			\psrotate(0,0){180}{
				\pscurve[linecolor=red](1,20)(1,18)(18,1)(20,1)
				\pscurve[linecolor=blue](1,20)(5,27)(24,4)(20,1)
				\pscurve[linecolor=red](-1,20)(-1,18)(-24,4)(-20,1)
				\pscurve[linecolor=blue](-1,20)(-5,27)(-13,25)(-18,1)(-20,1)
			}
			
			\rput(-20,0){
				\psellipse[linecolor=darkgreen,fillstyle=solid,fillcolor=white](0,0)(3,4)
				\psline[linecolor=darkgreen](0,-3)(0,-5)
				\psline[linecolor=darkgreen](0,3)(0,5)
			}
			
			\rput(0,-20){
				\pscurve[linecolor=red](-2,0)(-4,2)(-8,-3)(8,-3)(4,2)(2,0)
				\pscurve[linecolor=blue](-2,0)(-4,-2)(-8,3)(8,3)(4,-2)(2,0)
				\psellipse[linecolor=darkgreen,fillstyle=solid,fillcolor=white](0,0)(4,3)
				\psline[linecolor=darkgreen](-3,0)(-5,0)
				\psline[linecolor=darkgreen](3,0)(5,0)
			}
			
			\rput(20,0){
				\psellipse[linecolor=darkgreen,fillstyle=solid,fillcolor=white](0,0)(3,4)
				\psline[linecolor=darkgreen](0,-3)(0,-5)
				\psline[linecolor=darkgreen](0,3)(0,5)
			}
			
			\rput(0,20){
				\pscurve[linecolor=red](-2,0)(-4,2)(-8,-3)(8,-3)(4,2)(2,0)
				\pscurve[linecolor=blue](-2,0)(-4,-2)(-8,3)(8,3)(4,-2)(2,0)
				\psellipse[linecolor=darkgreen,fillstyle=solid,fillcolor=white](0,0)(4,3)
				\psline[linecolor=darkgreen](-3,0)(-5,0)
				\psline[linecolor=darkgreen](3,0)(5,0)
			}
			
			\rput(3.1,13.85){\psdot\rput{45}{\uput{1}[-155](0,0){$C$}}}
			\rput(-4.85,14.4){\psdot\rput{45}{\uput{1}[-135](0,0){$D$}}}
			\rput(3.5,26){\psdot\rput{45}{\uput{1}[65](0,0){$c$}}}
			\rput(-3.67,25.95){\psdot\rput{45}{\uput{1}[45](0,0){$d$}}}
			\rput(-15.15,10.9){\psdot\rput{45}{\uput{1.5}[-168](0,0){$\delta$}}}
			\rput(-7.23,20){\psdot\rput{45}{\uput{1}[135](0,0){$\zeta_1$}}}
			\rput(7.23,20){\psdot\rput{45}{\uput{1}[-45](0,0){$\zeta_2$}}}
			
			\psrotate(0,0){180}{
				\rput(3.1,13.85){\psdot\rput{-135}{\uput{1}[25](0,0){$a$}}}
				\rput(-4.85,14.4){\psdot\rput{-135}{\uput{1}[45](0,0){$b$}}}
				\rput(3.5,26){\psdot\rput{-135}{\uput{1}[-115](0,0){$A$}}}
				\rput(-3.67,25.95){\psdot\rput{-135}{\uput{1}[-135](0,0){$B$}}}
				\rput(-15.15,10.9){\psdot\rput{-135}{\uput{1.5}[12](0,0){$\beta$}}}
				\rput(-7.23,20){\psdot\rput{-135}{\uput{1}[-45](0,0){$\varepsilon_1$}}}
				\rput(7.23,20){\psdot\rput{-135}{\uput{1}[135](0,0){$\varepsilon_2$}}}
			}
			
			\rput(-8.7,-8.7){\rput{45}{\red $a$}}
			\rput(8.7,8.7){\rput{45}{\blue $C$}}
			
			\rput(15,21){\rput{45}{\red $c$}}
			\rput(-15,-21){\rput{45}{\blue $A$}}
			
			\rput(-17,5){\rput{45}{\red $d$}}
			\rput(17,-5){\rput{45}{\blue $B$}}
			
			\rput(22,-6){\rput{45}{\red $b$}}
			\rput(-22,6){\rput{45}{\blue $D$}}
			
			\rput(-0.6,12){\rput{45}{\red $f$}}
			\rput(-0.6,28){\rput{45}{\blue $F$}}
			
			\rput(0.6,-28){\rput{45}{\red $e$}}
			\rput(0.6,-12){\rput{45}{\blue $E$}}
			
			\rput(-5,16.5){\rput{45}{\red $P_4$}}
			\rput(0,15.5){\rput{45}{\red $Q_4$}}
			\rput(5,16.5){\rput{45}{\red $P_4$}}
			
			\rput(-5,23.5){\rput{45}{\blue $p_4$}}
			\rput(0,24.5){\rput{45}{\blue $q_4$}}
			\rput(5,23.5){\rput{45}{\blue $p_4$}}
			
			\rput(-5,-16.5){\rput{45}{\blue $q_2$}}
			\rput(0,-15){\rput{45}{\blue $p_2$}}
			\rput(5,-16.5){\rput{45}{\blue $q_2$}}
			
			\rput(-5,-23.5){\rput{45}{\red $Q_2$}}
			\rput(0,-24.5){\rput{45}{\red $P_2$}}
			\rput(5,-23.5){\rput{45}{\red $Q_2$}}
			
			\rput(-25,0){\rput{45}{\red $P_1$}}
			\rput(-15,0){\rput{45}{\blue $p_1$}}
			
			\rput(25,0){\rput{45}{\blue $q_3$}}
			\rput(15,0){\rput{45}{\red $Q_3$}}
			
			\rput(-12,16){\rput{45}{$\textcolor{white}{_\delta}\square_\delta$}}
			\rput(12,-16){\rput{45}{$\textcolor{white}{_\beta}\square_\beta$}}
		}
		\end{pspicture}
		\caption{A Heegaard diagram for $X$. The orientation of the surface is such that the normal vector (determined by the right-hand rule) points out of the projection plane. 
		}\label{fig:GlueingHD}
	\end{subfigure}
	\begin{subfigure}[b]{\textwidth}
			\begin{align*}
			{\red a}{\blue a}\co &Cc\beta\delta\varepsilon_2
			&
			{\red b}{\blue a}\co &ACbc\delta
			&
			{\red c}{\blue c}\co &Aa\beta\delta\zeta_2
			&
			{\red d}{\blue c}\co &ACad\beta
			\\
			\overline{{\red a}{\blue a}}\co &BCbc\delta
			&
			{\red a}{\blue b}\co &BCac\delta
			&
			\overline{{\red c}{\blue c}}\co &ADad\beta
			&
			{\red c}{\blue d}\co &ADac\beta
			\\
			\underline{{\red a}{\blue a}}\co &Cc\beta\delta\varepsilon_1
			&
			{\red b}{\blue b}\co &ACac\delta
			&
			\underline{{\red c}{\blue c}}\co &Aa\beta\delta\zeta_1
			&
			{\red d}{\blue d}\co &ACac\beta
			\end{align*}
		\caption{Generators in idempotents $\mathcal{I}'_\beta$ and $\mathcal{I}'_\alpha$ for the Heegaard diagram above. The five-letter word to the right of a colon specifies the intersection points of that generator; the two-letter word to the left of that colon denotes the name of the generator which indicates the corresponding idempotents. Namely, the idempotents of a generator ${\red x}{\blue y}$ are $\iota_x\in\mathcal{I}'_\beta$ in the first component and $\iota_y\in\mathcal{I}'_\alpha$ in the second. }\label{fig:GlueingGeneratorList}
	\end{subfigure}
	\begin{subfigure}[b]{\textwidth}\centering
		\[
		\begin{tikzcd}[row sep=1.5cm, column sep=3.5cm]
		&
		\textcolor{red}{a}\blue b
		\arrow[leftarrow,bend right=20]{ld}[description]{\{{\blue q_{2}},\square_\beta\}}
		\arrow[leftarrow,pos=0.45,bend right=11.5]{ldd}[description]{\{{\red P_{2}}\}}
		\arrow[dashed,pos=0.35,bend left=6]{rrdd}[description]{\{{\red P_4},{\red P_{1}}/{\blue q_{3}},\square_\delta\}}
		&
		\textcolor{red}{d}\blue c
		\arrow[leftarrow,bend right=20]{l}[description]{\{{\red P_{1}}/{\blue q_{3}}\}}
		\arrow[dashed,leftarrow,bend right=10]{lld}[description]{\{{\red P_{1}}/{\blue q_3},{\blue q_{2}},\square_\beta\}}
		\arrow[dashed,leftarrow,pos=0.35,bend right=6]{lldd}[description]{\{{\red P_2},{\red P_{1}}/{\blue q_{3}}\}}
		\\
		\textcolor{red}{a}\blue a
		\arrow[dashed,pos=0.65,bend right=6]{rrdd}[description]{\{{\red Q_2},{\red Q_{3}}/{\blue p_{1}},\square_\beta\}}
		&&&
		\textcolor{red}{d}\blue d
		\arrow[leftarrow,bend right=20]{lu}[description]{\{{\blue q_{4}}\}}
		\arrow[dashed,leftarrow,pos=0.53,bend right=10]{llu}[description]{\{{\red P_{1}}/{\blue q_{3}},{\blue q_4}\}}
		\arrow[dashed,leftarrow,bend right=1.5]{lll}[description]{\{{\red Q_2},{\red Q_4},{\red Q_{3}}/{\blue p_{1}},\square_\beta\}+\{{\red P_{1}}/{\blue q_{3}},{\blue q_2},{\blue q_4},\square_\beta\}}
		\arrow[dashed,leftarrow,bend right=7]{llld}[description]{\{{\red Q_4},{\red Q_{3}}/{\blue p_{1}},{\blue p_2}\}+\{{\red P_2},{\red P_{1}}/{\blue q_{3}},{\blue q_4}\}}
		\arrow[dashed,leftarrow,pos=0.65,bend left=6]{lldd}[description]{\{{\red Q_4},{\red Q_{3}}/{\blue p_{1}}\}}
		\arrow[leftarrow,pos=0.55,bend left=11.5]{ldd}[description]{\{{\red Q_4}\}}
		\\
		\textcolor{red}{b}\blue b
		\arrow[dashed,bend right=10,pos=0.53]{rrd}[description]{\{{\red Q_{3}}/{\blue p_{1}},{\blue p_2}\}}
		&&&
		\textcolor{red}{c}\blue c
		\arrow[leftarrow,pos=0.53,bend right=11.5]{luu}[description]{\{{\red P_{4}},\square_\delta\}}
		\arrow[dashed,leftarrow,bend left=7]{lllu}[description]{\{{\red Q_2},{\red Q_{3}}/{\blue p_{1}},{\blue p_4},\square_\beta,\square_\delta\}+\{{\red P_4},{\red P_{1}}/{\blue q_{3}},{\blue q_2},\square_\beta,\square_\delta\}}
		\arrow[dashed,leftarrow,bend left=1.5]{lll}[description]{\{{\red Q_{3}}/{\blue p_{1}},{\blue p_4},{\blue p_2},\square_\delta\}+\{{\red P_2},{\red P_4},{\red P_{1}}/{\blue q_{3}},\square_\delta\}}
		\arrow[dashed,leftarrow,bend left=10]{lld}[description]{\{{\red Q_{3}}/{\blue p_{1}},{\blue p_4},\square_\delta\}}
		\arrow[leftarrow,bend left=20]{ld}[description]{\{{\blue p_{4}},\square_\delta\}}
		\\
		&
		\textcolor{red}{b}\blue a
		\arrow[leftarrow,pos=0.47,bend left=11.5]{luu}[description]{\{{\red Q_{2}},\square_\beta\}}
		\arrow[leftarrow,bend left=20]{lu}[description]{\{{\blue p_{2}}\}}
		&
		\textcolor{red}{c}\blue d
		\arrow[leftarrow,bend left=20]{l}[description]{\{{\red Q_{3}}/{\blue p_{1}}\}}
		\end{tikzcd}\]	
		\caption{The domains that contribute to the type~AA structure $\mathcal{P}'$.}\label{fig:GlueingDomains}
	\end{subfigure}
	\caption{A Heegaard diagram for the bordered sutured manifold~$X$ from Figure~\ref{fig:GlueingAAP} and some computations of generators and domains.}\label{fig:GlueingCAT}
\end{figure}
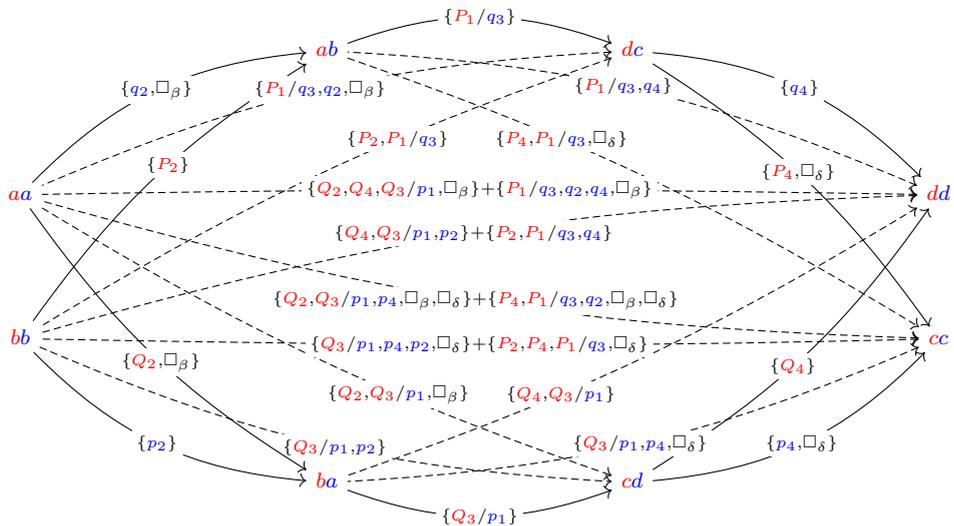

\begin{proof}[of Theorem~\ref{thm:CFTdGeneralGlueing}]
	The strategy of the proof is to glue the tangle complements $X_{T_1}$ and $X_{T_2}$ to opposite sides of a thickened 4-punctured sphere $X=I\times (S^2\smallsetminus4 D^2)$, ie along $\{0\}\times (S^2\smallsetminus4 D^2)$ and $\{1\}\times (S^2\smallsetminus4 D^2)$, respectively. For this, we equip $X_{T_1}$, $X_{T_2}$ and $X$ with the structures of bordered sutured manifolds specified by the arc diagrams in Figure~\ref{fig:GlueingMflds}. Note that each closed component of $T_1$ and $T_2$ carries two oppositely oriented meridional sutures. By glueing $X_{T_1}$ to the outside and $X_{T_2}$ to the inside of the thickened 4-punctured sphere $X$ as shown in Figure~\ref{fig:GlueingAAP}, we obtain the complement $X_L$ of the link $L$. In addition to those sutures on closed components of $T_1$ and $T_2$, the sutured manifold $X_L$ carries one meridional suture at each of the four places where the tangles have been glued together.

	Let $\mathcal{A}$ be the bordered sutured algebra corresponding to the arc diagram on $X$ consisting of $\alpha$-arcs and $\mathcal{I}_\alpha$ the corresponding ring of idempotents. Let $\mathcal{I}'_\alpha$ be the subring of idempotents occupying the $\alpha$-arcs ${\red e}$ and ${\red f}$ and $\iota_{\alpha}\co\mathcal{A}'\hookrightarrow\mathcal{A}$ the subalgebra $\mathcal{I}'_\alpha.\mathcal{A}.\mathcal{I}'_\alpha$ of~$\mathcal{A}$. Similarly define $\mathcal{I}_\beta$, $\mathcal{I}'_\beta$ and $\iota_{\beta}\co\mathcal{B}'\hookrightarrow\mathcal{B}$. 
	
	Zarev's Glueing Theorem (Theorem~\ref{thm:GlueingZarev}) tells us that
	\begin{equation}\label{eqn:FirstGlueingProofStep1}
	\SFC(X_L)\cong \BSD(X_{T_1})^\mathcal{B}\boxtimes\typeA{\mathcal{B},\mathcal{A}}{\BSAA(X)}\boxtimes\BSD(X_{T_2})^\mathcal{A}
	\end{equation}
	The left-hand side agrees with the left-hand term of the identity from Theorem~\ref{thm:CFTdGeneralGlueing}. So the goal is to identify the three tensor-factors on the right-hand sides with each other.
	
	First of all, observe that we can choose a Heegaard diagram for $X_{T_1}$ where the two $\beta$-arcs that have ends on the same suture do not intersect any $\alpha$-curve. Thus, generators of its type~D structure belong to the idempotents that occupy one of the four arcs $\blue A$, $\blue B$, $\blue C$ or $\blue D$. The same is true for $X_{T_2}$ and its $\alpha$-arcs. Moreover, the labels of the type~D structures for $X_{T_1}$ and $X_{T_2}$ are contained in $\mathcal{B}'$ and $\mathcal{A}'$, respectively. In other words, $\BSD(X_{T_1})$ and $\BSD(X_{T_2})$ lie in the images of the functors $\mathcal{F}^D_{\iota_\beta}$ and $\mathcal{F}^D_{\iota_\alpha}$ induced by the inclusions $\iota_\beta$ and $\iota_\alpha$, respectively (see Definition~\ref{def:InducedFunctors}). Thus, by the Pairing Adjunction (Theorem~\ref{thm:PairingAdjunction}), the right-hand side of (\ref{eqn:FirstGlueingProofStep1}) is equal to	
	\begin{equation}\label{eqn:FirstGlueingProofStep2}
	\BSD(X_{T_1})^{\mathcal{B}'}\boxtimes\typeA{\mathcal{B}',\mathcal{A}'}{\mathcal{F}^{AA}_{\iota_\beta,\iota_\alpha}(\BSAA(X))}\boxtimes\BSD(X_{T_2})^{\mathcal{A}'}
	\end{equation}
	where $\mathcal{F}^{AA}_{\iota_\beta,\iota_\alpha}$ is the functor induced by the inclusions $\iota_\beta$ and $\iota_\alpha$.
	
	So let us compute $\mathcal{F}^{AA}_{\iota_\beta,\iota_\alpha}(\BSAA(X))$. A Heegaard diagram for $X$ is shown in Figure~\ref{fig:GlueingHD}, which is obtained by ``flattening'' the thickened 4-punctured sphere from Figure~\ref{fig:GlueingAAP}. The regions adjacent to a basepoint are shaded green (\textcolor{lightgreen}{$\blacksquare$}). $\alpha$- and $\beta$-arcs are labelled by ${\red a}$, ${\red b}$, ${\red c}$, ${\red d}$, ${\red e}$, ${\red f}$ and ${\blue A}$, ${\blue B}$, ${\blue C}$, ${\blue D}$, ${\blue E}$, ${\blue F}$, respectively, as in Figure~\ref{fig:GlueingAAP}. 
	Intersection points of these are suggestively labelled by black Greek and Roman letters which indicate which $\alpha$- and $\beta$-arcs the intersection points lie on. For example, the intersection point $\beta$ lies on both ${\red b}$ and ${\blue B}$; the same principle applies to all Greek letter labels. The intersection points labelled by Roman letters lie on exactly one of $\{{\red e, f},{\blue E, F}\}$. This makes it easy to determine the generators of $\mathcal{F}^{AA}_{\iota_\beta,\iota_\alpha}(\BSAA(X))$ and their corresponding idempotents, which are shown in Figure~\ref{fig:GlueingGeneratorList}. 
	
	Next, let us consider the domains of this Heegaard diagram. The two square regions with vertices $\{d,\delta,D,\zeta_1\}$ and $\{\beta,b,\varepsilon_1,B\}$ are labelled by $\square_\delta$ and $\square_\beta$, respectively. All other regions in the Heegaard diagram have at least one boundary component and are labelled by ${\red P_i}$, ${\red Q_i}$, ${\blue p_i}$ and ${\blue q_i}$. There are two domains that have two labels, namely ${\red P_1}/{\blue q_3}$ and ${\red Q_3}/{\blue p_1}$. There are four pairs of regions that have the same label. Note however, that the coefficients of any paired regions agree in any domain connecting generators in $\mathcal{F}^{AA}_{\iota_\beta,\iota_\alpha}(\BSAA(X))$, ie those generators that occupy $\red e$, $\red f$, $\blue E$ and~$\blue F$. Thus, we may use these labels to describe all domains that contribute to $\mathcal{F}^{AA}_{\iota_\beta,\iota_\alpha}(\BSAA(X))$. 
		
	In the following, let us write domains $D$ as formal differences $D_+-D_-$ of unordered sets of regions $D_+$ and $D_-$ with $D_+\cap D_-=\emptyset$ such that 
	\[D=\sum_{r\in D_+}r-\sum_{r\in D_-}r.\]
	Let us calculate some connecting domains between the generators. First of all, here are some bigons with a single boundary puncture:
	\begin{align*}
			\{{\blue p_2}\}\co &{\red b}{\blue b}\rightarrow{\red b}{\blue a}, &
			\{{\red P_2}\}\co &{\red b}{\blue b}\rightarrow{\red a}{\blue b},&
			\{{\blue q_4}\}\co &{\red d}{\blue c}\rightarrow{\red d}{\blue d}, &
			\{{\red Q_4}\}\co &{\red c}{\blue d}\rightarrow{\red d}{\blue d}.
	\end{align*}
	The following domains consist of two bigons, each with a single boundary puncture:
	\begin{align*}
			\{{\blue q_2},\square_\beta\}\co &{\red a}{\blue a} \rightarrow{\red a}{\blue b}, &
			\{{\red Q_2},\square_\beta\}\co &{\red a}{\blue a}\rightarrow {\red b}{\blue a},&
			\{{\blue p_4},\square_\delta\}\co &{\red c}{\blue d}\rightarrow {\red c}{\blue c}, &
			\{{\red P_4},\square_\delta\}\co &{\red d}{\blue c}\rightarrow {\red c}{\blue c}.
	\end{align*}
	The polygonal regions ${\red P_1}/{\blue q_3}$ and ${\red Q_3}/{\blue p_1}$ connect the following generators:
	\begin{align*}
			\{{\red P_1}/{\blue q_3}\}\co &{\red a}{\blue b}\rightarrow {\red d}{\blue c}, &
			\{{\red Q_3}/{\blue p_1}\}\co &{\red b}{\blue a}\rightarrow {\red c}{\blue d}.
	\end{align*}
	All those domains above contribute to the type~AA structure. They correspond to the solid arrows in Figure~\ref{fig:GlueingDomains}. From these, we can compute the connecting domains labelling the dashed arrows in the same figure. We claim that these are in fact all domains that connect these eight generators and that have non-negative multiplicities. To show this, we can argue as follows: Pick any two generators $x$ and $y$ and calculate a connecting domain between them, eg by following along the arrows in Figure~\ref{fig:GlueingDomains}. We observe that if $x$ and $y$ are among those eight generators from Figure~\ref{fig:GlueingDomains}, we can always choose a domain 
	\[X=(X_+-X_-)\co x\rightarrow y\]
	with multiplicities in $\{-1,0,+1\}$. 	
	We can obtain any other connecting domain between $x$ and $y$ by adding a periodic domain $P=P_+-P_-$ to $X$. Suppose this new connecting domain $(X_++P_+)-(X_-+P_-)$ has non-negative (respectively non-positive) coefficients only.  Then 
	\begin{align*}
	X_++P_+\subseteq X_-+P_-, &\text{ so } P_+\subseteq X_-, X_+\subseteq P_-,\text{ or respecively}\\
	X_++P_+\supseteq X_-+P_-, &\text{ so } P_+\supseteq X_-, X_+\supseteq P_-.
	\end{align*}
	In particular, since neither $X_-$ nor $X_+$ contain any region more than once, we only need to consider periodic domains which have multiplicities $\geq -1$ or $\leq +1$.
	
	Let us compute the group of periodic domains of our Heegaard diagram. It is easy to see that it is freely generated by the following three domains:
	\begin{align*}
	\mathcal{D}_1&=\{{\red Q_4},{\red Q_3}/{\blue p_1},{\blue p_2}\}-\{{\red P_2},{\red P_1}/{\blue q_3},{\blue q_4}\},\\
	\mathcal{D}_2&=\{{\blue p_2},{\blue q_2}\}-\{{\red P_2},{\red Q_2}\},\\
	\mathcal{D}_3&=\{{\blue p_4},{\blue q_4}\}-\{{\red P_4},{\red Q_4}\}.
	\end{align*}
	The periodic domains which have multiplicities $\geq -1$ or $\leq +1$ are given by $\mathcal{D}_1$, $\mathcal{D}_2$, $\mathcal{D}_3$,
	\begin{align*}
	\mathcal{D}_1-\mathcal{D}_2&=\{{\red Q_4},{\red Q_2},{\red Q_3}/{\blue p_1}\}-\{{\red P_1}/{\blue q_3},{\blue q_4},{\blue q_2}\},\\
	\mathcal{D}_1+\mathcal{D}_3&=\{{\red Q_3}/{\blue p_1},{\blue p_4},{\blue p_2}\}-\{{\red P_4},{\red P_2},{\red P_1}/{\blue q_3}\},\\
	\mathcal{D}_1-\mathcal{D}_2+\mathcal{D}_3&=\{{\red Q_2},{\red Q_3}/{\blue p_1},{\blue p_4}\}-\{{\red P_4},{\red P_1}/{\blue q_3},{\blue q_2}\},\\
	\mathcal{D}_2+\mathcal{D}_3&=\{{\blue p_2},{\blue q_2},{\blue p_4},{\blue q_4}\}-\{{\red P_2},{\red Q_2},{\red P_4},{\red Q_4}\},\\
	\mathcal{D}_2-\mathcal{D}_3&=\{{\red P_4},{\red Q_4},{\blue p_2},{\blue q_2}\}-\{{\red P_2},{\red Q_2},{\blue p_4},{\blue q_4}\}
	\end{align*}
	and their negatives.
	It is now elementary to check that indeed, Figure~\ref{fig:GlueingDomains} shows all connecting domains with non-negative multiplicities. 
	
	There are four generators that we have not considered yet, namely $\overline{{\red a}{\blue a}}$, $\underline{{\red a}{\blue a}}$, $\overline{{\red c}{\blue c}}$ and $\underline{{\red c}{\blue c}}$. They are connected by the following two contributing domains
	\begin{align*}
			\{\square_\beta\}\co &\underline{{\red a}{\blue a}}\rightarrow\overline{{\red a}{\blue a}}, &
			\{\square_\delta\}\co &\overline{{\red c}{\blue c}}\rightarrow\underline{{\red c}{\blue c}}.
	\end{align*}
	Thus the two generator pairs can be cancelled. Let us connect these two generators to those from Figure~\ref{fig:GlueingDomains}. There are for example the two domains 
	\begin{align*}
			\{{\red P_2},{\red Q_2}\}\co &{\red a}{\blue a}\rightarrow\underline{{\red a}{\blue a}}, &
			\{{\red P_4},{\red Q_4}\}\co &\underline{{\red c}{\blue c}}\rightarrow{\red c}{\blue c}.
	\end{align*}
	We can use the same arguments as above to verify that there are no domains with non-negative domains leaving $\underline{{\red a}{\blue a}}$ or terminating at $\underline{{\red c}{\blue c}}$ other than $\{\square_\beta\}$ and $\{\square_\delta\}$. Hence, after cancellation, the remaining complex is the same as before. 
	We can now use the $d^2$-relation for type AA structures to deduce that all domains from Figure~\ref{fig:GlueingDomains} -- not only those on the solid arrows -- contribute. For example, the contribution to the $d^2$-relation of the composition ${\red a}{\blue a}\rightarrow{\red a}{\blue b}\rightarrow{\red d}{\blue c}$ can only be cancelled by the differential of ${\red a}{\blue a}\rightarrow{\red d}{\blue c}$. We can argue similarly for all other dashed arrows. 	
	Thus, we obtain a type~AA structure $\mathcal{P}'$ which, by definition, is homotopic to $\mathcal{F}^{AA}_{\iota_\beta,\iota_\alpha}(\BSAA(X))$. Thus, the expression (\ref{eqn:FirstGlueingProofStep2}) is chain homotopic to
	\begin{equation}\label{eqn:FirstGlueingProofStep3}
	\BSD(X_{T_1})^{\mathcal{B}'}\boxtimes\typeA{\mathcal{B}',\mathcal{A}'}{\mathcal{P}'}\boxtimes \BSD(X_{T_2})^{\mathcal{A}'}.
	\end{equation}
	Let $\Anull$ be the quotient algebra obtained from $\Ad$ by setting $p_3=0=q_1$. Now identify $\Id$ with $\mathcal{I}'_\alpha$ and $\mathcal{I}'_\beta$ such that an idempotent for site $s$ corresponds to an idempotent which does not occupy the $\alpha$- and respectively $\beta$-arc labelled $s$ in Figure~\ref{fig:GlueingAAP}. Under this identification, there are unique $\Id$-algebra epimorphisms
	\[\pi_\alpha\co \mathcal{A}'\rightarrow\Anull\quad\text{and}\quad\pi_\beta\co \mathcal{B}'\rightarrow\Anull.\]
	Note that, as the notation for our domains suggests, each domain from Figure~\ref{fig:GlueingDomains} is recorded in $\mathcal{P}'$ by the algebra elements in $\mathcal{A}'$ and $\mathcal{B}'$ corresponding to the algebra elements in $\Anull$ obtained by the product of all blue and red labels of the domain, respectively. 
	In fact, $\mathcal{P}'$ is equal to the image of the type~AA structure $\mathcal{P}$ from Figure~\ref{fig:GlueingTypeAAstructure} (viewed as a bimodule over $\Anull$) under the induced functor $\mathcal{F}^{AA}_{\pi_\beta,\pi_\alpha}$. Thus, by the pairing adjuction (Theorem~\ref{thm:PairingAdjunction}), the expression (\ref{eqn:FirstGlueingProofStep3}) is equal to
	\begin{equation}\label{eqn:FirstGlueingProofStep4}
	\mathcal{F}_{\pi_\beta}(\BSD(X_{T_1}))^{\Anull}\boxtimes\typeA{\Anull,\Anull}{\mathcal{P}}\boxtimes \mathcal{F}_{\pi_\alpha}(\BSD(X_{T_2}))^{\Anull}.
	\end{equation}
Let $\pi_{31}\co\Ad\rightarrow\Anull$ be the quotient map defining $\Anull$. Then by a final application of the Pairing Adjunction, it is sufficient to identify $\mathcal{F}_{\pi_\beta}(\BSD(X_{T_1}))$ and $\mathcal{F}_{\pi_\alpha}(\BSD(X_{T_2}))$ with $\mathcal{F}_{\pi_{31}}(\CFTd(\rr(T_1)))$ and $\mathcal{F}_{\pi_{31}}(\CFTd(T_2))$, respectively. 

The second identification is immediate from Lemma~\ref{lem:IdentificationCFTdBSD} below and the observation that the bordered sutured Heegaard diagram for $X_{T_2}$ can be regarded as a tangle Heegaard diagram for $T_2$, up to a minor modification at the tangle ends. The first identification follows likewise, observing that switching the roles of $\alpha$- and $\beta$-curves while at the same time switching the orientation of the Heegaard surface leaves the resulting complex unchanged, up to a reversal of Alexander gradings.

Finally, the $\delta$-grading on $\mathcal{P}$ is calculated as usual, see~\cite[Definition~5.13]{HDsForTangles}. By the additivity of the $\delta$-grading under glueing, the $\delta$-gradings on both sides of our glueing formula agree.
Similarly, if $A$ is the Alexander grading induced by an ordered matching for $T_2$, then the three generating periodic domains have vanishing Alexander gradings. This shows that the Alexander grading on $\mathcal{P}$ is well-defined (as a relative grading). Again, by additivity under glueing, the Alexander grading on both sides of the glueing formula agree.
\end{proof}

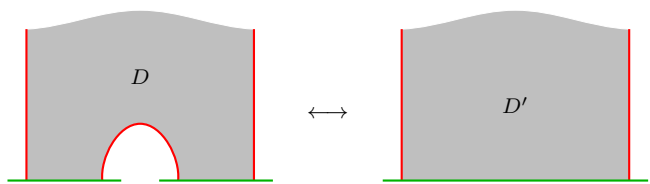
\begin{figure}[t]
	\centering
	$
	{\raisebox{-0.9cm}{
			\psset{unit=0.5}
			\begin{pspicture}(-4,-0.1)(4,4.6)
			\pscustom[fillstyle=solid,fillcolor=lightgray,linewidth=0pt,linecolor=white]{
				\psline(-3,4)(-3,0)
				\psline(-3,0)(3,0)
				\psline(3,0)(3,4)
				\psecurve(10,4.5)(3,4)(0,4.5)(-3,4)(-10,4.5)
			}
			\pscustom[fillstyle=solid,fillcolor=white,linewidth=0pt,linecolor=white]{
				\psecurve(0,-1)(-1,0)(0,1.5)(1,0)(0,-1)
				\psline(1,0)(1,-1)(-1,-1)(-1,0)
			}
			\rput(0,2.75){$D$}
			\psline[linecolor=red](3,0)(3,4)
			\psline[linecolor=red](-3,0)(-3,4)
			\psecurve[linecolor=red](0,-1)(-1,0)(0,1.5)(1,0)(0,-1)
			\psline[linecolor=darkgreen](-3.5,0)(-0.5,0)
			\psline[linecolor=darkgreen](0.5,0)(3.5,0)
			\end{pspicture}
		}
		\longleftrightarrow
		\raisebox{-0.9cm}{
			\psset{unit=0.5}
			\begin{pspicture}(-4,-0.1)(4,4.6)
			\pscustom[fillstyle=solid,fillcolor=lightgray,linewidth=0pt,linecolor=white]{
				\psline(-3,4)(-3,0)
				\psline(-3,0)(3,0)
				\psline(3,0)(3,4)
				\psecurve(10,4.5)(3,4)(0,4.5)(-3,4)(-10,4.5)
			}
			\rput(0,2){$D'$}
			\psline[linecolor=red](3,0)(3,4)
			\psline[linecolor=red](-3,0)(-3,4)
			\psline[linecolor=darkgreen](-3.5,0)(3.5,0)
			\end{pspicture}
		}
	}
	$
	\caption{A domain with multiplicity 1 near a tangle end.}\label{fig:DomainNearSillyArc}
\end{figure}

\begin{lemma}\label{lem:IdentificationCFTdBSD}
	Let \(D\) and \(D'\) be two domains which only differ in a small region of multiplicity 1 as shown in Figure~\ref{fig:DomainNearSillyArc}. Then, for a suitable choice of complex structure, \(D\) contributes iff \(D'\) does. 
\end{lemma}
\begin{proof}
	This follows from the same arguments as \cite[Proposition~2.7]{Hanselman}.
\end{proof}

%% file: sections/Classification.tex

\section{Curved complexes for marked surfaces}\label{sec:classification}

In this section, we prove the main classification results which allow us to interpret peculiar modules and morphisms between them geometrically in terms of Lagrangian intersection Floer theory. 
These classification results are not particular to peculiar modules. 
They actually hold in a more general framework, namely for curved complexes over certain algebras constructed from \textit{arbitrary} surfaces with boundary (plus some extra data). 
Therefore, we will work in this general framework throughout this section. In section~\ref{sec:glueingrevisited}, we will return to the discussion of our tangle invariants and summarize the consequences of our main classification results from this section to peculiar modules. 

Many ideas in this section originate from~\cite{HRW} and~\cite{Kontsevich}. In the first paper, Hanselman, Rasmussen and Watson classify extendable type D structures in the context of a bordered Heegaard Floer theory for 3-manifolds with torus boundary. While our tangle invariants are classified by immersed curves on the 4-punctured sphere, their immersed curves live on the once-punctured torus. The reader should feel encouraged to read the relevant passages in their paper in parallel to this section, as our discussion will stay rather close to theirs, see also Remark~\ref{rem:comparisonToHRW}. 

In~\cite{Kontsevich}, Haiden, Katzarkov and Kontsevich classify twisted complexes over marked surfaces. Their results, unlike the ones from~\cite{HRW}, do not include a correspondence between Lagrangian intersection Floer theory and morphism spaces. Nonetheless, the main ideas about how to set up the correct framework in which a classification result for curved complexes could hold come from this paper. For a slightly more detailed comparison between their setting and ours, see Remark~\ref{rem:comparisonToKontsevich}. 

The broad outline of this section is as follows: after setting up the general framework of curved complexes over marked surfaces with arc systems, we introduce the category of precurves. The notion of precurves serves as a useful intermediary between the algebraic nature of curved complexes and the geometric nature of immersed curves with local systems. In subsections~\ref{subsec:precurves}, \ref{subsec:precurves_geometry} and~\ref{subsec:simplify_precurves}, we show that the category of curved complexes and the category of precurves over a fixed marked surface with arc system are equivalent and we discuss how the simplification algorithm from~\cite{HRW} can be adapted to precurves. In subsections~\ref{subsec:ClassificationMor} and~\ref{subsec:formulaMor}, we classify the homology of morphism spaces for a particularly simple class of precurves in terms of Lagrangian intersection theory, before finally proving the full classification of precurves in subsection~\ref{subsec:complete_classification}.

\subsection{Curved complexes over marked surfaces with arc systems}

\begin{definition}\label{def:MarkedSurface}
	A \textbf{marked surface} is a pair $(S,M)$ where $S$ is a connected oriented surface with non-empty boundary and $M$ is a (possibly empty) set of basepoints on  $\partial S$. 
	An \textbf{arc} $a$ of a marked surface $(S,M)$ is (the image of) an embedding of an oriented closed interval
	\[(I, \partial I) \hookrightarrow (S,\partial S\smallsetminus M)\]
	such that its class $[a]\in\pi_1(S,\partial S\smallsetminus M)$ is non-trivial.
	Suppose $A$ is a non-empty set of pairwise disjoint arcs on a marked surface $(S,M)$. For each arc $a\in A$, choose a closed neighbourhood $N(a)$ of $a$ such that $N(a)\cap N(a')=\emptyset$ for all $a'\in A\smallsetminus \{a\}$. Let $s_1(a)$ and $s_2(a)$ be the closures of the two components of $\partial N(a)\smallsetminus\partial S$ such that $s_1(a)$ lies to the right of the oriented arc $a$ and $s_2(a)$ to its left; we call $s_1(a)$ and $s_2(a)$ the two \textbf{sides} of $a$. We also fix a foliation $\mathcal{F}_a=I\times I$ of $N(a)$ such that $s_1(a)$, $s_2(a)$ and $a$ are leaves of $\mathcal{F}_a$. 

	Furthermore, we call the closures of the connected components of $S\smallsetminus \bigcup_{a\in A}N(a)$ \textbf{faces} and denote the set of all faces by $F(S,M,A)$. We call~$A$ (together with a choice of fixed~$N(a)$ and foliation~$\mathcal{F}_a$ as above) an \textbf{arc system} if 
	each face $f\in F(S,M,A)$ is a topological disc containing at most one point in $M$. If it does contain a point in $M$, we call it an \textbf{open face}. Otherwise, we call it \textbf{closed}. 
	%
	Given a face $f\in F(S,M,A)$, we denote the set of all sides adjacent to $f$ by $S(f)$ and write $n_f$ for the cardinality of $S(f)$. Note that $n_f\geq2$ for closed faces $f$, because $n_f=0$ would imply $A=\emptyset$ and $n_f=1$ would imply that the arc adjacent to $f$ can be pushed into $\partial S\smallsetminus M$.
	\end{definition}
	
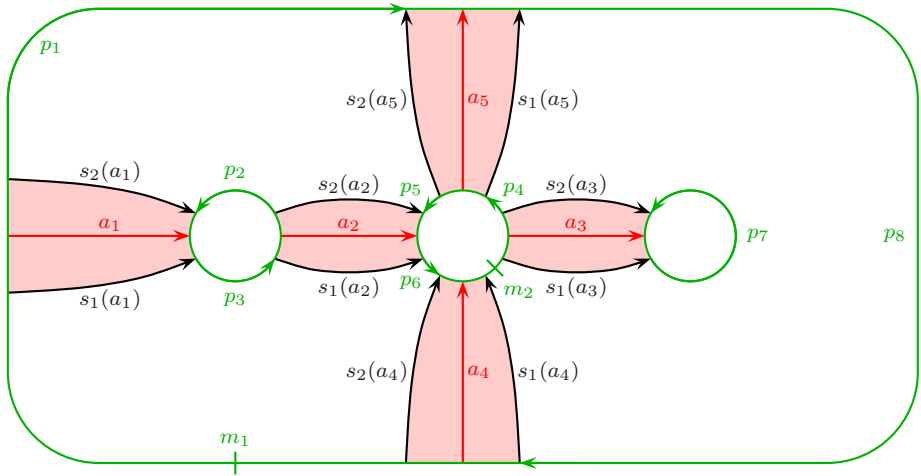
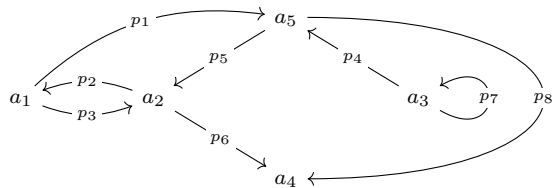
\begin{figure}[t]
	\centering
	
	\begin{subfigure}[b]{\textwidth}
		\centering
		\psset{unit=0.15}
		\begin{pspicture}(-41,-21)(41,21)
		\pnode(-20,0){A}
		\pnode(20,0){C}
		\pnode([nodesep=4,angle=30]A){A30}
		\pnode([nodesep=4,angle=150]A){A150}
		\pnode([nodesep=4,angle=-30]A){Am30}
		\pnode([nodesep=4,angle=-150]A){Am150}
		\pnode([nodesep=4,angle=30]C){C30}
		\pnode([nodesep=4,angle=150]C){C150}
		\pnode([nodesep=4,angle=-30]C){Cm30}
		\pnode([nodesep=4,angle=-150]C){Cm150}
		
		\pscustom*[linecolor=lightred]{
			\pscurve(-40,5)(-33,4.5)(-28,3.6)(A150)
			\psline(A150)(Am150)
			\pscurve(Am150)(-28,-3.6)(-33,-4.5)(-40,-5)
		}
		\psline[linecolor=red]{<-}(-24,0)(-40,0)
		\pscurve{->}(-40,5)(-33,4.5)(-28,3.6)(A150)
		\pscurve{->}(-40,-5)(-33,-4.5)(-28,-3.6)(Am150)	
		\uput{0.4}[90](-31,0){\textcolor{red}{$a_1$}}
		\uput{4.7}[90](-31,0){$s_2(a_1)$}
		\uput{-4.7}[-90](-31,0){$s_1(a_1)$}
		
		\pscustom*[linecolor=lightred]{
			\pscurve(4;150)(-7,3)(-13,3)(A30)
			\psline(A30)(Am30)
			\pscurve(Am30)(-13,-3)(-7,-3)(4;-150)
		}
		\psline[linecolor=red]{<-}(-4,0)(-16,0)
		\pscurve{<-}(4;150)(-7,3)(-13,3)(A30)
		\pscurve{<-}(4;-150)(-7,-3)(-13,-3)(Am30)	
		\uput{0.4}[90](-10,0){\textcolor{red}{$a_2$}}
		\uput{3.6}[90](-10,0){$s_2(a_2)$}
		\uput{-3.6}[-90](-10,0){$s_1(a_2)$}
		
		\pscustom*[linecolor=lightred]{
			\pscurve(4;30)(7,3)(13,3)(C150)
			\psline(C150)(Cm150)
			\pscurve(Cm150)(13,-3)(7,-3)(4;-30)
		}
		\psline[linecolor=red]{->}(4,0)(16,0)
		\pscurve{->}(4;30)(7,3)(13,3)(C150)
		\pscurve{->}(4;-30)(7,-3)(13,-3)(Cm150)	
		\uput{0.4}[90](10,0){\textcolor{red}{$a_3$}}
		\uput{3.6}[90](10,0){$s_2(a_3)$}
		\uput{-3.6}[-90](10,0){$s_1(a_3)$}
		
		\pscustom*[linecolor=lightred]{
			\pscurve(5,-20)(4.5,-13)(3.6,-8)(4;-60)
			\psline(4;-60)(4;-120)
			\pscurve[liftpen=1](4;-120)(-3.6,-8)(-4.5,-13)(-5,-20)
		}
		\psline[linecolor=red]{<-}(0,-4)(0,-20)
		\pscurve{->}(5,-20)(4.5,-13)(3.6,-8)(4;-60)
		\pscurve{->}(-5,-20)(-4.5,-13)(-3.6,-8)(4;-120)
		\uput{0.4}[0](0,-12){\textcolor{red}{$a_4$}}
		\uput{4.8}[180](0,-12){$s_2(a_4)$}
		\uput{-4.8}[0](0,-12){$s_1(a_4)$}
		
		\pscustom*[linecolor=lightred]{
			\pscurve(5,20)(4.5,13)(3.6,8)(4;60)
			\psline(4;60)(4;120)
			\pscurve[liftpen=1](4;120)(-3.6,8)(-4.5,13)(-5,20)
		}
		\psline[linecolor=red]{->}(0,4)(0,20)
		\pscurve{<-}(5,20)(4.5,13)(3.6,8)(4;60)
		\pscurve{<-}(-5,20)(-4.5,13)(-3.6,8)(4;120)
		\uput{0.4}[0](0,12){\textcolor{red}{$a_5$}}
		\uput{4.8}[180](0,12){$s_2(a_5)$}
		\uput{-4.8}[0](0,12){$s_1(a_5)$}
		
		\psarc[fillcolor=white,fillstyle=solid,linecolor=darkgreen](0,0){4}{0}{360}
		\psarc[fillcolor=white,fillstyle=solid,linecolor=darkgreen](20,0){4}{0}{360}
		\psarc[fillcolor=white,fillstyle=solid,linecolor=darkgreen](-20,0){4}{0}{360}
		
		\psline[linecolor=darkgreen,linearc=8,cornersize=absolute](0,-20)(-40,-20)(-40,20)(40,20)(40,-20)(0,-20)
		\psline[linecolor=darkgreen](-20,-21)(-20,-19)
		\rput(-20,-18){\textcolor{darkgreen}{$m_1$}}
		\psline[linecolor=darkgreen](3;-45)(5;-45)
		\rput(7;-45){\textcolor{darkgreen}{$m_2$}}
		
		\psarc[linecolor=darkgreen]{->}(-20,0){4}{-150}{-30}
		\psarc[linecolor=darkgreen]{->}(-20,0){4}{30}{150}
		\psarc[linecolor=darkgreen]{->}(0,0){4}{30}{60}
		\psarc[linecolor=darkgreen]{->}(0,0){4}{120}{150}
		\psarc[linecolor=darkgreen]{->}(0,0){4}{-150}{-120}
		\psarc[linecolor=darkgreen]{->}(20,0){4}{-150}{150}
		\psline[linecolor=darkgreen,linearc=8,cornersize=absolute]{->}(-40,5)(-40,20)(-5,20)
		
		\uput{4}[-45](-40,20){$\textcolor{darkgreen}{p_1}$}
		
		\uput{5}[90](-20,0){$\textcolor{darkgreen}{p_2}$}
		\uput{5}[-90](-20,0){$\textcolor{darkgreen}{p_3}$}
		
		\uput{5}[45](0,0){$\textcolor{darkgreen}{p_4}$}
		\uput{5}[135](0,0){$\textcolor{darkgreen}{p_5}$}
		\uput{5}[-135](0,0){$\textcolor{darkgreen}{p_6}$}
		
		\uput{5}[0](20,0){$\textcolor{darkgreen}{p_7}$}
		\rput(38,0){\textcolor{darkgreen}{$p_8$}}
		\psline[linecolor=darkgreen]{->}(10,-20)(5,-20)
		
		\end{pspicture}
		\caption{An arc system $A$ on the marked surface $(S,M)=(D^2\smallsetminus 3D^2, \{m_1, m_2\})$. The boundary components are labelled by the elementary algebra elements $p_i$, which correspond to the arrows in the quiver $Q(S,M,A)$ below. }\label{fig:MarkedSurfacePic}
	\end{subfigure}
	\begin{subfigure}[b]{\textwidth}
		\[\begin{tikzcd}[row sep=0.7cm, column sep=1.2cm]
		& &
		a_5
		\arrow{ld}[description]{p_5}
		\arrow[out=0, in=0,looseness=5]{dd}[description]{p_8}
		\\
		a_1
		\arrow[bend left=25]{rru}[description]{p_1}
		\arrow[bend right=20]{r}[description]{p_3}
		&
		a_2
		\arrow[bend right=20]{l}[description]{p_2}
		\arrow{rd}[description]{p_6}
		&&
		a_3
		\arrow{ul}[description]{p_4}
		\arrow[out=-30, in=30, loop]{lr}[description]{p_7}
		\\
		&&
		a_4
		\end{tikzcd}\]
		\caption{The quiver~$Q(S,M,A)$ for the arc system $A$ on the marked surface $(S,M)$ above.}\label{fig:MarkedSurfaceQuiver}
	\end{subfigure}
	\caption{An illustration of the Definitions~\ref{def:MarkedSurface} and~\ref{def:MarkedSurfaceQuiver}.}\label{fig:ExampleMarkedSurface}
\end{figure} 

\begin{remark}
	Our conventions are slightly different from those in~\cite{Kontsevich}: their markings $M$ correspond to the subsets $\partial S\smallsetminus M$. Also, their arc systems allow faces of arbitrary genus. Our arc systems correspond to their admissible arc systems, except that they require $n_f\geq3$. 
\end{remark}

\begin{example}\label{exa:MarkedSurface}
	Figure~\ref{fig:MarkedSurfacePic} shows an example of a marked surface $(S,M)$ with an arc system~$A$. We usually draw the boundary of the surface in green and mark elements in~$M$ by dashes through the boundary, in this case labelled by~$\textcolor{darkgreen}{m_1}$ and~$\textcolor{darkgreen}{m_2}$. The oriented arcs in $A$ are drawn in red, labelled by $\textcolor{red}{a_i}$, $i=1,\dots,5$. The neighbourhood of each arc $\textcolor{red}{a_i}$ is shaded in light red, bounded by solid arrows representing the two sides of $\textcolor{red}{a_i}$. In this particular example, the arcs cut the surface into three components, which are the faces in $F(S,M,A)$. For each face $f$, we have labelled the components of $\partial S\cap f$ which do not contain the two points $\textcolor{darkgreen}{m_1}$ and $\textcolor{darkgreen}{m_2}$ by variables $\textcolor{darkgreen}{p_i}$. These correspond to the basic algebra elements from the next definition.	
	
	Note that this particular marked surface will be irrelevant for our tangle invariants; it simply serves as an example to illustrate the generality of the arguments in this section.
\end{example}

\begin{wrapfigure}{r}{0.4\textwidth}
	\centering
	\psset{unit=0.8}\medskip
	\begin{pspicture}(-3,-1)(3,1)
	\pscustom*[linecolor=lightred,linewidth=0pt]{
		\psline(-2,0.3)(2,0.3)
		\psline(2,-0.3)(-2,-0.3)
	}
	\psline[linecolor=red](-2,0)(2,0)
	\psline(-2,0.3)(2,0.3)
	\psline(-2,-0.3)(2,-0.3)
	\pscustom*[linecolor=lightred,linewidth=0pt]{
		\psline(-0.2,0.2)(0.2,0.2)
		\psline(0.2,-0.2)(-0.2,-0.2)
	}
	\rput[c](0,0){\textcolor{red}{$a$}}

	\rput[r](-2.2,0){\textcolor{darkgreen}{$\partial S$}}
	\rput[l](2.2,0){\textcolor{darkgreen}{$\partial S$}}
	\psline[linecolor=darkgreen,linewidth=1pt](-2,0.5)(-2,-0.5)
	\psline[linecolor=darkgreen,linewidth=1pt](2,0.5)(2,-0.5)
	\psline[linecolor=darkgreen,linewidth=1pt]{->}(-2,0.3)(-2,1)\rput[l](-1.8,0.7){$\textcolor{darkgreen}{p_{1}}$}
	\psline[linecolor=darkgreen,linewidth=1pt]{->}(-2,-1)(-2,-0.3)\rput[l](-1.8,-0.7){$\textcolor{darkgreen}{q_{1}}$}
	\psline[linecolor=darkgreen,linewidth=1pt]{<-}(2,0.3)(2,1)\rput[r](1.8,0.7){$\textcolor{darkgreen}{p_{2}}$}
	\psline[linecolor=darkgreen,linewidth=1pt]{<-}(2,-1)(2,-0.3)\rput[r](1.8,-0.7){$\textcolor{darkgreen}{q_{2}}$}
	\end{pspicture}
	\caption{A typical neighbourhood of an arc $a$. }\label{fig:TypicalNeighbourhoodOfArc}
\end{wrapfigure}
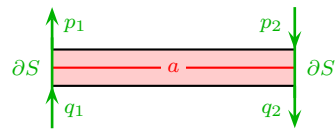
\myfixwrapfig

\begin{definition}\label{def:MarkedSurfaceQuiver}
	Given an arc system $A$ on a marked surface $(S,M)$, consider the graph $Q(S,M,A)$ obtained by contracting the closed neighbourhood of each arc in $A$ to a single point (which defines the vertices of $Q(S,M,A)$) and removing the interior of each face $f$ as well as any component of $\partial S\cap f$ containing a basepoint in $M$ (which defines the edges of $Q(S,M,A)$). By choosing the induced boundary orientation of $\partial f$, the 1-cells inherit an orientation from $(S,M)$, which turns $Q(S,M,A)$ into a quiver. Figure~\ref{fig:MarkedSurfaceQuiver} shows this quiver for the triple $(S,M,A)$ from Example~\ref{exa:MarkedSurface}.
\\\indent
	A typical neighbourhood of an arc $a\in A$ is shown in Figure~\ref{fig:TypicalNeighbourhoodOfArc}. As we can see, each arc $a$ is the source of at most two arrows and the target of at most two arrows. Using the notation from Figure~\ref{fig:TypicalNeighbourhoodOfArc}, we associate with the triple $(S,M,A)$ the path algebra
	\[\mathcal{A}:=\mathbb{F}_2 Q(S,M,A)/\{p_1q_1=0=q_2p_2\mid\text{arcs }a\in A\}.\]
	Note that we follow the convention to read algebra elements from right to left.
	Every arrow in the quiver corresponds to an algebra element which we call the \textbf{elementary algebra elements}. 
	For each closed face $f\in F(S,M,A)$, we write $U_f$ for the sum of all paths that start on a side on $f$ and go around $\partial f$ exactly once. We write  
	$U$ for the sum of $U_f$ over all closed faces. For each arc $a\in A$, denote the idempotent corresponding to~$a$ by~$\iota_a$ and let $\mathcal{I}$ be the ring of all idempotents. 
	We often consider $\mathcal{A}$ as a category: the underlying objects are given by the arcs in $A$. Then for any $a, b\in  A$, 
	\[\Mor_{\mathcal{A}}(a, b):=\iota_b.\mathcal{A}.\iota_a\]
	and multiplication is given by algebra multiplication. 
\end{definition}

\begin{definition}
	Let $A$ be an arc system on a marked surface $(S,M)$. An  
	$\mathbb{R}^{\geq0}$-grading on $\mathcal{A}$ is called a \textbf{$\delta$-grading}, if for all closed faces $f\in F(S,M,A)$, $U_f$ is homogeneous of degree 2 and the subalgebra of grading $0$ is $\mathcal{I}$. (Note that such gradings always exist: for example, if $n=\text{lcm}(n_f)$, we can define a $\frac{2}{n}\mathbb{Z}$-grading $\delta$ by setting $\delta(p)=\frac{2}{n_f}$ for any elementary algebra elements $p$ of a face $f$.)
\end{definition}

\begin{definition}\label{def:CCMarkedSurfaces}
	Let $A$ be an arc system on a marked surface $(S,M)$. We denote the dg category $\Cx^{U}(\Mat(\mathcal{A}(S,M,A)))$ by $\CC(S,M,A)$ and call it the \textbf{category of curved complexes} associated with $(S,M,A)$. 
\end{definition}

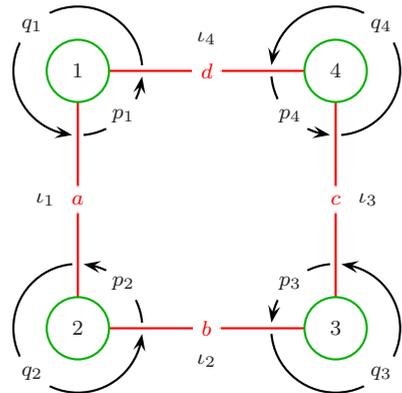
\begin{wrapfigure}{r}{0.4\textwidth}
	\centering
	\psset{unit=1.7}
	\begin{pspicture}(-1.52,-1.52)(1.52,1.52)
	
	\psline[linecolor=red](-1,-1)(-1,1)(1,1)(1,-1)(-1,-1)
	
	\rput(-1,0){\fcolorbox{white}{white}{\red $a$}}
	\rput(0,-1){\fcolorbox{white}{white}{\red $b$}}
	\rput(1,0){\fcolorbox{white}{white}{\red $c$}}
	\rput(0,1){\fcolorbox{white}{white}{\red $d$}}
	
	\rput(-1.25,0){$\iota_1$}
	\rput(0,-1.25){$\iota_2$}
	\rput(1.25,0){$\iota_3$}
	\rput(0,1.25){$\iota_4$}
	
	\pscircle[fillstyle=solid, fillcolor=white,linecolor=darkgreen](1,1){0.25}
	\pscircle[fillstyle=solid, fillcolor=white,linecolor=darkgreen](1,-1){0.25}
	\pscircle[fillstyle=solid, fillcolor=white,linecolor=darkgreen](-1,1){0.25}
	\pscircle[fillstyle=solid, fillcolor=white,linecolor=darkgreen](-1,-1){0.25}
	
	\rput[c](-1,1){1}
	\rput[c](-1,-1){2}
	\rput[c](1,-1){3}
	\rput[c](1,1){4}
	
	\rput(-1,1){
		\psarcn{<-}(0,0){0.5}{-5}{-85}
		\psarcn{<-}(0,0){0.5}{-95}{5}
		\pscircle*[linecolor=white](0.5;-45){8pt}
		\rput(0.5;-45){$p_1$}
		\pscircle*[linecolor=white](0.5;135){8pt}
		\rput(0.5;135){$q_1$}
	}
	
	\rput(-1,-1){
		\psarcn{<-}(0,0){0.5}{85}{5}
		\psarcn{<-}(0,0){0.5}{-5}{95}
		\pscircle*[linecolor=white](0.5;45){8pt}
		\rput(0.5;45){$p_2$}
		\pscircle*[linecolor=white](0.5;-135){8pt}
		\rput(0.5;-135){$q_2$}
	}
	
	\rput(1,-1){
		\psarcn{<-}(0,0){0.5}{175}{95}
		\psarcn{<-}(0,0){0.5}{85}{-175}
		\pscircle*[linecolor=white](0.5;135){8pt}
		\rput(0.5;135){$p_3$}
		\pscircle*[linecolor=white](0.5;-45){8pt}
		\rput(0.5;-45){$q_3$}
	}
	
	\rput(1,1){
		\psarcn{<-}(0,0){0.5}{-95}{-175}
		\psarcn{<-}(0,0){0.5}{175}{-85}
		\pscircle*[linecolor=white](0.5;-135){8pt}
		\rput(0.5;-135){$p_4$}
		\pscircle*[linecolor=white](0.5;45){8pt}
		\rput(0.5;45){$q_4$}
	}
	
	\end{pspicture}
	\caption{The marked surface and arc system which are relevant for peculiar modules, see Example~\ref{exa:pqModSpecialCaseofCC}.}\label{fig:nutshelll}\vspace*{-20pt}
\end{wrapfigure}
\myfixwrapfig

\begin{example}\label{exa:pqModSpecialCaseofCC}
	Let $(S,M)=(S^2\smallsetminus 4 D^2,\emptyset)$ and let $A$ be the arc system consisting of four arcs which divide the surface into exactly two faces each of which has four boundary components, as shown in Figure~\ref{fig:nutshelll}. Then $\mathcal{A}=\Ad$ and $\CC(S,M,A)=\pqMod$.
\end{example}

\begin{remark}
	Usually, we fix a $\delta$-grading on $\mathcal{A}$ in the definition of the category $\CC(S,M,A)$ of curved complexes. This plays the role of the $\mathbb{Z}$-grading in section~\ref{sec:AlgStructFromGDCats}, except that the differential here \textit{increases} $\delta$-grading by 1. So actually, with respect to the $\delta$-grading, one should call $\CC(S,M,A)$ the category of curved \textit{co}complexes. We have chosen this convention, because, as seen in the previous example, the category $\pqMod$ of peculiar modules is a special case of $\CC(S,M,A)$, and for $\pqMod$, there is also a homological grading $h$, which decreases along differentials by 1. The grading $h$ is defined as a combination of the $\delta$-grading and the Alexander grading, see Definition~\ref{def:RecallGradingsFromHDsForTangles}. The latter is preserved by the differentials, so for the present section it is irrelevant.
\end{remark}

\begin{remark}\label{rem:comparisonToKontsevich}
	In~\cite{Kontsevich}, Haiden, Katzarkov and Kontsevich associate with the triple $(S,M,A)$ the path algebra
	\[\Atw:=\mathbb{F}_2 Q(S,M,A)/\{p_1p_2=0=q_2q_1\mid\text{arcs }a\in A\}.\]
	Note that the relations they impose on the free quiver algebra $\mathbb{F}_2 Q(S,M,A)$ are in some sense ``dual'' to ours. Haiden, Katzarkov and Kontsevich define an $A_\infty$-structure on $\Atw$, which then gives rise to the notion of twisted complexes using an $A_\infty$-version of $\Cx^{0}(\Mat(\cdot))$. In \cite[Theorem~4.3]{Kontsevich}, they show that twisted complexes over $\Atw$ are classified by immersed curves on $S$ with local systems.
	
	Curved complexes over $\mathcal{A}$ and twisted complexes over $\Atw$ are closely related. In~\cite[Theorem~5.10, Proposition~5.15]{MyThesis}, I define two functors between the category of peculiar modules and the category of twisted complexes for the triple $(S,M,A)$ from Example~\ref{exa:pqModSpecialCaseofCC} which I expect to set up an equivalence of categories. However, to define these functors, one first needs to pass to infinitely generated twisted complexes and peculiar modules, which complicates matters. 
	
	The classification results below for peculiar modules and more generally curved complexes (which we assume to be finitely generated throughout this paper) suggest an alternative approach. We may argue that twisted complexes and peculiar modules encode the same geometric objects, namely immersed curves on the 4-punctured sphere with local systems. The only difference is that finitely generated peculiar modules correspond to \textit{compact} curves, while finitely generated twisted complexes can also represent non-compact ones. This explains why we cannot expect to obtain an equivalence for peculiar modules and twisted complexes in the finitely generated setting. 
\end{remark}

\begin{remark}\label{rem:comparisonToHRW}
	It is worthwhile to compare our definition of curved complexes associated with marked surfaces to extendable type D structures studied by Hanselman, Rasmussen and Watson in~\cite{HRW}. In fact, for the proof of Theorem~\ref{thm:EverythingIsLoopTypeUpToLocalSystems} below,  we will recycle a simplification algorithm which is a central ingredient in their classification result~\cite[Theorem~32]{HRW}. The main technical difference is that in~\cite{HRW}, the algebra $\mathcal{A}$ is truncated in the sense that any path that traverses certain elementary algebra elements more than once is set equal to zero. We work instead with the full algebra $\mathcal{A}$. \textit{A posteriori}, we will see that these two perspectives coincide and that it is actually sufficient to work with marked surfaces and arc systems in which each face has a basepoint, see Corollary~\ref{cor:AddingBasepointFunctorIsFaithfulUpToHom} and Corollary~\ref{cor:PeculiarModulesFromNiceDiagrams}.
	
	Apart from this rather technical difference (which could easily have been avoided), Theorem~\ref{thm:EverythingIsLoopTypeUpToLocalSystems} below follows directly from~\cite[Theorem~32]{HRW}. 
	However, in~\cite{HRW}, the classification of morphisms relies on some ad-hoc arguments that seem to be particular to the torus algebra. We propose alternative arguments which work for arbitrary marked surfaces with arc systems. The notion of precurves, which we define in the next subsection, seems to provide a good framework for these kinds of arguments. 
\end{remark}

\subsection{An algebraic reinterpretation of curved complexes}\label{subsec:precurves}

\begin{definition}\label{def:AlgebraFromGlueingFaces}
	Let $A$ be an arc system on a marked surface $(S,M)$. We will now give an alternative description of the algebra $\mathcal{A}$ in terms of a subalgebra of a larger algebra built out of simpler pieces. For each face $f\in F(S,M,A)$, let $\mathcal{A}_f$ be the free path algebra of the cyclic or linear quiver with $n_f$ vertices, depending on whether $f$ is closed or open. If $f$ is open, we may think of the linear quiver as a deformation retract of $f$ to its boundary minus the basepoint. If $f$ is closed, let $f^{\ast}$ be the face $f$ punctured at a fixed point; then we may think of the cyclic quiver as a deformation retract of $f^{\ast}$ to the boundary of $f$. Thus, each side $s$ of the face $f$ corresponds to a vertex and hence to the constant path $\iota_s$ at that vertex. 
	
	We denote the vector space generated by the idempotents of $\mathcal{A}_f$ by $\mathcal{I}_f$ and the (non-unital) subalgebra generated by paths of length $\geq1$ by $\mathcal{A}^+_f$. These fit into a short exact sequence
	$$
	\begin{tikzcd}[row sep=10pt]
	0
	\arrow{r}{}
	&
	\mathcal{A}^+_f
	\arrow{r}{}
	&
	\mathcal{A}_f
	\arrow{r}{}
	&
	\mathcal{I}_f
	\arrow{r}{}
	&
	0.
	\end{tikzcd}
	$$
	By taking the direct sum over all faces $f\in F(S,M,A)$, we obtain the short exact sequence  
	$$
	\begin{tikzcd}[row sep=10pt]
	0
	\arrow{r}{}
	&
	\overline{\mathcal{A}}^+:=\bigoplus_f\mathcal{A}^+_f
	\arrow{r}{}
	&
	\overline{\mathcal{A}}:=\bigoplus_f\mathcal{A}_f
	\arrow{r}{}
	&
	\overline{\mathcal{I}}:=\bigoplus_f\mathcal{I}_f
	\arrow{r}{}
	&
	0.
	\end{tikzcd}
	$$
	Then $\mathcal{I}$ agrees with the subring of $\overline{\mathcal{I}}$ given by
	$$\{\iota_{s_1(a)}+\iota_{s_2(a)}\mid a\in A\}$$
	and $\mathcal{A}$ agrees with the subalgebra of $\overline{\mathcal{A}}$ whose underlying vector space is given by $\mathcal{I}\oplus\overline{\mathcal{A}}^+$.
	For each face $f\in F(S,M,A)$, let $p_f$ be the sum of all length 1 paths in $\mathcal{A}_f$, $U_f=p_f^{n_f}$ and $U$ the sum of $U_f$ over all faces $f$. (Note that $U_f=0$ for open faces.) Our \textbf{standard basis} of $\mathcal{A}^+_f.\iota_s$ is given by $\{p_f^n.\iota_s\}_{n\geq1}$. We call $n$ the \textbf{length of a basis element} $p_f^n.\iota_s$. We denote the shortest element of the standard basis between two sides $s$ and $t$ of the same face by $p^s_t$. 
\end{definition}

\begin{remark}\label{rem:precurveFunctors}
	If we view the algebras $\mathcal{A}$ and $\overline{\mathcal{A}}$ as categories whose objects correspond to constant paths, we can view the inclusions $\mathcal{I}\hookrightarrow\overline{\mathcal{I}}$ and $\mathcal{A}\hookrightarrow\overline{\mathcal{A}}$ in terms of a functor between the additive enlargements of $\mathcal{A}$ and  $\overline{\mathcal{A}}$
	$$\mathcal{F}\co\Mat\mathcal{A}\rightarrow\Mat\overline{\mathcal{A}}$$
	which sends an object $\iota_a$, $a\in A$, to $\iota_{s_1(a)}\oplus \iota_{s_2(a)}$ and a morphism $\varphi\in\Mor_{\mathcal{A}}(\iota_a,\iota_b)$ to the map
	$$
	\begin{tikzcd}[column sep=5cm,ampersand replacement=\&]
	\iota_{s_1(a)}\oplus \iota_{s_2(a)}
	\arrow{r}{\left(\begin{matrix}
		\iota_{s_1(b)}.\varphi.\iota_{s_1(a)} & \iota_{s_1(b)}.\varphi.\iota_{s_2(a)} \\
		\iota_{s_2(b)}.\varphi.\iota_{s_1(a)} & \iota_{s_2(b)}.\varphi.\iota_{s_2(a)}
		\end{matrix}\right)}
	\&
	\iota_{s_1(b)}\oplus \iota_{s_2(b)}.
	\end{tikzcd}
	$$
	In order to recover a given object $C$ in $\Mat\mathcal{A}$ from $\mathcal{F}(C)$, one needs to remember how the idempotents are matched up. This motivates the following definition.
\end{remark} 

\begin{definition}\label{def:SplittingCatsForEquivalenceFunctors}
	Consider the category $\Mat_i\overline{\mathcal{A}}$ enriched over $\delta$-graded vector spaces
	\begin{itemize}
		\item whose objects are pairs $(C, \{P_a\}_{a\in A})$, where $C$ is a $\delta$-graded right module over $\overline{\mathcal{I}}$ and 
		$$P_a\co C.\iota_{s_1(a)}\rightarrow C.\iota_{s_2(a)}$$ 
		is a $\delta$-grading preserving vector space isomorphism for every arc $a$, and
		\item whose morphism spaces are defined as follows: for any two right $\overline{\mathcal{I}}$-modules $C$ and $C'$, the short exact sequences from Definition~\ref{def:AlgebraFromGlueingFaces} induce a split exact sequence
		$$
		\begin{tikzcd}[row sep=10pt]
		0
		\arrow{r}{}
		&
		\Mor_{\Mat\overline{\mathcal{A}}^+}(C,C')
		\arrow{r}{}
		&
		\Mor_{\Mat\overline{\mathcal{A}}}(C,C')
		\arrow{r}{}
		&
		\Mor_{\Mat\overline{\mathcal{I}}}(C,C')
		\arrow{r}{}
		&
		0.
		\end{tikzcd}
		$$
		In other words, we may write each morphism $\varphi$ in the middle uniquely as $\varphi^++\varphi^\times$, where $\varphi^+$ is an element of the morphism space on the left and $\varphi^\times$ an element of the morphism space on the right. 
		We then define the morphism spaces 
		$$
		\Mor((C, \{P_a\}_{a\in A}),(C', \{P'_a\}_{a\in A}))$$
		of $\Mat_i\overline{\mathcal{A}}$ by
		$$\{\varphi\in\Mor_{\Mat\overline{\mathcal{A}}}(C,C')\mid \forall a\in A\co (\iota_{s_2(a)}.\varphi^\times.\iota_{s_2(a)})\circ P_a=P'_a\circ(\iota_{s_1(a)}.\varphi^\times.\iota_{s_1(a)})\}.
		$$
		\item Finally,  $\Mat_i\overline{\mathcal{A}}$ inherits its composition from $\Mat\overline{\mathcal{A}}$, which is indeed well-defined. 
	\end{itemize}
	We may extend the functor $\mathcal{F}$ from Remark~\ref{rem:precurveFunctors} to a functor
	$$\mathcal{F}\co \Mat\mathcal{A}\longrightarrow\Mat_i\overline{\mathcal{A}}$$
	by setting, for a generator $x.\iota_a\in C.\iota_a$, $P_a(x.\iota_{s_1(a)})=x.\iota_{s_2(a)}$. (Note that this makes the image of any morphism under $\mathcal{F}$ well-defined.) Conversely, we may also define a functor 
	$$\mathcal{G}\co \Mat_i\overline{\mathcal{A}}\longrightarrow\Mat\mathcal{A}$$
	as follows: 
	given an object $(C,\{P_a\}_{a\in A})$ of $\Mat_i\overline{\mathcal{A}}$, we define an object $\mathcal{G}(X)$ in $\Mat\mathcal{A}$ by setting
	$\mathcal{G}(X).\iota_a=C.\iota_{s_1(a)}$
	as vector spaces. Furthermore, if $\varphi\in\Mor_{\Mat_i\overline{\mathcal{A}}}((C,\{P_a\}),(C',\{P'_a\}))$, we define a morphism $\mathcal{G}(\varphi)$ in $\Mat\mathcal{A}$ by setting for each $a, a'\in A$
	\begin{align*}
	\iota_{a'}.\mathcal{G}(\varphi).\iota_a:=(\iota_{s_1(a')}.\varphi^\times.\iota_{s_1(a)})&+(\iota_{s_1(a')}.\varphi^+.\iota_{s_1(a)})+(P'_{a'})^{-1}\circ(\iota_{s_2(a')}.\varphi^+.\iota_{s_1(a)})\\
	&+(\iota_{s_1(a')}.\varphi^+.\iota_{s_2(a)})\circ P_a+(P'_{a'})^{-1}\circ(\iota_{s_2(a')}.\varphi^+.\iota_{s_2(a)})\circ P_a.
	\end{align*}
\end{definition}  

\begin{lemma}\label{lem:SplittingCatsForEquivalenceAdditive}
	\(\mathcal{G}\circ\mathcal{F}=\id_{\Mat\mathcal{A}}\) and \(\mathcal{F}\circ\mathcal{G}\cong\id_{\Mat_i\overline{\mathcal{A}}}\). In particular, \(\Mat\mathcal{A}\) and \(\Mat_i\overline{\mathcal{A}}\) are equivalent. 
\end{lemma}

\begin{proof}
	The first identity is obvious from the definitions. For the second, consider the two mutually inverse natural transformations
	$$
	\begin{tikzcd}[column sep=2cm]
	(C, \{\id_{C.\iota_{s_1(a)}}\}_{a\in A})=\mathcal{F}(\mathcal{G}(C, \{P_a\}_{a\in A}))
	\arrow[shift left]{r}{\eta_{(C, \{P_a\})}}
	&
	(C, \{P_a\}_{a\in A})
	\arrow[shift left]{l}{\varepsilon_{(C, \{P_a\})}}
	\end{tikzcd}
	$$
	given by 
	$$
	\iota_a.\eta_{(C, \{P_a\})}.\iota_a=
	\left(\begin{matrix}
	\id_{C.\iota_{s_1(a)}} & 0 \\
	0 & P_a
	\end{matrix}\right)
	\quad\text{ and }\quad
	\iota_a.\varepsilon_{(C, \{P_a\})}.\iota_a=
	\left(\begin{matrix}
	\id_{C.\iota_{s_1(a)}} & 0 \\
	0 & P^{-1}_a
	\end{matrix}\right)
	$$
	for all $a\in A$ and 0 in all other components. 
\end{proof}

\begin{definition}\label{def:precurvesAlgebraic}
	Let $\CC_i(S,M,A)=\Cx^{U}(\Mat_i(\overline{\mathcal{A}}))$. More explicitly, 
	\begin{itemize}
		\item objects are triples $(C, \{P_a\}_{a\in A},\partial)$, where $C$ is an object in $\Mat\overline{\mathcal{A}}$, 
		$$P_a\co C.\iota_{s_1(a)}\rightarrow C.\iota_{s_2(a)}$$ 
		is a $\delta$-grading preserving isomorphism for every arc $a$ and
		$$\partial\co C\longrightarrow C$$
		defines an endomorphism of $(C, \{P_a\}_{a\in A})\in\ob\Mat_i\overline{\mathcal{A}}$ which increases $\delta$-grading by 1 and satisfies $\partial^2=U\cdot\id$. 
		\item Morphism spaces are morphism spaces of the underlying objects in $\Mat_i\overline{\mathcal{A}}$. 
		The differentials on these morphism spaces are given as usual by pre- and post-composition with the endomorphisms $\partial$. 
		\item $\CC_i(S,M,A)$ inherits its composition from $\Mat_i\overline{\mathcal{A}}$. 
	\end{itemize}
	We call an object in $\CC_i(S,M,A)$ a \textbf{precurve}. To simplify notation, we often denote a precurve $(C, \{P_a\}_{a\in A},\partial)$ by $C$. We call a precurve \textbf{reduced} if its differential does not contain any identity component, ie $\partial^+=\partial$.
\end{definition}

\begin{corollary}\label{cor:SplittingCatsForEquivalenceComplexes}
	\(\CC(S,M,A)\) and \(\CC_i(S,M,A)\) are equivalent as dg categories.
\end{corollary}  

\begin{proof}
	This follows directly from the previous lemma. 
\end{proof}

\begin{lemma}\label{lem:CCreduced}
	Every curved complex is chain homotopic to a reduced one. The same holds for precurves. 
\end{lemma}

\begin{proof}
	Any arrow on a $\delta$-graded curved complex which is labelled by an idempotent does not have any other labels because of the $\delta$-grading. Therefore we can always reduce the number of generators of an unreduced curved complex by cancellation until the complex is reduced. 
	
	By the previous corollary, any given precurve is isomorphic to one in the image of $\mathcal{F}$. Since $\mathcal{F}$ preserves chain homotopy classes, this implies that any precurve is chain homotopic to $\mathcal{F}(C,\partial)$ for some reduced curved complex $(C,\partial)$ such that $\mathcal{F}(C,\partial)$ is a reduced precurve. Alternatively, we may also apply Lemma~\ref{lem:AbstractCancellation} directly. 
\end{proof}

\subsection{A geometric interpretation of precurves}\label{subsec:precurves_geometry}

\begin{figure}[t]
	\centering
	\begin{subfigure}[b]{\textwidth}
		\centering
		\psset{unit=0.15}
		\begin{pspicture}(-41,-21)(41,21)
		\psecurve(0,-2)(4.5,-13)(4;-52.5)(-13,4.5)
		\psecurve(-10,-36.75)(-18,-20)(-10,-3.25)(-20,13.5)

		\pnode(-20,0){A}
		\pnode(20,0){C}
		\pnode([nodesep=4,angle=30]A){A30}
		\pnode([nodesep=4,angle=150]A){A150}
		\pnode([nodesep=4,angle=-30]A){Am30}
		\pnode([nodesep=4,angle=-150]A){Am150}
		\pnode([nodesep=4,angle=30]C){C30}
		\pnode([nodesep=4,angle=150]C){C150}
		\pnode([nodesep=4,angle=-30]C){Cm30}
		\pnode([nodesep=4,angle=-150]C){Cm150}
		
		\pscustom*[linecolor=lightred]{
			\pscurve(-40,5)(-33,4.5)(-28,3.6)(A150)
			\psline(A150)(Am150)
			\pscurve(Am150)(-28,-3.6)(-33,-4.5)(-40,-5)
		}
		\psline[linecolor=red]{->}(-40,0)(-24,0)
		\pscurve{->}(-40,5)(-33,4.5)(-28,3.6)(A150)
		\pscurve{->}(-40,-5)(-33,-4.5)(-28,-3.6)(Am150)	
		\pscircle*[linecolor=lightred](-29,0){1}
		\rput[c]{-90}(-29,0){\textcolor{red}{$a_1$}}
		
		\pscustom*[linecolor=lightred]{
			\pscurve(4;150)(-7,3)(-13,3)(A30)
			\psline(A30)(Am30)
			\pscurve(Am30)(-13,-3)(-7,-3)(4;-150)
		}
		\psline[linecolor=red]{<-}(-4,0)(-16,0)
		\pscurve{<-}(4;150)(-7,3)(-13,3)(A30)
		\pscurve{<-}(4;-150)(-7,-3)(-13,-3)(Am30)	
		\pscircle*[linecolor=lightred](-13,0){1}
		\rput[c]{-90}(-13,0){\textcolor{red}{$a_2$}}
		
		\pscustom*[linecolor=lightred]{
			\pscurve(4;30)(7,3)(13,3)(C150)
			\psline(C150)(Cm150)
			\pscurve(Cm150)(13,-3)(7,-3)(4;-30)
		}
		\psline[linecolor=red]{->}(4,0)(16,0)
		\pscurve{->}(4;30)(7,3)(13,3)(C150)
		\pscurve{->}(4;-30)(7,-3)(13,-3)(Cm150)	
		\pscircle*[linecolor=lightred](6,0){1}
		\rput[c]{-90}(6,0){\textcolor{red}{$a_3$}}
		
		\pscustom*[linecolor=lightred]{
			\pscurve(5,-20)(4.5,-13)(3.6,-8)(4;-60)
			\psline(4;-60)(4;-120)
			\pscurve[liftpen=1](4;-120)(-3.6,-8)(-4.5,-13)(-5,-20)
		}
		\psline[linecolor=red]{<-}(0,-4)(0,-20)
		\pscurve{->}(5,-20)(4.5,-13)(3.6,-8)(4;-60)
		\pscurve{->}(-5,-20)(-4.5,-13)(-3.6,-8)(4;-120)
		\pscircle*[linecolor=lightred](0,-9){1}
		\rput[c](0,-9){\textcolor{red}{$a_4$}}
		
		\pscustom*[linecolor=lightred]{
			\pscurve(5,20)(4.5,13)(3.6,8)(4;60)
			\psline(4;60)(4;120)
			\pscurve[liftpen=1](4;120)(-3.6,8)(-4.5,13)(-5,20)
		}
		\psline[linecolor=red]{->}(0,4)(0,20)
		\pscurve{<-}(5,20)(4.5,13)(3.6,8)(4;60)
		\pscurve{<-}(-5,20)(-4.5,13)(-3.6,8)(4;120)
		\pscircle*[linecolor=lightred](0,16){1}
		\rput[c](0,16){\textcolor{red}{$a_5$}}
		
		\psarc[fillcolor=white,fillstyle=solid,linecolor=darkgreen](0,0){4}{0}{360}
		\psarc[fillcolor=white,fillstyle=solid,linecolor=darkgreen](20,0){4}{0}{360}
		\psarc[fillcolor=white,fillstyle=solid,linecolor=darkgreen](-20,0){4}{0}{360}
		
		\psline[linecolor=darkgreen,linearc=8,cornersize=absolute](0,-20)(-40,-20)(-40,20)(40,20)(40,-20)(0,-20)
		\psline[linecolor=darkgreen](-20,-21)(-20,-19)
		\rput(-20,-18){\textcolor{darkgreen}{$m_1$}}
		\psline[linecolor=darkgreen](3;-45)(5;-45)
		\rput(7;-40){\textcolor{darkgreen}{$m_2$}}
		
		\psarc[linecolor=darkgreen]{->}(-20,0){4}{-150}{-30}
		\psarc[linecolor=darkgreen]{->}(-20,0){4}{30}{150}
		\psarc[linecolor=darkgreen]{->}(0,0){4}{30}{60}
		\psarc[linecolor=darkgreen]{->}(0,0){4}{120}{150}
		\psarc[linecolor=darkgreen]{->}(0,0){4}{-150}{-120}
		\psarc[linecolor=darkgreen]{->}(20,0){4}{-150}{150}
		\psline[linecolor=darkgreen,linearc=8,cornersize=absolute]{->}(-40,5)(-40,20)(-5,20)
		
		\uput{4}[-45](-40,20){$\textcolor{darkgreen}{p_1}$}
		
		\uput{5}[90](-20,0){$\textcolor{darkgreen}{p_2}$}
		\uput{5}[-90](-20,0){$\textcolor{darkgreen}{p_3}$}
		
		\uput{5}[45](0,0){$\textcolor{darkgreen}{p_4}$}
		\uput{5}[135](0,0){$\textcolor{darkgreen}{p_5}$}
		\uput{5}[-135](0,0){$\textcolor{darkgreen}{p_6}$}
		
		\uput{5}[0](20,0){$\textcolor{darkgreen}{p_7}$}
		\rput(38,0){\textcolor{darkgreen}{$p_8$}}
		\psline[linecolor=darkgreen]{->}(10,-20)(5,-20)
		
		\rput[l](6,-16.5){$P_{a_3}=\left(
			\begin{matrix}
			0 & 1 \\ 
			1 & 0
			\end{matrix}\right)
			$}
		\rput[r](-8,16.5){$P_{a_5}=\left(
			\begin{matrix}
			1 & 1 \\ 
			0 & 1
			\end{matrix}\right)
			$}

		\psdots(4.5,13)(3.6,8)(-4.5,13)(-3.6,8)
		\psdots(7,3)(13,3)(7,-3)(13,-3)

		\psdots(-33,4.5)(-33,-4.5)(-10,3.25)(-10,-3.25)(-4.5,-13)(4.5,-13)
		
		\psecurve(-3.6,-4)(-33,4.5)(-3.6,8)(33,4.5)
		\psecurve(-3.6,11)(-33,-4.5)(-4.5,-13)(33,-4.5)
		\psecurve(1,3)(-4.5,13)(-10,3.1)(-4.5,-7)
		\psecurve(26,0)(13,3)(26,6)(26,-6)(13,-3)(26,0)
		\psecurve(0,3)(3.6,8)(7,3)(3,-2)
		\psecurve(-32,12)(4.5,13)(29,10)(29,-10)(7,-3)(29,4)
		
		\psecurve(-12.6,8)(-4.5,13)(3.6,8)(11.7,13)
		\psecurve(12.6,8)(4.5,13)(-3.6,8)(-11.7,13)
		
		\psline{->}(1,9.1)(1.5,8.65)
		\psline{->}(-1,9.1)(-1.5,8.65)
		
		\psecurve(-32,10)(-33,4.5)(-33,-4.5)(-32,-10)
		\psecurve(-10,10)(-10,3.1)(-10,-3.1)(-10,-10)
		\psecurve(10,12)(4.5,13)(-4.5,13)(-10,12)
		\psecurve(10,7)(3.6,8)(-3.6,8)(-10,8)
		\psecurve(13,9)(7,3)(13,-3)(7,-9)
		\psecurve(13,-9)(7,-3)(13,3)(7,9)
		\psecurve(-14.5,-12.5)(-4.5,-13)(4.5,-13)(8,-12)
		
		\end{pspicture}
		\caption{A simply-faced precurve for the triple $(S,M,A)$ from Figure~\ref{fig:ExampleMarkedSurface}.}\label{fig:ExamplePreloopPic}
	\end{subfigure}
	\begin{subfigure}[b]{\textwidth}\centering
		{$
			\begin{tikzcd}[row sep=1.3cm, column sep=1.4cm]
			&&
			e^{a_5}_2
			\arrow[swap,out=180,in=90]{lldd}{p_{2}p_{5}}
			\\
			&&
			e^{a_5}_1
			\arrow[leftarrow,bend right=10]{ld}{p_{1}p_{2}}
			\\
			e^{a_1}_1
			\arrow[bend right,swap,in=200]{rrd}{p_{6}p_{3}}
			\arrow[bend right,out=45,in=150,pos=0.4]{rru}{p_{1}}
			\arrow[leftarrow,bend left,out=25,in=170,pos=0.2,swap]{rru}{p_{2}p_{5}}
			&
			e^{a_2}_1
			\arrow[leftarrow,bend right,swap,out=30,in=160,pos=0.8]{uur}{p_{5}}
			\arrow[bend left,out=50,in=140,pos=0.6]{uur}{p_{1}p_{2}}
			&
			&
			e^{a_3}_1
			\arrow[bend right=30,swap]{uul}{p_{4}p_{7}}
			&
			e^{a_3}_2
			\arrow[bend right=15,swap,pos=0.3]{ull}{p_{4}}
			\arrow{l}{p_{7}}
			\\
			&&
			e^{a_4}_1
			\end{tikzcd}
			$}
		\caption{The image of the precurve shown above under the functor $\mathcal{G}$.}\label{fig:ExamplePreloopGraph}
	\end{subfigure}
	\caption{The geometric representation of a precurve and its corresponding curved complex.}\label{fig:ExamplePrecurve}
\end{figure} 

\begin{definition}
	Given three positive integers $i,j,n$ satisfying $i,j\leq n$ and $i\neq j$, we define the elementary matrix $E^{j}_{i}\in \GL_n(\mathbb{F}_2)$ by
	\renewcommand{\kbldelim}{(}
	\renewcommand{\kbrdelim}{)}
	$$E^{j}_{i}:=(\delta_{i'j'}+\delta_{ii'}\delta_{jj'})_{i'j'}=
	\kbordermatrix{
		&   & i & j &   \\
		& 1 &   &   &   \\
		i &   & \ddots  & 1 &   \\
		j &   &   & \ddots  &   \\
		&   &   &   & 1 
	}.
	$$
	Furthermore, let $P_{ij}\in \GL_n(\mathbb{F}_2)$ be the permutation matrix corresponding to the permutation $(i,j)$.
\end{definition}

\begin{definition}\label{def:precurvesGeometric}
	We interpret a reduced precurve $C$ geometrically on the marked surface $(S,M)$ with arc system $A$ as follows: 
	for each side $s$, we represent the fixed basis $\{e^s_{i}\}_{i=1,\dots,n_s}$ of $C.\iota_s$ by pairwise distinct dots $\bullet$ on $s$, which we order according to the orientation of $s$. We label the $i^\text{th}$ dot on $s$ by $(s,i)$. Next, since we are assuming that the precurve is reduced, every component of the differential lies in $\mathcal{A}_f^+$ for some face $f$. So let us consider such a component
	$$e^s_i\xrightarrow{p_f^n.\iota_s}e^{s'}_{i'}.$$
	We represent this component by an oriented or unoriented immersed interval connecting the dots $\bullet(s,i)$ and $\bullet(s',i')$, depending on the following three cases:	
	\begin{enumerate}
		\item If $f$ is open, we draw an unoriented curve between $\bullet(s,i)$ and $\bullet(s',i')$ on $f$. 
		\item If $f$ is closed and there is a component $e^{s'}_{i'}\rightarrow e^s_i$ of $\partial$, we do the same thing as in the first case. However, here, the unoriented curve represents both $e^s_i\rightarrow e^{s'}_{i'}$ and $e^{s'}_{i'}\rightarrow e^s_i$. 
		Also note that because of the $\delta$-grading, $p_f^n.\iota_s=p^s_{s'}$ and the component $e^{s'}_{i'}\rightarrow e^s_i$ is labelled by $p^{s'}_s$, such that the compositions of the two components are equal to $U_f.\iota_s$ and $U_f.\iota_{s'}$, respectively.
		\item If $f$ is closed and there is no component $e^{s'}_{i'}\rightarrow e^s_i$ in $\partial$, we draw an oriented immersed interval from $\bullet(s,i)$ to $\bullet(s',i')$ on the punctured face $f^{\ast}$ which is homotopic to the path $p_f^n.\iota_s$ in the quiver of $f$. 
	\end{enumerate}
	Up to homotopy in $f$, respectively $f^\ast$, this representation of $\partial$ is unique and uniquely determines the differential $\partial$. We call the oriented and unoriented immersed intervals \textbf{two-sided $f$-joins}.
	The $\partial^2$-relation for precurves says that each dot $\bullet(s,i)$ on a closed face $f$ is connected to some dot $\bullet(s',i')$ on the same face via an unoriented two-sided $f$-join. The same need not be true for open faces. So we connect those dots that do not lie on any two-sided $f$-joins to the boundary segment containing the puncture of $f$ by unoriented immersed curves on $f$. We call those intervals \textbf{one-sided $f$-joins}.
	We define the \textbf{length of an $f$-join} to be the length of the corresponding algebra element if the $f$-join is two-sided and 0 if it is one-sided. 
	
	Finally, we label each arc $a\in A$ by the matrix $P_a$. However, we can also represent the matrix $P_a$ itself geometrically if we have a decomposition of $P_a$ into elementary matrices
	$$P_a=P_a^{l_a}\cdots P_a^1$$
	for some $l_a\geq 0$, where for each $k=1,\dots,l_a$, $P_a^k$ is equal to some $E^{j}_{i}$ or $P_{ij}$, for some $i=i(k)$ and $j=j(k)$ with $i\neq j$. We divide the neighbourhood $N(a)$ of $a$ along some leaves of $\mathcal{F}_a$ into $l_a$~segments, ordered from $s_1(a)$ (right) to $s_2(a)$ (left), and label the $k^\text{th}$ segment by the matrix $P_a^k$. Finally, we represent each matrix $P_a^k$ graphically, as shown on the left of Figure~\ref{fig:traintrackmoves}. We call the crossing arcs for $P_a^k=P_{ij}$ \textbf{crossings} and the pair of arrows for $P_a^k=E^{j}_{i}$ \textbf{crossover arrows}. 
	This notation is borrowed from~\cite[section~3.3]{HRW}, except that we restrict crossings and crossover arrows to the neighbourhoods of the arcs. 

	Of course, the decomposition of $P_a$ is not unique. However, linear algebra tells us that it is unique up to the following moves, some of which are illustrated on the right of Figure~\ref{fig:traintrackmoves}:
	
	\begin{description}
		\item[(T1)] A crossings can be expressed in terms of crossover arrows, since $P_{ij}=E^{j}_{i}E^{i}_{j}E^{j}_{i}$.
		\item[(T2)] Two adjacent identical crossover arrows can be cancelled, since $E^{j}_{i}E^{j}_{i}$ is the identity. 
		\item[(T3)] Any two crossover arrows can be moved past one another unless one strand is the start of one and the end of the other; in such a case, a crossover arrow connecting the other two strands needs to be added, as shown in Figure~\ref{fig:traintrackmoves}. This corresponds to the identities
		$$
		E^{j}_{i}E^{j'}_{i'}=\begin{cases}
		E^{j'}_{i'}E^{j}_{i} & \text{if $i\neq j'$ and $j\neq i'$} \\
		E^{j'}_{i'}E^{j}_{i}E^{j}_{i'} & \text{if $i=j'$ and $j\neq i'$}\\
		E^{j'}_{i'}E^{j}_{i}E^{j'}_{i} & \text{if $i\neq j'$ and $j=i'$.}
		\end{cases}
		$$
	\end{description}
\end{definition}

\begin{remark}\label{rem:PathsForTrainTracks}
	The point of the graphical notation for the matrices $P_a$ is that we can recover the entry in the $j^\text{th}$ column and $i^\text{th}$ row of each $P_a$ by counting (up to homotopy and modulo 2) all paths in the restriction of the precurve to $N(a)$
	\begin{enumerate}
		\item which start at $\bullet(s_1(a),j)$ and end at $\bullet(s_2(a),i)$,
		\item which are transverse to the foliation $\mathcal{F}_a$ and 
		\item whose orientation agree with each crossover arrow they traverse. 
	\end{enumerate}
	Since
	$$P_a^{-1}=(P_a^{l_a}\cdots P_a^1)^{-1}=(P_a^1)^{-1}\cdots(P_a^{l_a})^{-1}=P_a^1\cdots P_a^{l_a},$$ 
	we can similarly read off the entry in the $j^\text{th}$ column and $i^\text{th}$ row of each $P_a^{-1}$ by following paths from $\bullet(s_2(a),j)$ to $\bullet(s_1(a),i)$ which also satisfy the other two conditions above. 
\end{remark}

\begin{remark}
	Up to the moves (T1)--(T3) and isotopies of the immersed curves, this geometric interpretation of a precurve is unique and also uniquely defines a precurve. So we will no longer carefully distinguish between precurves and their geometric representations. The following two definitions are examples of this principle. 
\end{remark}

\begin{figure}[t]
	\centering
	\psset{unit=0.2}
	\begin{subfigure}{0.35\textwidth}\centering
		\begin{pspicture}(-13,-10.1)(13,10.1)
		
		\pscustom*[linecolor=lightred]{
			\psline(-9,-10)(-9,10)
			\psline(9,10)(9,-10)
		}
		\psline[linecolor=red,linestyle=dotted](-2,-10)(-2,10)
		\psline[linecolor=red,linestyle=dotted](2,-10)(2,10)
		
		\psline[linecolor=red,linestyle=dotted](-6,-10)(-6,10)
		\psline[linecolor=red,linestyle=dotted](6,-10)(6,10)

		\rput[l](9.5,7.5){$s_1(a)$}
		\rput[r](-9.5,7.5){$s_2(a)$}
		\rput[c](10.5,1.5){$i$}
		\rput[c]{-90}(10.5,-0.25){$>$}
		\rput[c](10.5,-2){$j$}
		
		\rput[c](4,7.5){\textcolor{red}{$E^{j}_{i}$}}
		\rput[c](-4,7.5){\textcolor{red}{$P_{ij}$}}
		\rput[c](7.5,7.5){\textcolor{red}{$\cdots$}}
		\rput[c](-7.5,7.5){\textcolor{red}{$\cdots$}}
		\rput[c](0,7.5){\textcolor{red}{$\cdots$}}
		
		\pscustom{
			\psline(9,5)(-9,5)
		}
		\pscustom{
			\psline(-9,1.5)(-6,1.5)
			\psecurve(-10,-2)(-6,1.5)(-2,-2)(2,1.5)
			\psline(-2,-2)(9,-2)
		}
		\pscustom{
			\psline(-9,-2)(-6,-2)
			\psecurve(-10,1.5)(-6,-2)(-2,1.5)(2,-2)
			\psline(-2,1.5)(9,1.5)
		}
		\pscustom{
			\psline(9,-5.5)(-9,-5.5)
		}
		
		\psecurve(10,1.5)(6,-2)(2,1.5)(-2,-2)
		\psecurve(10,-2)(6,1.5)(2,-2)(-2,1.5)
		\psline{->}(4.85,0.87)(5.15,1.2)
		\psline{->}(3.15,0.87)(2.85,1.2)
		
		\psline{->}(-9,-10)(-9,10)
		\psline{->}(9,-10)(9,10)
		
		\psline[linecolor=darkgreen](-10,10)(10,10)
		\psline[linecolor=darkgreen](-10,-10)(10,-10)
		
		\psdots(9,5)(9,1.5)(9,-2)(9,-5.5)
		\psdots(-9,5)(-9,1.5)(-9,-2)(-9,-5.5)
		\end{pspicture}
	\end{subfigure}
	\begin{minipage}{0.6\textwidth}\centering
	\begin{subfigure}{\textwidth}\centering
		\begin{pspicture}(-21.5,-5.1)(13.5,5.1)
		\rput(10,0){
			\pscustom*[linecolor=lightred,linewidth=0pt]{
				\psline(3,-5)(3,5)
				\psline(-7,5)(-7,-5)
			}
			\psline[linecolor=red,linestyle=dotted](-4,-5)(-4,5)
			\psline[linecolor=red,linestyle=dotted](0,-5)(0,5)
			
			\pscustom{
				\psline(1,-2)(0,-2)
				\psecurve(4,2)(0,-2)(-4,2)(-8,-2)
				\psline(-4,2)(-5,2)
			}
			\psline[linestyle=dotted,dotsep=1pt](-5,2)(-6,2)
			
			\pscustom{
				\psline(1,2)(0,2)
				\psecurve(4,-2)(0,2)(-4,-2)(-8,2)
				\psline(-4,-2)(-5,-2)
			}
			\psline[linestyle=dotted,dotsep=1pt](-5,-2)(-6,-2)
			
			\psline[linestyle=dotted,dotsep=1pt](2,2)(1,2)
			\psline[linestyle=dotted,dotsep=1pt](2,-2)(1,-2)
		}
		
		\rput[b](0,1){(T1)}
		\rput(0,0){$\longleftrightarrow$}

		\rput(-14,0){
			\pscustom*[linecolor=lightred,linewidth=0pt]{
				\psline(11,-5)(11,5)
				\psline(-7,5)(-7,-5)
			}
			
			\psline[linecolor=red,linestyle=dotted](-4,-5)(-4,5)
			\psline[linecolor=red,linestyle=dotted](0,-5)(0,5)
			\psline[linecolor=red,linestyle=dotted](4,-5)(4,5)
			\psline[linecolor=red,linestyle=dotted](8,-5)(8,5)

			
			\psline[linestyle=dotted,dotsep=1pt](-5,2)(-6,2)
			\psline[linestyle=dotted,dotsep=1pt](-5,-2)(-6,-2)
			\psline[linestyle=dotted,dotsep=1pt](9,2)(10,2)
			\psline[linestyle=dotted,dotsep=1pt](9,-2)(10,-2)
			\psline(-5,2)(9,2)
			\psline(-5,-2)(9,-2)
			
			\rput(-2,0){
				\psecurve(-6,2)(-2,-2)(2,2)(6,-2)
				\psecurve(-6,-2)(-2,2)(2,-2)(6,2)
				\psline{->}(0.7,1.12)(1,1.55)
				\psline{->}(-0.7,1.12)(-1,1.55)
			}
			\rput(2,0){
				\psecurve(-6,2)(-2,-2)(2,2)(6,-2)
				\psecurve(-6,-2)(-2,2)(2,-2)(6,2)
				\psline{->}(0.7,-1.12)(1,-1.55)
				\psline{->}(-0.7,-1.12)(-1,-1.55)
			}
			\rput(6,0){
				\psecurve(-6,2)(-2,-2)(2,2)(6,-2)
				\psecurve(-6,-2)(-2,2)(2,-2)(6,2)
				\psline{->}(0.7,1.12)(1,1.55)
				\psline{->}(-0.7,1.12)(-1,1.55)
			}
		}
		\end{pspicture}
	\end{subfigure}
	\begin{subfigure}{\textwidth}\centering
		\begin{pspicture}(-17.5,-5.1)(21.5,5.1)
		\rput(-10,0){
			\pscustom*[linecolor=lightred,linewidth=0pt]{
				\psline(7,-5)(7,5)
				\psline(-7,5)(-7,-5)
			}
			
			\psline[linecolor=red,linestyle=dotted](-4,-5)(-4,5)
			\psline[linecolor=red,linestyle=dotted](0,-5)(0,5)
			\psline[linecolor=red,linestyle=dotted](4,-5)(4,5)		
			
			\psline[linestyle=dotted,dotsep=1pt](-5,4)(-6,4)
			\psline[linestyle=dotted,dotsep=1pt](-5,0)(-6,0)
			\psline[linestyle=dotted,dotsep=1pt](-5,-4)(-6,-4)
			\psline[linestyle=dotted,dotsep=1pt](5,4)(6,4)
			\psline[linestyle=dotted,dotsep=1pt](5,0)(6,0)
			\psline[linestyle=dotted,dotsep=1pt](5,-4)(6,-4)
			\psline(-5,4)(5,4)
			\psline(-5,0)(5,0)
			\psline(-5,-4)(5,-4)
			
			\rput(-2,2){
				\psecurve(-6,2)(-2,-2)(2,2)(6,-2)
				\psecurve(-6,-2)(-2,2)(2,-2)(6,2)
				\psline{->}(0.7,1.12)(1,1.55)
				\psline{->}(-0.7,1.12)(-1,1.55)
			}
			\rput(2,-2){
				\psecurve(-6,2)(-2,-2)(2,2)(6,-2)
				\psecurve(-6,-2)(-2,2)(2,-2)(6,2)
				\psline{->}(0.7,1.12)(1,1.55)
				\psline{->}(-0.7,1.12)(-1,1.55)
			}
		}
		
		\rput[b](0,1){(T3)}
		\rput(0,0){$\longleftrightarrow$}

		\rput(10,0){
			\pscustom*[linecolor=lightred,linewidth=0pt]{
				\psline(11,-5)(11,5)
				\psline(-7,5)(-7,-5)
			}
			
			\psline[linecolor=red,linestyle=dotted](-4,-5)(-4,5)
			\psline[linecolor=red,linestyle=dotted](0,-5)(0,5)
			\psline[linecolor=red,linestyle=dotted](4,-5)(4,5)
			\psline[linecolor=red,linestyle=dotted](8,-5)(8,5)
			
			
			\psline[linestyle=dotted,dotsep=1pt](-5,4)(-6,4)
			\psline[linestyle=dotted,dotsep=1pt](-5,0)(-6,0)
			\psline[linestyle=dotted,dotsep=1pt](-5,-4)(-6,-4)
			\psline[linestyle=dotted,dotsep=1pt](9,4)(10,4)
			\psline[linestyle=dotted,dotsep=1pt](9,0)(10,0)
			\psline[linestyle=dotted,dotsep=1pt](9,-4)(10,-4)
			\psline(-5,4)(9,4)
			\psline(-5,0)(9,0)
			\psline(-5,-4)(9,-4)
			
			\rput(-2,-2){
				\psecurve(-6,2)(-2,-2)(2,2)(6,-2)
				\psecurve(-6,-2)(-2,2)(2,-2)(6,2)
				\psline{->}(0.7,1.12)(1,1.55)
				\psline{->}(-0.7,1.12)(-1,1.55)
			}
			\rput(2,2){
				\psecurve(-6,2)(-2,-2)(2,2)(6,-2)
				\psecurve(-6,-2)(-2,2)(2,-2)(6,2)
				\psline{->}(0.7,1.12)(1,1.55)
				\psline{->}(-0.7,1.12)(-1,1.55)
			}
			\rput(6,0){
				\psecurve(-6,4)(-2,-4)(2,4)(6,-4)
				\psecurve(-6,-4)(-2,4)(2,-4)(6,4)
				\psline{->}(0.78,3.12)(1,3.55)
				\psline{->}(-0.78,3.12)(-1,3.55)
			}
		}
		\end{pspicture}
	\end{subfigure}
	\end{minipage}
	\caption{A precurve in a neighbourhood of an arc (left) and some moves corresponding to a change of matrix decomposition in Definition~\ref{def:precurvesGeometric} (right).}\label{fig:traintrackmoves}
\end{figure}
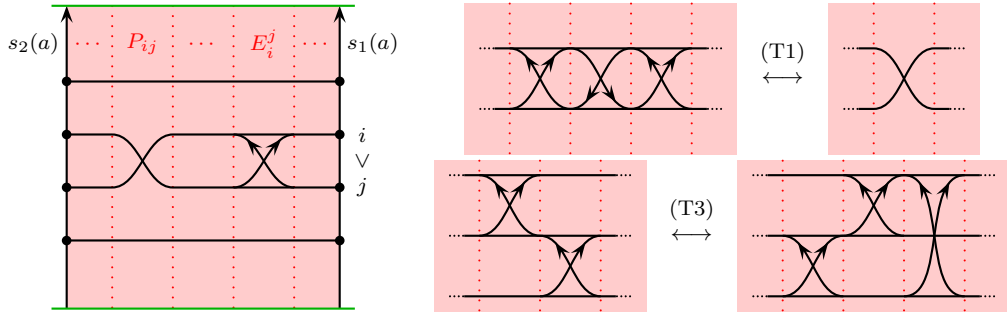

\begin{definition}
	We say a precurve is \textbf{simply-faced} if it is reduced and every dot lies on exactly one $f$-join. In particular, there are no oriented $f$-joins. For an illustration of a simply-faced precurve and its associated curved complex in $\CC(S,M,A)$, see Figure~\ref{fig:ExamplePrecurve}.
\end{definition}

\begin{definition}
	We say a precurve is \textbf{compact} if it is simply-faced and does not have any one-sided $f$-joins. 
\end{definition}

\subsection{From precurves to curves}\label{subsec:simplify_precurves}

One might call reduced precurves that lie in the image of the functor $\mathcal{F}$ ``simple arced'', because in the arc neighbourhoods, they just look like parallel strands without any crossings and crossover arrows. Using Corollary~\ref{cor:SplittingCatsForEquivalenceComplexes}, one can easily see that every reduced precurve is chain isomorphic to a ``simply-arced'' one. The same is true for simply-faced precurves: 

\begin{proposition}\label{prop:PreloopToCC}
	Every reduced precurve is chain isomorphic to a simply-faced precurve.
\end{proposition}

\begin{figure}[t]
	\centering
	\begin{subfigure}{0.48\textwidth}\centering
		{
			\begin{pspicture}(-2.2,-1.2)(2.2,3.2)
			\psrotate(0,1){-45}{
				\psline(-2,0)(-2,2)
				\psline(2,0)(2,2)
				\psline(1,3)(-1,3)
				
				\psdots(-2,1)(2,1)(0,3)
				
				\psecurve{->}(-12,2)(-2,1)(0,3)(-1,13)
				\psecurve[linestyle=dashed]{<-}(12,2)(2,1)(0,3)(1,13)
				\pscurve{->}(-2,1)(-1,0.8)(1,0.8)(2,1)
				
				\rput{45}(0.8,1.8){\psframebox*[framesep=1pt]{$(m-n)$}}
				\rput{45}(-0.8,1.8){\pscirclebox*[framesep=1pt]{$n$}}
				\rput{45}(0,0.7){\pscirclebox*[framesep=1pt]{$m$}}
				
				\uput{0.1}[180]{45}(-2,1){$e^s_i$}
				\uput{0.1}[90]{45}(0,3){$e^t_j$}
				\uput{0.1}[0]{45}(2,1){$e^r_k$}
			}
			\end{pspicture}
		}
		\caption{}\label{fig:BasicIsosI}
	\end{subfigure}
	\begin{subfigure}{0.48\textwidth}\centering
		{
			\begin{pspicture}(-2.2,-1.2)(2.2,3.2)
			\psrotate(0,1){45}{
				\psline(-2,0)(-2,2)
				\psline(2,0)(2,2)
				\psline(1,3)(-1,3)
				
				\psdots(-2,1)(2,1)(0,3)
				
				\psecurve[linestyle=dashed]{->}(-12,2)(-2,1)(0,3)(-1,13)
				\psecurve{<-}(12,2)(2,1)(0,3)(1,13)
				\pscurve{->}(-2,1)(-1,0.8)(1,0.8)(2,1)
				
				\rput{-45}(0.8,1.8){\pscirclebox*[framesep=1pt]{$n$}}
				\rput{-45}(-0.8,1.8){\psframebox*[framesep=1pt]{$(m-n)$}}
				\rput{-45}(0,0.7){\pscirclebox*[framesep=1pt]{$m$}}
				
				\uput{0.1}[180]{-45}(-2,1){$e^r_k$}
				\uput{0.1}[90]{-45}(0,3){$e^s_i$}
				\uput{0.1}[0]{-45}(2,1){$e^t_j$}
			}
			\end{pspicture}
		}
		\caption{}\label{fig:BasicIsosII}
	\end{subfigure}
	\caption{The morphisms $h_1$ and $h_2$ (given by the dashed arrows of length $(m-n)$) for applications of the Clean-Up Lemma in the proof of Proposition~\ref{prop:PreloopToCC}.}\label{fig:BasicIsos}
\end{figure}
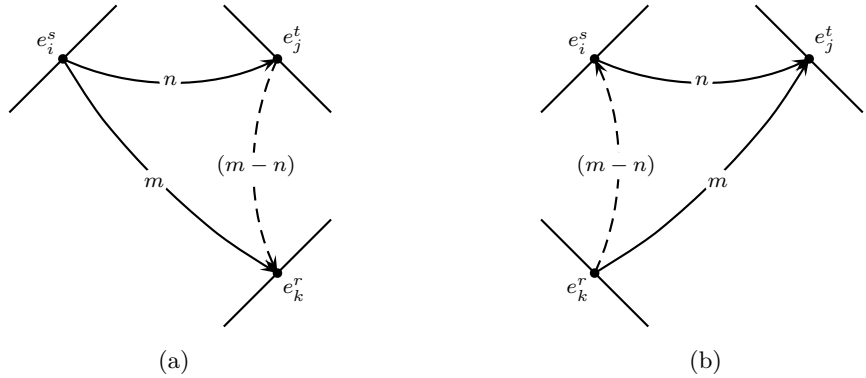
\begin{proof}	
	We simplify the precurve for each face separately. For this, we carefully choose two basis elements $e^s_i$ and $e^t_j$ such that there is an (oriented or unoriented) $f$-join going from one to the other whose length $n$ is smallest among all arrows leaving $e^s_i$ and arrows arriving at $e^t_j$. In particular, this implies that $e^s_i$ and $e^t_j$ do not lie on the same side. For open faces, this follows from the assumption that the precurve is reduced; for closed faces~$f$, we use the $\partial^2$-relation to see that the length of such a shortest arrow is strictly less than $n_f$. 
	
	Suppose there is an $f$-join of length $m$ leaving $e_i^s$ and going to a generator $e_k^r$. Consider an arrow $h_1$ of length $(m-n)\geq0$ going from $e_j^t$ to $e_k^r$, see Figure~\ref{fig:BasicIsosI}. 
	We want to apply  the Clean-Up Lemma (\ref{lem:AbstractCleanUp}) to our precurve with $h=h_1$, so let us verify the hypotheses of this lemma. The first hypothesis is $h_1^2=0$. Clearly, the composition of $h_1$ with itself can only be non-zero if $e^t_j=e_k^r$. Also, $h_1$ is a morphism of $\delta$-grading 0, so $e^t_j=e_k^r$ implies $m=n$. But this means that the two $f$-joins actually agree, which is a contradiction. 
	Given $h_1^2=0$, the second and third hypotheses of the Clean-Up Lemma are equivalent to $h_1\partial h_1=0$. Suppose there were an arrow from $e^r_k$ to $e^t_j$ in the differential $\partial$. Then its composition with $h_1$ would be a power of $U_f$ and thus have an even $\delta$-grading. However, the composition $h_1\partial$ has $\delta$-grading $1$, so we also have $h_1\partial h_1=0$. Hence, we may apply the Clean-Up Lemma to the precurve with $h=h_1$. 
	If $m>n$, this strictly reduces the number of $f$-joins leaving $e^s_i$. Also, it does not affect the precurve in other faces of $(S,M,A)$. This is not necessarily true if $m=n$, ie if $h_1$ contains an identity component. So instead, we remove the $f$-join from $e^s_i$ to $e_k^r$ in this case by pre- or postcomposing the matrix decorating the arc of the side $t$ by the elementary matrix $E^{k}_{j}$ or $E^{j}_{k}$. We now repeat this procedure until there is only one $f$-join leaving $e^s_i$ left. 
	
	Similarly, we can achieve that there are no other $f$-joins going into $e^t_j$. Indeed, by assumption, the length of any such $f$-join has to be at least $n$, so by the same argument as above, we can apply the Clean-Up Lemma with $h=h_2$ from Figure~\ref{fig:BasicIsosII}. After this, the $\partial^2$-relation ensures that there are no other $f$-joins leaving $e^t_j$ nor ending at $e^s_i$. 
	We can now ignore the two generators $e_i^s$ and $e_j^t$ and the one or two arrows between them and apply induction on the number of generators on the face $f$ until there are no two-sided $f$-joins left.
\end{proof}

\begin{remark}
	In some sense, to be made precise in the proof of Theorem~\ref{thm:EverythingIsLoopTypeUpToLocalSystems} below, our simply-faced precurves correspond to train-tracks from~\cite[Definition~17]{HRW} of a special form. Proposition~\ref{prop:PreloopToCC} above corresponds roughly to~\cite[Propositions~22 and~23]{HRW}. While their proofs rely on the fact that extendable type D structures are defined over the truncated algebra $\mathcal{A}$ (see Remark~\ref{rem:comparisonToHRW}), our proof relies in essential points on the $\delta$-grading of our curved complexes over the full algebra $\mathcal{A}$: indeed, in the main step of the proof, we repeatedly apply the Clean-Up Lemma. To verify that all hypotheses of the lemma are satisfied, we use the $\delta$-grading. Also note that the $\delta$-grading played a key role in the proof of Lemma~\ref{lem:CCreduced}. If we dropped the $\delta$-grading, we might also see arrows labelled by $1+U_f$, which do not have an inverse in $\mathcal{A}$!
\end{remark}

\begin{figure}[t]
	\centering
	\psset{unit=0.2}
	\begin{subfigure}{0.49\textwidth}\centering
		\begin{pspicture}(-8,-7)(23,7)
		
		\pscustom*[linecolor=lightred,linewidth=0pt]{
			\psline(0,-5)(0,5)
			\psline(-7,5)(-7,-5)
			}
		\psline[linecolor=red,linestyle=dotted](-4,-5)(-4,5)
		\psline(0,5)(0,-5)
		
		\pscustom{
			\psecurve(4,2)(0,-2)(-4,2)(-8,-2)
			\psline(-4,2)(-5,2)
			}
		\psline[linestyle=dotted,dotsep=1pt](-5,2)(-6,2)
		
		\pscustom{
			\psecurve(4,-2)(0,2)(-4,-2)(-8,2)
			\psline(-4,-2)(-5,-2)
		}
		\psline[linestyle=dotted,dotsep=1pt](-5,-2)(-6,-2)
		
		\psecurve(-4,4)(0,2)(4,4)(5,6)
		\psecurve(-4,-4)(0,-2)(4,-4)(5,-6)
		
		\psecurve[linestyle=dotted,dotsep=1pt](0,2)(4,4)(5,6)(6,20)
		\psecurve[linestyle=dotted,dotsep=1pt](0,2)(4,-4)(5,-6)(6,-20)
		
		\psdots(0,2)(0,-2)

		\rput(7,0){$\longleftrightarrow$}
		
		\rput(13,0){
			\pscustom*[linecolor=lightred,linewidth=0pt]{
				\psline(0,-5)(0,5)
				\psline(-3,5)(-3,-5)
			}
			\psline(0,5)(0,-5)
			
			\psline(0,2)(-1,2)
			\psline[linestyle=dotted,dotsep=1pt](-1,2)(-2,2)
			
			\psline(0,-2)(-1,-2)
			\psline[linestyle=dotted,dotsep=1pt](-1,-2)(-2,-2)
			
			\psecurve(-8,4)(0,-2)(8,4)(9,6)
			\psecurve(-8,-4)(0,2)(8,-4)(9,-6)
			
			\psecurve[linestyle=dotted,dotsep=1pt](0,2)(8,4)(9,6)(10,20)
			\psecurve[linestyle=dotted,dotsep=1pt](0,2)(8,-4)(9,-6)(10,-20)
			
			\rput(1,3.3){$j$}
			\rput(1,-3.5){$i$}
			
			\psdots(0,2)(0,-2)
			}
			
			\rput(-2,3){$\textcolor{red}{P_{ij}}$}
			\rput(1,3.3){$j$}
			\rput(1,-3.5){$i$}
		
		\end{pspicture}\\
		(M1)
	\end{subfigure}
	\begin{subfigure}{0.49\textwidth}\centering
		\begin{pspicture}(-10.5,-7)(20.5,7)
		\pscustom*[linecolor=lightred,linewidth=0pt]{
			\psline(0,-5)(0,5)
			\psline(-7,5)(-7,-5)
		}
		\psline[linecolor=red,linestyle=dotted](-4,-5)(-4,5)
		\psline(0,5)(0,-5)
		
		\pscustom*[linecolor=lightred,linewidth=0pt]{
			\psline(10,-5)(10,5)
			\psline(17,5)(17,-5)
		}
		\psline[linecolor=red,linestyle=dotted](14,-5)(14,5)
		\psline(10,5)(10,-5)
		
		
		\psline[linestyle=dotted,dotsep=1pt](-5,2)(-6,2)
		\psline[linestyle=dotted,dotsep=1pt](-5,-2)(-6,-2)
		\psline[linestyle=dotted,dotsep=1pt](15,2)(16,2)
		\psline[linestyle=dotted,dotsep=1pt](15,-2)(16,-2)
		\psline(-5,2)(15,2)
		\psline(-5,-2)(15,-2)
		
		\psdots(0,2)(0,-2)
		\psdots(10,2)(10,-2)
		
		\rput(1,3.3){$j$}
		\rput(1,-3.5){$i$}
		\rput[r](9.7,3.3){$j'$}
		\rput[r](9.7,-3.3){$i'$}
		
		\rput(-2,0){
		\psecurve(-6,2)(-2,-2)(2,2)(6,-2)
		\psecurve(-6,-2)(-2,2)(2,-2)(6,2)
		\psline{->}(0.7,-1.12)(1,-1.55)
		\psline{->}(-0.7,-1.12)(-1,-1.55)
		}
		\rput(12,0){
			\psecurve(-6,2)(-2,-2)(2,2)(6,-2)
			\psecurve(-6,-2)(-2,2)(2,-2)(6,2)
			\psline{->}(0.7,-1.12)(1,-1.55)
			\psline{->}(-0.7,-1.12)(-1,-1.55)
		}
		\end{pspicture}\\
		(M2)
	\end{subfigure}
	\\
	\begin{subfigure}{0.24\textwidth}\centering
		\begin{pspicture}(-9,-11)(9,11)
		
		\pscustom*[linecolor=lightred,linewidth=0pt]{
			\psline(0,-5)(0,5)
			\psline(-7,5)(-7,-5)
		}
		\psline[linecolor=red,linestyle=dotted](-4,-5)(-4,5)
		\psline(0,5)(0,-5)
		
		\psline(0,2)(-5,2)
		\psline[linestyle=dotted,dotsep=1pt](-5,2)(-6,2)
		
		\psline(0,-2)(-5,-2)
		\psline[linestyle=dotted,dotsep=1pt](-5,-2)(-6,-2)
		
		\rput(-2,0){
			\psecurve(-6,2)(-2,-2)(2,2)(6,-2)
			\psecurve(-6,-2)(-2,2)(2,-2)(6,2)
			\psline{->}(0.7,-1.12)(1,-1.55)
			\psline{->}(-0.7,-1.12)(-1,-1.55)
		}

		
		\psecurve(-4.5,7)(0,2)(4.5,7)(0,12)
		\psecurve(-4.5,-7)(0,-2)(4.5,-7)(0,-12)
		
		\psdots(0,2)(0,-2)
		
		\pscustom*[linecolor=lightred,linewidth=0pt]{
			\psline(2,7)(7,7)
			\psline(7,10)(2,10)
		}
		\psline(2,7)(7,7)
		\psline(4.5,7)(4.5,8)
		\psline[linestyle=dotted,dotsep=1pt](4.5,8)(4.5,9)
		\psdot(4.5,7)
		
		\pscustom*[linecolor=lightred,linewidth=0pt]{
			\psline(2,-7)(7,-7)
			\psline(7,-10)(2,-10)
		}
		\psline(2,-7)(7,-7)
		\psline(4.5,-7)(4.5,-8)
		\psline[linestyle=dotted,dotsep=1pt](4.5,-8)(4.5,-9)
		\psdot(4.5,-7)
		
		\rput(1,3.3){$j$}
		\rput(1,-3.5){$i$}
		
		\end{pspicture}
		(M3a)
	\end{subfigure}
	\begin{subfigure}{0.24\textwidth}\centering
		\begin{pspicture}(-9,-11)(9,11)
		
		\pscustom*[linecolor=lightred,linewidth=0pt]{
			\psline(0,-5)(0,5)
			\psline(-7,5)(-7,-5)
		}
		\psline[linecolor=red,linestyle=dotted](-4,-5)(-4,5)
		\psline(0,5)(0,-5)
		
		\psline(0,2)(-5,2)
		\psline[linestyle=dotted,dotsep=1pt](-5,2)(-6,2)
		
		\psline(0,-2)(-5,-2)
		\psline[linestyle=dotted,dotsep=1pt](-5,-2)(-6,-2)
		
		\rput(-2,0){
			\psecurve(-6,2)(-2,-2)(2,2)(6,-2)
			\psecurve(-6,-2)(-2,2)(2,-2)(6,2)
			\psline{->}(0.7,-1.12)(1,-1.55)
			\psline{->}(-0.7,-1.12)(-1,-1.55)
		}

		
		\psecurve(-4.5,7)(0,2)(4.5,7)(0,12)
		\psecurve(-4.5,-7)(0,-2)(4.5,-7)(0,-12)
		
		\psdots(0,2)(0,-2)
		
		\pscustom*[linecolor=lightred,linewidth=0pt]{
			\psline(2,7)(7,7)
			\psline(7,10)(2,10)
		}
		\psline(2,7)(7,7)
		\psline(4.5,7)(4.5,8)
		\psline[linestyle=dotted,dotsep=1pt](4.5,8)(4.5,9)
		\psdot(4.5,7)
		
		\psline[linecolor=darkgreen](2,-7)(7,-7)
		\psline[linecolor=darkgreen](4.5,-8)(4.5,-6)
		\rput[l](3.5,-8.5){\textcolor{darkgreen}{$m\!\in\!M$}}
		
		\rput(1,3.3){$j$}
		\rput(1,-3.5){$i$}
		
		\end{pspicture}
		(M3b)
	\end{subfigure}
	\begin{subfigure}{0.24\textwidth}\centering
		\begin{pspicture}(-9,-11)(9,11)
		
		\pscustom*[linecolor=lightred,linewidth=0pt]{
			\psline(0,-5)(0,5)
			\psline(-7,5)(-7,-5)
		}
		\psline[linecolor=red,linestyle=dotted](-4,-5)(-4,5)
		\psline(0,5)(0,-5)
		
		\psline(0,2)(-5,2)
		\psline[linestyle=dotted,dotsep=1pt](-5,2)(-6,2)
		
		\psline(0,-2)(-5,-2)
		\psline[linestyle=dotted,dotsep=1pt](-5,-2)(-6,-2)
		
		\rput(-2,0){
			\psecurve(-6,2)(-2,-2)(2,2)(6,-2)
			\psecurve(-6,-2)(-2,2)(2,-2)(6,2)
			\psline{->}(0.7,-1.12)(1,-1.55)
			\psline{->}(-0.7,-1.12)(-1,-1.55)
		}

		
		\psecurve(-4.5,7)(0,2)(4.5,7)(0,12)
		\psecurve(-4.5,-7)(0,-2)(4.5,-7)(0,-12)
		
		\psdots(0,2)(0,-2)
		
		\psline[linecolor=darkgreen](2,7)(7,7)
		\psline[linecolor=darkgreen](4.5,8)(4.5,6)
		\rput[l](3.5,8.8){\textcolor{darkgreen}{$m\!\in\!M$}}
		
		\pscustom*[linecolor=lightred,linewidth=0pt]{
			\psline(2,-7)(7,-7)
			\psline(7,-10)(2,-10)
		}
		\psline(2,-7)(7,-7)
		\psline(4.5,-7)(4.5,-8)
		\psline[linestyle=dotted,dotsep=1pt](4.5,-8)(4.5,-9)
		\psdot(4.5,-7)
		
		\rput(1,3.3){$j$}
		\rput(1,-3.5){$i$}
		
		\end{pspicture}
		(M3c)
	\end{subfigure}
	\begin{subfigure}{0.24\textwidth}\centering
		\begin{pspicture}(-9,-11)(9,11)
		
		\pscustom*[linecolor=lightred,linewidth=0pt]{
			\psline(0,-5)(0,5)
			\psline(-7,5)(-7,-5)
		}
		\psline[linecolor=red,linestyle=dotted](-4,-5)(-4,5)
		\psline(0,5)(0,-5)
		
		\psline(0,2)(-5,2)
		\psline[linestyle=dotted,dotsep=1pt](-5,2)(-6,2)
		
		\psline(0,-2)(-5,-2)
		\psline[linestyle=dotted,dotsep=1pt](-5,-2)(-6,-2)
		
		\rput(-2,0){
			\psecurve(-6,2)(-2,-2)(2,2)(6,-2)
			\psecurve(-6,-2)(-2,2)(2,-2)(6,2)
			\psline{->}(0.7,-1.12)(1,-1.55)
			\psline{->}(-0.7,-1.12)(-1,-1.55)
		}

		
		\psecurve(-5,0)(0,2)(5,0)(10,2)
		\psecurve(-5,0)(0,-2)(5,0)(10,-2)
		
		\psdots(0,2)(0,-2)
		
		\psline[linecolor=darkgreen](5,2)(5,-2)
		\psline[linecolor=darkgreen](4,0)(6,0)
		\rput{-90}(7,-2){\textcolor{darkgreen}{$m\!\in\!M$}}

		\rput(1,3.3){$j$}
		\rput(1,-3.5){$i$}		
		\end{pspicture}
		(M3d)
	\end{subfigure}
	\caption{An illustration of the moves (M1)--(M3) from Lemma~\ref{lem:CalculusForPreloops}. For (M3), there are four different cases to consider, depending on whether the $f$-joins are two- or one-sided.}\label{fig:GraphicCalculus}
\end{figure}
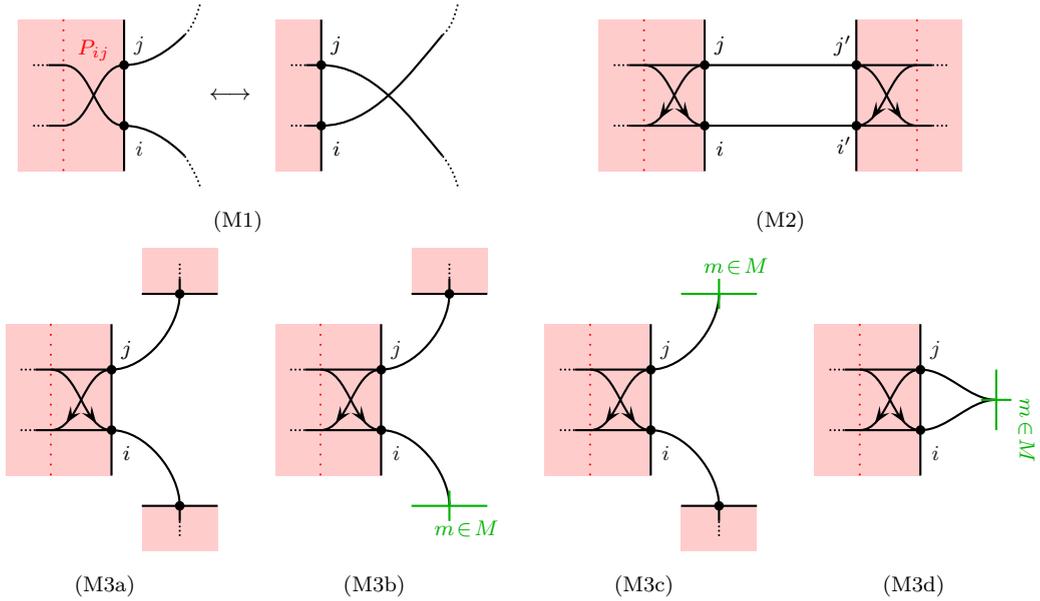

For the proof of the previous proposition, it was fairly irrelevant how the matrices $P_a$ change under the various base changes. For the main classification result, however, we need more control over these equivalences:

\begin{lemma}\label{lem:CalculusForPreloops}
	Let \((C, \{P_a\}_{a\in A},\partial)\) be a simply-faced precurve on a marked surface \((S,M)\) with arc system~\(A\). For some \(a\in A\), let \(s\) be a side of \(a\) and \(f\) the face adjacent to \(s\). Let \(\bullet(s,i)\) and \(\bullet(s,j)\) be two distinct dots on \(s\). Then \((C, \{P_a\}_{a\in A},\partial)\) is chain isomorphic to the precurve obtained by one of the following three moves, which are illustrated in Figure~\ref{fig:GraphicCalculus}.
	\begin{description}
		\item[(M1)] Multiply \(P_a\) on the right/left by \(P_{ij}\), depending on whether \(s=s_1(a)\) or \(s=s_2(a)\), and switch the endpoints of the two \(f\)-joins ending in \(\bullet(s,i)\) and \(\bullet(s,j)\). 
		\item[(M2)] Assume that \(\bullet(s,i)\) and \(\bullet(s,j)\) have the same \(\delta\)-grading. Suppose also that both \(\bullet(s,i)\) and \(\bullet(s,j)\) are connected by two two-sided \(f\)-joins which connect the same two sides of \(f\) and let \(\bullet(s',i')\) and \(\bullet(s',j')\) be the other endpoints, respectively. Then multiply \(P_a\) on the right/left by \(E^{j}_{i}\), depending on whether \(s=s_1(a)\) or \(s=s_2(a)\), and multiply \(P_{a'}\) on the right/left by \(E^{j'}_{i'}\), depending on whether \(s'=s_1(a')\) or \(s'=s_2(a')\).  
		\item[(M3)] Assume that \(\bullet(s,i)\) and \(\bullet(s,j)\) have the same \(\delta\)-grading. If the \(f\)-joins starting at \(\bullet(s,i)\) and \(\bullet(s,j)\) are both two-sided, let us assume that they end on different sides. Unless the \(f\)-joins are both one-sided, assume also that if we follow the oriented boundary of \(f\), starting at \(s\), we meet the other end of the second \(f\)-join before the other end of the first. Then multiply \(P_a\) on the right/left by \(E^{j}_{i}\) depending on whether \(s=s_1(a)\) or \(s=s_2(a)\). 
	\end{description}
\end{lemma}

\begin{wrapfigure}{r}{0.25\textwidth}
	\centering
	\psset{unit=0.2}
	\begin{pspicture}(-5,-10.5)(13,10.5)
	
	\pscustom*[linecolor=lightred,linewidth=0pt]{
		\psline(0,-5)(0,5)
		\psline(-3,5)(-3,-5)
	}
	\pscustom*[linecolor=lightred,linewidth=0pt]{
		\psline(2,7)(7,7)
		\psline(7,10)(2,10)
	}
	\pscustom*[linecolor=lightred,linewidth=0pt]{
		\psline(2,-7)(7,-7)
		\psline(7,-10)(2,-10)
	}
	
	
	\psline[linestyle=dotted,dotsep=1pt](-2,2)(-1,2)
	\psline(-1,2)(0,2)
	\psline[linestyle=dotted,dotsep=1pt](-2,-2)(-1,-2)
	\psline(-1,-2)(0,-2)
	
	\psline[linestyle=dotted,dotsep=1pt](4.5,8)(4.5,9)
	\psline(4.5,7)(4.5,8)
	\psline[linestyle=dotted,dotsep=1pt](4.5,-8)(4.5,-9)
	\psline(4.5,-7)(4.5,-8)
	
	\psline(0,5)(0,-5)

	\psecurve(-4.5,7)(0,2)(4.5,7)(3,12)
	\psecurve(-4.5,-7)(0,-2)(4.5,-7)(3,-12)
	\psdots(0,2)(0,-2)
	
	\psline(2,7)(7,7)
	\psdot(4.5,7)
	
	\psline(2,-7)(7,-7)
	\psdot(4.5,-7)
	
	\psecurve{->}(0,10)(4.5,7)(0,-2)(-4.5,-2)
	\psecurve{<-}(0,-10)(4.5,-7)(0,2)(-4.5,2)
	
	\psecurve[linestyle=dashed]{<-}(0,-8)(4.5,-7)(4.5,7)(0,8)
	
	\rput(-1,0){$s$}
	\rput(5.5,8){$t$}
	\rput(5.5,-8){$t'$}
	
	\rput(10,0){$h=p^{t}_{t'}$}
	\rput(5,2){$p^t_s$}
	\rput(5,-1.2){$p^s_{t'}$}

	\end{pspicture}
	\caption{The morphism~$h$ for the proof of invariance of (M3).}\label{fig:CalculusForPreloopsProofHomotopy}
	\vspace*{-20pt}
\end{wrapfigure}
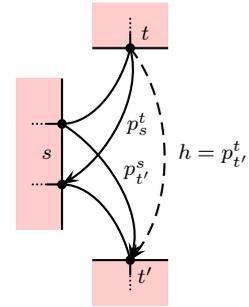
\myfixwrapfig

\begin{remark}
	Using (T2), we can reformulate (M2) as follows: if there is only one crossover arrow in the picture for (M2) in Figure~\ref{fig:GraphicCalculus}, then we may remove that crossover arrow and replace it by the other crossover arrow, thereby pushing the crossover arrow along parallel joins. 
\end{remark}

\begin{proof}[of Lemma~\ref{lem:CalculusForPreloops}]
	All three parts of the lemma can be shown in the same way, namely by doing a base change, which in our graphical notation shifts the outermost segment of $N(a)$ adjacent to $s$ into the face $f$. 
	\\\indent
	For (M1), the result after such a base change is obviously chain isomorphic to the original precurve. For (M2), we do two such base changes, one for $s$ and one for $s'$. For (M3), we do one such base change; the precurve then differs in at most two $f$-joins from the first, see Figure~\ref{fig:CalculusForPreloopsProofHomotopy}. However, we can remove these new $f$-joins using the Clean-Up Lemma and the morphism $h$ from the same figure.
\end{proof}

\begin{definition}\label{def:curves}
	Let $(S,M)$ be a marked surface with an arc system $A$. A \textbf{curve} on $(S,M,A)$ is a pair $(\gamma, X)$, where either
	\begin{enumerate}
		\item $\gamma$ is an immersion of an oriented circle into $S$, representing a non-trivial primitive element of $\pi_1(S)$ and 
		$X\in \GL_n(\mathbb{F}_2)$ for some $n$; or \label{def:curves:loops}
		\item $\gamma$ is an immersion of an interval $(I,\partial I)$ into $(S,M)$, defining a non-trivial element of $\pi_1(S,M)$ and 
		$X=\id\in\GL_n(\mathbb{F}_2)$ for some positive integer $n$ \label{def:curves:paths}
	\end{enumerate}
	satisfying the following properties:
	\begin{itemize}
		\item $\gamma$ restricted to each component of the preimage of each face is an embedding and 
		\item $\gamma$ restricted to each component of the preimage of the neighbourhood $N(a)$ of each arc $a$ is an embedding, intersecting each leaf of $\mathcal{F}_a$ exactly once.
	\end{itemize}
	In case \ref{def:curves:loops}, we call $(\gamma,X)$ a \textbf{compact curve} or a \textbf{loop}, in case \ref{def:curves:paths}, we call it a \textbf{non-compact curve} or a \textbf{path}. We say that a curve $(\gamma,X)$ is \textbf{supported} on the immersed curve~$\gamma$ and call $\gamma$ the \textbf{underlying curve} and $X$ its \textbf{local system}. Note that the local system of paths only records some positive integer $n$. We consider curves up to homotopy of the underlying immersed curves through curves. Furthermore, we consider the local systems of loops up to matrix similarity.
	
	A \textbf{$\delta$-grading} on a curve $(\gamma,X)$ is an $\mathbb{R}$-grading on the set of intersection points of the underlying curve with arcs in $A$ satisfying the following property: let $x$ and $y$ be two intersection points of $\delta$-grading $\delta(x)$ and $\delta(y)$, respectively. Suppose $x$ and $y$ are joined by a component of $\gamma\smallsetminus A$. Then such a component is mapped to a path in $Q(S,M,A)$ corresponding to an algebra element $p^s_t$, from $x$ to $y$, say. Then we ask that $\delta(y)-\delta(x)+\delta(p^s_t)=1$.
	
	A \textbf{collection of ($\delta$-graded) curves} is a finite set of ($\delta$-graded) curves such that all underlying curves are pairwise non-homotopic as unoriented ($\delta$-graded) curves. We denote the set of all collections of $\delta$-graded curves up to equivalence by $\loops(S,M,A)$.
\end{definition}

\begin{remark}
	Since any immersed curve $\gamma$ in a pair $(\gamma,X)$ of the form \ref{def:curves:loops} or \ref{def:curves:paths} can be homotoped to one that satisfies the two conditions of the previous definition, the arc system $A$ is only required to define the $\delta$-grading on elements of $\loops(S,M,A)$. Apart from that, $\loops(S,M,A)$ is independent of $A$.
\end{remark}

\begin{definition}
	Given an arc system $A$ on $(S,M)$, we define a map
	\[\Pi_i\co \loops(S,M,A)\rightarrow \ob(\CC_i(S,M,A))/\text{(chain homotopy)}\]
	as follows: given a single curve $(\gamma,X)$, choose a small immersed tubular neighbourhood of $\gamma$ and replace $\gamma$ by $\dim X$ parallel copies thereof in this neighbourhood. Then, for each face $f\in F(S,M,A)$, the $f$-joins of~$\Pi_i(\gamma,X)$ are given by the intersection of these curves with~$f$. Then pick an intersection point $x$ of an arc $a$ with $\gamma$. Let the matrix $P_a$ be the diagonal block matrix with blocks of dimension $\dim X$ such that all blocks are equal to the identity matrix except the one corresponding to the intersection point $x$. We define this block to be equal to $X$ if $\gamma$ goes through $x$ from the right of $a$ to its left (ie from $s_1(a)$ to $s_2(a)$), and set it equal to $X^t$ otherwise. On all other arcs, we choose the identity matrix. 
	Finally, we extend $\Pi_i$ to collections of curves by taking unions/direct sums. 
	
	Note that the definition of $\Pi_i$ is indeed independent of the choice of $a$ and $x$ up to homotopy in $\CC_i(S,M,A)$, which can be seen by repeatedly applying (M2).  Similarly, we see that conjugation of the local systems $C$ of a loop $(\gamma,X)$ does not change the homotopy type of the image under $\Pi_i$. 
	
	We define the map
	\[\Pi\co  \loops(S,M,A)\rightarrow\ob(\CC(S,M,A))/\text{(chain homotopy)}\]
	as the composition of $\Pi_i$ and the functor $\mathcal{G}$ from Definition~\ref{def:SplittingCatsForEquivalenceFunctors}.
\end{definition}

\begin{theorem}\label{thm:EverythingIsLoopTypeUpToLocalSystems}
	Any reduced precurve is chain isomorphic to one in the image of \(\Pi_i\). Thus any reduced curved complex is chain isomorphic to one in the image of \(\Pi\). 
\end{theorem}

\begin{proof}
	Simply-faced precurves \((C, \{P_a\}_{a\in A},\partial)\) in our setting for which there are no one-sided $f$-joins correspond to certain types of train-tracks in the language of~\cite{HRW}, namely those in which all arrows come in pairs and only sit in the neighbourhoods of the arcs. 
	
	Then the train-track moves from \cite[Proposition~25]{HRW} correspond exactly to our moves (T1) to (T3) and (M1) to (M3). So we can apply the algorithm explained in~\cite[section~3.7]{HRW} for simplifying train-tracks, since the geometric objects agree, even though they represent algebraic objects defined over slightly different algebras. The output of the algorithm is train-tracks whose crossover arrows only connect parallel immersed curves, ie loops.  
	
	The same algorithm also works without any changes for simply-faced precurves which contain one-sided $f$-joins, since the additional moves for such $f$-joins from Lemma~\ref{lem:CalculusForPreloops}, namely (M3b), (M3c) and (M3d) from Figure~\ref{fig:GraphicCalculus}, can be regarded as generalisations of (M3a). Once all arrows only connect parallel immersed curves, we can remove all arrows on paths by applying moves (M2) followed by (M3d).
\end{proof}

\begin{remark}
	The key ingredient of the algorithm from~\cite{HRW} is a certain complexity that Hanselman, Rasmussen and Watson assign to each crossover arrow. This complexity, which they call weight, takes values in $(\mathbb{Z}\smallsetminus\{0\})^2\cup \infty$ and is defined as follows: the weight of a crossover arrow between two curves that always stay parallel is $\infty$. If the two segments diverge, the complexity is a pair of non-zero integers. Their absolute values record the maximum number of faces (plus 1) that the curve segments stay parallel to each other in each direction. Their signs are determined by whether the crossover arrow (if it were pushed into the face where the curves diverge) could be eliminated using the train-track move corresponding to our move (M3) or not. The depth of a crossover arrow is the minimum of the absolute values of these two integers (or $\infty$ if the weight is $\infty$). The depth of a curve configuration (ie precurve in our language) is defined as the minimum of the depths of all crossover arrows. Hanselman, Rasmussen and Watson then show that one can apply a sequence of train-track moves which strictly increases the depth of the curve configuration \cite[Proposition~29]{HRW}. Since the depth of any curve configuration is bounded (because there are only finitely many $f$-joins), this suffices to show that eventually one obtains a curve configuration of depth $\infty$, which means that all crossover arrows go between parallel curves. 
\end{remark}

\begin{observation}\label{obs:CompactPreCurves}
	Compactness of precurves is preserved under the moves (M1) to (M3), so if the original precurve in Theorem~\ref{thm:EverythingIsLoopTypeUpToLocalSystems} is compact, we can choose a compact precurve in the image of $\Pi_i$. Note that $\Pi_i$ maps compact curves to compact precurves and non-compact curves to non-compact precurves. 
\end{observation}

\subsection{Classification of morphisms between simply-faced precurves}\label{subsec:ClassificationMor}

Throughout this section, let $C=(C, \{P_a\}_{a\in A},\partial)$ and $C'=(C', \{P'_a\}_{a\in A},\partial')$ be a pair of simplify-faced precurves on a fixed marked surface \((S,M)\) with arc system \(A\). 

\begin{wrapfigure}{r}{0.3333\textwidth}
	\centering
	\medskip
	\psset{unit=0.18}
	\begin{pspicture}(-13.5,-10.1)(13.5,10.1)
	
	\pscustom*[linecolor=lightred]{
		\psline(-9,-10)(-9,10)
		\psline(9,10)(9,-10)
	}
	
	\rput[l](9.5,7.5){$s_1(a)$}
	\rput[r](-9.5,7.5){$s_2(a)$}

	\psline[linecolor=blue](-9,5.5)(-8,5.5)
	\psline[linecolor=blue](-9,3.75)(-8,3.75)
	\psline[linecolor=blue](-9,2)(-8,2)
	
	\psline[linecolor=red](-9,-5.5)(-8,-5.5)
	\psline[linecolor=red](-9,-3.75)(-8,-3.75)
	\psline[linecolor=red](-9,-2)(-8,-2)

	\psecurve[linecolor=blue](22,2)(9,-5.5)(-4,2)(-17,-5.5)
	\psecurve[linecolor=blue](22,3.75)(9,-3.75)(-4,3.75)(-17,-3.75)
	\psecurve[linecolor=blue](22,5.5)(9,-2)(-4,5.5)(-17,-2)
	
	\psecurve[linecolor=red](22,-2)(9,5.5)(-4,-2)(-17,5.5)
	\psecurve[linecolor=red](22,-3.75)(9,3.75)(-4,-3.75)(-17,3.75)
	\psecurve[linecolor=red](22,-5.5)(9,2)(-4,-5.5)(-17,2)
	
	\psframe[linecolor=blue](-4,1)(-8,6.5)
	\rput[c](-6,3.75){\textcolor{blue}{$P'_a$}}
	
	\psframe[linecolor=red](-4,-1)(-8,-6.5)
	\rput[c](-6,-3.75){\textcolor{red}{$P_a$}}
	
	\rput[b](0,5){\textcolor{blue}{$C'$}}
	\rput[t](0,-5){\textcolor{red}{$C$}}
	
	\psline{->}(-9,-10)(-9,10)
	\psline{->}(9,-10)(9,10)
	
	\psline[linecolor=darkgreen](-10,10)(10,10)
	\psline[linecolor=darkgreen](-10,-10)(10,-10)
	
	\psdots[linecolor=red](9,5.5)(9,3.75)(9,2)(-9,-5.5)(-9,-3.75)(-9,-2)
	\psdots[linecolor=blue](9,-5.5)(9,-3.75)(9,-2)(-9,5.5)(-9,3.75)(-9,2)
	
	\end{pspicture}
	\caption{A pair of precurves in pairing position in the neighbourhood of an arc $a$.}\label{fig:StandardPairingPosition}
\end{wrapfigure}
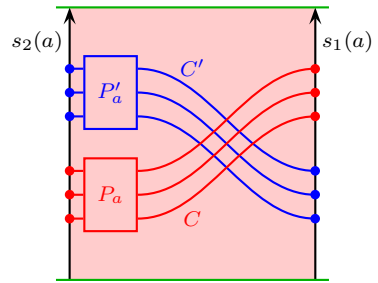

\myfixwrapfig

\begin{definition}\label{def:StandardPairingPosition}
	We say $C$ and $C'$ are in \textbf{pairing position} if the following holds:
	\begin{enumerate}[i]
		\item For each side $s$, all dots of $C$ on $s$ lie to the right of all dots of $C'$ on $s$, viewed from the face adjacent to $s$. \label{enu:pairingposSides}
		\item The crossings and crossover arrows lie in a small neighbourhood of the second side $s_2(a)$ for each arc $a\in A$; see Figure~\ref{fig:StandardPairingPosition} for an illustration. \label{enu:pairingposLocalSystems}
		\item For each open face $f$, the one-sided $f$-joins of $C$ end on the left of the basepoint, the ones of $C'$ end on the right of the basepoint, viewed from the face $f$. \label{enu:pairingposBasepoints}
		\item Any two $f$-joins intersect minimally with respect to the first three conditions. Similarly, the precurves intersect minimally on the neighbourhood of each arc.  \label{enu:pairingposMinimality}
	\end{enumerate}
\end{definition}

\begin{definition}
	With a pair of two simply-faced precurves $C$ and $C'$ in pairing position, we associate a chain complex $\LagrangianFC(C,C')$ as follows:
	as a vector space over $\mathbb{F}_2$, $\LagrangianFC(C,C')$ decomposes into two summands $\LagrangianFC^\times(C,C')$ and $\LagrangianFC^+(C,C')$. The former is generated by the intersection points between the curves in the neighbourhoods of the arcs, the latter by those on faces. We call the former \textbf{upper} and the latter \textbf{lower} intersection points/generators of $\LagrangianFC(C,C')$. 
	The differential on $\LagrangianFC(C,C')$ is a map 
	$$d\co\LagrangianFC^\times(C,C')\rightarrow\LagrangianFC^+(C,C')$$
	defined by counting bigons connecting upper intersection points to lower ones. More precisely, a bigon from an upper intersection point $x$ to a lower one $y$ is an orientation-preserving embedding
	$$\iota\co D^2\hookrightarrow S$$
	satisfying the following properties:
	\begin{itemize}
		\item the restriction of $\iota$ to the non-negative real part of $\partial D^2$ is a path from $x=\iota(-i)$ to $y=\iota(+i)$ on the first precurve $C$ such that the orientation is opposite to the orientation of any crossover arrows in $C$;
		\item the restriction of $\iota$ to the non-positive real part of $\partial D^2$ is a path from $y$ to $x$ on the second precurve $C'$ such that the orientation is opposite to the orientation of any crossover arrows in $C'$;
		\item $x$ and $y$ are convex corners of the image of $\iota$. 
	\end{itemize}
	See Figure~\ref{fig:LagrangianHFConventions} for an illustration of the above conventions.
	If $\mathcal{M}(x,y)$ denotes the set of such bigons up to reparametrization, the differential \(d\) is defined by
	$$d(x)=\sum \#\mathcal{M}(x,y)~y.$$
	By construction, it only connects upper generators to lower ones. We denote the homology of $\LagrangianFC(C,C')$ by $\LagrangianFH(C,C')$ and call it the \textbf{Lagrangian intersection Floer homology} of $C$ and~$C'$.
\end{definition}

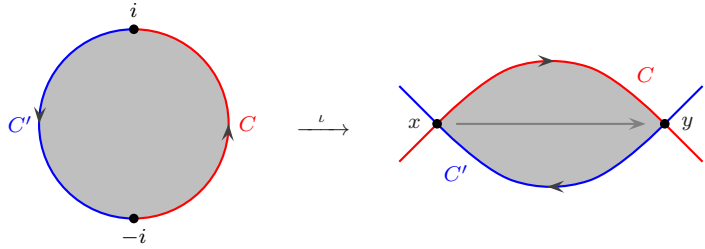
\begin{figure}[t]\centering
	{
		\psset{unit=0.5}
		\begin{pspicture}(-9,-3.5)(10,3.5)
		\rput(-5,0){
			\psarc*[linecolor=lightgray,linewidth=0pt](0,0){2.5}{0}{360}
			\psarc[linecolor=red](0,0){2.5}{-90}{90}
			\psarc[linecolor=blue](0,0){2.5}{90}{-90}
			\psarc[linecolor=darkgray]{->}(0,0){2.5}{-5}{0}
			\psarc[linecolor=darkgray]{->}(0,0){2.5}{175}{180}
			
			\psdot(2.5;90)
			\psdot(2.5;-90)
			
			\rput(3;-90){$-i$}
			\rput(3;90){$i$}
			
			\rput(3;0){$\textcolor{red}{C}$}
			\rput(3;180){$\textcolor{blue}{C'}$}
		}
		
		\rput(0,0){$\xrightarrow{~~\iota~~}$}
		
		\rput(6,0){
			\pscustom*[linecolor=lightgray]{
				\pscurve(-3,0)(-1,1.5)(1,1.5)(3,0)
				\pscurve[liftpen=2](3,0)(1,-1.5)(-1,-1.5)(-3,0)
			}

			\pscurve[linecolor=red](-3,0)(-1,1.5)(1,1.5)(3,0)
			\pscurve[linecolor=blue](-3,0)(-1,-1.5)(1,-1.5)(3,0)
			\psline[linecolor=darkgray]{->}(0.1,-1.66)(-0.1,-1.66)
			\psline[linecolor=darkgray]{<-}(0.1,1.66)(-0.1,1.66)
			
			\psline[linecolor=red](3,0)(4,-1)
			\psline[linecolor=red](-3,0)(-4,-1)
			\psline[linecolor=blue](3,0)(4,1)
			\psline[linecolor=blue](-3,0)(-4,1)
			
			\psdots(-3,0)(3,0)
			\rput(3.6,0){$y$}
			\rput(-3.6,0){$x$}
			
			\rput(2.5,1.3){$\textcolor{red}{C}$}
			\rput(-2.5,-1.3){$\textcolor{blue}{C'}$}
			\psline[linecolor=gray]{->}(-2.5,0)(2.5,0)
		}
		\end{pspicture}
	}
	\caption{Our orientation conventions for bigons: $\iota$ maps the unit disc in $\mathbb{C}$ to $S$. Thus, the normal vector, determined by the right-hand rule, points out of the plane on the left, but into the plane on the right. The arrows on the boundary of the disc and bigon indicate the induced boundary orientations. The generator $x$ is an upper generator (so it lies in the neighbourhood of an arc) and $y$ is a lower generator (so it lies in some face).  }\label{fig:LagrangianHFConventions}
\end{figure}

\begin{remark}
	The definition above is an adaptation of the Lagragian intersection Floer homology from Abouzaid's paper~\cite{AbouzaidSurfaces} to our more combinatorial setting, using the language of~\cite{HRW}. However, note that Hanselman, Rasmussen and Watson use slightly different orientation conventions in~\cite{HRW}: because of the way they express their glueing formula, they find it more convenient to interpret the two collections of (pre-)curves $C$ and $C'$ differently, namely one in terms of a type D structure and the other in terms of a type A structure; we treat both curves the same, which follows more standard conventions in Lagrangian intersection Floer theory. For example, in their setting~\cite[Definition~34]{HRW}, $\LagrangianFH(C,C)$ vanishes for some objects $C$; so in particular, there would be no identity morphism for such $C$. This does not happen with our conventions.
\end{remark}

\begin{definition}\label{def:resolution}
	With an intersection point $x$ between two simply-faced precurves $C$ and $C'$ in pairing position, we may associate a morphism of precurves $\varphi(x)$ from $C$ to $C'$ as follows. Consider the union of all paths $\gamma\co [0,1]\rightarrow C\cup C'$ (considered up to homotopy) satisfying the following conditions: 
	\begin{enumerate}
		\item the restriction of $\gamma$ to $[0,\frac{1}{2}]$ is a path on $C$ from a dot $\bullet(s,i)$ to $x$ which does not meet any other dot on $C$ and which follows the orientation of any crossover arrows,
		\item the restriction of $\gamma$ to $[\frac{1}{2},1]$ is a path on $C'$ from $x$ to a dot $\bullet(s',i')$ which does not meet any other dot on $C'$ and which follows the orientation of any crossover arrows,
		\item $\gamma$ turns left at the intersection point $x$: 
		$\raisebox{-3pt}{\psset{unit=0.2}%
			\begin{pspicture}(-1.1,-1)(1.1,1)
			\psline[linecolor=blue](1,-1)(-1,1)
			\psline[linecolor=red](-1,-1)(1,1)	
			\psdot(0,0)
			\rput{-45}(0,0){
				\psline[linearc=0.25,arrowsize= 1pt 2]{->}(0.3,1.3)(0.3,0.3)(1.3,0.3)
			}
			\rput{135}(0,0){
				\psline[linearc=0.25,arrowsize= 1pt 2]{->}(0.3,1.3)(0.3,0.3)(1.3,0.3)
			}
			\end{pspicture}}$.
	\end{enumerate}
	By labelling each of these paths by $p^{s}_{s'}$ and counting them modulo 2, we may regard them as a morphism $\varphi(x)$ from $C$ to $C'$, which we call the \textbf{resolution of $x$}. 
\end{definition}

\begin{lemma}\label{lem:ResolutionWellDefined}
	The resolution \(\varphi(x)\) is a well-defined morphism of precurves from \(C\) to \(C'\).
\end{lemma}

\begin{proof}
	If the intersection point is lower, then $\varphi(x)=\varphi^+(x)$, so there is nothing to check.
	If the intersection point is an upper point on an arc $a$, then $\varphi(x)=\varphi^\times(x)$. More explicitly, $\varphi_1:=(\iota_{s_1(a)}.\varphi^\times(x).\iota_{s_1(a)})$ contains exactly one non-zero component, say an arrow from a dot $\bullet(s_1(a),i)$ on the first precurve to a dot $\bullet(s_1(a),i')$ on the second precurve; see Figure~\ref{fig:BigonCounts} for an illustration. 
	We need to show that $\varphi_2:=(\iota_{s_2(a)}.\varphi^\times(x).\iota_{s_2(a)})$ is equal to  $P'_a\circ\varphi_1\circ P_a^{-1}$. 
	For this, recall that the component of $P_a^{-1}$ in the $j^\text{th}$ column and $i^\text{th}$ row is given by the number of paths (satisfying the conditions in Remark~\ref{rem:PathsForTrainTracks}) from $\bullet(s_2(a),j)$ to $\bullet(s_1(a),i)$ of the first precurve; similarly, the component of $P'_a$ in the $(i')^\text{th}$ column and $(j')^\text{th}$ row is given by the number of paths from $\bullet(s_1(a),i')$ to $\bullet(s_2(a),j')$ of the second precurve. Thus, $P'_a\circ\varphi_1\circ P_a^{-1}$ has a non-zero component from $\bullet(s_2(a),j)$ to $\bullet(s_2(a),j')$ iff the number of paths from $\bullet(s_2(a),j)$ to $\bullet(s_2(a),j')$ via $x$ is odd. This is agrees with the definition of $\varphi_2$.
\end{proof}

\begin{lemma}
	The resolution \(\varphi(x)\) of any generator \(x\) of \(\LagrangianFC(C,C')\) is homogeneous with respect to the \(\delta\)-grading.
\end{lemma}
\begin{proof}
	If $x$ is upper, this is true because dots connected by curve segments on an arc neighbourhood have the same $\delta$-grading by the definition of $\delta$-gradings on precurves. For lower intersection points, $\varphi(x)$ has at most two components. If it has exactly two components, both $f$-joins are two-sided; for an illustration, see Figures~\ref{fig:GensOfMorNoM2}, \ref{fig:GensOfMorNoM4} and~\ref{fig:GensOfMorNoM5}. Suppose the two components correspond to paths from $\bullet(s,i)$ to $\bullet(s',i')$ and from $\bullet(t,j)$ to $\bullet(t',j')$. Assume without loss of generality that if $f$ is open, the basepoint of $f$ lies between the side $s'$ and $t$. Then the $\delta$-gradings of the two components are
	$$\delta(\bullet(s',i'))-\delta(\bullet(s,i))+\delta(p^{s}_{s'})
	\quad\text{ and }\quad \delta(\bullet(t',j'))-\delta(\bullet(t,j))+\delta(p^{t}_{t'}),$$
	respectively. Since 
	$$
	\delta(\bullet(s,i))-\delta(\bullet(t,j))+\delta(p^{t}_{s})
	=1=
	\delta(\bullet(s',i'))-\delta(\bullet(t',j'))+\delta(p^{t'}_{s'}),
	$$ 
	$p^{t}_{s}=p^{t'}_{s}p^{t}_{t'}$ and $p^{t'}_{s'}=p^{s}_{s'}p^{t'}_{s}$, these two gradings agree. 
\end{proof}

\begin{definition}\label{def:gradingVIAresolution}
	We endow $\LagrangianFC(C,C')$ with a \textbf{$\delta$-grading} by defining the $\delta$-grading of any intersection point to be the $\delta$-grading of its resolution. 
\end{definition}

\begin{lemma}
	The differential \(d\) on \(\LagrangianFC(C,C')\) increases \(\delta\)-grading by 1.
\end{lemma}

\begin{proof}
	This follows from a similar argument as the previous lemma. 
\end{proof}

\begin{definition}
	Given two reduced precurves $C$ and $C'$, let us write $(\Mor^+(C,C'),D)$ for the subcomplex of $(\Mor(C,C'),D)$ generated by those morphisms which do not contain an identity component. Let $\Mor^\times(C,C')$ be the vector space of all morphisms that consist of identity components only. By endowing it with the 0 differential, we obtain the following short exact sequence of chain complexes
	$$
	\begin{tikzcd}
	0
	\arrow{r}{}
	&
	(\Mor^+(C,C'),D^+)
	\arrow{r}{}
	&
	(\Mor(C,C'),D)
	\arrow{r}{}
	&
	(\Mor^\times(C,C'),0)
	\arrow{r}{}
	&
	0
	\end{tikzcd}
	$$
	as in Definitions~\ref{def:AlgebraFromGlueingFaces} and~\ref{def:SplittingCatsForEquivalenceFunctors}.
	Via the induced long exact sequence, it gives rise to a graded chain homotopy equivalence
	$$
	H_*(\Mor(C,C'),D)\cong
	\left(
	\begin{tikzcd}
	H_*(\Mor^\times(C,C'),0)
	\arrow{r}{\beta}
	&
	H_*(\Mor^+(C,C'),D^+)
	\end{tikzcd}
	\right),
	$$
	where $\beta$ is the boundary map from the long exact sequence. 
\end{definition}

The central result of this subsection is the following:

\begin{theorem}\label{thm:PairingMorLagrangianFH}
	Given a pair of simply-faced precurves \(C\) and \(C'\) in pairing position on a marked surface \((S,M)\) with arc system \(A\), the chain complex \(\LagrangianFC(C,C')\) and the mapping cone of the map
	$$
	\beta\co H_*(\Mor^\times(C,C'),0)
	\longrightarrow
	H_*(\Mor^+(C,C'),D^+)
	$$
	from the previous definition are graded chain isomorphic via the correspondence between generators of \(\LagrangianFC(C,C')\) and their resolutions. In particular, 
	$$\LagrangianFH(C,C')\cong H_*(\Mor(C,C'),D).$$
\end{theorem}

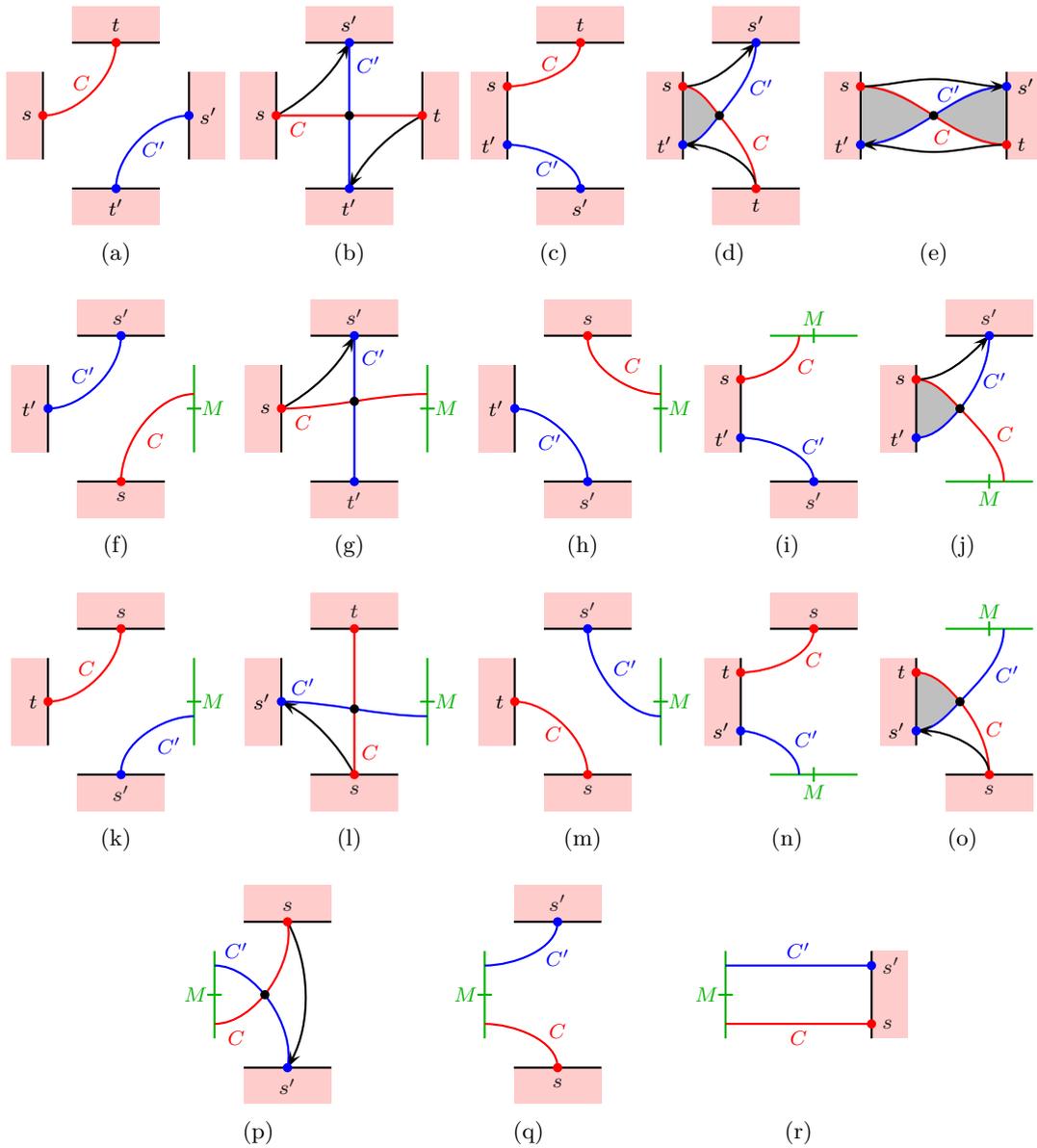
\begin{figure}[p]
	\centering
	\psset{unit=0.2}
	\begin{subfigure}{0.215\textwidth}\centering
		\begin{pspicture}(-7.5,-7.5)(7.5,7.5)
		
		\pscustom*[linecolor=lightred,linewidth=0pt]{
			\psline(5,-3)(5,3)
			\psline(7.5,3)(7.5,-3)
		}
		\pscustom*[linecolor=lightred,linewidth=0pt]{
			\psline(-3,5)(3,5)
			\psline(3,7.5)(-3,7.5)
		}
		\pscustom*[linecolor=lightred,linewidth=0pt]{
			\psline(-5,-3)(-5,3)
			\psline(-7.5,3)(-7.5,-3)
		}
		\pscustom*[linecolor=lightred,linewidth=0pt]{
			\psline(-3,-5)(3,-5)
			\psline(3,-7.5)(-3,-7.5)
		}
		
		\psline(-5,-3)(-5,3)
		\psline(5,-3)(5,3)
		\psline(-3,5)(3,5)
		\psline(-3,-5)(3,-5)
		
		\psecurve[linecolor=blue](10,-5)(5,0)(0,-5)(5,-10)
		\psecurve[linecolor=red](-10,5)(-5,0)(0,5)(-5,10)
		
		\psdots[linecolor=red](-5,0)
		\uput{0.75}[180](-5,0){$s$}
		
		\psdots[linecolor=blue](5,0)
		\uput{0.75}[0](5,0){$s'$}
		
		\psdots[linecolor=red](0,5)
		\uput{0.75}[90](0,5){$t$}
		
		\psdots[linecolor=blue](0,-5)
		\uput{0.75}[-90](0,-5){$t'$}
		
		\uput{2.5}[-45](0,0){$\textcolor{blue}{C'}$}
		\uput{2.5}[135](0,0){$\textcolor{red}{C}$}
		
		\end{pspicture}
		\subcaption{}\label{fig:GensOfMorNoM1}
	\end{subfigure}
	\begin{subfigure}{0.215\textwidth}\centering
		\begin{pspicture}(-7.5,-7.5)(7.5,7.5)
		\pscustom*[linecolor=lightred,linewidth=0pt]{
			\psline(5,-3)(5,3)
			\psline(7.5,3)(7.5,-3)
		}
		\pscustom*[linecolor=lightred,linewidth=0pt]{
			\psline(-3,5)(3,5)
			\psline(3,7.5)(-3,7.5)
		}
		\pscustom*[linecolor=lightred,linewidth=0pt]{
			\psline(-5,-3)(-5,3)
			\psline(-7.5,3)(-7.5,-3)
		}
		\pscustom*[linecolor=lightred,linewidth=0pt]{
			\psline(-3,-5)(3,-5)
			\psline(3,-7.5)(-3,-7.5)
		}
		
		\psline(-5,-3)(-5,3)
		\psline(5,-3)(5,3)
		\psline(-3,5)(3,5)
		\psline(-3,-5)(3,-5)
		
		\psline[linecolor=blue](0,5)(0,-5)
		\psline[linecolor=red](5,0)(-5,0)
		
		\psecurve{->}(10,1)(5,0)(0,-5)(-1,-10)
		\psecurve{->}(-10,-1)(-5,0)(0,5)(1,10)
		
		\psdots[linecolor=red](-5,0)
		\uput{0.75}[180](-5,0){$s$}
		
		\psdots[linecolor=red](5,0)
		\uput{0.75}[0](5,0){$t$}
		
		\psdots[linecolor=blue](0,5)
		\uput{0.75}[90](0,5){$s'$}
		
		\psdots[linecolor=blue](0,-5)
		\uput{0.75}[-90](0,-5){$t'$}
		
		\uput{0.5}[-90](-3.5,0){$\textcolor{red}{C}$}
		\uput{0.5}[0](0,3.5){$\textcolor{blue}{C'}$}
		
		\psdot(0,0)
		
		\end{pspicture}
		\subcaption{}\label{fig:GensOfMorNoM2}
	\end{subfigure}
	\begin{subfigure}{0.16\textwidth}\centering
		\begin{pspicture}(-7.5,-7.5)(4,7.5)
		
		\pscustom*[linecolor=lightred,linewidth=0pt]{
			\psline(-3,5)(3,5)
			\psline(3,7.5)(-3,7.5)
		}
		\pscustom*[linecolor=lightred,linewidth=0pt]{
			\psline(-5,-3)(-5,3)
			\psline(-7.5,3)(-7.5,-3)
		}
		\pscustom*[linecolor=lightred,linewidth=0pt]{
			\psline(-3,-5)(3,-5)
			\psline(3,-7.5)(-3,-7.5)
		}
		
		\psline(-5,-3)(-5,3)
		\psline(-3,5)(3,5)
		\psline(-3,-5)(3,-5)


		\psecurve[linecolor=blue](-10,-5)(-5,-2)(0,-5)(-5,-8)
		\psecurve[linecolor=red](-10,5)(-5,2)(0,5)(-5,8)
		
		\psdots[linecolor=blue](-5,-2)
		\uput{0.75}[180](-5,-2){$t'$}
		
		\psdots[linecolor=red](-5,2)
		\uput{0.75}[180](-5,2){$s$}
		
		\psdots[linecolor=red](0,5)
		\uput{0.75}[90](0,5){$t$}
		
		\psdots[linecolor=blue](0,-5)
		\uput{0.75}[-90](0,-5){$s'$}
		
		\uput{3.4}[-120](0,0){$\textcolor{blue}{C'}$}
		\uput{3.4}[120](0,0){$\textcolor{red}{C}$}
		\end{pspicture}
		\subcaption{}\label{fig:GensOfMorNoM3}
	\end{subfigure}
	\begin{subfigure}{0.16\textwidth}\centering
		\begin{pspicture}(-7.5,-7.5)(4,7.5)
		
		\pscustom*[linecolor=lightgray]{
			\psecurve(-7.5,0)(-5,2)(-2.5,0)(0,-5)
			\psecurve(0,5)(-2.5,0)(-5,-2)(-7.5,0)
			\psline(-5,-2)(-5,2)
		}
		
		\pscustom*[linecolor=lightred,linewidth=0pt]{
			\psline(-3,5)(3,5)
			\psline(3,7.5)(-3,7.5)
		}
		\pscustom*[linecolor=lightred,linewidth=0pt]{
			\psline(-5,-3)(-5,3)
			\psline(-7.5,3)(-7.5,-3)
		}
		\pscustom*[linecolor=lightred,linewidth=0pt]{
			\psline(-3,-5)(3,-5)
			\psline(3,-7.5)(-3,-7.5)
		}
		
		\psline(-5,-3)(-5,3)
		\psline(-3,5)(3,5)
		\psline(-3,-5)(3,-5)
		
		\psecurve[linecolor=red](-7.5,0)(-5,2)(-2.5,0)(0,-5)(-2.5,-10)
		\psecurve[linecolor=blue](-7.5,0)(-5,-2)(-2.5,0)(0,5)(-2.5,10)
		
		\psecurve{<-}(-10,-5)(-5,-2)(0,-5)(-5,-8)
		\psecurve{->}(-10,5)(-5,2)(0,5)(3,10)
		
		\psdot(-2.5,0)
		
		\psdots[linecolor=blue](-5,-2)
		\uput{0.75}[180](-5,-2){$t'$}
		
		\psdots[linecolor=red](-5,2)
		\uput{0.75}[180](-5,2){$s$}
		
		\psdots[linecolor=blue](0,5)
		\uput{0.75}[90](0,5){$s'$}
		
		\psdots[linecolor=red](0,-5)
		\uput{0.75}[-90](0,-5){$t$}
		
		\uput{1.4}[-80](0,0){$\textcolor{red}{C}$}
		\uput{1.4}[80](0,0){$\textcolor{blue}{C'}$}
		
		\end{pspicture}
		\subcaption{}\label{fig:GensOfMorNoM4}
	\end{subfigure}
	\begin{subfigure}{0.215\textwidth}\centering
		\begin{pspicture}(-7.5,-7.5)(7.5,7.5)
		
		\pscustom*[linecolor=lightgray]{
			\psecurve(-10,2)(-5,-2)(5,2)(10,-2)
			\psline(5,2)(5,-2)
			\psecurve(10,2)(5,-2)(-5,2)(-10,-2)
			\psline(-5,2)(-5,-2)
		}
		\pscustom*[linecolor=white,linewidth=0pt]{
			\psecurve(-10,-1.5)(-5,-2)(0,-1.5)(5,-2)(10,-1.5)
			\psline(5,-2)(5,-3)(-5,-3)(-5,-2)
		}
		\pscustom*[linecolor=white,linewidth=0pt]{
			\psecurve(-10,1.5)(-5,2)(0,1.5)(5,2)(10,1.5)
			\psline(5,2)(5,3)(-5,3)(-5,2)
		}
		
		\pscustom*[linecolor=lightred,linewidth=0pt]{
			\psline(5,-3)(5,3)
			\psline(7.5,3)(7.5,-3)
		}
		\pscustom*[linecolor=lightred,linewidth=0pt]{
			\psline(-5,-3)(-5,3)
			\psline(-7.5,3)(-7.5,-3)
		}
		
		\psline(-5,-3)(-5,3)
		\psline(5,-3)(5,3)
		
		\psecurve[linecolor=blue](-15,2)(-5,-2)(5,2)(15,-2)
		\psecurve[linecolor=red](15,2)(5,-2)(-5,2)(-15,-2)
		
		\psecurve{->}(-10,2.5)(-5,2)(0,2.5)(5,2)(10,2.5)
		\psecurve{<-}(-10,-2.5)(-5,-2)(0,-2.5)(5,-2)(10,-2.5)
		
		\psdots[linecolor=red](-5,2)
		\uput{0.75}[180](-5,2){$s$}
		\psdots[linecolor=blue](-5,-2)
		\uput{0.75}[180](-5,-2){$t'$}
		
		\psdot(0,0)
		
		\psdots[linecolor=blue](5,2)
		\uput{0.75}[0](5,2){$s'$}
		\psdots[linecolor=red](5,-2)
		\uput{0.75}[0](5,-2){$t$}
		
		\uput{1}[60](0,0){$\textcolor{blue}{C'}$}
		\uput{1}[-70](0,0){$\textcolor{red}{C}$}
		\end{pspicture}
		\subcaption{}\label{fig:GensOfMorNoM5}
	\end{subfigure}
	\medskip\\
	\begin{subfigure}{0.215\textwidth}\centering
		\begin{pspicture}(-8,-7.5)(8,7.5)
		
		\pscustom*[linecolor=lightred,linewidth=0pt]{
			\psline(-3,5)(3,5)
			\psline(3,7.5)(-3,7.5)
		}
		\pscustom*[linecolor=lightred,linewidth=0pt]{
			\psline(-5,-3)(-5,3)
			\psline(-7.5,3)(-7.5,-3)
		}
		\pscustom*[linecolor=lightred,linewidth=0pt]{
			\psline(-3,-5)(3,-5)
			\psline(3,-7.5)(-3,-7.5)
		}
		
		\psline(-5,-3)(-5,3)
		\psline[linecolor=darkgreen](5,-3)(5,3)
		\psline(-3,5)(3,5)
		\psline(-3,-5)(3,-5)
		
		\psecurve[linecolor=red](10,-5)(5,1)(0,-5)(5,-11)
		\psecurve[linecolor=blue](-10,5)(-5,0)(0,5)(-5,10)
		
		\psdots[linecolor=blue](-5,0)
		\uput{0.75}[180](-5,0){$t'$}
		
		\psline[linecolor=darkgreen](4.5,0)(5.5,0)
		\uput{0.5}[0](5,0){\textcolor{darkgreen}{$M$}}
		
		\psdots[linecolor=blue](0,5)
		\uput{0.75}[90](0,5){$s'$}
		
		\psdots[linecolor=red](0,-5)
		\uput{0.75}[-90](0,-5){$s$}
		
		\uput{2.5}[-45](0,0){$\textcolor{red}{C}$}
		\uput{2.5}[135](0,0){$\textcolor{blue}{C'}$}
		
		\end{pspicture}
		\subcaption{}\label{fig:GensOfMorM1}
	\end{subfigure}
	\begin{subfigure}{0.215\textwidth}\centering
		\begin{pspicture}(-8,-7.5)(8,7.5)
		
		\pscustom*[linecolor=lightred,linewidth=0pt]{
			\psline(-3,5)(3,5)
			\psline(3,7.5)(-3,7.5)
		}
		\pscustom*[linecolor=lightred,linewidth=0pt]{
			\psline(-5,-3)(-5,3)
			\psline(-7.5,3)(-7.5,-3)
		}
		\pscustom*[linecolor=lightred,linewidth=0pt]{
			\psline(-3,-5)(3,-5)
			\psline(3,-7.5)(-3,-7.5)
		}
		
		\psline(-5,-3)(-5,3)
		\psline[linecolor=darkgreen](5,-3)(5,3)
		\psline(-3,5)(3,5)
		\psline(-3,-5)(3,-5)
		
		\psline[linecolor=blue](0,5)(0,-5)
		\psecurve[linecolor=red](-15,1)(-5,0)(5,1)(15,0)
		
		\psecurve{->}(-10,-1)(-5,0)(0,5)(1,10)
		
		\psdot(0,0.5)
		
		\psdots[linecolor=red](-5,0)
		\uput{0.75}[180](-5,0){$s$}
		
		\psline[linecolor=darkgreen](4.5,0)(5.5,0)
		\uput{0.5}[0](5,0){\textcolor{darkgreen}{$M$}}
		
		\psdots[linecolor=blue](0,5)
		\uput{0.75}[90](0,5){$s'$}
		
		\psdots[linecolor=blue](0,-5)
		\uput{0.75}[-90](0,-5){$t'$}
		
		\uput{0.5}[-90](-3.5,0){$\textcolor{red}{C}$}
		\uput{0.5}[0](0,3.5){$\textcolor{blue}{C'}$}
		
		\end{pspicture}
		\subcaption{}\label{fig:GensOfMorM2}
	\end{subfigure}
	\begin{subfigure}{0.215\textwidth}\centering
		\begin{pspicture}(-8,-7.5)(8,7.5)
		
		\pscustom*[linecolor=lightred,linewidth=0pt]{
			\psline(-3,5)(3,5)
			\psline(3,7.5)(-3,7.5)
		}
		\pscustom*[linecolor=lightred,linewidth=0pt]{
			\psline(-5,-3)(-5,3)
			\psline(-7.5,3)(-7.5,-3)
		}
		\pscustom*[linecolor=lightred,linewidth=0pt]{
			\psline(-3,-5)(3,-5)
			\psline(3,-7.5)(-3,-7.5)
		}
		
		\psline(-5,-3)(-5,3)
		\psline[linecolor=darkgreen](5,-3)(5,3)
		\psline(-3,5)(3,5)
		\psline(-3,-5)(3,-5)
		
		\psecurve[linecolor=red](10,5)(5,1)(0,5)(5,8)
		\psecurve[linecolor=blue](-10,-5)(-5,0)(0,-5)(-5,-10)
		
		\psdots[linecolor=blue](-5,0)
		\uput{0.75}[180](-5,0){$t'$}
		
		\psline[linecolor=darkgreen](4.5,0)(5.5,0)
		\uput{0.5}[0](5,0){\textcolor{darkgreen}{$M$}}
		
		\psdots[linecolor=blue](0,-5)
		\uput{0.75}[-90](0,-5){$s'$}
		
		\psdots[linecolor=red](0,5)
		\uput{0.75}[90](0,5){$s$}
		
		\uput{3.5}[45](0,0){$\textcolor{red}{C}$}
		\uput{2.5}[-135](0,0){$\textcolor{blue}{C'}$}
		
		\end{pspicture}
		\subcaption{}\label{fig:GensOfMorM3}
	\end{subfigure}
	\begin{subfigure}{0.16\textwidth}\centering
		\begin{pspicture}(-7.5,-7.5)(4,7.5)
		
		\pscustom*[linecolor=lightred,linewidth=0pt]{
			\psline(-5,-3)(-5,3)
			\psline(-7.5,3)(-7.5,-3)
		}
		\pscustom*[linecolor=lightred,linewidth=0pt]{
			\psline(-3,-5)(3,-5)
			\psline(3,-7.5)(-3,-7.5)
		}
		
		\psline(-5,-3)(-5,3)
		\psline[linecolor=darkgreen](-3,5)(3,5)
		\psline(-3,-5)(3,-5)
		
		\psecurve[linecolor=blue](-10,-5)(-5,-2)(0,-5)(-5,-8)
		\psecurve[linecolor=red](-10,5)(-5,2)(-1,5)(-5,8)
		
		\psdots[linecolor=blue](-5,-2)
		\uput{0.75}[180](-5,-2){$t'$}
		
		\psdots[linecolor=red](-5,2)
		\uput{0.75}[180](-5,2){$s$}
		
		\psline[linecolor=darkgreen](0,4.5)(0,5.5)
		\uput{0.5}[90](0,5.2){\textcolor{darkgreen}{$M$}}
		
		\psdots[linecolor=blue](0,-5)
		\uput{0.75}[-90](0,-5){$s'$}
		
		\uput{2}[-90](0,0){$\textcolor{blue}{C'}$}
		\uput{2.4}[100](0,0){$\textcolor{red}{C}$}
		\end{pspicture}
		\subcaption{}\label{fig:GensOfMorM4}
	\end{subfigure}
	\begin{subfigure}{0.16\textwidth}\centering
		\begin{pspicture}(-7.5,-7.5)(4,7.5)
		\pscustom*[linecolor=lightgray]{
			\psecurve(-8,0)(-5,2)(-2,0)(0,-5)
			\psecurve(0,5)(-2,0)(-5,-2)(-8,0)
			\psline(-5,-2)(-5,2)
		}	
		
		\pscustom*[linecolor=lightred,linewidth=0pt]{
			\psline(-3,5)(3,5)
			\psline(3,7.5)(-3,7.5)
		}
		\pscustom*[linecolor=lightred,linewidth=0pt]{
			\psline(-5,-3)(-5,3)
			\psline(-7.5,3)(-7.5,-3)
		}
		
		\psline(-5,-3)(-5,3)
		\psline(-3,5)(3,5)
		\psline[linecolor=darkgreen](-3,-5)(3,-5)

		\psecurve[linecolor=red](-8,0)(-5,2)(-2,0)(1,-5)(-2,-10)
		\psecurve[linecolor=blue](-8,0)(-5,-2)(-2,0)(0,5)(-2,10)
		
		\psecurve{->}(-10,5)(-5,2)(0,5)(3,10)
		
		\psdot(-2,0)
		
		\psdots[linecolor=blue](-5,-2)
		\uput{0.75}[180](-5,-2){$t'$}
		
		\psdots[linecolor=red](-5,2)
		\uput{0.75}[180](-5,2){$s$}
		
		\psline[linecolor=darkgreen](0,-4.5)(0,-5.5)
		\uput{0.5}[-90](0,-5.2){\textcolor{darkgreen}{$M$}}
		
		\psdots[linecolor=blue](0,5)
		\uput{0.75}[90](0,5){$s'$}
		
		\uput{1.4}[-60](0,0){$\textcolor{red}{C}$}
		\uput{1.4}[80](0,0){$\textcolor{blue}{C'}$}
		
		\end{pspicture}
		\subcaption{}\label{fig:GensOfMorM5}
	\end{subfigure}
	\medskip\\
	\begin{subfigure}{0.215\textwidth}\centering
		\begin{pspicture}(-8,-7.5)(8,7.5)
		
		\pscustom*[linecolor=lightred,linewidth=0pt]{
			\psline(-3,5)(3,5)
			\psline(3,7.5)(-3,7.5)
		}
		\pscustom*[linecolor=lightred,linewidth=0pt]{
			\psline(-5,-3)(-5,3)
			\psline(-7.5,3)(-7.5,-3)
		}
		\pscustom*[linecolor=lightred,linewidth=0pt]{
			\psline(-3,-5)(3,-5)
			\psline(3,-7.5)(-3,-7.5)
		}
		
		\psline(-5,-3)(-5,3)
		\psline[linecolor=darkgreen](5,-3)(5,3)
		\psline(-3,5)(3,5)
		\psline(-3,-5)(3,-5)
		
		\psecurve[linecolor=blue](10,-5)(5,-1)(0,-5)(5,-8)
		\psecurve[linecolor=red](-10,5)(-5,0)(0,5)(-5,10)
		
		\psdots[linecolor=red](-5,0)
		\uput{0.75}[180](-5,0){$t$}
		
		\psline[linecolor=darkgreen](4.5,0)(5.5,0)
		\uput{0.5}[0](5,0){\textcolor{darkgreen}{$M$}}
		
		\psdots[linecolor=red](0,5)
		\uput{0.75}[90](0,5){$s$}
		
		\psdots[linecolor=blue](0,-5)
		\uput{0.75}[-90](0,-5){$s'$}
		
		\uput{3.5}[-45](0,0){$\textcolor{blue}{C'}$}
		\uput{2.5}[135](0,0){$\textcolor{red}{C}$}
		
		\end{pspicture}
		\subcaption{}\label{fig:GensOfMorMp1}
	\end{subfigure}
	\begin{subfigure}{0.215\textwidth}\centering
		\begin{pspicture}(-8,-7.5)(8,7.5)
		
		\pscustom*[linecolor=lightred,linewidth=0pt]{
			\psline(-3,5)(3,5)
			\psline(3,7.5)(-3,7.5)
		}
		\pscustom*[linecolor=lightred,linewidth=0pt]{
			\psline(-5,-3)(-5,3)
			\psline(-7.5,3)(-7.5,-3)
		}
		\pscustom*[linecolor=lightred,linewidth=0pt]{
			\psline(-3,-5)(3,-5)
			\psline(3,-7.5)(-3,-7.5)
		}
		
		\psline(-5,-3)(-5,3)
		\psline[linecolor=darkgreen](5,-3)(5,3)
		\psline(-3,5)(3,5)
		\psline(-3,-5)(3,-5)
		
		\psline[linecolor=red](0,5)(0,-5)
		\psecurve[linecolor=blue](-15,-1)(-5,0)(5,-1)(15,0)
		
		\psecurve{<-}(-10,1)(-5,0)(0,-5)(1,-10)
		
		\psdot(0,-0.5)
		
		\psdots[linecolor=blue](-5,0)
		\uput{0.75}[180](-5,0){$s'$}
		
		\psline[linecolor=darkgreen](4.5,0)(5.5,0)
		\uput{0.5}[0](5,0){\textcolor{darkgreen}{$M$}}
		
		\psdots[linecolor=red](0,5)
		\uput{0.75}[90](0,5){$t$}
		
		\psdots[linecolor=red](0,-5)
		\uput{0.75}[-90](0,-5){$s$}
		
		\uput{0.5}[90](-3.5,0){$\textcolor{blue}{C'}$}
		\uput{0.5}[0](0,-3.5){$\textcolor{red}{C}$}
		
		\end{pspicture}
		\subcaption{}\label{fig:GensOfMorMp2}
	\end{subfigure}
	\begin{subfigure}{0.215\textwidth}\centering
		\begin{pspicture}(-8,-7.5)(8,7.5)
		
		\pscustom*[linecolor=lightred,linewidth=0pt]{
			\psline(-3,5)(3,5)
			\psline(3,7.5)(-3,7.5)
		}
		\pscustom*[linecolor=lightred,linewidth=0pt]{
			\psline(-5,-3)(-5,3)
			\psline(-7.5,3)(-7.5,-3)
		}
		\pscustom*[linecolor=lightred,linewidth=0pt]{
			\psline(-3,-5)(3,-5)
			\psline(3,-7.5)(-3,-7.5)
		}
		
		\psline(-5,-3)(-5,3)
		\psline[linecolor=darkgreen](5,-3)(5,3)
		\psline(-3,5)(3,5)
		\psline(-3,-5)(3,-5)
		
		\psecurve[linecolor=blue](10,5)(5,-1)(0,5)(5,11)
		\psecurve[linecolor=red](-10,-5)(-5,0)(0,-5)(-5,-10)
		
		\psdots[linecolor=red](-5,0)
		\uput{0.75}[180](-5,0){$t$}
		
		\psline[linecolor=darkgreen](4.5,0)(5.5,0)
		\uput{0.5}[0](5,0){\textcolor{darkgreen}{$M$}}
		
		\psdots[linecolor=red](0,-5)
		\uput{0.75}[-90](0,-5){$s$}
		
		\psdots[linecolor=blue](0,5)
		\uput{0.75}[90](0,5){$s'$}
		
		\uput{2.5}[45](0,0){$\textcolor{blue}{C'}$}
		\uput{2.5}[-135](0,0){$\textcolor{red}{C}$}
		
		\end{pspicture}
		\subcaption{}\label{fig:GensOfMorMp3}
	\end{subfigure}
	\begin{subfigure}{0.16\textwidth}\centering
		\begin{pspicture}(-7.5,-7.5)(4,7.5)
		
		\pscustom*[linecolor=lightred,linewidth=0pt]{
			\psline(-3,5)(3,5)
			\psline(3,7.5)(-3,7.5)
		}
		\pscustom*[linecolor=lightred,linewidth=0pt]{
			\psline(-5,-3)(-5,3)
			\psline(-7.5,3)(-7.5,-3)
		}

		\psline(-5,-3)(-5,3)
		\psline(-3,5)(3,5)
		\psline[linecolor=darkgreen](-3,-5)(3,-5)
		
		\psecurve[linecolor=blue](-10,-5)(-5,-2)(-1,-5)(-5,-8)
		\psecurve[linecolor=red](-10,5)(-5,2)(0,5)(-5,8)
		
		\psdots[linecolor=blue](-5,-2)
		\uput{0.75}[180](-5,-2){$s'$}
		
		\psdots[linecolor=red](-5,2)
		\uput{0.75}[180](-5,2){$t$}
		
		\psline[linecolor=darkgreen](0,-4.5)(0,-5.5)
		\uput{0.5}[-90](0,-5.2){\textcolor{darkgreen}{$M$}}
		
		\psdots[linecolor=red](0,5)
		\uput{0.75}[90](0,5){$s$}
		
		\uput{2}[-100](0,0){$\textcolor{blue}{C'}$}
		\uput{2.4}[90](0,0){$\textcolor{red}{C}$}
		\end{pspicture}
		\subcaption{}\label{fig:GensOfMorMp4}
	\end{subfigure}
	\begin{subfigure}{0.16\textwidth}\centering
		\begin{pspicture}(-7.5,-7.5)(4,7.5)
		\pscustom*[linecolor=lightgray]{
			\psecurve(-8,0)(-5,-2)(-2,0)(0,5)
			\psecurve(0,-5)(-2,0)(-5,2)(-8,0)
			\psline(-5,2)(-5,-2)
		}	
		
		\pscustom*[linecolor=lightred,linewidth=0pt]{
			\psline(-5,-3)(-5,3)
			\psline(-7.5,3)(-7.5,-3)
		}
		\pscustom*[linecolor=lightred,linewidth=0pt]{
			\psline(-3,-5)(3,-5)
			\psline(3,-7.5)(-3,-7.5)
		}
		
		\psline(-5,-3)(-5,3)
		\psline[linecolor=darkgreen](-3,5)(3,5)
		\psline(-3,-5)(3,-5)

		\psecurve[linecolor=blue](-8,0)(-5,-2)(-2,0)(1,5)(-2,10)
		\psecurve[linecolor=red](-8,0)(-5,2)(-2,0)(0,-5)(-2,-10)
		
		\psecurve{<-}(-10,-5)(-5,-2)(0,-5)(-5,-8)
		
		\psdot(-2,0)

		\psdots[linecolor=blue](-5,-2)
		\uput{0.75}[180](-5,-2){$s'$}
		
		\psdots[linecolor=red](-5,2)
		\uput{0.75}[180](-5,2){$t$}
		
		\psline[linecolor=darkgreen](0,4.5)(0,5.5)
		\uput{0.5}[90](0,5.2){\textcolor{darkgreen}{$M$}}
		
		\psdots[linecolor=red](0,-5)
		\uput{0.75}[-90](0,-5){$s$}
		
		\uput{1.4}[60](0,0){$\textcolor{blue}{C'}$}
		\uput{1.4}[-80](0,0){$\textcolor{red}{C}$}
		
		\end{pspicture}
		\subcaption{}\label{fig:GensOfMorMp5}
	\end{subfigure}
	\medskip\\
	\begin{subfigure}{0.25\textwidth}\centering
		\begin{pspicture}(-8,-7.5)(4,7.5)
		
		\pscustom*[linecolor=lightred,linewidth=0pt]{
			\psline(-3,5)(3,5)
			\psline(3,7.5)(-3,7.5)
		}
		\pscustom*[linecolor=lightred,linewidth=0pt]{
			\psline(-3,-5)(3,-5)
			\psline(3,-7.5)(-3,-7.5)
		}
		\psline[linecolor=darkgreen](-5,-3)(-5,3)
		\psline(-3,5)(3,5)
		\psline(-3,-5)(3,-5)

		\psecurve{->}(-3,7)(0,5)(0,-5)(-3,-7)
		
		\psecurve[linecolor=blue](-10,-5)(-5,2)(0,-5)(-5,-8)
		\psecurve[linecolor=red](-10,5)(-5,-2)(0,5)(-5,8)
		
		\psdot(-1.55,0)
		
		\psline[linecolor=darkgreen](-4.5,0)(-5.5,0)
		\uput{0.5}[180](-5,0){\textcolor{darkgreen}{$M$}}
		
		\psdots[linecolor=red](0,5)
		\uput{0.75}[90](0,5){$s$}
		
		\psdots[linecolor=blue](0,-5)
		\uput{0.75}[-90](0,-5){$s'$}
		
		\uput{0.5}[90](-3.5,2){$\textcolor{blue}{C'}$}
		\uput{0.5}[-90](-3.5,-2){$\textcolor{red}{C}$}
		\end{pspicture}
		\subcaption{}\label{fig:GensOfMorBoth1}
	\end{subfigure}
	\begin{subfigure}{0.25\textwidth}\centering
		\begin{pspicture}(-8,-7.5)(4,7.5)
		
		\pscustom*[linecolor=lightred,linewidth=0pt]{
			\psline(-3,5)(3,5)
			\psline(3,7.5)(-3,7.5)
		}
		
		\pscustom*[linecolor=lightred,linewidth=0pt]{
			\psline(-3,-5)(3,-5)
			\psline(3,-7.5)(-3,-7.5)
		}
		
		\psline[linecolor=darkgreen](-5,-3)(-5,3)
		\psline(-3,5)(3,5)
		\psline(-3,-5)(3,-5)

		\psecurve[linecolor=red](-10,-5)(-5,-2)(0,-5)(-5,-8)
		\psecurve[linecolor=blue](-10,5)(-5,2)(0,5)(-5,8)

		\psline[linecolor=darkgreen](-4.5,0)(-5.5,0)
		\uput{0.5}[180](-5,0){\textcolor{darkgreen}{$M$}}
		
		\psdots[linecolor=blue](0,5)
		\uput{0.75}[90](0,5){$s'$}
		
		\psdots[linecolor=red](0,-5)
		\uput{0.75}[-90](0,-5){$s$}
		
		\uput{2}[90](0,0){$\textcolor{blue}{C'}$}
		\uput{2}[-90](0,0){$\textcolor{red}{C}$}
		
		\end{pspicture}
		\subcaption{}\label{fig:GensOfMorBoth2}
	\end{subfigure}
	\begin{subfigure}{0.25\textwidth}\centering
		\begin{pspicture}(-8,-7.5)(8,7.5)
		\pscustom*[linecolor=lightred,linewidth=0pt]{
			\psline(5,-3)(5,3)
			\psline(7.5,3)(7.5,-3)
		}
		
		\psline[linecolor=darkgreen](-5,-3)(-5,3)
		\psline(5,-3)(5,3)
		
		\psline[linecolor=blue](5,2)(-5,2)
		\psline[linecolor=red](5,-2)(-5,-2)
		
		\psline[linecolor=darkgreen](-4.5,0)(-5.5,0)
		\uput{0.5}[180](-5,0){\textcolor{darkgreen}{$M$}}

		\psdots[linecolor=blue](5,2)
		\uput{0.75}[0](5,2){$s'$}
		\psdots[linecolor=red](5,-2)
		\uput{0.75}[0](5,-2){$s$}
		
		\uput{2.5}[90](0,0){$\textcolor{blue}{C'}$}
		\uput{2.5}[-90](0,0){$\textcolor{red}{C}$}
		
		\end{pspicture}
		\subcaption{}\label{fig:GensOfMorBoth3}
	\end{subfigure}
	\caption{The computation of morphism spaces between $f$-joins of two precurves illustrating the proof of Theorem~\ref{thm:PairingMorLagrangianFH}. The first row shows all configurations with two-sided $f$-joins only. However, there might be a point of $M$ in the face, in which case some of the black arrows are zero. The second and third rows show those configurations with both two- and one-sided $f$-joins. The last row shows those configurations with one-sided $f$-joins only.}\label{fig:GensOfMor}
\end{figure}

\begin{proof}
	Let us consider lower intersection points first. According to the theorem, these should correspond to generators of the homology of $(\Mor^+(C,C'),D^+)$. This chain complex is equal to the direct sum of the morphism spaces between individual $f$-joins of $C$ and $C'$ for all faces $f$, so we may compute each summand separately. Let us fix a face $f$. Consider a single $f$-join of $C$ between $s$ and $t$ and a single $f$-join of $C'$ between $s'$ and $t'$ where $s$ and $s'$ are sides of $f$, and $t$ and $t'$ are either sides of $f$ or points near the basepoint of $f$. We consider all cyclic orders of $s$, $t$, $s'$ and $t'$ on the boundary of $f$ separately, modulo interchanging $s$ and $t$ and interchanging $s'$ and $t'$. All possibilities are illustrated in Figure~\ref{fig:GensOfMor}: 
	\begin{itemize}
		\item The first row shows the various cases in which both $f$-joins are two-sided. If $s$, $t$, $s'$ and $t'$ are all distinct, we are in case (a) or (b). If two sides coincide (without loss of generality $s$ and $t'$), we are in case (c) or (d). If both sides coincide, there is only one case, namely (e). In all five cases, we allow the face to be open. 
		\item The second row shows all cases for which the $f$-join of $C$ is one-sided and the one for $C'$ is two-sided. In cases (f), (g) and (h), all sides are distinct. If two sides coincide, we are in case (i) or (j). 
		\item The third row shows all cases for which the $f$-join of $C$ is two-sided and the one for $C'$ is one-sided. Again, in the first three cases, all sides are distinct and in the last two cases, two sides agree. 
		\item The last row shows the cases in which both $f$-joins are one-sided. (p) and (q) are the two cases in which the two sides are distinct, in case (r), the two sides coincide. 
	\end{itemize}
	In each of these cases, we may now study the morphism space separately and compare it to the number of intersection points. As we can see, in each case there is either one or no intersection point between the two $f$-joins. We claim that this agrees with the dimension of the homology of the morphism space between the $f$-joins. Moreover, we claim that the generator of the homology of each morphism space is given by the resolution of the corresponding intersection point.

	Let us only verify these claims in the examples of the first row, assuming the face $f$ is closed; all other cases follow similarly.
	\begin{enumerate}
		\item[(a)] Over $\mathbb{F}_2[U_f]$, the kernel of $D^+$ is generated by
		$$
		(s\xrightarrow{p^{s}_{s'}} s')
		\oplus
		(t\xrightarrow{p^{t}_{t'}} t')
		\quad\text{ and }\quad
		(s\xrightarrow{p^{s}_{t'}} t')
		\oplus
		(t\xrightarrow{U_fp^{t}_{s'}} s').
		$$
		The former is equal to $D^+(t\xrightarrow{p^{t}_{s'}} s')$ and the latter is equal to $D^+(s\xrightarrow{p^{s}_{s'}} s')=D^+(t\xrightarrow{p^{t}_{t'}} t')$. 
		\item[(b)] Over $\mathbb{F}_2[U_f]$, the kernel of $D^+$ is generated by
		$$
		(s\xrightarrow{p^{s}_{s'}} s')
		\oplus
		(t\xrightarrow{p^{t}_{t'}} t')
		\quad\text{ and }\quad
		(s\xrightarrow{p^{s}_{t'}} t')
		\oplus
		( t\xrightarrow{p^{t}_{s'}} s').
		$$
		The latter is equal to $D^+(s\xrightarrow{p^{s}_{s'}} s')=D^+(t\xrightarrow{p^{t}_{t'}} t')$. However, the former does not lie in the image of $D^+$, only its products with positive powers of $U_f$ do. 
		\item[(c)] This case is the same as case (a), except that two sides coincide. This does not change the morphism space of algebra elements in $\mathcal{A}_f^+$.
		\item[(d)] This case is the same as case (b), except that two sides coincide. Again, this does not change the morphism space of algebra elements in $\mathcal{A}_f^+$; however, note that while the surviving generator is not in the image of $D^+$, it is a component of $D(s\xrightarrow{\iota_{s}=\iota_{t'}} t')$. 
		\item[(e)] This case is similar to the previous one; here, the surviving generator is a component of $D(s\xrightarrow{\iota_{s}=\iota_{t'}} t')$ and $D(t\xrightarrow{\iota_{t}=\iota_{s'}} s')$. 
	\end{enumerate}
	Let us now turn to upper intersection points. These should correspond to basis elements of $\Mor^\times(C,C')$, which can be written as the direct sum of $\Mor^\times(C.\iota_a,C'.\iota_a)$ over all arcs $a\in A$. 
	Each summand $\Mor^\times(C.\iota_a,C'.\iota_a)$ can in turn be written as
	$$
	\left\{\left.\left(C.\iota_{s_1(a)}\xrightarrow{\varphi_1} C'.\iota_{s_1(a)}, C.\iota_{s_2(a)}\xrightarrow{\varphi_2} C'.\iota_{s_2(a)}\right)\right\vert \varphi_2\circ P_a=P'_a \circ \varphi_1\right\}.
	$$
	Since $P_a$ and $P'_a$ are invertible, the projection onto the first component $\varphi_1$ is an isomorphism. In other words, $\Mor^\times(C.\iota_a,C'.\iota_a)$ is isomorphic to the vector space of linear maps 
	$$\varphi_1\co C.\iota_{s_1(a)}\rightarrow C'.\iota_{s_1(a)}.$$ 
	By Lemma~\ref{lem:ResolutionWellDefined}, a resolution $\varphi$ of an intersection point $x$ between the precurves $C$ and $C'$ on an arc $a$ is an element of $\Mor^\times(C.\iota_a,C'.\iota_a)$. Moreover, it corresponds to a standard basis element of the vector space of linear maps $\varphi_1\co C.\iota_{s_1(a)}\rightarrow C'.\iota_{s_1(a)}$. So for each $a\in A$, the resolutions of upper intersection points in $N(a)$ form a basis of $\Mor^\times(C.\iota_a,C'.\iota_a)$.
	
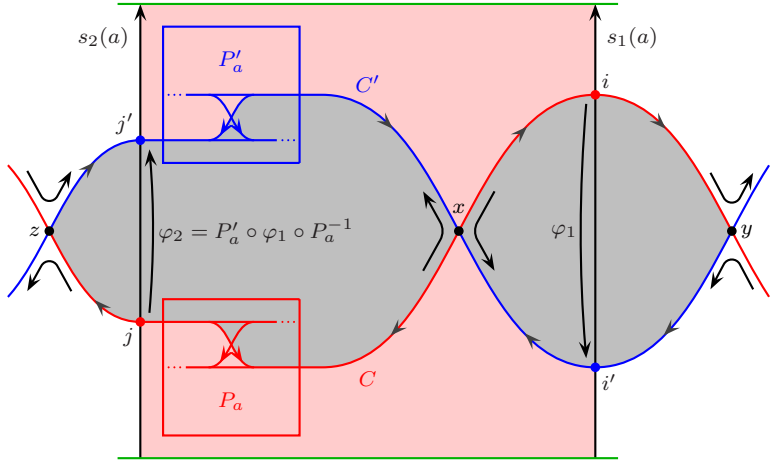
\begin{figure}[t]
	\centering
	\psset{unit=0.6}
	\begin{pspicture}(-13,-5.1)(9,5.1)
		\pscustom*[linecolor=lightred,linewidth=0pt]{
			\psline(-7,-5)(-7,5)
			\psline(3,5)(3,-5)
		}		
		\pscustom*[linecolor=lightgray]{%
			\psecurve(-9,3)(-3,-3)(3,3)(9,-3)(15,3)
			\psline(9,-3)(9,3)
			\psecurve(15,-3)(9,3)(3,-3)(-3,3)(-9,-3)
			\psline(-3,3)(-4.5,3)(-4.5,-3)(-3,-3)%
		}
		\pscustom*[linecolor=lightgray]{%
			\psecurve(-3.5,2)(-4.5,3)(-5.5,2)(-6.5,3)
			\psline(-5.5,2)(-7,2)(-7,-2)(-5.5,-2)%
			\psecurve(-6.5,-3)(-5.5,-2)(-4.5,-3)(-3.5,-2)
			\psline(-4.5,-3)(-3,-3)(-3,3)(-4.5,3)%
		}
		\pscustom*[linecolor=lightgray]{%
			\psecurve(-15,-2)(-11,2)(-7,-2)(-3,2)
			\psline(-7,-2)(-5.5,-2)(-5.5,2)(-7,2)
			\psecurve(-3,-2)(-7,2)(-11,-2)(-15,2)
			\psline(-11,-2)(-11,2)%
		}
		
		\psline*[linecolor=white,linewidth=0pt](10,4)(6,4)(6,-4)(10,-4)(10,4)
		\psline*[linecolor=white,linewidth=0pt](-9,3)(-12,3)(-12,-3)(-9,-3)(-9,3)
		
		\pscustom*[linecolor=white,linewidth=0pt]{
			\psecurve(-13,1.5)(-11,2)(-9,1.5)(-7,2)(-5,1.5)
			\psline(-7,2)(-7,3)(-11,3)(-11,2)
		}
		\pscustom*[linecolor=white,linewidth=0pt]{
			\psecurve(-13,-1.5)(-11,-2)(-9,-1.5)(-7,-2)(-5,-1.5)
			\psline(-7,-2)(-7,-3)(-11,-3)(-11,-2)
		}
		\pscustom*[linecolor=white,linewidth=0pt]{
			\psecurve(0,2.5)(3,3)(6,2.5)(9,3)(12,2.5)
			\psline(3,3)(3,4)(9,4)(9,3)
		}
		\pscustom*[linecolor=white,linewidth=0pt]{
			\psecurve(0,-2.5)(3,-3)(6,-2.5)(9,-3)(12,-2.5)
			\psline(3,-3)(3,-4)(9,-4)(9,-3)
		}
		
		\pscustom[linecolor=blue]{
			\psline(-6,3)(-3,3)
			\psecurve(-9,-3)(-3,3)(3,-3)(9,3)(15,-3)
		}
		\pscustom[linecolor=blue]{
			\psline(-4,2)(-7,2)		
			\psecurve(-3,-2)(-7,2)(-11,-2)(-15,2)
		}
		\pscustom[linecolor=red]{
			\psline(-6,-3)(-3,-3)
			\psecurve(-9,3)(-3,-3)(3,3)(9,-3)(15,3)
		}
		\pscustom[linecolor=red]{
			\psline(-4,-2)(-7,-2)		
			\psecurve(-3,2)(-7,-2)(-11,2)(-15,-2)
		}
		
		\psline[linecolor=blue,linestyle=dotted, dotsep=0.1](-6.1,3)(-6.4,3)
		\psline[linecolor=blue,linestyle=dotted, dotsep=0.1](-3.9,2)(-3.6,2)
		\psline[linecolor=red,linestyle=dotted, dotsep=0.1](-6.1,-3)(-6.4,-3)
		\psline[linecolor=red,linestyle=dotted, dotsep=0.1](-3.9,-2)(-3.6,-2)
		
		\psdot(-9,0)
		\psdot(0,0)
		\psdot(6,0)
		\uput{0.2}[0](6,0){$y$}
		\uput{0.4}[90](0,0){$x$}
		\uput{0.2}[180](-9,0){$z$}

		\pscircle*[linecolor=white](-17,8.5){10}
		\pscircle*[linecolor=white](-17,-8.5){10}
		\pscircle*[linecolor=white](12,-10){10}
		\pscircle*[linecolor=white](12,10){10}

		\rput(-5,2.5){\psset{linecolor=blue}
			\psecurve(-1.5,0.5)(-0.5,-0.5)(0.5,0.5)(1.5,-0.5)
			\psecurve(-1.5,-0.5)(-0.5,0.5)(0.5,-0.5)(1.5,0.5)
			\psline{->}(0.175,-0.28)(0.25,-0.4)
			\psline{->}(-0.175,-0.28)(-0.25,-0.4)
		}
		\rput(-5,-2.5){\psset{linecolor=red}
			\psecurve(-1.5,0.5)(-0.5,-0.5)(0.5,0.5)(1.5,-0.5)
			\psecurve(-1.5,-0.5)(-0.5,0.5)(0.5,-0.5)(1.5,0.5)
			\psline{->}(0.175,-0.28)(0.25,-0.4)
			\psline{->}(-0.175,-0.28)(-0.25,-0.4)
		}
		
		\psline[linecolor=blue](-3.5,1.5)(-6.5,1.5)(-6.5,4.5)(-3.5,4.5)(-3.5,1.5)
		\rput(-5,3.75){\textcolor{blue}{$P'_a$}}
		
		\psline[linecolor=red](-3.5,-1.5)(-6.5,-1.5)(-6.5,-4.5)(-3.5,-4.5)(-3.5,-1.5)
		\rput(-5,-3.75){\textcolor{red}{$P_a$}}
		
		\psline{->}(-7,-5)(-7,5)
		\rput[r](-7.25,4.25){$s_2(a)$}
		\psline{->}(3,-5)(3,5)
		\rput[l](3.25,4.25){$s_1(a)$}

		\psdot[linecolor=red](3,3)
		\psdot[linecolor=blue](3,-3)
		\psdot[linecolor=blue](-7,2)
		\psdot[linecolor=red](-7,-2)
		\uput{0.2}[45](3,3){$i$}
		\uput{0.2}[-45](3,-3){$i'$}
		\uput{0.2}[135](-7,2){$j'$}
		\uput{0.2}[-135](-7,-2){$j$}
		
		\rput{0}(0.3,0){\psline[linearc=0.4]{->}(1;61)(0,0)(1;-61)}
		\rput{180}(-0.3,0){\psline[linearc=0.4]{->}(1;61)(0,0)(1;-61)}
		\rput{0}(6,0.4){\psline[linearc=0.2]{<-}(1;62)(0,0)(1;118)}
		\rput{180}(6,-0.4){\psline[linearc=0.2]{<-}(1;62)(0,0)(1;118)}
		\rput{0}(-9,0.45){\psline[linearc=0.2]{<-}(1;61)(0,0)(1;119)}
		\rput{180}(-9,-0.45){\psline[linearc=0.2]{<-}(1;61)(0,0)(1;119)}
		
		\psline[linecolor=darkgray]{->}(1.45,2.25)(1.5,2.3)
		\psline[linecolor=darkgray]{->}(-1.45,-2.25)(-1.5,-2.3)
		\psline[linecolor=darkgray]{<-}(4.55,2.25)(4.5,2.3)
		\psline[linecolor=darkgray]{->}(-7.95,-1.57)(-8,-1.52)
		
		\psline[linecolor=darkgray]{<-}(-1.45,2.25)(-1.5,2.3)
		\psline[linecolor=darkgray]{<-}(1.45,-2.25)(1.5,-2.3)
		\psline[linecolor=darkgray]{->}(4.55,-2.25)(4.5,-2.3)
		\psline[linecolor=darkgray]{->}(-8.05,1.47)(-8,1.52)
		
		\psecurve{->}(3,4)(2.8,2.8)(2.8,-2.8)(3,-4)
		\psecurve{<-}(-7,3)(-6.8,1.8)(-6.8,-1.8)(-7,-3)
		\uput{0.4}[180](3,0){$\varphi_1$}
		\uput{0.4}[0](-7,0){$\varphi_2=P'_a\circ \varphi_1\circ P_a^{-1}$}
		
		\psline[linecolor=darkgreen](-7.5,5)(3.5,5)
		\psline[linecolor=darkgreen](-7.5,-5)(3.5,-5)
		
		\rput(-2,-3.25){$\textcolor{red}{C}$}
		\rput(-2,3.25){$\textcolor{blue}{C'}$}
		
		\end{pspicture}
	\caption{Illustration for the identification of bigons with the map $\beta$ in the proof of Theorem~\ref{thm:PairingMorLagrangianFH}.}\label{fig:BigonCounts}
\end{figure}

Finally, we need to identify the map $\beta$ with the differential on $\LagrangianFC(C,C')$. Given an upper intersection point $x$ and its resolution $(\varphi_1,\varphi_2)$, $\beta(\varphi_1,\varphi_2)$ is simply given by the differential of this morphism. Let us compute the components of this differential on $\varphi_1$ and $\varphi_2$ separately; for an illustration of the following argument, see Figure~\ref{fig:BigonCounts}. Suppose the morphism $\varphi_1$ goes from $\bullet(s_1(a),i)$ of the first precurve to $\bullet(s_1(a),i')$ of the second precurve. We claim that if the two $f$-joins starting at these two dots are disjoint, $D(\varphi_1)$ is null-homotopic; moreover, if they intersect at some point $y$, $D(\varphi_1)$ is equal to the resolution of $y$ and there is a single bigon from $x$ to $y$. These two claims can be easily verified for each case in which two sides coincide (ie cases (c), (d), (e), (i), (j), (n), (o) and (r)) separately. 

For $D(\varphi_2)$, we can argue similarly. However, we need to take the matrices $P_a$ and $P'_a$ into account. As we have seen in the proof of Lemma~\ref{lem:ResolutionWellDefined}, $\varphi_2$ has a non-zero component from $\bullet(s_2(a),j)$ of the first precurve to $\bullet(s_2(a),j')$ of the second precurve iff there are an odd number of paths from $\bullet(s_2(a),j)$ to $\bullet(s_2(a),j')$ via $x$. If the two $f$-joins starting at $\bullet(s_2(a),j)$ and $\bullet(s_2(a),j')$, respectively, intersect in a point $z$, then the number of bigons from $x$ to $z$ agrees with the number of such paths, noting that the boundary orientation of each bigon is opposite to the orientation of the crossover arrows. We may now argue as for $D(\varphi_1)$. 
\end{proof}

\subsection{A formula for computing the homology of morphism spaces between curves}\label{subsec:formulaMor}

If $C$ and $C'$ are two precurves, then by Theorem~\ref{thm:EverythingIsLoopTypeUpToLocalSystems}, there are two collections of curves $L$ and $L'$ such that $C$ and $\Pi_i(L)$ as well as $C'$ and $\Pi_i(L')$ are homotopic, respectively. Thus, 
$\Mor(C,C')$ and $\Mor(L,L'):=\Mor(\Pi_i(L),\Pi_i(L'))$ are chain homotopic. If we put $\Pi_i(L)$ and $\Pi_i(L')$ into pairing position, Theorem~\ref{thm:PairingMorLagrangianFH} says that the homology of $\Mor(L,L')$ is graded isomorphic to $\LagrangianFH(L,L'):=\LagrangianFH(\Pi_i(L),\Pi_i(L'))$. The main result of this section says that we can actually compute $\LagrangianFH(L,L')$ without putting the curves into pairing position.

\begin{definition}
	Let $(L,L')$ be a pair of $\delta$-graded curves $L=(\gamma,X)$ and $L'=(\gamma',X')$ on a marked surface with arc system $(S,M,A)$ with $\dim X=:n$ and $\dim X'=:n'$. Assume that $\gamma$ and $\gamma'$ intersect minimally. Let $\mathbb{F}_2\langle\gamma\cap\gamma'\rangle$ denote the vector space over $\mathbb{F}_2$ spanned by intersection points between $\gamma$ and $\gamma'$. 
	Each intersection point can be $\delta$-graded in exactly the same way as intersection points in $\LagrangianFC(C,C')$. 
	If $\gamma$ and $\gamma'$ are parallel, let $\delta(\gamma,\gamma')$ be the unique real number one needs to add to the $\delta$-grading of each intersection point of $\gamma$ with arcs in $A$ such that $\gamma$ and $\gamma'$ agree as $\delta$-graded curves.
	For any non-negative integer $m$, let $V_\delta(m)$ be an $m$-dimensional vector space in $\delta$-grading $\delta\in\mathbb{R}$. 
\end{definition}

\begin{theorem}\label{thm:PairingFormula}
	With the notation from above, \(\LagrangianFH(L,L')\) is graded isomorphic to 
	\begin{equation}\label{eqn:MorSpacesNonparallel}
	V_0(n\cdot n')\otimes \mathbb{F}_2\langle\gamma\cap\gamma'\rangle,
	\end{equation}
	unless \(\gamma\) and \(\gamma'\) are parallel. If they are parallel, let us assume without loss of generality that their orientations agree. Then, \(\LagrangianFH(L,L')\) is graded isomorphic to 
	\begin{equation}\label{eqn:MorSpacesParallel}
	\Big(V_0(n\cdot n')\otimes\mathbb{F}_2\langle\gamma\cap\gamma'\rangle\Big)
	\oplus
	\Big(\!\big(V_0(1)\oplus V_{1}(1)\big)\otimes V_{\delta(\gamma,\gamma')}\left(\dim\left(\ker\left((X^{-1})^t\otimes X'-\id\right)\right)\right)\!\!\Big).
	\end{equation}
\end{theorem}

\begin{figure}[t]\centering
	{
		\psset{unit=0.3}
		\begin{pspicture}(-18,-3.5)(18,3.5)
		
		\rput(-12,0){
			\pscustom*[linecolor=lightgray]{
				\pscurve(-3,0)(-1,1.5)(1,1.5)(3,0)
				\pscurve[liftpen=2](3,0)(1,-1.5)(-1,-1.5)(-3,0)
			}
			\pscurve[linecolor=red](-3,0)(-1,1.5)(1,1.5)(3,0)
			\pscurve[linecolor=blue](-3,0)(-1,-1.5)(1,-1.5)(3,0)
			\psline[linecolor=darkgray]{->}(-0.3,-1.66)(-0.4,-1.66)
			\psline[linecolor=darkgray]{<-}(0.4,1.66)(0.3,1.66)
			\psline[linecolor=gray]{->}(-2.5,0)(2.5,0)
		}
		
		\rput(-6,0){
			\pscustom*[linecolor=lightgray]{
				\pscurve(-3,0)(-1,1.5)(1,1.5)(3,0)
				\pscurve[liftpen=2](3,0)(1,-1.5)(-1,-1.5)(-3,0)
			}
			\pscurve[linecolor=blue](-3,0)(-1,1.5)(1,1.5)(3,0)
			\pscurve[linecolor=red](-3,0)(-1,-1.5)(1,-1.5)(3,0)
			\psline[linecolor=darkgray]{->}(-0.3,1.66)(-0.4,1.66)
			\psline[linecolor=darkgray]{<-}(0.4,-1.66)(0.3,-1.66)
			\psline[linecolor=gray]{<-}(-2.5,0)(2.5,0)
		}
		
		\rput(0,0){
			\pscustom*[linecolor=lightgray]{
				\pscurve(-3,0)(-1,1.5)(1,1.5)(3,0)
				\pscurve[liftpen=2](3,0)(1,-1.5)(-1,-1.5)(-3,0)
			}
			\pscurve[linecolor=red](-3,0)(-1,1.5)(1,1.5)(3,0)
			\pscurve[linecolor=blue](-3,0)(-1,-1.5)(1,-1.5)(3,0)
			\psline[linecolor=darkgray]{->}(-0.3,-1.66)(-0.4,-1.66)
			\psline[linecolor=darkgray]{<-}(0.4,1.66)(0.3,1.66)
			\psline[linecolor=gray]{->}(-2.5,0)(2.5,0)
		}
		
		\rput(6,0){
			\pscustom*[linecolor=lightgray]{
				\pscurve(-3,0)(-1,1.5)(1,1.5)(3,0)
				\pscurve[liftpen=2](3,0)(1,-1.5)(-1,-1.5)(-3,0)
			}
			\pscurve[linecolor=blue](-3,0)(-1,1.5)(1,1.5)(3,0)
			\pscurve[linecolor=red](-3,0)(-1,-1.5)(1,-1.5)(3,0)
			\psline[linecolor=darkgray]{->}(-0.3,1.66)(-0.4,1.66)
			\psline[linecolor=darkgray]{<-}(0.4,-1.66)(0.3,-1.66)
			\psline[linecolor=gray]{<-}(-2.5,0)(2.5,0)
		}
		
		\rput(12,0){
			\pscustom*[linecolor=lightgray]{
				\pscurve(-3,0)(-1,1.5)(1,1.5)(3,0)
				\pscurve[liftpen=2](3,0)(1,-1.5)(-1,-1.5)(-3,0)
			}
			\pscurve[linecolor=red](-3,0)(-1,1.5)(1,1.5)(3,0)
			\pscurve[linecolor=blue](-3,0)(-1,-1.5)(1,-1.5)(3,0)
			\psline[linecolor=darkgray]{->}(-0.3,-1.66)(-0.4,-1.66)
			\psline[linecolor=darkgray]{<-}(0.4,1.66)(0.3,1.66)
			\psline[linecolor=gray]{->}(-2.5,0)(2.5,0)
		}
		
		\rput(12,0){
			\psline*[linecolor=lightgray](3,0)(4,-1)(4,1)(3,0)
			\psline[linecolor=red](3,0)(4,-1)
			\psline[linecolor=blue](3,0)(4,1)	
		}
		\rput(-12,0){
			\psline*[linecolor=lightgray](-3,0)(-4,-1)(-4,1)(-3,0)
			\psline[linecolor=blue](-3,0)(-4,1)
			\psline[linecolor=red](-3,0)(-4,-1)
		}
		
		\psdots(-15,0)(-9,0)(-3,0)(3,0)(9,0)(15,0)
		
		\rput(0,3){$\textcolor{red}{L}$}
		\rput(0,-3){$\textcolor{blue}{L'}$}
		\rput(-16,0){$\dots$}
		\rput(16,0){$\dots$}
		
		\end{pspicture}
	}
	\caption{A ``zig-zag'' chain of bigons.}\label{fig:BigonChain}
\end{figure}
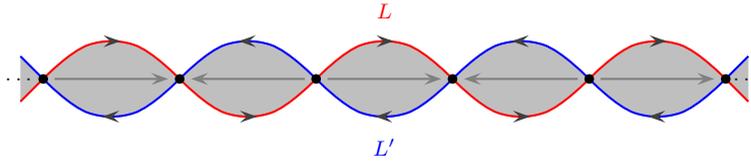

\begin{proof}
	Let us start by considering the case when the local systems of both curves are trivial, ie $n=n'=1$, so $X=X'=\id$. Then, any upper intersection point is the source of at most two bigons and each lower intersection point is the target of at most two bigons. Thus, if we start at any intersection point and follow the bigons in either direction, we obtain a ``zig-zag'' chain of intersection points which are connected by bigons, see Figure~\ref{fig:BigonChain} for an illustration. If $L$ and $L'$ are not parallel, $\LagrangianFC(L,L')$ decomposes into a direct sum of finitely many linear chains, which look like
	$$
	\begin{tikzcd}
	\bullet
	\arrow[leftarrow]{r}
	&
	\bullet
	&
	\cdots
	\arrow[leftarrow]{l}
	\arrow[leftarrow]{r}
	&
	\bullet
	&
	\bullet
	\arrow[leftarrow]{l}
	\end{tikzcd},
	$$
	$$
	\begin{tikzcd}
	\bullet
	\arrow{r}
	&
	\bullet
	&
	\cdots
	\arrow{l}
	\arrow{r}
	&
	\bullet
	&
	\bullet
	\arrow{l}
	\end{tikzcd}
	$$
	or
	$$
	\begin{tikzcd}
	\bullet
	\arrow{r}
	&
	\bullet
	&
	\cdots
	\arrow{l}
	&
	\bullet
	\arrow{l}
	\arrow{r}
	&
	\bullet
	\end{tikzcd}.
	$$
	In the first two cases, the number of intersection points is odd with an even number of upper, respectively lower intersection points. In the third case, the number of intersection points is even. The homology in the first two cases is 1 and in the third 0, which can be seen by cancelling one arrow at a time. This corresponds to sliding one curve over the other to remove one bigon at a time, until there is only one intersection point left or none. 
	
	If $L$ and $L'$ are parallel, there is also a cyclic ``zig-zag'' chain:
	$$
	\begin{tikzcd}
	\bullet
	\arrow{r}
	&
	\bullet
	&
	\bullet
	\arrow{l}
	\arrow{r}
	&
	\cdots
	\arrow{r}
	&
	\bullet
	&
	\bullet
	\arrow{r}
	\arrow{l}
	&
	\bullet
	\arrow[leftarrow,rounded corners,%
	to path={ -- ([xshift=2ex]\tikztostart.east)
		|-  ([xshift=-2ex,yshift=-4ex]\tikztostart.east)
		-| ([xshift=-2ex,yshift=-4ex]\tikztotarget.west)
		-|
		([xshift=-2ex]\tikztotarget.west)
		-- (\tikztotarget)}]{llllll}
	\end{tikzcd}\bigskip
	$$
	In this case, we can apply the same procedure to obtain a cyclic chain with just two intersection points:
	$$
	\begin{tikzcd}
	\bullet
	\arrow[bend left]{r}
	\arrow[bend right]{r}
	&
	\bullet
	\end{tikzcd}
	=\left(\bullet\oplus\bullet\right).
	$$
	So in this case, the homology is 2-dimensional. Geometrically, this corresponds to removing all bigons except the last two. In order to compute the minimal intersection  number of $\gamma$ and $\gamma'$, however, we need to remove these two bigon as well. So these additional two generators make up the extra term in formula~\eqref{eqn:MorSpacesParallel}. 
	
	The general case with potentially non-trivial local systems is shown similarly, by assuming that the underlying curves $\gamma$ and $\gamma'$, considered as precurves, are in pairing position. Each intersection point $x\in\LagrangianFC(\gamma,\gamma')$ corresponds to $n\cdot n'$ intersection points of $\Pi_i(L)$ and $\Pi_i(L')$.
	Let us number the copies of $\gamma$ and $\gamma'$ in $\Pi_i(L)$ and $\Pi_i(L')$ such that we can index the intersection points corresponding to $x$ by pairs $(i,i')$. Let us write $(i,i')\otimes x$ for the generator indexed by $(i,i')$. Thus, we can identify $\LagrangianFC(L,L')$ with $V_0(n\cdot n')\otimes \LagrangianFC(\gamma, \gamma')$.
	
	Suppose the precurves $\Pi_i(L)$ and $\Pi_i(L')$ do not have any generators on a common arc. In this case, they are certainly not parallel and there are also no bigons; so the theorem follows immediately. If $\Pi_i(L)$ and $\Pi_i(L')$ have generators on a common arc, let us put the matrices $X$ and $X'$ on such an arc $a$. Let us also assume for the moment, that $\gamma$ and $\gamma'$ pass through this arc from its right to its left, such that the non-identity block of $P_a$ is $X$ and the one of $P'_a$ is $X'$.	For an illustration, see Figure~\ref{fig:ParallelLoopsExtraMorphisms}. Then, the local systems only impact the bigons which cross the side $s_2(a)$. More precisely, suppose $x,z\in \LagrangianFC(\gamma, \gamma')$ and there is a bigon from $x$ to $z$. If this bigon does not cross $s_2(a)$, the bigon count from $V_0(n\cdot n')\otimes x$ to $V_0(n\cdot n')\otimes z$ is given by the identity matrix. Otherwise, the bigon count is given by $(X^{-1})^t\otimes X'$, because the bigon count from $(i,i')\otimes x$ to $(j,j')\otimes z$ is given by $(X^{-1})_{ij}\cdot X'_{j'i'}$ as in the proof of Theorem~\ref{thm:PairingMorLagrangianFH}. 
	
	Since $X$ and $X'$ have full rank, $(X^{-1})^t\otimes X'$ has full rank, too. So for linear chains, we can now argue as before by doing cancellations which remove all bigons corresponding to a single bigon between $\gamma$ and $\gamma'$ at a time. By applying this procedure to cyclic chains, we now arrive at
	$$
	\begin{tikzcd}
	\bullet
	\arrow[bend left]{r}{\id}
	\arrow[bend right,swap]{r}{(X^{-1})^t\otimes X'}
	&
	\bullet
	\end{tikzcd}.
	$$
	The homology of this complex is given by the direct sum of a copy of the kernel of $(X^{-1})^t\otimes X'-\id$ and another such copy shifted up by 1 in $\delta$-grading. 
	
	Finally, by definition, changing the orientation of $\gamma$ or $\gamma'$ only inverts and transposes the local systems. So for non-parallel curves, the formula does not change. If $\gamma$ and $\gamma'$ are parallel, we are assuming that they are oriented in the same direction. So if they pass through the arc $a$ from left to right, the non-identity block of $P_a$ is $X^t$ and the one of $P'_a$ is $X^{\prime t}$. Now observe that
	$X^{-1}\otimes X^{\prime t}-\id$ has the same rank as $(X^{-1})^t\otimes X'-\id$.
\end{proof}

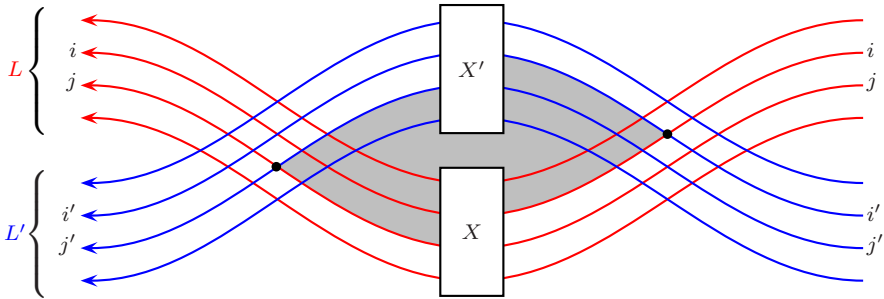
\begin{figure}[t]\centering
	\psset{unit=0.43}
	\begin{pspicture}(-15,-5.5)(15,5.5)
	
	\pscustom*[linecolor=lightgray]{
		\psecurve(-12,3)(0,-2)(12,3)(24,-2)
		\psline(12,3)(12,-2)
		\psecurve(24,3)(12,-2)(0,3)(-12,-2)
		\psline(0,3)(0,-2)
	}
	\psframe*[linecolor=white](6,-10)(20,10)
	
	\pscustom*[linecolor=lightgray]{
		\psecurve(12,2)(0,-3)(-12,2)(-24,-3)
		\psline(-12,2)(-12,-3)
		\psecurve(-24,2)(-12,-3)(0,2)(12,-3)
		\psline(0,2)(0,-3)
	}
	\psframe*[linecolor=white](-6,-10)(-20,10)
	
	\psframe*[linecolor=lightgray](-1,1.5)(1,-1.5)

	\rput(-11.2,2.5){$\left\{\textcolor{white}{\rule[-0.8cm]{1.8cm}{1.8cm}}\right.$}
	\rput{0}(-14,2.5){$\textcolor{red}{L}$}
	
	\rput(-11.2,-2.5){$\left\{\textcolor{white}{\rule[-0.8cm]{1.8cm}{1.8cm}}\right.$}
	\rput{0}(-14,-2.5){$\textcolor{blue}{L'}$}
	
	\psecurve[linecolor=red]{<-}(-24,-1)(-12,4)(0,-1)(12,4)(24,-1)
	\psecurve[linecolor=red]{<-}(-24,-2)(-12,3)(0,-2)(12,3)(24,-2)
	\psecurve[linecolor=red]{<-}(-24,-3)(-12,2)(0,-3)(12,2)(24,-3)
	\psecurve[linecolor=red]{<-}(-24,-4)(-12,1)(0,-4)(12,1)(24,-4)
	
	\psecurve[linecolor=blue]{<-}(-24,1)(-12,-4)(0,1)(12,-4)(24,1)
	\psecurve[linecolor=blue]{<-}(-24,2)(-12,-3)(0,2)(12,-3)(24,2)
	\psecurve[linecolor=blue]{<-}(-24,3)(-12,-2)(0,3)(12,-2)(24,3)
	\psecurve[linecolor=blue]{<-}(-24,4)(-12,-1)(0,4)(12,-1)(24,4)
	
	\psframe[fillcolor=white,fillstyle=solid](-1,0.5)(1,4.5)
	\psframe[fillcolor=white,fillstyle=solid](-1,-0.5)(1,-4.5)
	\rput(0,2.5){$X'$}
	\rput(0,-2.5){$X$}
	
	\rput[r](-12,3.1){$i\,$}
	\rput[r](-12,2.1){$j\,$}
	\rput[l](12,3.1){$\,i$}
	\rput[l](12,2.1){$\,j$}
	
	\rput[r](-12,-2.9){$j'\,$}
	\rput[r](-12,-1.9){$i'\,$}
	\rput[l](12,-2.9){$\,j'$}
	\rput[l](12,-1.9){$\,i'$}
	
	\psdots(6,0.5)(-6,-0.5)
	
	
	
	\end{pspicture}
	\caption{An illustration of the identification of bigons between parallel curves and the matrix $(X^{-1})^t\otimes X'$. Up to conventions, this is the same as~\protect\cite[Figure~43]{HRW}.}\label{fig:ParallelLoopsExtraMorphisms}
\end{figure}

\subsection{Classification of curved complexes}\label{subsec:complete_classification}

\begin{theorem}\label{thm:CompleteClassification}
	Let \(L=\{(\gamma_i,X_i)\}_{i\in I}\) and \(L'=\{(\gamma'_{i'},X'_{i'})\}_{i'\in I'}\) be two collections of loops. Then \(\Pi_i(L)\) is homotopic to~\(\Pi_i(L')\) iff there is a bijection \(\iota\co I\rightarrow I'\) such that \(\gamma_i\)~is homotopic to~\(\gamma'_{\iota(i)}\) and \(X_i\)~is similar to~\(X'_{\iota(i)}\). The same holds if we replace \(\Pi_i\) by \(\Pi\).
\end{theorem}
 \begin{remark}
 		We expect the same theorem to hold for general curves ie compact \emph{and non-compact} ones. The arguments that we use in the proof of Theorem~\ref{thm:CompleteClassification} rely on different growth properties of the dimensions of the two summands in \eqref{eqn:MorSpacesParallel} from Theorem~\ref{thm:PairingFormula} under pairing with particular test curves. However, the second term simply vanishes for non-compact curves and thus, a separate argument would be needed in this case.
 \end{remark}
 
 \begin{corollary}\label{cor:AddingBasepointFunctorIsFaithfulUpToHom}
 	Consider a marked surface \((S,\emptyset)\) with an arc system \(A\). Let \(M\) be a set of points on \(\partial S\), such that every face \(f\in (S,\emptyset,A)\) contains at most one point in \(M\). Then, two objects in \(\CC(S,\emptyset,A)\) are homotopic iff their images under the induced functor
 	\[\CC(S,\emptyset,A)\rightarrow\CC(S,M,A)\]
 	are homotopic.
 \end{corollary}
 \begin{proof}
 	The induced functor is a functor of dg categories, as it is induced by the quotient map 
 	\[\mathcal{A}(S,\emptyset,A)\rightarrow\mathcal{A}(S,\emptyset,A)/\{p_{m}=0\mid m\in M\}=\mathcal{A}(S,M,A),\]
 	where $p_{m}$ is the algebra element corresponding to the boundary component in which $m\in M$ lies. Thus, the images of two homotopic objects are homotopic. So we may assume that the two objects are direct sums of loops with local systems. The images of such objects are represented by the same loops with local systems, because adding a single basepoint to a face $f$ without any basepoints has the effect of removing exactly one of the two arrows that correspond to an $f$-join under $\Pi_i$. By Theorem~\ref{thm:CompleteClassification}, these loops with local systems represent homotopic objects in both $\CC(S,\emptyset,A)$ and $\CC(S,M,A)$ iff the curves are the same and the local systems are equivalent.
 \end{proof}

We now turn to the proof of Theorem~\ref{thm:CompleteClassification}.

\begin{definition}\label{def:companionmatrix}
	Given a polynomial 
	\[f=x^n+\sum_{i=0}^{n-1} a_i x^i\in\mathbb{F}_2[x],\]
	define the \textbf{companion matrix $X_f$ of $f$} to be the matrix
	\[
	X_f:=\left(
	\begin{matrix}
	0 & \phantom{a_1} & \phantom{a_1} & a_0\\
	1 & \ddots & \phantom{a_1} & a_1\\
	\phantom{a_1}& \ddots & 0 & \vdots\\
	\phantom{a_1} & \phantom{a_1} & 1 & a_{n-1}
	\end{matrix}\right)\in\GL_n(\mathbb{F}_2).
	\]
	Note that $X_f$ is invertible iff $a_0\neq 0$. Also, the minimal polynomial of $X_f$ is $f$, so that for any polynomial $g\in\mathbb{F}_2[x]$, $g(X_f)=0$ iff $f\vert g$. A diagonal block matrix of the form
	\[
	\left(
	\begin{matrix}
	X_{f_1} & \phantom{X_{f_1}}& \phantom{X_{f_1}}\\
	\phantom{X_{f_1}}& \ddots &\phantom{X_{f_1}}\\
	\phantom{X_{f_1}}& \phantom{X_{f_1}} & X_{f_r}
	\end{matrix}\right), \text{ where $f_1,\dots,f_r\in\mathbb{F}_2[x],$}
	\]
	is in Frobenius normal form if $f_{i+1}\vert f_{i}$ for all $i=1,\dots,r-1$. 
\end{definition}

\begin{theorem}
	Every matrix is similar to a matrix in Frobenius normal form. Two matrices are similar iff they have the same Frobenius normal form. 
\end{theorem}
\begin{proof}
	This is standard linear algebra.
\end{proof}

\begin{lemma}\label{lem:reformulationofparalleldimensioncount}
	Given a polynomial \(f\in \mathbb{F}_2[x]\), and \(X\in\GL_m(\mathbb{F}_2)\) for some integer \(m\),
	\[\dim(\ker((X^{-1})^t\otimes X_f-\id))=\dim(\ker(f(X))),\]
	where \(X_f\) is the companion matrix of~\(f\) from Definition~\ref{def:companionmatrix}.
\end{lemma}
\begin{proof}
	This follows from the same arguments as~\cite[Proposition~36]{HRW}. Let $n=\deg f$. By performing row and column operations, we can bring $((X^{-1})^t\otimes X_f-\id)$ into block diagonal form with the first block of dimension $(n-1)m$ equal to the identity matrix and the second block of dimension $m$ equal to the expression 
	\begin{equation}\label{eqn:reformulationofparalleldimensioncount}
	\id+a_{n-1}(X^{-1})^t\cdots +a_0((X^{-1})^t)^{n},
	\end{equation}
	where the $a_i$ are the coefficients of $f$ as in Definition~\ref{def:companionmatrix}. The kernel of the matrix \eqref{eqn:reformulationofparalleldimensioncount} has the same dimension as the kernel of $((X^{-1})^t\otimes X_f-\id)$. Now multiply \eqref{eqn:reformulationofparalleldimensioncount} by $(X^t)^n$ to obtain $f(X^t)$. Transposing a square matrix does not change the dimension of its kernel, so we are done.
\end{proof}

\begin{proof}[of Theorem~\ref{thm:CompleteClassification}]
	By Corollary~\ref{cor:SplittingCatsForEquivalenceComplexes}, it suffices to show the first part of the theorem. The if-direction is clear, so let us assume that $\Pi_i(L)$ and $\Pi_i(L')$ are homotopic. For every $i'\in I'$ such that there is no $i\in I$ with $\gamma'_{i'}=\gamma_i$, add a ``formal'' curve to $L$ which is supported on $\gamma'_{i'}$ and which has a 0-dimensional local system. Do the same for~$L'$. Note that this does not change $\LagrangianFH(L,L'')$ nor $\LagrangianFH(L',L'')$ for any curve $L''$ with local system. So by allowing 0-dimensional local systems, we may assume without loss of generality that there exists a bijection $\iota\co I\rightarrow I'$ such that $\gamma_i=\gamma'_{\iota(i)}$. Let us assume for simplicity that $I=I'$ and $\iota$ is the identity. 
	
	Let us fix some $j\in I$ and let $p$ and $p'$ be the minimal polynomials of the matrices $X_j$ and $X'_{j}$, respectively. Then for $N>\deg p+\deg p'$, let 
	\[
	f_N(x):=(x^{N-\deg p-\deg p'}+1) \cdot p(x)\cdot p'(x).
	\]
	Note that since $f_N$ has a non-zero constant term, its companion matrix is invertible. So $L''=(\gamma_j, X_{f_N})$ is a well-defined curve with local system, which we can use as a ``test curve''. By Theorem~\ref{thm:PairingFormula} and Lemma~\ref{lem:reformulationofparalleldimensioncount}, the dimensions of the morphism spaces from $L$ and $L'$ to $L''$ are equal to
	\[
	\left(\sum_{i\in I}\#\gamma_i\cap\gamma_j\cdot \dim X_{i}\right)\cdot N+2\dim X_j
	\]
	and 
	\[
	\left(\sum_{i\in I}\#\gamma_i\cap\gamma_j\cdot \dim X'_{i}\right) \cdot N+2\dim X'_{j},
	\]
	respectively. By considering these two terms as linear functions in $N$, we see that they coincide iff their coefficients coincide. Hence, in particular $\dim X_j=\dim X'_{j}$.
	
	So it only remains to show that $X_j$ and $X'_{j}$ are similar for all $j\in I$. 	
	For this, let us fix some $j\in I$ and assume that both $X_j$ and $X'_{j}$ are in Frobenius normal form defined by polynomials $f_1,\dots, f_r$ and $f'_1,\dots, f'_{r'}$ such that $f_{l+1}\vert f_{l}$ and $f'_{l'+1}\vert f'_{l'}$ for all $l=1,\dots,r-1$ and $l'=1,\dots,r'-1$.
	Then $X_j$ and $X'_{j}$ are similar iff $r=r'$ and $f_l=f'_l$ for all $l=1,\dots,r$. Suppose this is not the case. Then there exists some minimal $m\leq\min(r,r')$ such that $f_m\neq f'_m$. Assume without loss of generality that $f'_m\not\vert f_m$ and consider the ``test curve'' given by $(\gamma_j,X_{f_m})$. Let $N=\deg f_m$. Then the spaces of morphisms from $L$ and $L'$ to $(\gamma_j,X_{f_m})$ have dimension
	\begin{equation}\label{eqn:pairingdim3}
	\left(\sum_{i\in I}\#\gamma_i\cap\gamma_{j}\cdot \dim X_{i}\right)\cdot N+
	2\left(
	\sum_{l=1}^{m-1}\dim\ker(f_m(X_{f_l}))+
	\sum_{l=m}^{r}\dim \ker(f_m(X_{f_l}))
	\right)
	\end{equation}
	and 
	\begin{equation}\label{eqn:pairingdim4}
	\left(\sum_{i\in I}\#\gamma_{i}\cap\gamma_{j}\cdot \dim X'_{i}\right)\cdot N+
	2\left(
	\sum_{l=1}^{m-1}\dim\ker(f_m(X_{f'_l}))+
	\sum_{l=m}^{r'}\dim \ker(f_m(X_{f'_l}))
	\right),
	\end{equation}
	respectively. The first sums coincide by the results that we have already established. The second sums agree by minimality of $m$. The summands in the third sum of \eqref{eqn:pairingdim3} are equal to $\dim X_{f_l}=\deg f_l$, since $f_l\vert f_m$ for $l>m$. Now,
	\[\sum_{l=1}^{r}\dim X_{f_l}=\dim X_j=\dim X'_{j}=\sum_{l=1}^{r'}\dim X_{f'_l}.\]
	By minimality of $m$, we obtain
	\[\sum_{l=m}^{r}\dim X_{f_l}=\sum_{l=m}^{r'}\dim X_{f'_l}.\]
	Hence, the third sum in~\eqref{eqn:pairingdim4} is at most as large as the third sum in~\eqref{eqn:pairingdim3}. However, $f'_m\not\vert f_m$, hence $f_m(X_{f'_m})\neq 0$, so
	\[\dim \ker(f_m(X_{f'_m}))<\dim X_{f'_m}.\]
	Contradiction.
\end{proof}

%% file: sections/PeculiarModulesAsCurves.tex

\section{Peculiar modules as collections of immersed curves}\label{sec:glueingrevisited}
 
\begin{definition}\label{def:ImmersedCurveInvariantspqMod}
 	By Example~\ref{exa:pqModSpecialCaseofCC}, the category of peculiar modules is a special case of a category of curved complexes over a marked surface with arc system. Therefore, if $T$ is a 4-ended tangle in a homology 3-ball $M$ with spherical boundary, Theorem~\ref{thm:EverythingIsLoopTypeUpToLocalSystems} implies that there is a collection of immersed curves corresponding to the peculiar module $\CFTd(T)$. We denote such a collection of immersed curves by $L_T$. 
 	
 	If the tangle $T$ is oriented, the Alexander gradings on $\CFTd(T)$ give rise to an Alexander grading on $L_T$, which is defined similar to the $\delta$-grading. More precisely, the \textbf{Alexander grading of a curve} is an Alexander grading $A$ of intersection points of the curve with the four arcs parametrizing the 4-punctured sphere, such that if $x$ and $y$ are joined by a curve segment corresponding to an algebra element $p^s_t\in\Ad$ from $x$ to $y$, $A(y)-A(x)+A(p^s_t)=0$. The Alexander grading on the Lagrangian intersection Floer homology of two bigraded curves is defined in exactly the same way as the $\delta$-grading, namely via the Alexander grading of resolutions of intersection points, see Definitions~\ref{def:resolution} and~\ref{def:gradingVIAresolution}.
\end{definition}

As a special case of Theorem~\ref{thm:CompleteClassification}, we obtain:

\begin{theorem}\label{thm:ImmersedCurveInvariants}
	\(L_T\) is a well-defined tangle invariant up to homotopy of the underlying immersed curves and similarity of the local systems.\hfill\qedsymbol
\end{theorem}
 
Just as for peculiar modules, we can study how immersed curves behave under mirroring and reversing Alexander gradings. 

\begin{definition}
	Given a collection $L$ of bigraded curves, let $\rr(L)$ denote the collection of curves obtained from $L$ by reversing all Alexander gradings. Note that for this to be a well-defined Alexander grading, we also need to reverse the Alexander grading of $\Ad$. Moreover, let $\m(L)$ denote the collection of curves defined as follows: the underlying oriented curves of $\m(L)$ are obtained as the image of the underlying oriented curves of $L$ under the orientation reversing automorphism of the 4-punctured sphere which preserves the four parametrizing arcs pointwise and exchanges the front and back. The bigrading of $\m(L)$ is equal to the reversed bigrading of $L$. Note that this is well-defined over the same bigraded algebra $\Ad$. Finally, the local systems of $L$ and $\m(L)$ are the same. 
	
	As for the corresponding operations on peculiar modules and tangles from Definition~\ref{def:reversedmirror}, we write $\mr(L)$ for $\m(\rr(L))=\rr(\m(L))$.
\end{definition}

\begin{proposition}
	For any collection \(L\) of bigraded curves, \(\Pi(\rr(L))=\rr(\Pi(L))\) and \(\Pi(\m(L))=\m(\Pi(L))\). Moreover, if \(T\) is an oriented 4-ended tangle, then \(\rr(L_T)=L_{\rr(T)}\) and \(\m(L_T)=L_{\m(T)}\). 
\end{proposition}

\begin{proof}
	The second part follows from the first in conjunction with Proposition~\ref{prop:reversedmirror}. The first part follows from the definitions of the operations. To identify the local systems of $\Pi(\m(L))$ and $\m(\Pi(L))$, note that the underlying curves of $\m(L)$ pass through arcs in the opposite direction to those of $L$. This corresponds to transposing the local systems, ie reversing the orientation of any crossover arrows and reversing the order of crossings and crossover arrows on arc neighbourhoods. 
\end{proof}

\subsection{Peculiar modules from nice diagrams}

\begin{definition}
 	Given a tangle $T$, pick two basepoints $p_i$ and $q_j$ for some $i,j\in\{1,2,3,4\}$ as well as one of $z_j$ and $w_j$ for each closed component of $T$. 
 	A peculiar Heegaard diagram for $T$ is \textbf{nice} with respect to this choice of special basepoints, if all regions except those containing these basepoints are bigons or squares.
 \end{definition}
 \begin{theorem}
 	Every 4-ended tangle \(T\) has a nice peculiar Heegaard diagram with respect to any choice of basepoints. 
 \end{theorem}
 \begin{proof}
 	This is an application of Sarkar and Wang's main result in~\cite{SarkarWang}, where they describe an algorithm for niceifying any pointed Heegaard diagram with one basepoint in each component of the Heegaard surface minus the $\alpha$-circles.
 \end{proof}
 
 \begin{corollary}\label{cor:PeculiarModulesFromNiceDiagrams}
 	Peculiar modules for 4-ended tangles can be computed combinatorially.
 \end{corollary}
 \begin{proof}
 	We can use a nice peculiar Heegaard diagram for a tangle to compute the peculiar module $\CFTd(T)$. Since all regions away from the special basepoints are bigons or squares, the calculation of all domains that miss those basepoints is purely combinatorial. The complex $\overline{\CFTd(T)}$ corresponding to those domains is exactly the image of $\CFTd(T)$ under the functor induced by the quotient map $\Ad\rightarrow\Ad/(p_i=0=q_j)$. So by Corollary~\ref{cor:AddingBasepointFunctorIsFaithfulUpToHom}, we can recover the homotopy type of $\CFTd(T)$ from $\overline{\CFTd(T)}$. This can be done algorithmically, by finding a curve with local system representing $\overline{\CFTd(T)}$ as described in the previous section.
 \end{proof}

\begin{remark}
	The algorithm described in the proof above is implemented in the Mathematica package~\cite{PQM.m}.
\end{remark}

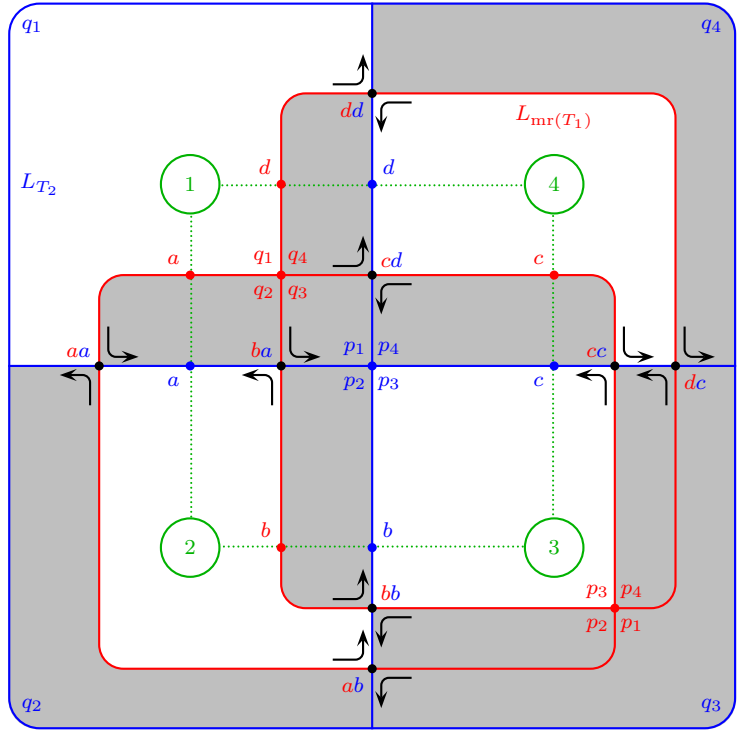
\begin{figure}[t]\centering
	\psset{unit=0.4}
	\begin{pspicture}(-12.5,-12.5)(12.5,12.5)
	
	\pscustom*[linecolor=lightgray]{
		\psline[liftpen=1,linearc=1,cornersize=absolute](12,0)(12,-12)(0,-12)
		\psline(0,-12)(0,0)(12,0)
	}
	\psline*[linecolor=white,linewidth=0pt](0,0)(0,-8)(8,-8)(8,0)(0,0)	
	
	\pscustom*[linecolor=lightgray]{
		\psline[liftpen=1,linearc=1,cornersize=absolute](-12,0)(-12,-12)(0,-12)
		\psline(0,-12)(0,0)(-12,0)
	}
	\pscustom*[linecolor=white,linewidth=0pt]{
		\psline[liftpen=1,linearc=0.8,cornersize=absolute](0,-10)(-9,-10)(-9,0)
		\psline(-9,0)(0,0)(0,-10)
	}
	
	\pscustom*[linecolor=lightgray]{
		\psline[liftpen=1,linearc=1,cornersize=absolute](12,0)(12,12)(0,12)
		\psline(0,12)(0,0)(12,0)
	}
	\pscustom*[linecolor=white,linewidth=0pt]{
		\psline[liftpen=1,linearc=0.8,cornersize=absolute](0,9)(10,9)(10,0)
		\psline(10,0)(0,0)(0,9)
	}
	
	\pscustom*[linecolor=lightgray]{
		\psline[liftpen=1,linearc=0.8,cornersize=absolute](-9,0)(-9,3)(8,3)(8,0)
		\psline(8,0)(-9,0)
	}
	\pscustom*[linecolor=lightgray]{
		\psline[liftpen=1,linearc=0.8,cornersize=absolute](0,-8)(-3,-8)(-3,9)(0,9)
		\psline(0,-8)(0,9)
	}
	
	\psframe[linecolor=blue,linearc=1,cornersize=absolute](-12,-12)(12,12)
	\psline[linecolor=blue](0,12)(0,-12)
	\psline[linecolor=blue](12,0)(-12,0)
	
	\psframe[linecolor=darkgreen,linestyle=dotted,dotsep=1pt](-6,-6)(6,6)
	\pscircle[fillstyle=solid,linecolor=darkgreen](-6,-6){1}
	\pscircle[fillstyle=solid,linecolor=darkgreen](6,-6){1}
	\pscircle[fillstyle=solid,linecolor=darkgreen](-6,6){1}
	\pscircle[fillstyle=solid,linecolor=darkgreen](6,6){1}
	
	\rput(-6,6){$\textcolor{darkgreen}{1}$}
	\rput(-6,-6){$\textcolor{darkgreen}{2}$}
	\rput(6,-6){$\textcolor{darkgreen}{3}$}
	\rput(6,6){$\textcolor{darkgreen}{4}$}

	\psline[linecolor=red,linearc=0.8,cornersize=absolute]
	(0,-8)(-3,-8)(-3,9)(10,9)(10,-8)(0,-8)
	\psline[linecolor=red,linearc=0.8,cornersize=absolute]
	(8,0)(8,3)(-9,3)(-9,-10)(8,-10)(8,0)
	
	\pscircle*[linecolor=blue](0,0){0.15}
	\rput(0.8;135){\blue $p_1$}
	\rput(0.8;-135){\blue $p_2$}
	\rput(0.8;-45){\blue $p_3$}
	\rput(0.8;45){\blue $p_4$}
	
	\rput(11.2,11.2){\blue $q_4$}
	\rput(11.2,-11.2){\blue $q_3$}
	\rput(-11.2,-11.2){\blue $q_2$}
	\rput(-11.2,11.2){\blue $q_1$}
	
	\pscircle*[linecolor=red](-3,3){0.15}
	\rput(-3,3){
		\rput(0.8;135){\textcolor{red}{$q_1$}}
		\rput(0.8;-135){\textcolor{red}{$q_2$}}
		\rput(0.8;-45){\textcolor{red}{$q_3$}}
		\rput(0.8;45){\textcolor{red}{$q_4$}}}
	
	\pscircle*[linecolor=red](8,-8){0.15}
	\rput(8,-8){
		\rput(0.8;135){\textcolor{red}{$p_3$}}
		\rput(0.8;-135){\textcolor{red}{$p_2$}}
		\rput(0.8;-45){\textcolor{red}{$p_1$}}
		\rput(0.8;45){\textcolor{red}{$p_4$}}}

	\pscircle*(-9,0){0.15}
	\uput{0.4}[135](-9,0){\red $a$\textcolor{blue}{$a$}}	
	\pscircle*(-3,0){0.15}
	\uput{0.4}[135](-3,0){\red $b$\textcolor{blue}{$a$}}
	
	\rput(-9,0){
		\rput{0}(0,0){\psline[linearc=0.25]{->}(0.3,1.3)(0.3,0.3)(1.3,0.3)}
		\rput{180}(0,0){\psline[linearc=0.25]{->}(0.3,1.3)(0.3,0.3)(1.3,0.3)}
	}
	\rput(-3,0){
		\rput{0}(0,0){\psline[linearc=0.25]{->}(0.3,1.3)(0.3,0.3)(1.3,0.3)}
		\rput{180}(0,0){\psline[linearc=0.25]{->}(0.3,1.3)(0.3,0.3)(1.3,0.3)}
	}
	
	\pscircle*(0,-8){0.15}
	\uput{0.4}[45](0,-8){\red $b$\textcolor{blue}{$b$}}
	\pscircle*(0,-10){0.15}
	\uput{0.4}[-135](0,-10){\red $a$\textcolor{blue}{$b$}}
	
	\rput(0,-8){
		\rput{90}(0,0){\psline[linearc=0.25]{->}(0.3,1.3)(0.3,0.3)(1.3,0.3)}
		\rput{-90}(0,0){\psline[linearc=0.25]{->}(0.3,1.3)(0.3,0.3)(1.3,0.3)}
	}
	\rput(0,-10){
		\rput{90}(0,0){\psline[linearc=0.25]{->}(0.3,1.3)(0.3,0.3)(1.3,0.3)}
		\rput{-90}(0,0){\psline[linearc=0.25]{->}(0.3,1.3)(0.3,0.3)(1.3,0.3)}
	}
	
	\pscircle*(8,0){0.15}
	\uput{0.4}[135](8,0){\red $c$\textcolor{blue}{$c$}}
	\pscircle*(10,0){0.15}
	\uput{0.4}[-45](10,0){\red $d$\textcolor{blue}{$c$}}
	
	\rput(8,0){
		\rput{0}(0,0){\psline[linearc=0.25]{->}(0.3,1.3)(0.3,0.3)(1.3,0.3)}
		\rput{180}(0,0){\psline[linearc=0.25]{->}(0.3,1.3)(0.3,0.3)(1.3,0.3)}
	}
	\rput(10,0){
		\rput{0}(0,0){\psline[linearc=0.25]{->}(0.3,1.3)(0.3,0.3)(1.3,0.3)}
		\rput{180}(0,0){\psline[linearc=0.25]{->}(0.3,1.3)(0.3,0.3)(1.3,0.3)}
	}
	
	\pscircle*(0,3){0.15}
	\uput{0.4}[45](0,3){\red $c$\textcolor{blue}{$d$}}
	\pscircle*(0,9){0.15}
	\uput{0.4}[-135](0,9){\red $d$\textcolor{blue}{$d$}}
	
	\rput(0,3){
		\rput{90}(0,0){\psline[linearc=0.25]{->}(0.3,1.3)(0.3,0.3)(1.3,0.3)}
		\rput{-90}(0,0){\psline[linearc=0.25]{->}(0.3,1.3)(0.3,0.3)(1.3,0.3)}
	}
	\rput(0,9){
		\rput{90}(0,0){\psline[linearc=0.25]{->}(0.3,1.3)(0.3,0.3)(1.3,0.3)}
		\rput{-90}(0,0){\psline[linearc=0.25]{->}(0.3,1.3)(0.3,0.3)(1.3,0.3)}
	}
	
	\uput{0.4}[-90](6,9){\textcolor{red}{$L_{\mr(T_1)}$}}
	\uput{0.4}[0](-12,6){\textcolor{blue}{$L_{T_2}$}}
	
	\uput{0.5}[-135](-6,0){\blue $a$}
	\uput{0.5}[45](0,-6){\blue $b$}
	\uput{0.5}[-135](6,0){\blue $c$}
	\uput{0.5}[45](0,6){\blue $d$}
	
	\uput{0.5}[135](-6,3){\textcolor{red}{$a$}}
	\uput{0.5}[135](-3,-6){\textcolor{red}{$b$}}
	\uput{0.5}[135](6,3){\textcolor{red}{$c$}}
	\uput{0.5}[135](-3,6){\textcolor{red}{$d$}}
	
	\pscircle*[linecolor=red](-6,3){0.15}
	\pscircle*[linecolor=red](6,3){0.15}
	\pscircle*[linecolor=red](-3,6){0.15}
	\pscircle*[linecolor=red](-3,-6){0.15}
	
	\pscircle*[linecolor=blue](0,6){0.15}
	\pscircle*[linecolor=blue](0,-6){0.15}
	\pscircle*[linecolor=blue](-6,0){0.15}
	\pscircle*[linecolor=blue](6,0){0.15}
	
	\end{pspicture}
	\caption{A geometric interpretation of the type~AA structure $\mathcal{P}$ for pairing in the wrapped Fukaya category of the 4-punctured sphere. The boundary of the picture is identified to a point. The {\blue blue} curves denote a 1-skeleton and the \textcolor{red}{red} ones a Hamiltonian translate thereof with reversed roles of $p_i$ and $q_i$. The orientation is chosen such that the normal vector (determined by the right-hand rule) points into the projection plane. The arrow pairs at the intersection points of the two skeletons serve as a reminder of our conventions for resolutions of two curves lying in a tubular neighbourhood of the skeletons, see Definition~\ref{def:resolution}. }\label{fig:GlueingInterpretationFUK}
\end{figure} 

\subsection{The Glueing Theorem revisited}

\begin{theorem}[(Glueing Theorem, version 2)]\label{thm:CFTdGlueingAsMorphism}
	With the same notation as in Theorem~\ref{thm:CFTdGeneralGlueing}, 
	\[\HFL(L)\otimes V^{i}\cong\LagrangianFH(L_{\mr(T_1)},L_{T_2})\cong H_\ast(\Mor(\CFTd(\mr(T_1)),\CFTd(T_2))),\]
	where \(\mr(T_1)\) denotes the reversed mirror of \(T_1\) (see Definition~\ref{def:reversedmirror}).
\end{theorem}

\begin{proof}
	The second equality is an application of Theorem~\ref{thm:PairingMorLagrangianFH}. As we saw in the proof of Theorem~\ref{thm:PairingFormula}, we can use any representatives of the underlying curves of $L_{\mr(T_1)}$ and $L_{T_2}$ to compute $\LagrangianFH(L_{\mr(T_1)},L_{T_2})$, as long as each pair of parallel curves bounds a cyclic chain of bigons (ie does not bound an immersed annulus). For example, we may homotope the curves in $L_{\mr(T_1)}$ and $L_{T_2}$ into neighbourhoods of the red and blue curves in Figure~\ref{fig:GlueingInterpretationFUK}, respectively. The first equality is then seen by identifying the intersection points between those curves and the connecting bigons with the generators and differentials of 
	\[
	\CFTd(\rr(T_1))\boxtimes\,\mathcal{P}\boxtimes\,\CFTd(T_2)
	\cong
	\Pi(L_{\rr(T_1)})\boxtimes\,\mathcal{P}\boxtimes \Pi(L_{T_2})
	\]
	from Theorem~\ref{thm:CFTdGeneralGlueing}. 
	
	Let us first consider the generators and their gradings. Obviously, there is a one-to-one correspondence between generators and intersection points. Let $\bullet(\textcolor{red}{s},i)$ be a generator of $\Pi(L_{\rr(T_1)})$ and $\bullet(\textcolor{blue}{s'},i')$ a generator of $\Pi(L_{T_2})$, where $\textcolor{red}{s}\in\{\textcolor{red}{a},\textcolor{red}{b},\textcolor{red}{c},\textcolor{red}{d}\}$ and $\textcolor{blue}{s'}\in\{\textcolor{blue}{a},\textcolor{blue}{b},\textcolor{blue}{c},\textcolor{blue}{d}\}$, such that there exists a generator $\textcolor{red}{s}\textcolor{blue}{s'}\in\mathcal{P}$. Then the $\delta$-grading of the corresponding intersection point agrees with the $\delta$-grading of $\bullet(\textcolor{red}{s},i)\boxtimes\mathcal{P}\boxtimes\bullet(\textcolor{blue}{s'},i')$; namely, it is equal to 
	$$
	\delta(\bullet(\textcolor{blue}{s'},i'))+\delta(\bullet(\textcolor{red}{s},i))+\delta(\textcolor{red}{s}\textcolor{blue}{s'})=	\delta(\bullet(\textcolor{blue}{s'},i'))-\delta(\m(\bullet(\textcolor{red}{s},i)))+\delta(\textcolor{red}{s}\textcolor{blue}{s'}),
	$$
	which can be checked for each generator $\textcolor{red}{s}\textcolor{blue}{s'}\in\mathcal{P}$ separately. 
	A similar argument also shows that the Alexander gradings agree on both sides. 
	
	Next, let us consider differentials. For this, let $\bullet(\textcolor{red}{s},i)\xrightarrow{\red p} \bullet(\textcolor{red}{t},j)$ be a differential in $\Pi(L_{\rr(T_1)})$ and $\bullet(\textcolor{blue}{s'},i')\xrightarrow{\blue p'} \bullet(\textcolor{blue}{t'},j')$ a differential in $\Pi(L_{T_2})$, where $\textcolor{red}{s},\textcolor{red}{t}\in\{\textcolor{red}{a},\textcolor{red}{b},\textcolor{red}{c},\textcolor{red}{d}\}$ and $\textcolor{blue}{s'},\textcolor{blue}{t'}\in\{\textcolor{blue}{a},\textcolor{blue}{b},\textcolor{blue}{c},\textcolor{blue}{d}\}$, such that there exist generators $\textcolor{red}{s}\textcolor{blue}{s'},\textcolor{red}{t}\textcolor{blue}{t'}\in\mathcal{P}$. 
	Now observe that there is a component $({\red p}\vert {\blue p'} )$ in $\mathcal{P}$ from $\textcolor{red}{s}\textcolor{blue}{s'}$ to $\textcolor{red}{t}\textcolor{blue}{t'}$, iff there is a bigon in Figure~\ref{fig:GlueingInterpretationFUK} between the corresponding intersection points which covers exactly those ${\red p_i}$ and ${\red q_i}$ in ${\red p}$ and those ${\blue p_i}$ and ${\blue q_i}$ in ${\blue p'}$. Note that the boundary orientation of the bigons is indeed opposite to the orientation of any crossover arrows in $\Pi(L_{\mr(T_1)})$ and $\Pi(L_{T_2})$.
\end{proof}

We end this section with a number of computations illustrating the above Glueing Theorem.

\begin{figure}[t]
	\centering
		{
		\psset{unit=0.2,linewidth=\stringwidth}
		\begin{pspicture}[showgrid=false](-11,-15)(11,15)
		\rput{-45}(0,0){
			\psline[linecolor=red](1,-4)(-9,-4)
			\psline[linewidth=\stringwhite,linecolor=white](-4,1)(-4,-9)
			\psline[linecolor=red](-4,1)(-4,-9)
			\psline[linecolor=red]{->}(-4,-7)(-4,-8)
			\psline[linecolor=red]{->}(-1,-4)(0,-4)
			
			\psline[linecolor=blue](-1,4)(9,4)
			\psline[linewidth=\stringwhite,linecolor=white](4,-1)(4,9)
			\psline[linecolor=blue](4,-1)(4,9)
			\psline[linecolor=blue]{->}(4,7)(4,8)
			\psline[linecolor=blue]{<-}(0,4)(1,4)
			
			\pscircle[linestyle=dotted](4,4){5}
			\pscircle[linestyle=dotted](-4,-4){5}

			\rput{45}(-10,0){$t_2$}
			\rput{45}(0,10){$t_1$}
			
			\rput{45}(1.5,1.5){$a$}
			\rput{45}(6.5,1.5){$b$}
			\rput{45}(6.5,6.5){$c$}
			\rput{45}(1.5,6.5){$d$}
			
			\rput{45}(-6.5,-6.5){$a$}
			\rput{45}(-1.5,-6.5){$b$}
			\rput{45}(-1.5,-1.5){$c$}
			\rput{45}(-6.5,-1.5){$d$}
			
			\psecurve{C-C}(8,-3)(-1,4)(-10,-3)(-9,-4)(-8,-3)
			\psecurve{C-C}(-3,8)(4,-1)(-3,-10)(-4,-9)(-3,-8)

			\psecurve[linewidth=\stringwhite,linecolor=white](-8,3)(1,-4)(10,3)(9,4)(8,3)
			\psecurve[linewidth=\stringwhite,linecolor=white](3,-8)(-4,1)(3,10)(4,9)(3,8)
			
			\psecurve{C-C}(-8,3)(1,-4)(10,3)(9,4)(8,3)
			\psecurve{C-C}(3,-8)(-4,1)(3,10)(4,9)(3,8)
		}
		\end{pspicture}
		}\qquad{\psset{unit=2}
		\begin{pspicture}(-2,-1.5)(1.7,1.5)
		\psecurve[linecolor=blue](1.3,1.3)(-0.25,0.25)(-1.3,-1.3)(0.25,-0.25)(1.3,1.3)(-0.25,0.25)(-1.3,-1.3)
		\rput{90}(0,0){
			\psecurve[linecolor=red](1.3,1.3)(-0.25,0.25)(-1.3,-1.3)(0.25,-0.25)(1.3,1.3)(-0.25,0.25)(-1.3,-1.3)
		}
		
		\psline[linestyle=dotted](1,1)(1,-1)
		\psline[linestyle=dotted](1,-1)(-1,-1)
		\psline[linestyle=dotted](-1,-1)(-1,1)
		\psline[linestyle=dotted](-1,1)(1,1)
		%
		
		\psset{dotsize=5pt}
		
		\psdot[linecolor=red](-1,0.45)
		\psdot[linecolor=blue](-1,-0.45)
		\psdot[linecolor=blue](0.45,1)
		\psdot[linecolor=red](-0.45,1)
		\psdot[linecolor=blue](1,0.45)
		\psdot[linecolor=red](1,-0.45)
		\psdot[linecolor=red](0.45,-1)
		\psdot[linecolor=blue](-0.45,-1)
		
		\uput{0.1}[45]{0}(1,1){$t_1$}
		\uput{0.1}[135]{0}(-1,1){$t_2$}
		\uput{0.1}[-135]{0}(-1,-1){$t_1$}
		\uput{0.1}[-45]{0}(1,-1){$t_2$}
		
		\uput{0.1}[90]{0}(0.45,1){$\delta^{\frac{1}{2}}t_1^\frac{1}{2}t_2^\frac{1}{2}$}
		\uput{0.1}[-90]{0}(-0.45,-1){$\delta^{\frac{1}{2}}t_1^{-\frac{1}{2}}t_2^{-\frac{1}{2}}$}
		\uput{0.1}[180]{0}(-1,-0.45){$\delta^0t_1^{\frac{1}{2}}t_2^{-\frac{1}{2}}$}
		\uput{0.1}[0]{0}(1,0.45){$\delta^0t_1^{-\frac{1}{2}}t_2^{\frac{1}{2}}$}
		
		\uput{0.1}[90]{0}(-0.45,1){$\delta^{-\frac{1}{2}}t_1^{-\frac{1}{2}}t_2^{-\frac{1}{2}}$}
		\uput{0.1}[-90]{0}(0.45,-1){$\delta^{-\frac{1}{2}}t_1^{\frac{1}{2}}t_2^{\frac{1}{2}}$}
		\uput{0.1}[180]{0}(-1,0.45){$\delta^0t_1^{-\frac{1}{2}}t_2^{\frac{1}{2}}$}
		\uput{0.1}[0]{0}(1,-0.45){$\delta^0t_1^{\frac{1}{2}}t_2^{-\frac{1}{2}}$}
		
		\uput{0.1}[180]{0}(-0.52,0){$\delta^1t_1^{1}t_2^{-1}$}
		\uput{0.1}[0]{0}(0.52,0){$\delta^1t_1^{-1}t_2^{1}$}
		\uput{0.1}[0]{0}(0,-0.52){$\delta^1t_1^{-1}t_2^{-1}$}
		\uput{0.1}[0]{0}(0,0.52){$\delta^1t_1^{1}t_2^{1}$}
		
		\uput{0.45}[180]{0}(-1,1){\textcolor{red}{$L_{\mr(T_1)}$}}
		\uput{0.45}[0]{0}(1,1){\textcolor{blue}{$L_{T_2}$}}
		
		\pscircle*[linecolor=white](-1,0.7){3pt}
		\rput(-1,0.7){$a$}
		\pscircle*[linecolor=white](0,-1){3pt}
		\rput(0,-1){$b$}
		\pscircle*[linecolor=white](1,0.7){3pt}
		\rput(1,0.7){$c$}
		\pscircle*[linecolor=white](0,1){3pt}
		\rput(0,1){$d$}
		
		\psdot(-0.52,0)
		\rput{0}(-0.52,0.1){%
			\psline[linearc=0.1]{->}(0.25;137)(0,0)(0.25;43)}
		\rput{180}(-0.52,-0.1){%
			\psline[linearc=0.1]{->}(0.25;137)(0,0)(0.25;43)}
		
		\psdot(0.52,0)
		\rput{0}(0.52,0.1){%
			\psline[linearc=0.1]{->}(0.25;137)(0,0)(0.25;43)}
		\rput{180}(0.52,-0.1){%
			\psline[linearc=0.1]{->}(0.25;137)(0,0)(0.25;43)}
		
		\psdot(0,-0.52)
		\rput{180}(0,-0.62){%
			\psline[linearc=0.1]{->}(0.25;133)(0,0)(0.25;47)}
		\rput{0}(0,-0.42){%
			\psline[linearc=0.1]{->}(0.25;133)(0,0)(0.25;47)}
		
		\psdot(0,0.52)
		\rput{0}(0,0.62){%
			\psline[linearc=0.1]{->}(0.25;133)(0,0)(0.25;47)}
		\rput{180}(0,0.42){%
			\psline[linearc=0.1]{->}(0.25;133)(0,0)(0.25;47)}
		
		\pscircle[fillstyle=solid, fillcolor=white](1,1){0.08}
		\pscircle[fillstyle=solid, fillcolor=white](-1,1){0.08}
		\pscircle[fillstyle=solid, fillcolor=white](1,-1){0.08}
		\pscircle[fillstyle=solid, fillcolor=white](-1,-1){0.08}
		\end{pspicture}
		}
\caption{A tangle decomposition of the Hopf link (left) and the intersection theory of the corresponding tangle invariants (right).}\label{fig:HopfLink}
\end{figure}
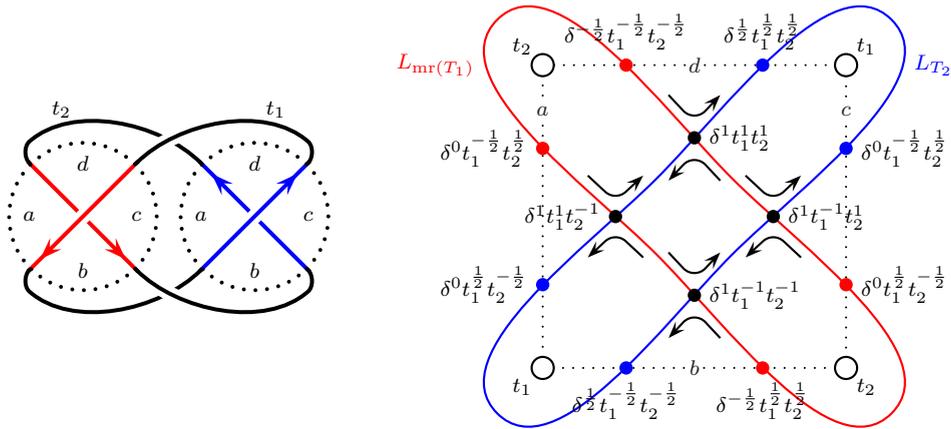 

\begin{example}[(Hopf link)]
	Consider the tangle decomposition of the Hopf link into two tangles $\textcolor{red}{T_1}$ and $\textcolor{blue}{T_2}$ from Figure~\ref{fig:HopfLink}, following the same conventions as in Definition~\ref{def:tanglepairing}. The diagram on the right of that figure shows a 4-punctured sphere which we consider as the  boundary of the 3-ball of the tangle $\textcolor{blue}{T_2}$. It is parametrized by four arcs which are drawn as dotted lines labelled $a$, $b$, $c$ and $d$ as usual. The intersection points of the curve $\textcolor{blue}{L_{T_2}}$ with those arcs are labelled by their gradings. 
	The tangle $\textcolor{red}{T_1}$ agrees with $\textcolor{blue}{T_2}$ except for the orientations of the tangle strands.  Thus, $\textcolor{red}{L_{\mr(T_1)}}=L_{\m(T_2)}$ is obtained from $\textcolor{blue}{L_{T_2}}$ by taking the mirror of the underlying curve and reversing all gradings. 
	
	The curves $\textcolor{red}{L_{\mr(T_1)}}$ and $\textcolor{blue}{L_{T_2}}$ intersect in four points. They are labelled by their respective gradings, which we can compute from their resolutions. As a reminder of our conventions, we have indicated these resolutions by pairs of arrows around the intersection points in this example. (We omit those arrows in the examples below.) 
	In summary, we conclude that $\HFL$ of the Hopf link is a 4-dimensional vector space supported in a single $\delta$-grading and Alexander gradings
	$$
	t_1^{-1}t_2^{-1}+
	t_1^{1}t_2^{-1}+
	t_1^{-1}t_2^{1}+
	t_1^{1}t_2^{1}.
	$$
\end{example}

\begin{figure}[t]
		{
			\psset{unit=0.2,linewidth=\stringwidth}
			\begin{pspicture}[showgrid=false](-11,-15)(11,15)
			\rput{-45}(0,0){
				\psline[linecolor=red](-4,1)(-4,-9)
				\psline[linewidth=\stringwhite,linecolor=white](1,-4)(-9,-4)
				\psline[linecolor=red](1,-4)(-9,-4)
				\psline[linecolor=red]{->}(-4,-7)(-4,-8)
				\psline[linecolor=red]{->}(-1,-4)(0,-4)
				
				\psline[linecolor=blue](-1,4)(9,4)
				\psline[linewidth=\stringwhite,linecolor=white](4,-1)(4,9)
				\psline[linecolor=blue](4,-1)(4,9)
				\psline[linecolor=blue]{->}(4,7)(4,8)
				\psline[linecolor=blue]{<-}(0,4)(1,4)
				
				\pscircle[linestyle=dotted](4,4){5}
				\pscircle[linestyle=dotted](-4,-4){5}

				\rput{45}(-10,0){$t_2$}
				\rput{45}(0,10){$t_1$}
				
				\rput{45}(1.5,1.5){$a$}
				\rput{45}(6.5,1.5){$b$}
				\rput{45}(6.5,6.5){$c$}
				\rput{45}(1.5,6.5){$d$}
				
				\rput{45}(-6.5,-6.5){$a$}
				\rput{45}(-1.5,-6.5){$b$}
				\rput{45}(-1.5,-1.5){$c$}
				\rput{45}(-6.5,-1.5){$d$}
				
				\psecurve{C-C}(8,-3)(-1,4)(-10,-3)(-9,-4)(-8,-3)
				\psecurve{C-C}(-3,8)(4,-1)(-3,-10)(-4,-9)(-3,-8)

				\psecurve[linewidth=\stringwhite,linecolor=white](-8,3)(1,-4)(10,3)(9,4)(8,3)
				\psecurve[linewidth=\stringwhite,linecolor=white](3,-8)(-4,1)(3,10)(4,9)(3,8)
				
				\psecurve{C-C}(-8,3)(1,-4)(10,3)(9,4)(8,3)
				\psecurve{C-C}(3,-8)(-4,1)(3,10)(4,9)(3,8)
			}
			\end{pspicture}
		}\qquad{\psset{unit=2}
		\begin{pspicture}(-1.7,-1.5)(1.7,1.5)
		\psecurve[linecolor=blue](1.3,1.3)(-0.25,0.25)(-1.3,-1.3)(0.25,-0.25)(1.3,1.3)(-0.25,0.25)(-1.3,-1.3)
		\psecurve[linecolor=red]%
		(-0.3,-1)%
		(-1.15,-1.45)(-1.45,-1.15)%
		(-1,-0.3)%
		(0.3,1)%
		(1.15,1.45)(1.45,1.15)%
		(1,0.3)%
		(0.15,-0.15)%
		(-0.3,-1)%
		(-1.15,-1.45)(-1.45,-1.15)%

		\psline[linestyle=dotted](1,1)(1,-1)
		\psline[linestyle=dotted](1,-1)(-1,-1)
		\psline[linestyle=dotted](-1,-1)(-1,1)
		\psline[linestyle=dotted](-1,1)(1,1)
		%
		
		\psset{dotsize=5pt}
		
		\psdot[linecolor=red](-1,-0.3)
		\psdot[linecolor=blue](-1,-0.45)
		\psdot[linecolor=blue](0.45,1)
		\psdot[linecolor=red](0.3,1)
		\psdot[linecolor=blue](1,0.45)
		\psdot[linecolor=red](1,0.3)
		\psdot[linecolor=red](-0.3,-1)
		\psdot[linecolor=blue](-0.45,-1)
		
		\uput{0.1}[45]{0}(1,1){$t_1$}
		\uput{0.1}[135]{0}(-1,1){$t_2$}
		\uput{0.1}[-135]{0}(-1,-1){$t_1$}
		\uput{0.1}[-45]{0}(1,-1){$t_2$}
		
		\uput{0.1}[-60]{0}(0.45,1){$\delta^{\frac{1}{2}}t_1^\frac{1}{2}t_2^\frac{1}{2}$}
		\uput{0.1}[120]{0}(-0.45,-1){$\delta^{\frac{1}{2}}t_1^{-\frac{1}{2}}t_2^{-\frac{1}{2}}$}
		\uput{0.1}[0]{0}(-1,-0.45){$\delta^0t_1^{\frac{1}{2}}t_2^{-\frac{1}{2}}$}
		\uput{0.1}[180]{0}(1,0.45){$\delta^0t_1^{-\frac{1}{2}}t_2^{\frac{1}{2}}$}
		
		\uput{0.1}[120]{0}(0.3,1){$\delta^{\frac{1}{2}}t_1^\frac{1}{2}t_2^\frac{1}{2}$}
		\uput{0.1}[-60]{0}(-0.3,-1){$\delta^{\frac{1}{2}}t_1^{-\frac{1}{2}}t_2^{-\frac{1}{2}}$}
		\uput{0.1}[180]{0}(-1,-0.3){$\delta^0t_1^{\frac{1}{2}}t_2^{-\frac{1}{2}}$}
		\uput{0.1}[0]{0}(1,0.3){$\delta^0t_1^{-\frac{1}{2}}t_2^{\frac{1}{2}}$}
		
		\uput{0.1}[-30]{0}(-0.09,-0.61){$\delta^0t_1^0t_2^0$}
		\uput{0.1}[-60]{0}(0.61,0.09){$\delta^1t_1^0t_2^0$}
		
		\uput{0.05}[135]{0}(0,0){\textcolor{blue}{$L_{T_2}$}}
		\uput{0.6}[135]{0}(0,0){\textcolor{red}{$L_{\mr(T_1)}$}}
		
		\psdot(-0.09,-0.61)
		\psdot(0.61,0.09)
		
		\pscircle*[linecolor=white](-1,0.5){3pt}
		\rput(-1,0.5){$a$}
		\pscircle*[linecolor=white](0.5,-1){3pt}
		\rput(0.5,-1){$b$}
		\pscircle*[linecolor=white](1,-0.5){3pt}
		\rput(1,-0.5){$c$}
		\pscircle*[linecolor=white](-0.5,1){3pt}
		\rput(-0.5,1){$d$}
		
		\pscircle[fillstyle=solid, fillcolor=white](1,1){0.08}
		\pscircle[fillstyle=solid, fillcolor=white](-1,1){0.08}
		\pscircle[fillstyle=solid, fillcolor=white](1,-1){0.08}
		\pscircle[fillstyle=solid, fillcolor=white](-1,-1){0.08}
		\end{pspicture}
	}
	\caption{A tangle decomposition of the unlink (left) and the intersection theory of the corresponding tangle invariants (right).}\label{fig:Unlink}
\end{figure}
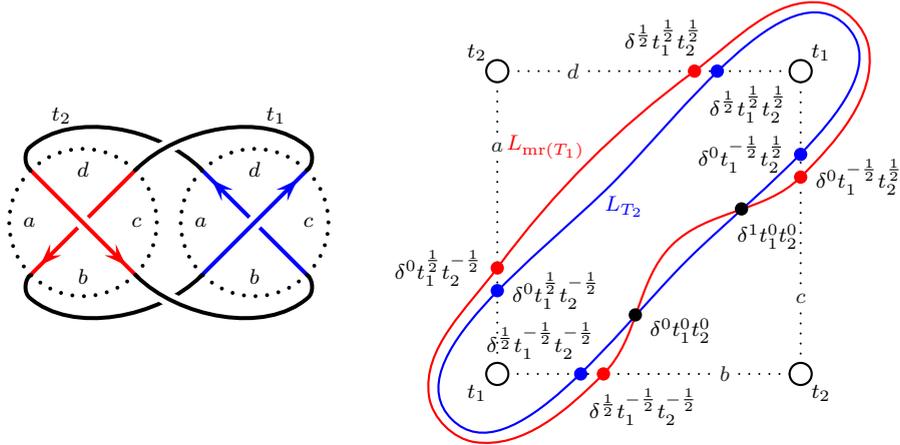

\begin{example}[(Unlink)]
	Figure~\ref{fig:Unlink} shows a tangle decomposition of the two-component unlink and the pairing of the two corresponding curves. This example is very similar to the previous one, so we let the pictures speak for themselves. It illustrates that admissibility of Heegaard diagrams corresponds to the fact that we need to remove immersed annuli before counting intersection points between parallel immersed curves. 
\end{example}

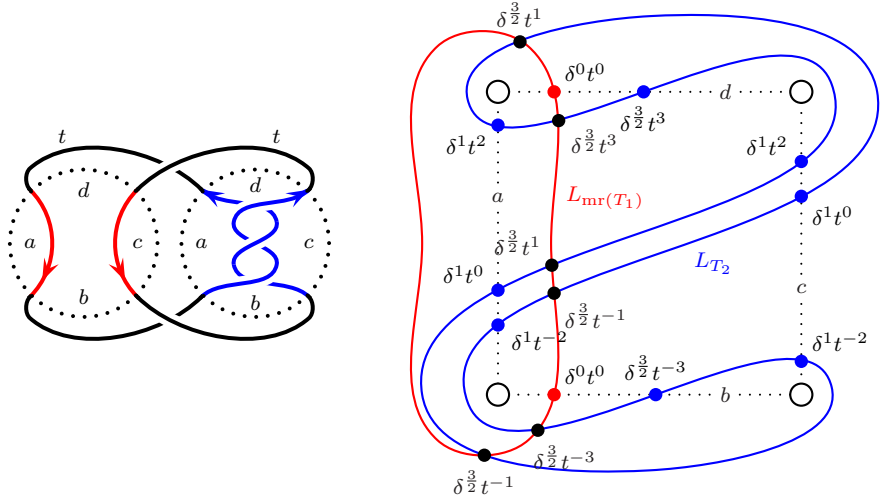
\begin{figure}[t]
	\centering
		{
			\psset{unit=0.2,linewidth=\stringwidth}
			\begin{pspicture}[showgrid=false](-11,-16.5)(11,16.5)
			\rput{-45}(0,0){
				\psarc[linecolor=red](1,1){5}{180}{270}
				\psarc[linecolor=red](-9,-9){5}{0}{90}
				
				\psarc[linecolor=red]{->}(1,1){5}{250}{255}
				\psarcn[linecolor=red]{->}(-9,-9){5}{20}{15}
				
				\psecurve[linecolor=blue](14,2)(9,4)(4,2)(4,6)(-1,4)(-6,6)
				\psecurve[linewidth=\stringwhite,linecolor=white](6,-6)(4,-1)(6,4)(2,4)(4,9)(2,14)
				\psecurve[linecolor=blue](6,-6)(4,-1)(6,4)(2,4)(4,9)(2,14)
				\psecurve[linewidth=\stringwhite,linecolor=white](9,4)(4,2)(4,6)(-1,4)
				\psecurve[linecolor=blue](9,4)(4,2)(4,6)(-1,4)
				
				
				\psline[linecolor=blue]{<-}(-1,4)(-0.9,4.03)
				\psline[linecolor=blue]{<-}(4,9)(3.97,8.9)
				
				\pscircle[linestyle=dotted](4,4){5}
				\pscircle[linestyle=dotted](-4,-4){5}

				\rput{45}(-10,0){$t$}
				\rput{45}(0,10){$t$}
				
				\rput{45}(1.5,1.5){$a$}
				\rput{45}(6.75,1.25){$b$}
				\rput{45}(6.5,6.5){$c$}
				\rput{45}(1.25,6.75){$d$}
				
				\rput{45}(-6.5,-6.5){$a$}
				\rput{45}(-1.5,-6.5){$b$}
				\rput{45}(-1.5,-1.5){$c$}
				\rput{45}(-6.5,-1.5){$d$}
				
				\psecurve{C-C}(8,-3)(-1,4)(-10,-3)(-9,-4)(-8,-3)
				\psecurve{C-C}(-3,8)(4,-1)(-3,-10)(-4,-9)(-3,-8)
				
				\psecurve[linewidth=\stringwhite,linecolor=white](-8,3)(1,-4)(10,3)(9,4)(8,3)
				\psecurve[linewidth=\stringwhite,linecolor=white](3,-8)(-4,1)(3,10)(4,9)(3,8)
				
				\psecurve{C-C}(-8,3)(1,-4)(10,3)(9,4)(8,3)
				\psecurve{C-C}(3,-8)(-4,1)(3,10)(4,9)(3,8)
			}
			\end{pspicture}
	}\qquad{\psset{unit=2}
		\begin{pspicture}(-1.7,-1.65)(1.7,1.65)
		\psecurve[linecolor=blue](1.2,-0.95)(0,-1.5)(-1.5,-1)(1.2,1.05)(-1.2,0.95)(0,1.5)(1.5,1)(-1.2,-1.05)(1.2,-0.95)(0,-1.5)(-1.5,-1)
		\rput{0}(-1.1,0){
			\psecurve[linecolor=red]%
			(0.3,1.3)(-0.3,1.3)%
			(-0.4,0)%
			(-0.3,-1.3)(0.3,-1.3)%
			(0.45,0)%
			(0.3,1.3)(-0.3,1.3)%
			(-0.4,0)%
		}
		
		\psline[linestyle=dotted](1,1)(1,-1)
		\psline[linestyle=dotted](1,-1)(-1,-1)
		\psline[linestyle=dotted](-1,-1)(-1,1)
		\psline[linestyle=dotted](-1,1)(1,1)
		%
		
		\psset{dotsize=5pt}
		
		\psdot[linecolor=blue](-1,0.78)
		\uput{0.1}[-135]{0}(-1,0.78){$\delta^1t^{2}$}
		\psdot[linecolor=blue](-1,-0.31)
		\uput{0.1}[145]{0}(-1,-0.31){$\delta^1t^{0}$}
		\psdot[linecolor=blue](-1,-0.54)
		\uput{0.1}[-35]{0}(-1,-0.54){$\delta^1t^{-2}$}
		\psdot[linecolor=blue](-0.04,1)
		\uput{0.1}[-90]{0}(-0.04,1){$\delta^{\frac{3}{2}}t^{3}$}
		\psdot[linecolor=blue](0.04,-1)
		\uput{0.1}[90]{0}(0.04,-1){$\delta^{\frac{3}{2}}t^{-3}$}
		\psdot[linecolor=blue](1,0.54)
		\uput{0.1}[145]{0}(1,0.54){$\delta^1t^{2}$}
		\psdot[linecolor=blue](1,0.31)
		\uput{0.1}[-45]{0}(1,0.31){$\delta^1t^{0}$}
		\psdot[linecolor=blue](1,-0.78)
		\uput{0.1}[45]{0}(1,-0.78){$\delta^1t^{-2}$}
		
		\psdot[linecolor=red](-0.63,1)
		\uput{0.1}[45]{0}(-0.63,1){$\delta^0t^{0}$}
		\psdot[linecolor=red](-0.63,-1)
		\uput{0.1}[45]{0}(-0.63,-1){$\delta^0t^{0}$}
		
		\uput{0.05}[0]{0}(-0.6,0.3){\textcolor{red}{$L_{\mr(T_1)}$}}
		\uput{0.3}[-20]{0}(0,0){\textcolor{blue}{$L_{T_2}$}}
		
		\psdot(-0.855,1.33)
		\uput{0.1}[90]{0}(-0.855,1.33){$\delta^{\frac{3}{2}}t^{1}$}
		\psdot(-0.6,0.81)
		\uput{0.1}[-40]{0}(-0.6,0.81){$\delta^{\frac{3}{2}}t^{3}$}
		
		\psdot(-0.645,-0.145)
		\uput{0.1}[135]{0}(-0.645,-0.145){$\delta^{\frac{3}{2}}t^{1}$}
		\psdot(-0.63,-0.33)
		\uput{0.1}[-45]{0}(-0.63,-0.33){$\delta^{\frac{3}{2}}t^{-1}$}
		
		\psdot(-0.737,-1.235)
		\uput{0.1}[-60]{0}(-0.737,-1.235){$\delta^{\frac{3}{2}}t^{-3}$}
		\psdot(-1.09,-1.4)
		\uput{0.1}[-90]{0}(-1.09,-1.4){$\delta^{\frac{3}{2}}t^{-1}$}

		\pscircle*[linecolor=white](-1,0.3){3pt}
		\rput(-1,0.3){$a$}
		\pscircle*[linecolor=white](0.5,-1){3pt}
		\rput(0.5,-1){$b$}
		\pscircle*[linecolor=white](1,-0.3){3pt}
		\rput(1,-0.3){$c$}
		\pscircle*[linecolor=white](0.5,1){3pt}
		\rput(0.5,1){$d$}

		\pscircle[fillstyle=solid, fillcolor=white](1,1){0.08}
		\pscircle[fillstyle=solid, fillcolor=white](-1,1){0.08}
		\pscircle[fillstyle=solid, fillcolor=white](1,-1){0.08}
		\pscircle[fillstyle=solid, fillcolor=white](-1,-1){0.08}
		\end{pspicture}
	}
	\caption{A tangle decomposition of the trefoil (left) and the intersection theory of the corresponding tangle invariants (right).}\label{fig:Trefoil}
\end{figure} 
\begin{example}[(Trefoil knot)]
	Figure~\ref{fig:Trefoil} shows a tangle decomposition of the trefoil knot and the pairing of the two corresponding curves. Again, this example is very similar to the previous two, except that the two open components of the two tangles are identified in the trefoil knot, so the same happens to the Alexander gradings: $t_1=t=t_2$. Intersection points lie in a single $\delta$-grading and Alexander gradings
	$$(t+t^{-1})(t^2+t^0+t^{-2}).$$
\end{example}

\begin{remark}
	The attentive reader will have noticed that in each of the previous three examples, the 4-punctured sphere together with the curves $\textcolor{red}{L_{\mr(T_1)}}$ and $\textcolor{blue}{L_{T_2}}$ constitutes a well-defined multi-pointed Heegaard diagram for the link $L(\textcolor{red}{T_1},\textcolor{blue}{T_2})$, where we interpret $\textcolor{red}{L_{\mr(T_1)}}$ as an $\textcolor{red}{\alpha}$-curve and $\textcolor{blue}{L_{T_2}}$ as a $\textcolor{blue}{\beta}$-curve. But this is only true more generally if both $\textcolor{red}{T_1}$ and $\textcolor{blue}{T_2}$ are rational. In general, we cannot find a Heegaard diagram for the link $L(\textcolor{red}{T_1},\textcolor{blue}{T_2})$ whose Heegaard surface is a 4-punctured sphere inducing the given tangle decomposition. 
	Theorem~\ref{thm:CFTdGlueingAsMorphism} says that nonetheless, we can regard the 4-punctured sphere together with $\textcolor{red}{L_{\mr(T_1)}}$ and  $\textcolor{blue}{L_{T_2}}$ as a generalized Heegaard diagram for $L(\textcolor{red}{T_1},\textcolor{blue}{T_2})$ and use it to compute $\HFL(L(\textcolor{red}{T_1},\textcolor{blue}{T_2}))$, even though $\textcolor{red}{L_{\mr(T_1)}}$ and  $\textcolor{blue}{L_{T_2}}$ might have multiple and possibly immersed components. Below, we compute two more examples to illustrate this point of view.
\end{remark}

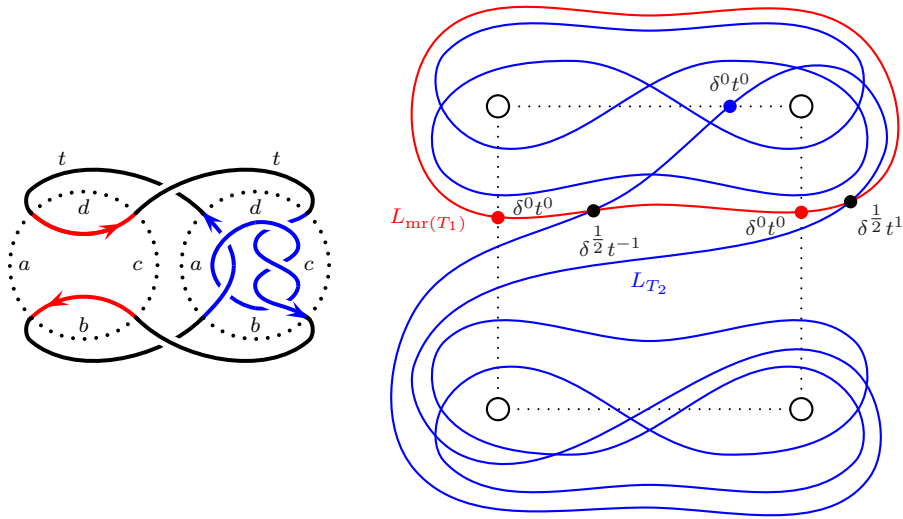
\begin{figure}[t]
	\centering
	{
		\psset{unit=0.2,linewidth=\stringwidth}
		\begin{pspicture}[showgrid=false](-11,-16.5)(11,16.5)
		\rput{-45}(0,0){
			\psarc[linecolor=red](-9,1){5}{-90}{0}
			\psarc[linecolor=red](1,-9){5}{90}{180}
			
			\psarc[linecolor=red]{->}(-9,1){5}{-90}{-15}
			\psarc[linecolor=red]{->}(1,-9){5}{160}{165}
			
			\psecurve[linecolor=blue]
			(-5,3)%
			(-1,4)(3,3)(4,-1)%
			(3,-5)
			\psecurve[linewidth=\stringwhite,linecolor=white](5,7)%
			(2,6)%
			(2,2)%
			(6,2)%
			(7,5)
			\psecurve[linecolor=blue](5,7)%
			(2,6)%
			(2,2)%
			(6,2)%
			(7,5)
			\psecurve[linewidth=\stringwhite,linecolor=white]
			(-1,4)(3,3)(4,-1)%
			(3,-5)
			\psecurve[linecolor=blue]
			(-1,4)(3,3)(4,-1)%
			(3,-5)
			
			\psecurve[linecolor=blue]%
			(2,2)%
			(6,2)%
			(7,5)(3,5)(4,9)(3,13)
			\psecurve[linewidth=\stringwhite,linecolor=white]%
			(13,3)(9,4)(5,3)(5,7)%
			(2,6)%
			(2,2)%
			\psecurve[linecolor=blue]%
			(13,3)(9,4)(5,3)(5,7)%
			(2,6)%
			(2,2)%
			\psecurve[linewidth=\stringwhite,linecolor=white]%
			(6,2)%
			(7,5)(3,5)(4,9)
			\psecurve[linecolor=blue]%
			(6,2)%
			(7,5)(3,5)(4,9)
			
			\psline[linecolor=blue]{<-}(-1,4)(-0.9,4)
			\psline[linecolor=blue]{<-}(9,4)(8.9,3.97)
			
			\pscircle[linestyle=dotted](4,4){5}
			\pscircle[linestyle=dotted](-4,-4){5}
			
			\rput{45}(-10,0){$t$}
			\rput{45}(0,10){$t$}
			
			\rput{45}(1.25,1.25){$a$}
			\rput{45}(6.75,1.25){$b$}
			\rput{45}(6.5,6.5){$c$}
			\rput{45}(1.25,6.75){$d$}
			
			\rput{45}(-6.75,-6.75){$a$}
			\rput{45}(-1.25,-6.75){$b$}
			\rput{45}(-1.5,-1.5){$c$}
			\rput{45}(-6.75,-1.25){$d$}
			
			\psecurve{C-C}(8,-3)(-1,4)(-10,-3)(-9,-4)(-8,-3)
			\psecurve{C-C}(-3,8)(4,-1)(-3,-10)(-4,-9)(-3,-8)

			\psecurve[linewidth=\stringwhite,linecolor=white](-8,3)(1,-4)(10,3)(9,4)(8,3)
			\psecurve[linewidth=\stringwhite,linecolor=white](3,-8)(-4,1)(3,10)(4,9)(3,8)
			
			\psecurve{C-C}(-8,3)(1,-4)(10,3)(9,4)(8,3)
			\psecurve{C-C}(3,-8)(-4,1)(3,10)(4,9)(3,8)
		}
		\end{pspicture}
	}\qquad{\psset{unit=2}
	\begin{pspicture}(-1.7,-1.7)(1.7,1.7)	
	
	\psline[linestyle=dotted](1,1)(1,-1)
	\psline[linestyle=dotted](1,-1)(-1,-1)
	\psline[linestyle=dotted](-1,-1)(-1,1)
	\psline[linestyle=dotted](-1,1)(1,1)
	
	\uput{0.05}[180]{0}(-1.15,0.25){\textcolor{red}{$L_{\mr(T_1)}$}}
	\uput{0.1}[-90]{0}(0,0){\textcolor{blue}{$L_{T_2}$}}

	\psecurve[linecolor=blue]%
	(0,1.5)%
	(-1.35,1.35)(-1.375,1)(-1,0.725)(0.5,1.3)(1.375,1)(1.35,0.55)%
	(0,0.55)%
	(-1.35,0.55)(-1.375,1)(-0.5,1.3)(1,0.725)(1.375,1)(1.35,1.35)%
	(0,1.5)%
	(-1.35,1.35)(-1.375,1)
	
	\rput{180}(0,0){
		\psecurve[linecolor=blue]%
		(0,1.5)%
		(-1.35,1.35)(-1.375,1)(-1,0.725)(0.5,1.3)(1.375,1)(1.35,0.55)%
		(0,0.55)%
		(-1.35,0.55)(-1.375,1)(-0.5,1.3)(1,0.725)(1.375,1)(1.35,1.35)%
		(0,1.5)%
		(-1.35,1.35)(-1.375,1)
	}
	
	\psecurve[linecolor=blue](0,-1.65)(-1.55,-1.4)(-1.55,-0.35)(-0.2,0.4)(1.25,1.25)(1.55,0.7)(-1.55,-0.7)(-1.25,-1.3)(1.3,-0.7)(1.45,-1.45)(0,-1.65)(-1.55,-1.4)(-1.55,-0.35)

	\rput{0}(0,1.1){
		\psecurve[linecolor=red]%
		(1.4,0.35)(1.55,-0.55)%
		(0,-0.75)%
		(-1.4,-0.65)(-1.4,0.35)%
		(0,0.5)%
		(1.4,0.35)(1.55,-0.55)%
		(0,-0.8)%
	}
	
	\psset{dotsize=5pt}
	
	\psdot(1.325,0.37)
	\uput{0.05}[-40]{0}(1.325,0.37){$\delta^{\tfrac{1}{2}}t^{1}$}
	
	\psdot(-0.37,0.31)
	\uput{0.1}[-70]{0}(-0.37,0.31){$\delta^{\tfrac{1}{2}}t^{-1}$}
	
	\psdot[linecolor=red](-1,0.265)
	\uput{0.1}[25]{0}(-1,0.265){$\delta^{0}t^{0}$}
	\psdot[linecolor=red](1,0.3)
	\uput{0.1}[-150]{0}(1,0.3){$\delta^{0}t^{0}$}
	
	\psdot[linecolor=blue](0.53,1)
	\uput{0.1}[90]{0}(0.53,1){$\delta^{0}t^{0}$}

	\pscircle[fillstyle=solid, fillcolor=white](1,1){0.08}
	\pscircle[fillstyle=solid, fillcolor=white](-1,1){0.08}
	\pscircle[fillstyle=solid, fillcolor=white](1,-1){0.08}
	\pscircle[fillstyle=solid, fillcolor=white](-1,-1){0.08}
	
	\end{pspicture}
}
\caption{A tangle decomposition of the unknot as a closure of the $(2,-3)$-pretzel tangle (left) and the intersection theory of the corresponding tangle invariants (right).}\label{fig:Unknot}
\end{figure}

\begin{example}[(unknot closure of the $(2,-3)$-pretzel tangle)]
	Figure~\ref{fig:Unknot} shows the unknot closure of the $(2,-3)$-pretzel tangle $T_{2,-3}$. Only the embedded component of $L_{T_{2,-3}}$ contributes to the pairing, since the immersed components can be homotoped such that they do not intersect the embedded curve of the trivial tangle. Note that the gradings are different from Example~\ref{exa:HFTdpretzeltangle}. This is because the orientation of one of the tangle strands is reversed, which,  for relative gradings, corresponds simply to reversing the Alexander grading of $t_2$ and leaving the $\delta$-grading unchanged. However, since I expect that the relative $\delta$-gradings of generators of peculiar modules can be lifted to absolute ones agreeing with those of the corresponding Kauffman states, I am using the $\delta$-grading from~\cite[Example~6.9]{HDsForTangles} here. The relevant generator of the peculiar module of $T_{2,-3}$ is $dy_3$, 
	which corresponds to the blue intersection point labelled by its $\delta$- and Alexander grading on the right of Figure~\ref{fig:Unknot}. From this, we can compute the gradings of the two intersection points and obtain the once stabilized knot Floer homology of the unknot.	
\end{example}

\begin{figure}[p]
	\centering
{\psset{unit=3}
	\begin{pspicture}(-2,-1.7)(2,1.7)
	\psline[linestyle=dotted](1,1)(1,-1)
	\psline[linestyle=dotted](1,-1)(-1,-1)
	\psline[linestyle=dotted](-1,-1)(-1,1)
	\psline[linestyle=dotted](-1,1)(1,1)
	
	\uput{0.05}[180]{0}(-1.1,0.3){\textcolor{red}{$L_{\mr(T_1)}$}}
	\uput{0.05}[0]{0}(1.1,-0.3){\textcolor{red}{$L_{\mr(T_1)}$}}
	\uput{0.05}[180]{0}(-1.1,-0.3){\textcolor{blue}{$L_{T_2}$}}
	\uput{0.05}[0]{0}(1.1,0.3){\textcolor{blue}{$L_{T_2}$}}
		
	\uput{0.1}[0]{0}(1,1){$t_2$}
	\uput{0.1}[180]{0}(-1,1){$t_1$}
	\uput{0.1}[180]{0}(-1,-1){$t_1$}
	\uput{0.1}[0]{0}(1,-1){$t_2$}
	
	\psecurve[linecolor=red]%
	(0.3,1.5)%
	(-1.4,1.4)(-1.45,1)(-1,0.65)(0.5,1.25)(1.3,1)(1.3,0.6)%
	(0.3,0.55)%
	(-1.4,0.5)(-1.45,1)(-0.5,1.35)(1,0.8)(1.3,1)(1.3,1.3)%
	(0.3,1.5)%
	(-1.4,1.4)(-1.45,1)
	
	
	\psecurve[linecolor=blue]%
	(-0.3,1.5)%
	(1.4,1.4)(1.45,1)(1,0.65)(-0.5,1.25)(-1.3,1)(-1.3,0.6)%
	(-0.3,0.55)%
	(1.4,0.5)(1.45,1)(0.5,1.35)(-1,0.8)(-1.3,1)(-1.3,1.3)%
	(-0.3,1.5)%
	(1.4,1.4)(1.45,1)
	
	\psecurve[linecolor=blue]%
	(0.3,-1.5)%
	(-1.4,-1.4)(-1.45,-1)(-1,-0.65)(0.5,-1.25)(1.3,-1)(1.3,-0.6)%
	(0.3,-0.55)%
	(-1.4,-0.5)(-1.45,-1)(-0.5,-1.35)(1,-0.8)(1.3,-1)(1.3,-1.3)%
	(0.3,-1.5)%
	(-1.4,-1.4)(-1.45,-1)
	
	
	\psecurve[linecolor=red]%
	(-0.3,-1.5)%
	(1.4,-1.4)(1.45,-1)(1,-0.65)(-0.5,-1.25)(-1.3,-1)(-1.3,-0.6)%
	(-0.3,-0.55)%
	(1.4,-0.5)(1.45,-1)(0.5,-1.35)(-1,-0.8)(-1.3,-1)(-1.3,-1.3)%
	(-0.3,-1.5)%
	(1.4,-1.4)(1.45,-1)
	
	
	\psset{dotsize=5pt}
	
	\psdot(0,0.505)
	\uput{0.05}[-90]{0}(0,0.505){$\delta^{0}t_1^{0}t_2^{0}$}
	\psdot(0,1.074)
	\psline{->}(0,0.8)(0,1.074)
	\uput{0.3}[-90]{0}(0,1.074){$\delta^{0}t_1^{-2}t_2^{-2}$}
	\psdot(0,1.2)
	\uput{0.05}[90]{0}(0,1.2){$\delta^{1}t_1^{2}t_2^{2}$}
	\psdot(0,1.53)
	\uput{0.05}[90]{0}(0,1.53){$\delta^{1}t_1^{0}t_2^{0}$}
	
	\psdot(-1.356,0.85)
	\uput{0.05}[180]{0}(-1.356,0.85){$\delta^{0}t_1^{0}t_2^{-2}$}
	\psdot(-1.34,1.13)
	\uput{0.05}[150]{0}(-1.34,1.13){$\delta^{1}t_1^{0}t_2^{2}$}
	\psdot(1.356,0.85)
	\uput{0.05}[0]{0}(1.356,0.85){$\delta^{0}t_1^{-2}t_2^{0}$}
	\psdot(1.34,1.13)
	\uput{0.05}[30]{0}(1.34,1.13){$\delta^{1}t_1^{2}t_2^{0}$}
	
	\uput{0.05}[0]{0}(-1,0.72){$\delta^{-\tfrac{3}{2}}t_1^{0}t_2^{-1}$}
	\uput{0.05}[0]{0}(-1,0.41){$\delta^{-\tfrac{3}{2}}t_1^{0}t_2^{-3}$}
	\uput{0.05}[180]{0}(1,0.72){$\delta^{-\tfrac{3}{2}}t_1^{-1}t_2^{-2}$}
	\uput{0.05}[180]{0}(1,0.41){$\delta^{-\tfrac{3}{2}}t_1^{1}t_2^{-2}$}
	\uput{0.05}[-90]{0}(-0.33,1){$\delta^{-1}t_1^{1}t_2^{-1}$}
	\uput{0.05}[-90]{0}(0.33,1){$\delta^{-2}t_1^{-1}t_2^{-3}$}
	
	\psdot[linecolor=blue](-1,0.8)
	\psdot[linecolor=blue](1,0.645)
	\psdot[linecolor=blue](-1,0.48)
	\psdot[linecolor=blue](1,0.34)
	\psdot[linecolor=blue](-0.33,1)
	\psdot[linecolor=blue](0.12,1)
	
	\psdot[linecolor=red](1,0.8)
	\psdot[linecolor=red](-1,0.645)
	\psdot[linecolor=red](1,0.48)
	\psdot[linecolor=red](-1,0.34)
	\psdot[linecolor=red](0.33,1)
	\psdot[linecolor=red](-0.12,1)

	\psdot(0,-0.505)
	\uput{0.05}[90]{0}(0,-0.505){$\delta^{0}t_1^{0}t_2^{0}$}
	\psdot(0,-1.074)
	\psline{->}(0,-0.8)(0,-1.074)
	\uput{0.3}[90]{0}(0,-1.074){$\delta^{0}t_1^{2}t_2^{2}$}
	\psdot(0,-1.2)
	\uput{0.05}[-90]{0}(0,-1.2){$\delta^{1}t_1^{-2}t_2^{-2}$}
	\psdot(0,-1.53)
	\uput{0.05}[-90]{0}(0,-1.53){$\delta^{1}t_1^{0}t_2^{0}$}
	
	\psdot(-1.356,-0.85)
	\uput{0.05}[180]{0}(-1.356,-0.85){$\delta^{0}t_1^{0}t_2^{2}$}
	\psdot(-1.34,-1.13)
	\uput{0.05}[210]{0}(-1.34,-1.13){$\delta^{1}t_1^{0}t_2^{-2}$}
	\psdot(1.356,-0.85)
	\uput{0.05}[0]{0}(1.356,-0.85){$\delta^{0}t_1^{2}t_2^{0}$}
	\psdot(1.34,-1.13)
	\uput{0.05}[-30]{0}(1.34,-1.13){$\delta^{1}t_1^{-2}t_2^{0}$}
	
	\uput{0.05}[0]{0}(-1,-0.41){$\delta^{-\tfrac{3}{2}}t_1^{0}t_2^{1}$}
	\uput{0.05}[0]{0}(-1,-0.72){$\delta^{-\tfrac{3}{2}}t_1^{0}t_2^{3}$}
	\uput{0.05}[180]{0}(1,-0.41){$\delta^{-\tfrac{3}{2}}t_1^{1}t_2^{2}$}
	\uput{0.05}[180]{0}(1,-0.72){$\delta^{-\tfrac{3}{2}}t_1^{-1}t_2^{2}$}
	\uput{0.05}[90]{0}(0.33,-1){$\delta^{-1}t_1^{-1}t_2^{1}$}
	\uput{0.05}[90]{0}(-0.33,-1){$\delta^{-2}t_1^{1}t_2^{3}$}
	
	\psdot[linecolor=red](-1,-0.8)
	\psdot[linecolor=red](1,-0.645)
	\psdot[linecolor=red](-1,-0.48)
	\psdot[linecolor=red](1,-0.34)
	\psdot[linecolor=red](-0.33,-1)
	\psdot[linecolor=red](0.12,-1)
	
	\psdot[linecolor=blue](1,-0.8)
	\psdot[linecolor=blue](-1,-0.645)
	\psdot[linecolor=blue](1,-0.48)
	\psdot[linecolor=blue](-1,-0.34)
	\psdot[linecolor=blue](0.33,-1)
	\psdot[linecolor=blue](-0.12,-1)

	\pscircle[fillstyle=solid, fillcolor=white](1,1){0.08}
	\pscircle[fillstyle=solid, fillcolor=white](-1,1){0.08}
	\pscircle[fillstyle=solid, fillcolor=white](1,-1){0.08}
	\pscircle[fillstyle=solid, fillcolor=white](-1,-1){0.08}
	
	\end{pspicture}
	}
	{\psset{unit=3}
		\begin{pspicture}(-2,-1.7)(2,1.64)
		\psline[linestyle=dotted](1,1)(1,-1)
		\psline[linestyle=dotted](1,-1)(-1,-1)
		\psline[linestyle=dotted](-1,-1)(-1,1)
		\psline[linestyle=dotted](-1,1)(1,1)
		\uput{0.45}[45]{0}(0,0){\textcolor{red}{$L_{\mr(T_1)}$}}
		\uput{0.4}[-90]{0}(0,0){\textcolor{red}{$L_{\mr(T_1)}$}}
		\uput{0.05}[90]{0}(0,0){\textcolor{blue}{$L_{T_2}$}}
				
		\psecurve[linecolor=red]%
		(0,1.5)%
		(-1.35,1.35)(-1.375,1)(-1,0.725)(0.5,1.3)(1.375,1)(1.35,0.55)%
		(0,0.55)%
		(-1.35,0.55)(-1.375,1)(-0.5,1.3)(1,0.725)(1.375,1)(1.35,1.35)%
		(0,1.5)%
		(-1.35,1.35)(-1.375,1)
		
		\rput{180}(0,0){
			\psecurve[linecolor=red]%
			(0,1.5)%
			(-1.35,1.35)(-1.375,1)(-1,0.725)(0.5,1.3)(1.375,1)(1.35,0.55)%
			(0,0.55)%
			(-1.35,0.55)(-1.375,1)(-0.5,1.3)(1,0.725)(1.375,1)(1.35,1.35)%
			(0,1.5)%
			(-1.35,1.35)(-1.375,1)
		}
		
		\psecurve[linecolor=blue](0,-1.65)(-1.55,-1.4)(-1.55,-0.35)(-0.2,0.4)(1.25,1.25)(1.55,0.7)(-1.55,-0.7)(-1.25,-1.3)(1.3,-0.7)(1.45,-1.45)(0,-1.65)(-1.55,-1.4)(-1.55,-0.35)
		
		\psset{dotsize=5pt}
		
		\psdot(0.015,0.55)
		\uput{0.05}[150]{0}(0.015,0.55){$\delta^{1}t_1^{0}t_2^{4}$}
		\psdot(0.4,0.89)
		\uput{0.05}[-90]{0}(0.4,0.89){$\delta^{1}t_1^{2}t_2^{4}$}
		\psdot(0.97,1.25)
		\uput{0.05}[90]{0}(0.97,1.25){$\delta^{1}t_1^{0}t_2^{2}$}
		\psdot(1.41,1.15)
		\uput{0.05}[30]{0}(1.41,1.15){$\delta^{1}t_1^{2}t_2^{2}$}
		
		\psdot(-1.41,-1.17)
		\uput{0.05}[-150]{0}(-1.41,-1.17){$\delta^{1}t_1^{0}t_2^{-4}$}
		\psdot(-0.63,-1.3)
		\uput{0.05}[-80]{0}(-0.63,-1.3){$\delta^{1}t_1^{0}t_2^{-2}$}
		\psdot(-0.12,-1.065)
		\uput{0.05}[-80]{0}(-0.12,-1.065){$\delta^{1}t_1^{-2}t_2^{-4}$}
		\psdot(1.43,-0.84)
		\uput{0.05}[20]{0}(1.43,-0.84){$\delta^{1}t_1^{-2}t_2^{-2}$}
		
		\psdot[linecolor=blue](0.53,1)
		\uput{0.05}[-30]{0}(0.53,1){$\delta^{-1}t_1^{1}t_2^{1}$}
		\psdot[linecolor=blue](1,0.19)
		\uput{0.05}[-30]{0}(1,0.19){$\delta^{-\tfrac{3}{2}}t_1^{1}t_2^{0}$}
		\psdot[linecolor=blue](1,-0.61)
		\uput{0.05}[150]{0}(1,-0.61){$\delta^{-\tfrac{3}{2}}t_1^{-1}t_2^{0}$}
		\psdot[linecolor=blue](-0.01,-1)
		\uput{0.05}[90]{0}(-0.01,-1){$\delta^{-1}t_1^{-1}t_2^{-1}$}
		\psdot[linecolor=blue](-1,-0.2)
		\uput{0.05}[150]{0}(-1,-0.2){$\delta^{-\tfrac{3}{2}}t_1^{0}t_2^{-1}$}
		\psdot[linecolor=blue](-1,0.07)
		\uput{0.05}[150]{0}(-1,0.07){$\delta^{-\tfrac{3}{2}}t_1^{0}t_2^{1}$}
		
		\psdot[linecolor=red](-1,0.413)
		\uput{0.05}[-150]{0}(-1,0.413){$\delta^{-\tfrac{3}{2}}t_1^{0}t_2^{-3}$}
		\psdot[linecolor=red](-1,0.726)
		\uput{0.05}[-150]{0}(-1,0.726){$\delta^{-\tfrac{3}{2}}t_1^{0}t_2^{-1}$}
		\psdot[linecolor=red](-0.22,1)
		\uput{0.05}[120]{0}(-0.22,1){$\delta^{-1}t_1^{1}t_2^{-1}$}
		\psdot[linecolor=red](1,0.413)
		\uput{0.05}[-30]{0}(1,0.413){$\delta^{-\tfrac{3}{2}}t_1^{1}t_2^{-2}$}
		\psdot[linecolor=red](1,0.726)
		\uput{0.05}[-30]{0}(1,0.726){$\delta^{-\tfrac{3}{2}}t_1^{-1}t_2^{-2}$}
		\psdot[linecolor=red](0.22,1)
		\uput{0.05}[90]{0}(0.22,1){$\delta^{-2}t_1^{-1}t_2^{-3}$}
		
		\psdot[linecolor=red](-1,-0.413)
		\uput{0.05}[30]{0}(-1,-0.413){$\delta^{-\tfrac{3}{2}}t_1^{0}t_2^{1}$}
		\psdot[linecolor=red](-1,-0.726)
		\uput{0.05}[30]{0}(-1,-0.726){$\delta^{-\tfrac{3}{2}}t_1^{0}t_2^{3}$}
		\psdot[linecolor=red](-0.22,-1)
		\uput{0.05}[210]{0}(-0.22,-1){$\delta^{-2}t_1^{1}t_2^{3}$}
		\psdot[linecolor=red](1,-0.413)
		\uput{0.05}[30]{0}(1,-0.413){$\delta^{-\tfrac{3}{2}}t_1^{1}t_2^{2}$}
		\psdot[linecolor=red](1,-0.726)
		\uput{0.05}[-135]{0}(1,-0.726){$\delta^{-\tfrac{3}{2}}t_1^{-1}t_2^{2}$}
		\psdot[linecolor=red](0.22,-1)
		\uput{0.05}[-40]{0}(0.22,-1){$\delta^{-1}t_1^{-1}t_2^{1}$}
				
		\pscircle[fillstyle=solid, fillcolor=white](1,1){0.08}
		\pscircle[fillstyle=solid, fillcolor=white](-1,1){0.08}
		\pscircle[fillstyle=solid, fillcolor=white](1,-1){0.08}
		\pscircle[fillstyle=solid, fillcolor=white](-1,-1){0.08}
		
		\end{pspicture}
	}
	\caption{The main calculation for Example~\ref{exa:KinoshitaTerasakaLink}.}\label{fig:KTLinkMixedPairing}
\end{figure}

\begin{figure}[p]
	\centering
\begin{subfigure}[b]{0.49\textwidth}\centering
	{
		\psset{unit=0.2,linewidth=\stringwidth}
		\begin{pspicture}[showgrid=false](-11,-9)(11,9)
		\rput{-45}(0,0){
			\psecurve[linecolor=red]
			(-13,-5)%
			(-9,-4)(-5,-5)(-4,-9)%
			(-5,-13)
			\psecurve[linewidth=\stringwhite,linecolor=white](-3,-1)%
			(-6,-2)%
			(-6,-6)%
			(-2,-6)%
			(-1,-3)
			\psecurve[linecolor=red](-3,-1)%
			(-6,-2)%
			(-6,-6)%
			(-2,-6)%
			(-1,-3)
			\psecurve[linewidth=\stringwhite,linecolor=white]
			(-13,-5)%
			(-9,-4)(-5,-5)(-4,-9)%
			\psecurve[linecolor=red]
			(-13,-5)%
			(-9,-4)(-5,-5)(-4,-9)%
			
			\psecurve[linecolor=red]%
			(5,-5)(1,-4)(-3,-5)(-3,-1)%
			(-6,-2)%
			(-6,-6)%
			\psecurve[linewidth=\stringwhite,linecolor=white]%
			(-6,-6)%
			(-2,-6)%
			(-1,-3)(-5,-3)(-4,1)(-5,5)
			\psecurve[linecolor=red]%
			(-6,-6)%
			(-2,-6)%
			(-1,-3)(-5,-3)(-4,1)(-5,5)
			\psecurve[linewidth=\stringwhite,linecolor=white]%
			(1,-4)(-3,-5)(-3,-1)%
			(-6,-2)%
			\psecurve[linecolor=red]%
			(1,-4)(-3,-5)(-3,-1)%
			(-6,-2)%
		
			\psline[linecolor=red]{<-}(1,-4)(0.9,-4.025)
			\psline[linecolor=red]{<-}(-4,-9)(-4,-8.9)
			
			\psecurve[linecolor=blue]
			(-5,3)%
			(-1,4)(3,3)(4,-1)%
			(3,-5)
			\psecurve[linewidth=\stringwhite,linecolor=white](5,7)%
			(2,6)%
			(2,2)%
			(6,2)%
			(7,5)
			\psecurve[linecolor=blue](5,7)%
			(2,6)%
			(2,2)%
			(6,2)%
			(7,5)
			\psecurve[linewidth=\stringwhite,linecolor=white]
			(-1,4)(3,3)(4,-1)%
			(3,-5)
			\psecurve[linecolor=blue]
			(-1,4)(3,3)(4,-1)%
			(3,-5)
			
			\psecurve[linecolor=blue]%
			(2,2)%
			(6,2)%
			(7,5)(3,5)(4,9)(3,13)
			\psecurve[linewidth=\stringwhite,linecolor=white]%
			(13,3)(9,4)(5,3)(5,7)%
			(2,6)%
			(2,2)%
			\psecurve[linecolor=blue]%
			(13,3)(9,4)(5,3)(5,7)%
			(2,6)%
			(2,2)%
			\psecurve[linewidth=\stringwhite,linecolor=white]%
			(6,2)%
			(7,5)(3,5)(4,9)
			\psecurve[linecolor=blue]%
			(6,2)%
			(7,5)(3,5)(4,9)

			\psline[linecolor=blue]{<-}(-1,4)(-0.9,4)
			\psline[linecolor=blue]{<-}(4,9)(3.975,8.9)
			
			\pscircle[linestyle=dotted](4,4){5}
			\pscircle[linestyle=dotted](-4,-4){5}
			
			\rput{45}(-10,0){$t_1$}
			\rput{45}(0,10){$t_2$}
			
			\rput{45}(1.25,1.25){$a$}
			\rput{45}(6.75,1.25){$b$}
			\rput{45}(6.5,6.5){$c$}
			\rput{45}(1.25,6.75){$d$}
			
			\rput{45}(-6.75,-6.75){$a$}
			\rput{45}(-1.25,-6.75){$b$}
			\rput{45}(-1.5,-1.5){$c$}
			\rput{45}(-6.75,-1.25){$d$}
			\psecurve{C-C}(8,-3)(-1,4)(-10,-3)(-9,-4)(-8,-3)
			\psecurve{C-C}(-3,8)(4,-1)(-3,-10)(-4,-9)(-3,-8)

			\psecurve[linewidth=\stringwhite,linecolor=white](-8,3)(1,-4)(10,3)(9,4)(8,3)
			\psecurve[linewidth=\stringwhite,linecolor=white](3,-8)(-4,1)(3,10)(4,9)(3,8)
			
			\psecurve{C-C}(-8,3)(1,-4)(10,3)(9,4)(8,3)
			\psecurve{C-C}(3,-8)(-4,1)(3,10)(4,9)(3,8)
		}
		\end{pspicture}
	}
	\caption{}\label{fig:KTLinkDiagram}
\end{subfigure}
\begin{subfigure}[b]{0.49\textwidth}\centering
{\psset{unit=1.5}
	\begin{pspicture}(-1.7,-1.2)(1.7,1.2)
	\psline[linecolor=lightgray](-1.5,0.6)(1.5,0.6)
	\psline[linecolor=lightgray](-1.5,-0.6)(1.5,-0.6)
	\psline[linecolor=lightgray](-1.2,0.9)(-1.2,-0.9)
	\psline[linecolor=lightgray](-0.6,0.9)(-0.6,-0.9)
	\psline[linecolor=lightgray](0.6,0.9)(0.6,-0.9)
	\psline[linecolor=lightgray](1.2,0.9)(1.2,-0.9)
	\psline{->}(0,1.1)(0,-1.1)
	\psline{->}(-1.7,0)(1.7,0)
	\uput{0.1}[-90]{0}(1.6,0){$t_2$}
	\uput{0.1}[0]{0}(0,-1){$t_1$}
	
	\rput(-1.2,0.6){
		\pscircle*(-0.025,0.025){0.05}
		\pscircle[fillcolor=white,fillstyle=solid](0.025,-0.025){0.05}
	}
	\rput(-1.2,0){
		\pscircle*(-0.025,0.025){0.05}
		\pscircle[fillcolor=white,fillstyle=solid](0.025,-0.025){0.05}
	}
	
	\rput(-0.6,0.6){
		\uput{0.1}[135]{0}(-0.025,0.025){$2$}
		\uput{0.1}[-45]{0}(0.025,-0.025){$2$}
		\pscircle*(-0.025,0.025){0.05}
		\pscircle[fillcolor=white,fillstyle=solid](0.025,-0.025){0.05}
	}
	\rput(-0.6,0){
		\uput{0.1}[135]{0}(-0.025,0.025){$2$}
		\uput{0.1}[-45]{0}(0.025,-0.025){$2$}
		\pscircle*(-0.025,0.025){0.05}
		\pscircle[fillcolor=white,fillstyle=solid](0.025,-0.025){0.05}
	}
	
	\rput(0,0.6){
		\pscircle*(-0.025,0.025){0.05}
		\pscircle[fillcolor=white,fillstyle=solid](0.025,-0.025){0.05}
	}
	\rput(0,0){
		\uput{0.1}[135]{0}(-0.025,0.025){$3$}
		\uput{0.1}[-45]{0}(0.025,-0.025){$3$}
		\pscircle*(-0.025,0.025){0.05}
		\pscircle[fillcolor=white,fillstyle=solid](0.025,-0.025){0.05}
	}
	\rput(0,-0.6){
		\pscircle*(-0.025,0.025){0.05}
		\pscircle[fillcolor=white,fillstyle=solid](0.025,-0.025){0.05}
	}
	
	\rput(0.6,0){
		\uput{0.1}[135]{0}(-0.025,0.025){$2$}
		\uput{0.1}[-45]{0}(0.025,-0.025){$2$}
		\pscircle*(-0.025,0.025){0.05}
		\pscircle[fillcolor=white,fillstyle=solid](0.025,-0.025){0.05}
	}
	\rput(0.6,-0.6){
		\uput{0.1}[135]{0}(-0.025,0.025){$2$}
		\uput{0.1}[-45]{0}(0.025,-0.025){$2$}
		\pscircle*(-0.025,0.025){0.05}
		\pscircle[fillcolor=white,fillstyle=solid](0.025,-0.025){0.05}
	}
	
	\rput(1.2,0){
		\pscircle*(-0.025,0.025){0.05}
		\pscircle[fillcolor=white,fillstyle=solid](0.025,-0.025){0.05}
	}
	\rput(1.2,-0.6){
		\pscircle*(-0.025,0.025){0.05}
		\pscircle[fillcolor=white,fillstyle=solid](0.025,-0.025){0.05}
	}
	\end{pspicture}
}
\caption{}\label{fig:KTLinkFinalResult}
\end{subfigure}
\\
\begin{subfigure}[b]{\textwidth}\centering
	{\psset{unit=0.75}
\begin{tabular}{cccc}
	\begin{pspicture}(-1.7,-1.7)(1.7,1.7)	
	\uput{0}[90]{0}(0,0.5){\textcolor{blue}{$L_{T_2}$}}
	\psecurve{->}(1.7,1)(0,0.5)(1.7,0)(3.4,0.5)
	\uput{0}[0]{0}(-1.7,0){\textcolor{red}{$L_{\mr(T_1)}$}}
	\psecurve{->}(-0.6,-1.7)(-0.3,0)(0,-1.7)(-0.3,-3.4)
	\end{pspicture}
	&
	\begin{pspicture}(-1.7,-1.7)(1.7,1.7)
	\psline[linestyle=dotted](1,1)(1,-1)
	\psline[linestyle=dotted](1,-1)(-1,-1)
	\psline[linestyle=dotted](-1,-1)(-1,1)
	\psline[linestyle=dotted](-1,1)(1,1)
		
	\psecurve[linecolor=blue]%
	(0,1.5)%
	(-1.35,1.35)(-1.375,1)(-1,0.725)(0.5,1.3)(1.375,1)(1.35,0.55)%
	(0,0.55)%
	(-1.35,0.55)(-1.375,1)(-0.5,1.3)(1,0.725)(1.375,1)(1.35,1.35)%
	(0,1.5)%
	(-1.35,1.35)(-1.375,1)
	
	\pscircle[fillstyle=solid, fillcolor=white](1,1){0.08}
	\pscircle[fillstyle=solid, fillcolor=white](-1,1){0.08}
	\pscircle[fillstyle=solid, fillcolor=white](1,-1){0.08}
	\pscircle[fillstyle=solid, fillcolor=white](-1,-1){0.08}
	
	\end{pspicture}
	&
	\begin{pspicture}(-1.7,-1.7)(1.7,1.7)
	\psline[linestyle=dotted](1,1)(1,-1)
	\psline[linestyle=dotted](1,-1)(-1,-1)
	\psline[linestyle=dotted](-1,-1)(-1,1)
	\psline[linestyle=dotted](-1,1)(1,1)
		
	\psecurve[linecolor=blue]%
	(1.2,-0.95)(0,-1.5)(-1.5,-1)(1.25,1.15)(-1.2,-1.05)(1.2,-0.95)(0,-1.5)(-1.5,-1)
	
	\pscircle[fillstyle=solid, fillcolor=white](1,1){0.08}
	\pscircle[fillstyle=solid, fillcolor=white](-1,1){0.08}
	\pscircle[fillstyle=solid, fillcolor=white](1,-1){0.08}
	\pscircle[fillstyle=solid, fillcolor=white](-1,-1){0.08}
	\end{pspicture}
	&
	\begin{pspicture}(-1.7,-1.7)(1.7,1.7)
	\psline[linestyle=dotted](1,1)(1,-1)
	\psline[linestyle=dotted](1,-1)(-1,-1)
	\psline[linestyle=dotted](-1,-1)(-1,1)
	\psline[linestyle=dotted](-1,1)(1,1)
		
	\rput{180}(0,0){
		\psecurve[linecolor=blue]%
		(0,1.5)%
		(-1.35,1.35)(-1.375,1)(-1,0.725)(0.5,1.3)(1.375,1)(1.35,0.55)%
		(0,0.55)%
		(-1.35,0.55)(-1.375,1)(-0.5,1.3)(1,0.725)(1.375,1)(1.35,1.35)%
		(0,1.5)%
		(-1.35,1.35)(-1.375,1)
	}
	
	\pscircle[fillstyle=solid, fillcolor=white](1,1){0.08}
	\pscircle[fillstyle=solid, fillcolor=white](-1,1){0.08}
	\pscircle[fillstyle=solid, fillcolor=white](1,-1){0.08}
	\pscircle[fillstyle=solid, fillcolor=white](-1,-1){0.08}
	\end{pspicture}
	\\
	\begin{pspicture}(-1.7,-1.7)(1.7,1.7)	
	\psline[linestyle=dotted](1,1)(1,-1)
	\psline[linestyle=dotted](1,-1)(-1,-1)
	\psline[linestyle=dotted](-1,-1)(-1,1)
	\psline[linestyle=dotted](-1,1)(1,1)
	
	\psecurve[linecolor=red]%
	(0,1.5)%
	(-1.35,1.35)(-1.375,1)(-1,0.725)(0.5,1.3)(1.375,1)(1.35,0.55)%
	(0,0.55)%
	(-1.35,0.55)(-1.375,1)(-0.5,1.3)(1,0.725)(1.375,1)(1.35,1.35)%
	(0,1.5)%
	(-1.35,1.35)(-1.375,1)
	
	\pscircle[fillstyle=solid, fillcolor=white](1,1){0.08}
	\pscircle[fillstyle=solid, fillcolor=white](-1,1){0.08}
	\pscircle[fillstyle=solid, fillcolor=white](1,-1){0.08}
	\pscircle[fillstyle=solid, fillcolor=white](-1,-1){0.08}
	
	\end{pspicture}
	&
	\begin{pspicture}(-1.7,-1.7)(1.7,1.7)
	\psline[linecolor=lightgray](-1.5,0.6)(1.5,0.6)
	\psline[linecolor=lightgray](-1.5,-0.6)(1.5,-0.6)
	\psline[linecolor=lightgray](-1.2,0.9)(-1.2,-0.9)
	\psline[linecolor=lightgray](-0.6,0.9)(-0.6,-0.9)
	\psline[linecolor=lightgray](0.6,0.9)(0.6,-0.9)
	\psline[linecolor=lightgray](1.2,0.9)(1.2,-0.9)
	\psline{->}(0,1.1)(0,-1.1)
	\psline{->}(-1.7,0)(1.7,0)
	\uput{0.2}[-90]{0}(1.6,0){$t_2$}
	\uput{0.2}[0]{0}(0,-1){$t_1$}
	
	\pscircle*(0.6,-0.6){0.1}
	\pscircle*(0,-0.6){0.1}
	\pscircle*(0.6,0){0.1}
	\pscircle*(0.05,-0.05){0.1}
	\pscircle[fillcolor=white,fillstyle=solid](-0.6,0.6){0.1}
	\pscircle[fillcolor=white,fillstyle=solid](0,0.6){0.1}
	\pscircle[fillcolor=white,fillstyle=solid](-0.6,0){0.1}
	\pscircle[fillcolor=white,fillstyle=solid](-0.05,0.05){0.1}
	\end{pspicture}
	&
	\begin{pspicture}(-1.7,-1.7)(1.7,1.7)
	\psline[linecolor=lightgray](-1.5,0.6)(1.5,0.6)
	\psline[linecolor=lightgray](-1.5,-0.6)(1.5,-0.6)
	\psline[linecolor=lightgray](-1.2,0.9)(-1.2,-0.9)
	\psline[linecolor=lightgray](-0.6,0.9)(-0.6,-0.9)
	\psline[linecolor=lightgray](0.6,0.9)(0.6,-0.9)
	\psline[linecolor=lightgray](1.2,0.9)(1.2,-0.9)
	\psline{->}(0,1.1)(0,-1.1)
	\psline{->}(-1.7,0)(1.7,0)
	\uput{0.2}[-90]{0}(1.6,0){$t_2$}
	\uput{0.2}[0]{0}(0,-1){$t_1$}
	
	\pscircle*(1.2,-0.6){0.1}
	\pscircle*(0.6,-0.6){0.1}
	\pscircle*(1.2,0){0.1}
	\pscircle*(0.6,0){0.1}
	\end{pspicture}
	&
	\begin{pspicture}(-1.7,-1.7)(1.7,1.7)
	\psline[linecolor=lightgray](-1.5,0.6)(1.5,0.6)
	\psline[linecolor=lightgray](-1.5,-0.6)(1.5,-0.6)
	\psline[linecolor=lightgray](-1.2,0.9)(-1.2,-0.9)
	\psline[linecolor=lightgray](-0.6,0.9)(-0.6,-0.9)
	\psline[linecolor=lightgray](0.6,0.9)(0.6,-0.9)
	\psline[linecolor=lightgray](1.2,0.9)(1.2,-0.9)
	\psline{->}(0,1.1)(0,-1.1)
	\psline{->}(-1.7,0)(1.7,0)
	\uput{0.2}[-90]{0}(1.6,0){$t_2$}
	\uput{0.2}[0]{0}(0,-1){$t_1$}
	\end{pspicture}
	\\
	\begin{pspicture}(-1.7,-1.7)(1.7,1.7)
	\psline[linestyle=dotted](1,1)(1,-1)
	\psline[linestyle=dotted](1,-1)(-1,-1)
	\psline[linestyle=dotted](-1,-1)(-1,1)
	\psline[linestyle=dotted](-1,1)(1,1)
		
	\psecurve[linecolor=red]%
	(1.2,-0.95)(0,-1.5)(-1.5,-1)(1.25,1.15)(-1.2,-1.05)(1.2,-0.95)(0,-1.5)(-1.5,-1)
	
	\pscircle[fillstyle=solid, fillcolor=white](1,1){0.08}
	\pscircle[fillstyle=solid, fillcolor=white](-1,1){0.08}
	\pscircle[fillstyle=solid, fillcolor=white](1,-1){0.08}
	\pscircle[fillstyle=solid, fillcolor=white](-1,-1){0.08}
	\end{pspicture}
	&
	\begin{pspicture}(-1.7,-1.7)(1.7,1.7)
	\psline[linecolor=lightgray](-1.5,0.6)(1.5,0.6)
	\psline[linecolor=lightgray](-1.5,-0.6)(1.5,-0.6)
	\psline[linecolor=lightgray](-1.2,0.9)(-1.2,-0.9)
	\psline[linecolor=lightgray](-0.6,0.9)(-0.6,-0.9)
	\psline[linecolor=lightgray](0.6,0.9)(0.6,-0.9)
	\psline[linecolor=lightgray](1.2,0.9)(1.2,-0.9)
	\psline{->}(0,1.1)(0,-1.1)
	\psline{->}(-1.7,0)(1.7,0)
	\uput{0.2}[-90]{0}(1.6,0){$t_2$}
	\uput{0.2}[0]{0}(0,-1){$t_1$}
	
	\pscircle[fillstyle=solid, fillcolor=white](-1.2,0.6){0.1}
	\pscircle[fillstyle=solid, fillcolor=white](-0.6,0.6){0.1}
	\pscircle[fillstyle=solid, fillcolor=white](-1.2,0){0.1}
	\pscircle[fillstyle=solid, fillcolor=white](-0.6,0){0.1}
	\end{pspicture}
	&
	\begin{pspicture}(-1.7,-1.7)(1.7,1.7)
	\psline[linecolor=lightgray](-1.5,0.6)(1.5,0.6)
	\psline[linecolor=lightgray](-1.5,-0.6)(1.5,-0.6)
	\psline[linecolor=lightgray](-1.2,0.9)(-1.2,-0.9)
	\psline[linecolor=lightgray](-0.6,0.9)(-0.6,-0.9)
	\psline[linecolor=lightgray](0.6,0.9)(0.6,-0.9)
	\psline[linecolor=lightgray](1.2,0.9)(1.2,-0.9)
	\psline{->}(0,1.1)(0,-1.1)
	\psline{->}(-1.7,0)(1.7,0)
	\uput{0.2}[-90]{0}(1.6,0){$t_2$}
	\uput{0.2}[0]{0}(0,-1){$t_1$}
	
	\pscircle*(0.05,-0.05){0.1}
	\pscircle[fillcolor=white,fillstyle=solid](-0.05,0.05){0.1}
	\end{pspicture}
	&
	\begin{pspicture}(-1.7,-1.7)(1.7,1.7)
	\psline[linecolor=lightgray](-1.5,0.6)(1.5,0.6)
	\psline[linecolor=lightgray](-1.5,-0.6)(1.5,-0.6)
	\psline[linecolor=lightgray](-1.2,0.9)(-1.2,-0.9)
	\psline[linecolor=lightgray](-0.6,0.9)(-0.6,-0.9)
	\psline[linecolor=lightgray](0.6,0.9)(0.6,-0.9)
	\psline[linecolor=lightgray](1.2,0.9)(1.2,-0.9)
	\psline{->}(0,1.1)(0,-1.1)
	\psline{->}(-1.7,0)(1.7,0)
	\uput{0.2}[-90]{0}(1.6,0){$t_2$}
	\uput{0.2}[0]{0}(0,-1){$t_1$}
	
	\pscircle[fillstyle=solid, fillcolor=white](1.2,-0.6){0.1}
	\pscircle[fillstyle=solid, fillcolor=white](0.6,-0.6){0.1}
	\pscircle[fillstyle=solid, fillcolor=white](1.2,0){0.1}
	\pscircle[fillstyle=solid, fillcolor=white](0.6,0){0.1}
	\end{pspicture}
	\\
	\begin{pspicture}(-1.7,-1.7)(1.7,1.7)
	\psline[linestyle=dotted](1,1)(1,-1)
	\psline[linestyle=dotted](1,-1)(-1,-1)
	\psline[linestyle=dotted](-1,-1)(-1,1)
	\psline[linestyle=dotted](-1,1)(1,1)
		
	\rput{180}(0,0){
		\psecurve[linecolor=red]%
		(0,1.5)%
		(-1.35,1.35)(-1.375,1)(-1,0.725)(0.5,1.3)(1.375,1)(1.35,0.55)%
		(0,0.55)%
		(-1.35,0.55)(-1.375,1)(-0.5,1.3)(1,0.725)(1.375,1)(1.35,1.35)%
		(0,1.5)%
		(-1.35,1.35)(-1.375,1)
	}
	
	\pscircle[fillstyle=solid, fillcolor=white](1,1){0.08}
	\pscircle[fillstyle=solid, fillcolor=white](-1,1){0.08}
	\pscircle[fillstyle=solid, fillcolor=white](1,-1){0.08}
	\pscircle[fillstyle=solid, fillcolor=white](-1,-1){0.08}
	\end{pspicture}
	&
	\begin{pspicture}(-1.7,-1.7)(1.7,1.7)
	\psline[linecolor=lightgray](-1.5,0.6)(1.5,0.6)
	\psline[linecolor=lightgray](-1.5,-0.6)(1.5,-0.6)
	\psline[linecolor=lightgray](-1.2,0.9)(-1.2,-0.9)
	\psline[linecolor=lightgray](-0.6,0.9)(-0.6,-0.9)
	\psline[linecolor=lightgray](0.6,0.9)(0.6,-0.9)
	\psline[linecolor=lightgray](1.2,0.9)(1.2,-0.9)
	\psline{->}(0,1.1)(0,-1.1)
	\psline{->}(-1.7,0)(1.7,0)
	\uput{0.2}[-90]{0}(1.6,0){$t_2$}
	\uput{0.2}[0]{0}(0,-1){$t_1$}
	\end{pspicture}
	&
	\begin{pspicture}(-1.7,-1.7)(1.7,1.7)
	\psline[linecolor=lightgray](-1.5,0.6)(1.5,0.6)
	\psline[linecolor=lightgray](-1.5,-0.6)(1.5,-0.6)
	\psline[linecolor=lightgray](-1.2,0.9)(-1.2,-0.9)
	\psline[linecolor=lightgray](-0.6,0.9)(-0.6,-0.9)
	\psline[linecolor=lightgray](0.6,0.9)(0.6,-0.9)
	\psline[linecolor=lightgray](1.2,0.9)(1.2,-0.9)
	\psline{->}(0,1.1)(0,-1.1)
	\psline{->}(-1.7,0)(1.7,0)
	\uput{0.2}[-90]{0}(1.6,0){$t_2$}
	\uput{0.2}[0]{0}(0,-1){$t_1$}
	
	\pscircle*(-1.2,0.6){0.1}
	\pscircle*(-0.6,0.6){0.1}
	\pscircle*(-1.2,0){0.1}
	\pscircle*(-0.6,0){0.1}
	\end{pspicture}
	&
	\begin{pspicture}(-1.7,-1.7)(1.7,1.7)
	\psline[linecolor=lightgray](-1.5,0.6)(1.5,0.6)
	\psline[linecolor=lightgray](-1.5,-0.6)(1.5,-0.6)
	\psline[linecolor=lightgray](-1.2,0.9)(-1.2,-0.9)
	\psline[linecolor=lightgray](-0.6,0.9)(-0.6,-0.9)
	\psline[linecolor=lightgray](0.6,0.9)(0.6,-0.9)
	\psline[linecolor=lightgray](1.2,0.9)(1.2,-0.9)
	\psline{->}(0,1.1)(0,-1.1)
	\psline{->}(-1.7,0)(1.7,0)
	\uput{0.2}[-90]{0}(1.6,0){$t_2$}
	\uput{0.2}[0]{0}(0,-1){$t_1$}
	
	\pscircle*(-0.6,0.6){0.1}
	\pscircle*(0,0.6){0.1}
	\pscircle*(-0.6,0){0.1}
	\pscircle*(-0.05,0.05){0.1}
	\pscircle[fillcolor=white,fillstyle=solid](0.6,-0.6){0.1}
	\pscircle[fillcolor=white,fillstyle=solid](0,-0.6){0.1}
	\pscircle[fillcolor=white,fillstyle=solid](0.6,0){0.1}
	\pscircle[fillcolor=white,fillstyle=solid](0.05,-0.05){0.1}
	\end{pspicture}
\end{tabular}
}
\caption{}\label{fig:KTLinkTable}
\end{subfigure}
\caption{A tangle decomposition of the Kinoshita-Terasaka link L10n36 \protect\cite[Figure~17]{OSHFLThurston} (a) and the intersection theory of the corresponding tangle invariants (b). The table (c) shows the intersection theory for each pair of components separately. The main calculation is done in  Figure~\ref{fig:KTLinkMixedPairing}. In (b) and (c), generators of the homology groups are denoted by
\protect\begin{pspicture}(-0.1,-0.1)(0.1,0.1)
\protect\pscircle[fillcolor=white,fillstyle=solid]{0.075}
\protect\end{pspicture}
or 
\protect\begin{pspicture}(-0.1,-0.1)(0.1,0.1)
\protect\pscircle*{0.075}
\protect\end{pspicture}, 
	depending on whether their $\delta$-grading is 0 or 1. They are arranged in the coordinate systems according to their Alexander gradings. In (b), the labels $2$ and $3$ of some generators indicate their multiplicity. }\label{fig:KTLink}
\end{figure}

%
%
%
%
%
%
%
%
%
%
%

\begin{example}[(Kinoshita-Terasaka link)]\label{exa:KinoshitaTerasakaLink}
	The Kinoshita-Terasaka link $L_{KT}$ can be decomposed into the $(2,-3)$-pretzel tangle $\textcolor{blue}{T_2}=T_{2,-3}$ and $\textcolor{red}{T_1}=\mr(T_{2,-3})$, as shown in Figure~\ref{fig:KTLinkDiagram}. Thus, its link Floer homology can be computed as the self-pairing of $L_{T_{2,-3}}$, ie
	$$\HFL(L_{KT})\cong\LagrangianFH(L_{T_{2,-3}},L_{T_{2,-3}}).$$
	The Lagrangian Floer homology can be computed for each pair of components separately. The result is shown in the table of Figure~\ref{fig:KTLinkTable}. The final result is shown in Figure~\ref{fig:KTLinkFinalResult}. Oszváth and Szabó computed the link Floer homology of $L_{KT}$ in~\cite[Figure~24]{OSHFLThurston}. Note that our orientation of the strand $t_1$ is opposite to the one in~\cite[Figure~17]{OSHFLThurston}, so our result agrees with theirs after reversing the Alexander grading $t_1$. 
	
	Let us discuss the computation of the table in Figure~\ref{fig:KTLinkTable} in more detail. The two opposite immersed components are easiest to pair, since they can be homotoped such that they do not intersect each other. The intersection of each immersed component with itself, however, gives eight generators each. This computation is done in the upper half of Figure~\ref{fig:KTLinkMixedPairing}. The self-intersection of the embedded component with itself is equal to the link Floer homology of the unlink, ie two generators in gradings $\delta^0t_1^0t_2^0$ and $\delta^1t_1^0t_2^0$, respectively. The lower half of Figure~\ref{fig:KTLinkMixedPairing} shows the pairing of the two immersed components with the embedded components. The opposite pairing is done similarly by switching red and blue curves, which reverses the Alexander gradings of intersection points between the curves and acts like $\delta\mapsto 1-\delta$ on their $\delta$-gradings. 
\end{example}

%% file: sections/Applications.tex

\section{Applications}\label{sec:ApplicationsGlueing}

\subsection{Rational tangle detection}

 \begin{observation}\label{obs:AlexGradingOfCFTdLoops}
 	The Alexander grading on $\CFTd(T,M)$ implies that for a tangle $T$, each loop in $L_T$ lies in the kernel of 
 	\[\pi_1(\partial M\smallsetminus \partial T)\rightarrow\pi_1(M\smallsetminus \nu(T))\rightarrow H_1(M\smallsetminus \nu(T)),\]
 	where the first map is induced by the inclusion and the second is the Abelianization map.
 \end{observation}

\begin{theorem}\label{thm:CFTdDetectsRatTan}
	A 4-ended tangle \(T\) in the 3-ball is rational iff \(L_T\) is a single embedded loop with the unique 1-dimensional local system.
\end{theorem}

\begin{proof}
	The only-if direction is simply a calculation, see Example~\ref{exa:CFTdRatTang}. Conversely, suppose $L_T$ is a single loop which corresponds to an embedded loop on the 4-punctured sphere. It divides the sphere into two disc components, each of which has at least one puncture, since the loop is not nullhomotopic. By Observation~\ref{obs:AlexGradingOfCFTdLoops}, there are exactly two punctures in each disc, so $L_T$ agrees with $L_{T'}$ for some rational tangle $T'$. Then also $L_{\mr(T)}$ agrees with $L_{\mr(T')}$. Let $L=L(T_1,T_2)$ be the link obtained by pairing $T_2=T$ with $T_1=\mr(T)$. If we pair $T'$ with $\mr(T')$, we obtain the 2-component unlink $\bigcirc\amalg\bigcirc$. So by Theorem~\ref{thm:CFTdGlueingAsMorphism},
	\[
	\HFL(L)\cong\LagrangianFH(L_T,L_T)\cong\LagrangianFH(L_{T'},L_{T'})\cong\HFL(\bigcirc\amalg\bigcirc).
	\]
	We now apply the fact that link Floer homology detects unlinks \cite{OSHFLThurston}, so $L$ is the 2-component unlink. The following lemma finishes the proof.
\end{proof}
\begin{lemma}\label{lem:RatTanDet}
	Let \(T_1\) and \(T_2\) be two 4-ended tangles without closed components that glue together to the 2-component unlink. Then either \(T_1\) or \(T_2\) is a rational tangle. 
\end{lemma}
\begin{proof}
	Let $S$ be the 4-punctured sphere along which we glue $T_1$ and $T_2$. Let $U$ be the sphere that separates the two unknot components and assume that $S$ and $U$ intersect transversely in a disjoint union of circles. We now proceed by induction on the number of circles in $S\cap U$. First of all, this intersection is non-empty, since $U$ is separating. So we can always find a curve $\gamma$ that bounds a disc $D$ in $U$ which does not contain any other curves in $U\cap S$. If $\gamma$ bounds a disc $D'$ in $S$, $D\cup D'$ bounds a 3-ball, which we can use as a homotopy for $U$ to remove $\gamma$ (along with any other components of $S\cap U$ in $D'$), so we are done by the induction hypothesis. If $\gamma$ does not bound a disc in $S$, it separates two punctures from the other two. So $D$ separates the two strands in $T_1$ or $T_2$. They must obviously be unknotted, since the connected sum of two knots is the unknot iff both knots are unknots. Thus either $T_1$ or $T_2$ is rational.
\end{proof}

\begin{wrapfigure}{r}{0.3333\textwidth}
	\centering
	\psset{unit=0.35}\vspace*{-25pt}
	\begin{subfigure}[b]{0.116\textwidth}\centering
		$n\left\{\raisebox{-1.1cm}{
			\begin{pspicture}(-1.1,-3.3)(1.1,3.3)
			\psset{linewidth=\stringwidth}
			\psecurve(1,5)(-1,3)(1,1)(-1,-1)
			\psecurve[linecolor=white,linewidth=\stringwhite](-1,5)(1,3)(-1,1)(1,-1)
			\psecurve(-1,5)(1,3)(-1,1)(1,-1)
			\psline[linestyle=dotted,dotsep=0.4](0,-0.6)(0,0.6)
			\psecurve(-1,-5)(1,-3)(-1,-1)(1,1)
			\psecurve[linecolor=white,linewidth=\stringwhite](1,-5)(-1,-3)(1,-1)(-1,1)
			\psecurve(1,-5)(-1,-3)(1,-1)(-1,1)
			\psline{->}(1,3)(1.07,3.2)
			\psline{->}(-1,3)(-1.07,3.2)
			\end{pspicture}}\right.\quad
		$
		\caption{$T_{n}$}\label{fig:OrientedSkeinRelationTn}
	\end{subfigure}
	\begin{subfigure}[b]{0.116\textwidth}\centering
		$n\left\{\raisebox{-1.1cm}{
			\begin{pspicture}(-1.1,-3.3)(1.1,3.3)
			\psset{linewidth=\stringwidth}
			\psecurve(-1,5)(1,3)(-1,1)(1,-1)
			\psecurve[linecolor=white,linewidth=\stringwhite](1,5)(-1,3)(1,1)(-1,-1)
			\psecurve(1,5)(-1,3)(1,1)(-1,-1)
			\psline[linestyle=dotted,dotsep=0.4](0,-0.6)(0,0.6)
			\psecurve(1,-5)(-1,-3)(1,-1)(-1,1)
			\psecurve[linecolor=white,linewidth=\stringwhite](-1,-5)(1,-3)(-1,-1)(1,1)
			\psecurve(-1,-5)(1,-3)(-1,-1)(1,1)
			\psline{->}(1,3)(1.07,3.2)
			\psline{->}(-1,3)(-1.07,3.2)
			\end{pspicture}}\right.\quad
		$
		\caption{$T_{-n}$}\label{fig:OrientedSkeinRelationTmn}
	\end{subfigure}
	\begin{subfigure}[b]{0.085\textwidth}\centering
		\begin{pspicture}(-1.1,-3.3)(1.1,3.3)
		\psset{linewidth=\stringwidth}
		\psecurve{<-}(-2,6)(-1,3.1)(-1,-3)(-2,-6)
		\psecurve{<-}(2,6)(1,3.1)(1,-3)(2,-6)
		\end{pspicture}
		\caption{$T_0$}\label{fig:OrientedSkeinRelationT0}
	\end{subfigure}
	\caption{Basic tangles.}\label{fig:OrientedSkeinRelationTangles}\vspace*{-25pt}
\end{wrapfigure}
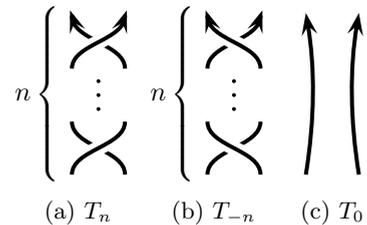

\subsection{Skein exact sequences}

We start with a slight generalisation of Ozsv\'{a}th and Szab\'{o}'s exact triangle \cite{OSHFK} which categorifies the oriented skein relation for the Alexander polynomial. However, we remind the reader that all gradings on link Floer homology should be regarded as relative, see Remark~\ref{rmk:RelativeGradings}, so the graded version of the following theorem is not quite as strong as Ozsv\'{a}th and Szab\'{o}'s result in the case $n=1$.

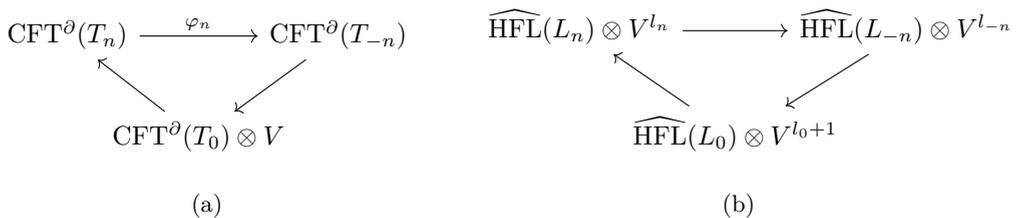
\begin{figure}[b]
\centering
\psset{unit=0.3}
\begin{subfigure}[b]{0.48\textwidth}
	{\normalsize
	\[\begin{tikzcd}[row sep=0.7cm, column sep=-0.5cm]
	\CFTd(T_{n})
	\arrow{rr}{\varphi_n}
	& &
	\CFTd(T_{-n})
	\arrow{dl}
	\\
	&
	\CFTd(T_{0})\otimes V
	\arrow{lu}
	\end{tikzcd}\]
	}
	\caption{}\label{fig:nTwistSkeinRelationTangle}
\end{subfigure}
\begin{subfigure}[b]{0.48\textwidth}
	{\normalsize
	\[\begin{tikzcd}[row sep=0.7cm, column sep=-0.8cm]
	\HFL(L_{n})\otimes V^{l_{n}}
	\arrow{rr}
	& &
	\HFL(L_{-n})\otimes V^{l_{-n}}
	\arrow{dl}
	\\
	&
	\HFL(L_{0})\otimes V^{l_0+1}
	\arrow{lu}
	\end{tikzcd}\]	
	}
	\caption{}\label{fig:nTwistSkeinRelationLinks}
\end{subfigure}
\caption{The skein exact triangles from Theorem~\ref{thm:nTwistSkeinRelation}.}\label{fig:OrientedSkeinRelation}
\end{figure}

\begin{theorem}[($n$-twist skein exact triangle)]\label{thm:nTwistSkeinRelation}
	Let \(T_n\) be the positive \(n\)-twist tangle, \(T_{-n}\) the negative \(n\)-twist tangle and \(T_0\) the trivial tangle, see Figure~\ref{fig:OrientedSkeinRelationTangles}. Furthermore, let \(V\) be a 2-dimensional vector space supported in degrees \(\delta^0 t^{n}\) and \(\delta^0 t^{-n}\), where \(t\) is the colour of the two open strands. Then there is an exact triangle shown in Figure~\ref{fig:nTwistSkeinRelationTangle}.
	\(\varphi_n\) preserves the (univariate) Alexander grading and changes \(\delta\)- and homological gradings by \(+1\) and \(-1\), respectively; the other two maps preserve all three gradings. Moreover, given three links \(L_{n}\), \(L_{-n}\) and \(L_0\) in \(S^3\), which agree outside a closed 3-ball and in this closed 3-ball agree with the 4-ended tangles \(T_{n}\), \(T_{-n}\) and \(T_{0}\), respectively, then the above triangle together with the Glueing Theorem induces an exact triangle shown in Figure~\ref{fig:nTwistSkeinRelationLinks}, where for \(i\in\{n,-n,0\}\), \(l_i\) is either \(0\) or \(1\), depending on whether the two strands in \(T_i\) belong to different or the same components in \(L_i\), respectively. 
\end{theorem}

\begin{remark}\label{rem:nTwistSkeinRelation}
Similar results hold for other orientations; for $n$ even, also multivariate Alexander gradings are preserved.
\end{remark}

\begin{figure}[t]
	\centering
	\begin{subfigure}[b]{0.23\textwidth}\centering
		{\psset{unit=1}
			\begin{pspicture}(-1.6,-1.7)(1.6,1.7)
			\psecurve(1.2,-0.95)(0,-1.5)(-1.5,-1)(1.25,1.15)(-1.2,-1.05)(1.2,-0.95)(0,-1.5)(-1.5,-1)
			
			\psline[linestyle=dotted](1,1)(1,-1)
			\psline[linestyle=dotted](1,-1)(-1,-1)
			\psline[linestyle=dotted](-1,-1)(-1,1)
			\psline[linestyle=dotted](-1,1)(1,1)
			
			\pscircle[fillstyle=solid, fillcolor=white](1,1){0.08}
			\pscircle[fillstyle=solid, fillcolor=white](-1,1){0.08}
			\pscircle[fillstyle=solid, fillcolor=white](1,-1){0.08}
			\pscircle[fillstyle=solid, fillcolor=white](-1,-1){0.08}
			
			\psset{dotsize=5pt}
			\psdot[linecolor=red](-1,0.18)
			\psdot[linecolor=red](-1,-0.48)
			\psdot[linecolor=gold](0,1)
			\psdot[linecolor=blue](0.04,-1)
			\psdot[linecolor=darkgreen](1,0.65)
			\psdot[linecolor=darkgreen](1,-0.78)
			
			\end{pspicture}
		}
		\caption{$L_{T_2}$}\label{fig:OrientSkeinRatTangleII}
	\end{subfigure}
	\begin{subfigure}[b]{0.23\textwidth}\centering
		{\psset{unit=1}
			\begin{pspicture}(-1.6,-1.7)(1.6,1.7)
			\psecurve(-1.2,-0.95)(0,-1.5)(1.5,-1)(-1.25,1.15)(1.2,-1.05)(-1.2,-0.95)(0,-1.5)(1.5,-1)
			
			\psline[linestyle=dotted](1,1)(1,-1)
			\psline[linestyle=dotted](1,-1)(-1,-1)
			\psline[linestyle=dotted](-1,-1)(-1,1)
			\psline[linestyle=dotted](-1,1)(1,1)
			
			\pscircle[fillstyle=solid, fillcolor=white](1,1){0.08}
			\pscircle[fillstyle=solid, fillcolor=white](-1,1){0.08}
			\pscircle[fillstyle=solid, fillcolor=white](1,-1){0.08}
			\pscircle[fillstyle=solid, fillcolor=white](-1,-1){0.08}
			
			\psset{dotsize=5pt}
			\psdot[linecolor=darkgreen](1,0.18)
			\psdot[linecolor=darkgreen](1,-0.48)
			\psdot[linecolor=gold](0,1)
			\psdot[linecolor=blue](-0.04,-1)
			\psdot[linecolor=red](-1,0.65)
			\psdot[linecolor=red](-1,-0.78)
			
			\end{pspicture}
		}
		\caption{$L_{T_{-2}}$}\label{fig:OrientSkeinRatTanglemII}
	\end{subfigure}
	\begin{subfigure}[b]{0.23\textwidth}\centering
		{\psset{unit=1}
			\begin{pspicture}(-1.6,-1.7)(1.6,1.7)
			\psecurve(1.2,-0.95)(0,-1.5)(-1.5,-1)(1.2,1.05)(-1.2,0.95)(0,1.5)(1.5,1)(-1.2,-1.05)(1.2,-0.95)(0,-1.5)(-1.5,-1)
			
			\psline[linestyle=dotted](1,1)(1,-1)
			\psline[linestyle=dotted](1,-1)(-1,-1)
			\psline[linestyle=dotted](-1,-1)(-1,1)
			\psline[linestyle=dotted](-1,1)(1,1)
			
			\pscircle[fillstyle=solid, fillcolor=white](1,1){0.08}
			\pscircle[fillstyle=solid, fillcolor=white](-1,1){0.08}
			\pscircle[fillstyle=solid, fillcolor=white](1,-1){0.08}
			\pscircle[fillstyle=solid, fillcolor=white](-1,-1){0.08}
			
			\psset{dotsize=5pt}
			\psdot[linecolor=red](-1,0.78)
			\psdot[linecolor=red](-1,-0.31)
			\psdot[linecolor=red](-1,-0.54)
			\psdot[linecolor=gold](-0.04,1)
			\psdot[linecolor=blue](0.04,-1)
			\psdot[linecolor=darkgreen](1,0.54)
			\psdot[linecolor=darkgreen](1,0.31)
			\psdot[linecolor=darkgreen](1,-0.78)
			
			\end{pspicture}
		}
		\caption{$L_{T_3}$}\label{fig:OrientSkeinRatTangleIII}
	\end{subfigure}
	\begin{subfigure}[b]{0.23\textwidth}\centering
		{\psset{unit=1}
			\begin{pspicture}(-1.6,-1.7)(1.6,1.7)
			\psecurve(-1.2,-0.95)(0,-1.5)(1.5,-1)(-1.2,1.05)(1.2,0.95)(0,1.5)(-1.5,1)(1.2,-1.05)(-1.2,-0.95)(0,-1.5)(1.5,-1)
			
			\psline[linestyle=dotted](1,1)(1,-1)
			\psline[linestyle=dotted](1,-1)(-1,-1)
			\psline[linestyle=dotted](-1,-1)(-1,1)
			\psline[linestyle=dotted](-1,1)(1,1)
			
			\pscircle[fillstyle=solid, fillcolor=white](1,1){0.08}
			\pscircle[fillstyle=solid, fillcolor=white](-1,1){0.08}
			\pscircle[fillstyle=solid, fillcolor=white](1,-1){0.08}
			\pscircle[fillstyle=solid, fillcolor=white](-1,-1){0.08}
			
			\psset{dotsize=5pt}
			\psdot[linecolor=darkgreen](1,0.78)
			\psdot[linecolor=darkgreen](1,-0.31)
			\psdot[linecolor=darkgreen](1,-0.54)
			\psdot[linecolor=gold](0.04,1)
			\psdot[linecolor=blue](-0.04,-1)
			\psdot[linecolor=red](-1,0.54)
			\psdot[linecolor=red](-1,0.31)
			\psdot[linecolor=red](-1,-0.78)
			
			\end{pspicture}
		}
		\caption{$L_{T_{-3}}$}\label{fig:OrientSkeinRatTanglemIII}
	\end{subfigure}
	\caption{Immersed curves of $T_n$ and $T_{-n}$ for $n=2$ and $n=3$.}\label{fig:OrientSkeinRatTangle}
\end{figure}
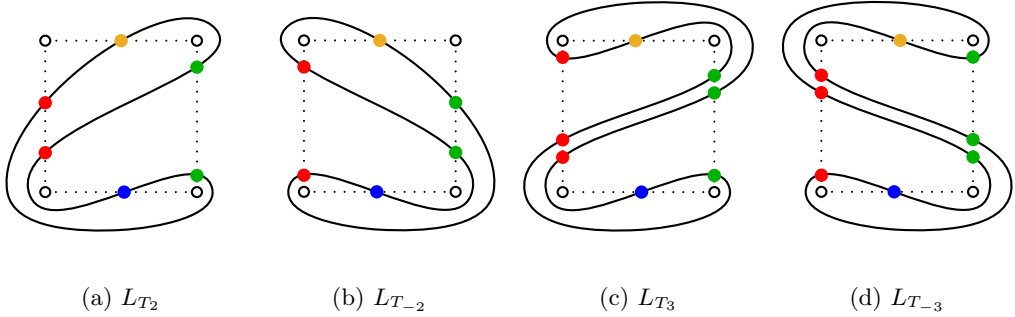

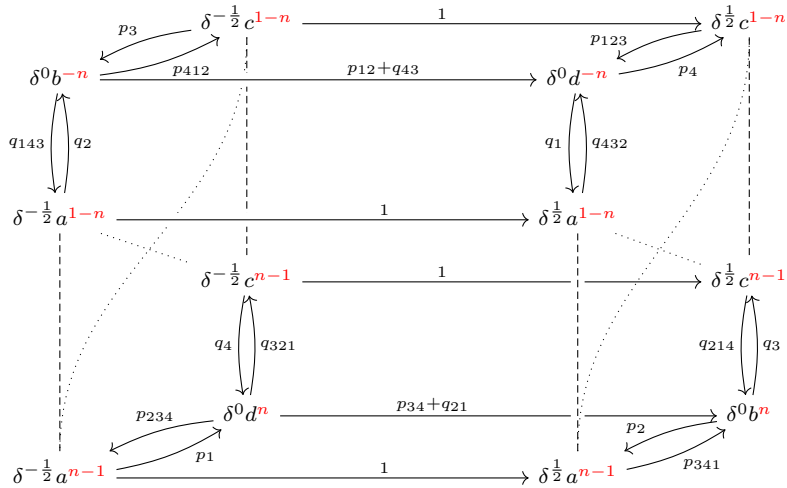
\begin{figure}[t]
	\centering
	$\begin{tikzcd}[row sep=0.3cm, column sep=1cm]
	&
	\delta^{-\frac{1}{2}}c^\Red{1-n}
	\arrow[pos=0.35]{rrrr}{1}
	\arrow[-, dotted,in=90,out=-90]{dddddddl}
	\arrow[-, dashed,in=90,out=-90]{dddd}
	&
	&
	&
	&
	\delta^{\frac{1}{2}}c^\Red{1-n}
	\arrow[-, dashed,in=90,out=-90]{dddd}
	\\
	\delta^0 b^\Red{-n}
	\arrow[crossing over,pos=0.65]{rrrr}{p_{12}+q_{43}}
	\arrow[bend left=10,leftarrow]{ru}{p_3}
	\arrow[bend right=10,swap,pos=0.55]{ru}{p_{412}}
	\arrow[bend left=10,leftarrow]{dd}{q_{2}}
	\arrow[bend right=10,swap]{dd}{q_{143}}
	&
	&
	&
	&
	\delta^0 d^\Red{-n}
	\arrow[bend left=10,leftarrow,pos=0.2]{ru}{p_{123}}
	\arrow[bend right=10,swap]{ru}{p_{4}}
	\arrow[bend left=10,leftarrow]{dd}{q_{432}}
	\arrow[bend right=10,swap]{dd}{q_{1}}
	\\
	&
	\phantom{\vdots}
	&
	&
	&
	&
	\phantom{\vdots}
	\\
	\delta^{-\frac{1}{2}}a^\Red{1-n}
	\arrow[-, dashed,in=90,out=-90]{dddd}
	\arrow[crossing over,pos=0.65]{rrrr}{1}
	\arrow[-, dotted]{dr}
	&
	&
	&
	&
	\delta^{\frac{1}{2}}a^\Red{1-n}
	\arrow[-, dotted]{dr}
	\\
	&
	\delta^{-\frac{1}{2}}c^\Red{n-1}
	\arrow[pos=0.35]{rrrr}{1}
	&
	&
	&
	&
	\delta^{\frac{1}{2}}c^\Red{n-1}
	\\
	\phantom{\vdots}
	&
	&
	&
	&
	\phantom{\vdots}
	\\
	&
	\delta^0 d^\Red{n}
	\arrow[bend left=10,leftarrow,pos=0.3]{ld}{p_1}
	\arrow[bend right=10,swap,pos=0.3]{ld}{p_{234}}
	\arrow[bend left=10,leftarrow]{uu}{q_{4}}
	\arrow[bend right=10,swap]{uu}{q_{321}}
	\arrow[pos=0.35]{rrrr}{p_{34}+q_{21}}
	&
	&
	&
	&
	\delta^0 b^\Red{n}
	\arrow[bend left=10,leftarrow]{ld}{p_{341}}
	\arrow[bend right=10,swap,pos=0.7]{ld}{p_{2}}
	\arrow[bend left=10,leftarrow]{uu}{q_{214}}
	\arrow[bend right=10,swap]{uu}{q_{3}}
	\\
	\delta^{-\frac{1}{2}}a^\Red{n-1}
	\arrow[pos=0.65]{rrrr}{1}
	&
	&
	&
	&
	\delta^{\frac{1}{2}}a^\Red{n-1}
	\arrow[-, dashed,in=-90,out=90,crossing over]{uuuu}
	\arrow[-, dotted,in=-90,out=90]{uuuuuuur}
	\end{tikzcd}$
	\caption{The morphism $\varphi_n:\CFTd(T_{n})\rightarrow\CFTd(T_{-n})$.}\label{fig:OrientSkeinBoundaryMap}
\end{figure}

\begin{proof}
$T_n$ and $T_{-n}$ are rational tangles. As such, their immersed curve invariants are straightforward to compute from genus 0 Heegaard diagrams, since (as explained in Example~\ref{exa:CFTdRatTang}) the $\beta$-curves of such Heegaard diagrams \emph{are} the invariants. Figure~\ref{fig:OrientSkeinRatTangle} illustrates this for the cases $n=2$ and $n=3$.
The peculiar modules $\CFTd(T_{n})$ and $\CFTd(T_{-n})$ for general $n$ are shown in Figure~\ref{fig:OrientSkeinBoundaryMap}, the former on the left, the latter on the right. If $n$ is odd, the dashed lines denote a sequence of alternating generators in sites $a$ and $c$, connected by pairs of morphisms, labelled alternatingly by $p_i$s and $q_i$s. For even $n$, the two components are connected by similar sequences along the dotted lines. The horizontal arrows in Figure~\ref{fig:OrientSkeinBoundaryMap} describe $\varphi_n$. 
By cancelling all identity components of the mapping cone of $\varphi_n$, we get two copies of 
\[\begin{tikzcd}[row sep=0.3cm, column sep=1cm]
\nmathphantom{\CFTd(T_0)=}\CFTd(T_0)=\delta^0b^\Red{0}
\arrow[bend left=10,leftarrow]{r}{p_{34}+q_{21}}
\arrow[bend right=10,swap,pos=0.55]{r}{p_{12}+q_{43}}
&
\delta^0d^\Red{0}.
\end{tikzcd}\]
Thus, we can write $\CFTd(T_0)\otimes V$ as a cone of $\CFTd(T_n)$ and $\CFTd(T_{-n})$, which gives rise to the exact triangle of the required form.
\end{proof}

Next, we give a new proof of a theorem by Manolescu \cite[Theorem~1]{Manolescu}. Like Manolescu's triangle, ours does not preserve any gradings. Note that Manolescu uses slightly different conventions from ours, so the two triangles only look the same after reversing the direction of the three arrows.

\begin{figure}[t]
	\centering
	\begin{subfigure}[b]{0.48\textwidth}
		{\normalsize
			\[\begin{tikzcd}[row sep=0.7cm, column sep=-0.5cm]
			\CFTd(\!\!
			\raisebox{-6pt}{
				\psset{unit=0.2}
				\begin{pspicture}(-1.05,-1.45)(1.05,1.05)
				\psline(-0.9,-0.9)(0.9,0.9)
				\psline(0.9,-0.9)(-0.9,0.9)
				\pscircle*[linecolor=white](0,0){0.5}
				\psarc(0.7071,0){0.5}{135}{225}
				\psarc(-0.7071,0){0.5}{-45}{45}
				\end{pspicture}
			}
			\!\!)
			\arrow{rr}{\varphi}
			& &
			\CFTd(\!\!
			\raisebox{-6pt}{
				\psset{unit=0.2}
				\begin{pspicture}(-1.05,-1.45)(1.05,1.05)
				\psline(-0.9,-0.9)(0.9,0.9)
				\psline(0.9,-0.9)(-0.9,0.9)
				\pscircle*[linecolor=white](0,0){0.5}
				\psarc(0,0.7071){0.5}{-135}{-45}
				\psarc(0,-0.7071){0.5}{45}{135}
				\end{pspicture}
			}
			\!\!)
			\arrow{dl}
			\\
			&
			\CFTd(\!\!
			\raisebox{-6pt}{
				\psset{unit=0.2}
				\begin{pspicture}(-1.05,-1.45)(1.05,1.05)
				\psline(-0.9,-0.9)(0.9,0.9)
				\pscircle*[linecolor=white](0,0){0.3}
				\psline(0.9,-0.9)(-0.9,0.9)
				\end{pspicture}
			}
			\!\!)
			\arrow{lu}
			\end{tikzcd}\]
		}
		\caption{}\label{fig:ResolutionExactTriangleTangle}
	\end{subfigure}
	\begin{subfigure}[b]{0.48\textwidth}
		{\normalsize
			\[\begin{tikzcd}[row sep=0.7cm, column sep=-1.1cm]
			\HFL(L_{0})\otimes V^{l_{0}}
			\arrow{rr}
			& &
			\HFL(L_{1})\otimes V^{l_{1}}
			\arrow{dl}
			\\
			&
			\HFL(L_{X})\otimes V^{l_X}
			\arrow{lu}
			\end{tikzcd}\]
		}
		\caption{}\label{fig:ResolutionExactTriangleLinks}
	\end{subfigure}
	\caption{The skein exact triangles from Theorem~\ref{thm:ResolutionExactTriangle}.}\label{fig:ResolutionExactTriangle}
\end{figure}
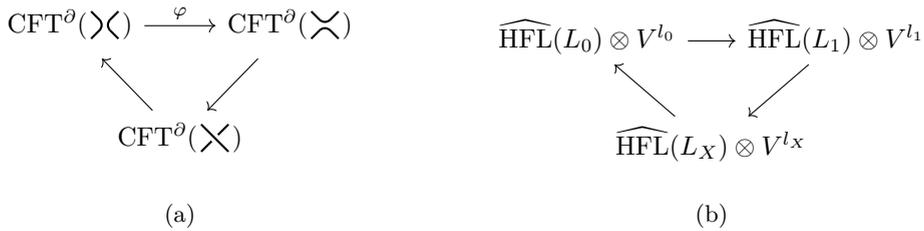

\begin{theorem}[(resolution skein exact triangle)]\label{thm:ResolutionExactTriangle}
There is an exact triangle as shown in Figure~\ref{fig:ResolutionExactTriangleTangle}.
Moreover, given three links \(L_{0}\), \(L_{1}\) and \(L_X\) in \(S^3\) which agree outside a closed 3-ball and in this closed 3-ball agree with the 4-ended tangles 
\(T_0=\!\!\raisebox{-5pt}{
\psset{unit=0.2}
\begin{pspicture}(-1.05,-1.45)(1.05,1.05)
\psline(-0.9,-0.9)(0.9,0.9)
\psline(0.9,-0.9)(-0.9,0.9)
\pscircle*[linecolor=white](0,0){0.5}
\psarc(0.7071,0){0.5}{135}{225}
\psarc(-0.7071,0){0.5}{-45}{45}
\end{pspicture}
}\!\!,\) 
\(T_1=\!\!\raisebox{-5pt}{
\psset{unit=0.2}
\begin{pspicture}(-1.05,-1.45)(1.05,1.05)
\psline(-0.9,-0.9)(0.9,0.9)
\psline(0.9,-0.9)(-0.9,0.9)
\pscircle*[linecolor=white](0,0){0.5}
\psarc(0,0.7071){0.5}{-135}{-45}
\psarc(0,-0.7071){0.5}{45}{135}
\end{pspicture}
}\)\!\!
 and  
 \(T_X=
\!\!\raisebox{-5pt}{
\psset{unit=0.2}
\begin{pspicture}(-1.05,-1.45)(1.05,1.05)
\psline(-0.9,-0.9)(0.9,0.9)
\pscircle*[linecolor=white](0,0){0.3}
\psline(0.9,-0.9)(-0.9,0.9)
\end{pspicture}
}\!\!,\)
respectively, then the above triangle, together with the Glueing Theorem induces the exact triangle from Figure~\ref{fig:ResolutionExactTriangleLinks}, 
where for \(i\in\{0,1,X\}\), \(l_i\) is either \(0\) or \(1\), depending on whether the two strands in \(T_i\) belong to different or the same components in \(L_i\), respectively. 
\end{theorem}

\begin{proof}
The map $\varphi$ is given by (the horizontal arrows in) the following diagram on the left:
\[\begin{tikzcd}[row sep=2cm, column sep=2.8cm]
b
\arrow[bend left=10,leftarrow]{d}{p_{34}+q_{21}}
\arrow[bend right=10,swap]{d}{p_{12}+q_{43}}
\arrow{r}{q_3}
\arrow[pos=0.3]{rd}{p_2}
&
c
\arrow[bend left=10,leftarrow]{d}{p_{41}+q_{32}}
\arrow[bend right=10,swap]{d}{p_{23}+q_{14}}
\\
d
\arrow{r}{q_1}
\arrow[pos=0.3,swap]{ru}{p_4}
&
a
\end{tikzcd}
\quad\cong\quad
\begin{tikzcd}[row sep=2cm, column sep=2.8cm]
b
\arrow[bend left=10,leftarrow,pos=0.8]{dr}{p_{341}}
\arrow[bend right=10,swap,pos=0.2]{dr}{p_{2}}
\arrow[bend left=10,leftarrow]{r}{q_{214}}
\arrow[bend right=10,swap]{r}{q_{3}}
&
c
\\
d
\arrow[bend right=10,swap,leftarrow,pos=0.8]{ru}{p_{123}}
\arrow[bend left=10,pos=0.2]{ru}{p_{4}}
\arrow[bend right=10,swap,leftarrow]{r}{q_{432}}
\arrow[bend left=10]{r}{q_{1}}
&
a
\end{tikzcd}\]
Using the Clean-Up Lemma (\ref{lem:AbstractCleanUp}) twice with $h=(a\oplus c\xrightarrow{(q_2~p_3)}b)$ and $h=(a\oplus c\xrightarrow{(p_1~q_4)}d)$, respectively, we see that it is chain isomorphic to the diagram on the right,
which is $\CFTd(\!\!\raisebox{-5pt}{
\psset{unit=0.2}
\begin{pspicture}(-1.05,-1.45)(1.05,1.05)
\psline(-0.9,-0.9)(0.9,0.9)
\pscircle*[linecolor=white](0,0){0.3}
\psline(0.9,-0.9)(-0.9,0.9)
\end{pspicture}
}\!\!)$. Now apply the same arguments as in the proof of Theorem~\ref{thm:nTwistSkeinRelation}.
\end{proof}

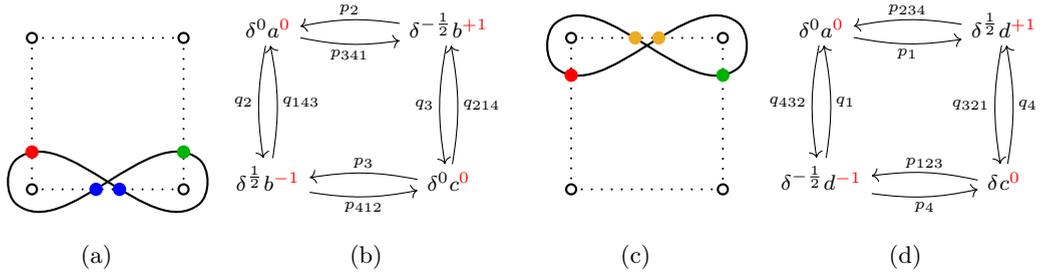
\begin{figure}[t]
	\begin{subfigure}[b]{0.19\textwidth}\centering
		\begin{pspicture}(-1.5,-1.5)(1.5,1.5)
		\psrotate(0,0){180}{
			\psecurve(-1.2,0.6)(-1.2,1.2)(1.2,0.6)(1.2,1.2)(-1.2,0.6)(-1.2,1.2)(1.2,0.6)
			
			\psline[linestyle=dotted](1,1)(1,-1)
			\psline[linestyle=dotted](1,-1)(-1,-1)
			\psline[linestyle=dotted](-1,-1)(-1,1)
			\psline[linestyle=dotted](-1,1)(1,1)
			%
			
			\psset{dotsize=5pt}
			
			\psdot[linecolor=darkgreen](-1,0.507)
			
			\psdot[linecolor=blue](0.155,1)
			\psdot[linecolor=blue](-0.155,1)
			
			\psdot[linecolor=red](1,0.507)

			\pscircle[fillstyle=solid, fillcolor=white](1,1){0.08}
			\pscircle[fillstyle=solid, fillcolor=white](-1,1){0.08}
			\pscircle[fillstyle=solid, fillcolor=white](1,-1){0.08}
			\pscircle[fillstyle=solid, fillcolor=white](-1,-1){0.08}
		}
		\end{pspicture}
		\caption{}\label{fig:figure8loop1}
	\end{subfigure}
	\begin{subfigure}[b]{0.29\textwidth}
		{
		\[
		\begin{tikzcd}[row sep=1.5cm, column sep=1.2cm,crossing over clearance=3pt]
		\delta^{0}a^{\red 0}
		\arrow[bend left=10,leftarrow,pos=0.5]{d}{q_{143}}
		\arrow[bend right=10,swap,pos=0.5]{d}{q_2}
		\arrow[bend left=10,leftarrow,pos=0.5]{r}{p_2}
		\arrow[bend right=10,swap,pos=0.5]{r}{p_{341}}
		&
		\delta^{-\frac{1}{2}}b^{\red +1}
		\\
		\delta^{\frac{1}{2}}b^{\red -1}
		&
		\delta^{0}c^{\red 0}
		\arrow[bend left=10,leftarrow,pos=0.5]{u}{q_{3}}
		\arrow[bend right=10,swap,pos=0.5]{u}{q_{214}}
		\arrow[bend left=10,leftarrow,pos=0.5]{l}{p_{412}}
		\arrow[bend right=10,swap,pos=0.5]{l}{p_3}
		\end{tikzcd}
		\]
		}
		\caption{}\label{fig:figure8loop1pm}
	\end{subfigure}
	\begin{subfigure}[b]{0.19\textwidth}\centering
		\begin{pspicture}(-1.5,-1.5)(1.5,1.5)
		\psecurve(-1.2,0.6)(-1.2,1.2)(1.2,0.6)(1.2,1.2)(-1.2,0.6)(-1.2,1.2)(1.2,0.6)
		
		\psline[linestyle=dotted](1,1)(1,-1)
		\psline[linestyle=dotted](1,-1)(-1,-1)
		\psline[linestyle=dotted](-1,-1)(-1,1)
		\psline[linestyle=dotted](-1,1)(1,1)
		%
		
		\psset{dotsize=5pt}
		
		\psdot[linecolor=red](-1,0.507)
		
		\psdot[linecolor=gold](0.155,1)
		\psdot[linecolor=gold](-0.155,1)
		
		\psdot[linecolor=darkgreen](1,0.507)

		\pscircle[fillstyle=solid, fillcolor=white](1,1){0.08}
		\pscircle[fillstyle=solid, fillcolor=white](-1,1){0.08}
		\pscircle[fillstyle=solid, fillcolor=white](1,-1){0.08}
		\pscircle[fillstyle=solid, fillcolor=white](-1,-1){0.08}
		\end{pspicture}
		\caption{}\label{fig:figure8loop2}
	\end{subfigure}
	\begin{subfigure}[b]{0.29\textwidth}
		{
		\[
		\begin{tikzcd}[row sep=1.5cm, column sep=1.2cm,crossing over clearance=3pt]
		\delta^{0}a^{\red 0}
		\arrow[bend left=10,leftarrow,pos=0.5]{d}{q_1}
		\arrow[bend right=10,swap,pos=0.5]{d}{q_{432}}
		\arrow[bend left=10,leftarrow,pos=0.5]{r}{p_{234}}
		\arrow[bend right=10,swap,pos=0.5]{r}{p_1}
		&
		\delta^{\frac{1}{2}}d^{\red +1}
		\\
		\delta^{-\frac{1}{2}}d^{\red -1}
		&
		\delta^{}c^{\red 0}
		\arrow[bend left=10,leftarrow,pos=0.5]{u}{q_{321}}
		\arrow[bend right=10,swap,pos=0.5]{u}{q_4}
		\arrow[bend left=10,leftarrow,pos=0.5]{l}{p_4}
		\arrow[bend right=10,swap,pos=0.5]{l}{p_{123}}
		\end{tikzcd}
		\]
		}
		\caption{}\label{fig:figure8loop2pm}
	\end{subfigure}
	\caption{Two figure-8 loops (a) and (c) and their peculiar modules (b) and (d).}\label{fig:figure8loops}
\end{figure} 

\begin{proposition}\label{prop:singularcrossing}
There are two morphisms
\[
\delta^{-\frac{1}{2}}t^{\pm1}\CFTd\left(\raisebox{-5pt}{
\psset{unit=0.3}
\begin{pspicture}(-0.91,-0.91)(0.91,0.91)
\psline{->}(-0.9,-0.9)(0.9,0.9)
\psline{->}(0.9,-0.9)(-0.9,0.9)
\pscircle*[linecolor=white](0,0){0.5}
\psarc(0.7071,0){0.5}{135}{225}
\psarc(-0.7071,0){0.5}{-45}{45}
\end{pspicture}
}\right)\rightarrow
\CFTd\left(\raisebox{-5pt}{
\psset{unit=0.3}
\begin{pspicture}(-0.91,-0.91)(0.91,0.91)
\psline{->}(0.9,-0.9)(-0.9,0.9)
\pscircle*[linecolor=white](0,0){0.3}
\psline{->}(-0.9,-0.9)(0.9,0.9)
\end{pspicture}
}\right)
\]
whose mapping cones are homotopic to the peculiar modules represented by the ``figure-8'' loops shown in Figure~\ref{fig:figure8loops}. Since these loops are invariant under taking the mirror, they agree with the mapping cones of maps from the negative crossing to the trivial tangle. 
\end{proposition}

\begin{figure}[t]
	\begin{subfigure}[b]{0.48\textwidth}
		{
		\[
		\begin{tikzcd}[row sep=0.5cm, column sep=0.22cm,crossing over clearance=3pt]
		&
		\delta^{-\frac{1}{2}}d^{\red +1}
		\arrow[out=0,in=180]{rrrr}{1}
		&&&&
		\delta^{\frac{1}{2}}d^{\red +1}
		\arrow[bend left=10,leftarrow,pos=0.5]{dd}{q_4}
		\arrow[bend right=10,swap,pos=0.6]{dd}{q_{321}}
		\\
		&&&&
		\delta^{0}a^{\red 0}
		\arrow[crossing over, bend left=10,leftarrow,pos=0.35]{ur}{p_{234}}
		\arrow[crossing over, bend right=10,swap,pos=0.4]{ur}{p_1}
		\\
		&&&&&
		\delta^{0}c^{\red 0}
		\\
		\delta^{-\frac{1}{2}}b^{\red +1}
		\arrow[out=30,in=180,pos=0.3]{rrrruu}{p_2}
		\arrow[swap,out=20,in=180,pos=0.5]{rrrrru}{q_3}
		\arrow[bend left=10,leftarrow,pos=0.5]{uuur}{p_{34}+q_{21}}
		\arrow[bend right=10,swap,pos=0.85]{uuur}{p_{12}+q_{43}}
		&&&&
		\delta^{\frac{1}{2}}b^{\red -1}
		\arrow[bend left=10,leftarrow,pos=0.65]{ur}{p_3}
		\arrow[bend right=10,swap,pos=0.65]{ur}{p_{412}}
		\arrow[crossing over, bend left=10,leftarrow,pos=0.55]{uu}{q_2}
		\arrow[crossing over, bend right=10,swap,pos=0.75]{uu}{q_{143}}
		\end{tikzcd}
		\]
		}
		\caption{}
	\end{subfigure}
	\begin{subfigure}[b]{0.48\textwidth}
		{
		\[
		\begin{tikzcd}[row sep=0.5cm, column sep=0.22cm,crossing over clearance=3pt]
		&
		\delta^{-\frac{1}{2}}d^{\red -1}
		\arrow[out=-20,in=180,pos=0.2]{rrrrdd}{p_4}
		\arrow[crossing over, out=-30,in=180,swap,pos=0.3]{rrrd}{q_1}
		&&&&
		\delta^{\frac{1}{2}}d^{\red +1}
		\arrow[bend left=10,leftarrow,pos=0.5]{dd}{p_4}
		\arrow[bend right=10,swap,pos=0.6]{dd}{q_{321}}
		\\
		&&&&
		\delta^{0}a^{\red 0}
		\arrow[crossing over, bend left=10,leftarrow,pos=0.35]{ur}{p_{234}}
		\arrow[crossing over, bend right=10,swap,pos=0.4]{ur}{p_1}
		\\
		&&&&&
		\delta^{0}c^{\red 0}
		\\
		\delta^{-\frac{1}{2}}b^{\red -1}
		\arrow[out=0,in=180,swap]{rrrr}{1}
		\arrow[bend left=10,leftarrow,pos=0.5]{uuur}{p_{34}+q_{21}}
		\arrow[bend right=10,swap,pos=0.5]{uuur}{p_{12}+q_{43}}
		&&&&
		\delta^{\frac{1}{2}}b^{\red -1}
		\arrow[bend left=10,leftarrow,pos=0.65]{ur}{p_3}
		\arrow[bend right=10,swap,pos=0.65]{ur}{p_{412}}
		\arrow[crossing over, bend left=10,leftarrow,pos=0.55]{uu}{q_2}
		\arrow[crossing over, bend right=10,swap,pos=0.75]{uu}{q_{143}}
		\end{tikzcd}
		\]
		}
		\caption{}
	\end{subfigure}
	\caption{The two maps from Proposition~\ref{prop:singularcrossing}.}\label{fig:singularcrossing}
\end{figure}
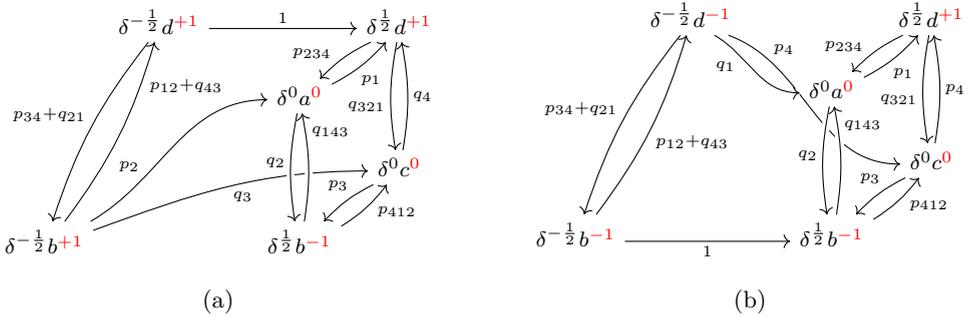 

\begin{proof}
The two maps between the peculiar invariants of the two tangles are shown in Figure~\ref{fig:singularcrossing}. After cancelling the identity arrows in both mapping cones, we obtain the two peculiar modules shown in Figure~\ref{fig:figure8loops}. Also, reversing all arrows, swapping $p_i$ and $q_i$ and reversing the Alexander grading leaves both of them invariant. Doing this to the mapping cone gives us maps between the mirrors of the two tangles, but in the opposite direction.
\end{proof}

\begin{figure}[t]\centering
	{
		$
		\begin{tikzcd}[row sep=0.5cm, column sep=0.45cm,crossing over clearance=3pt]
		&
		b^{\red 2-n}
		\arrow[bend left=10,leftarrow,pos=0.35]{dd}{q_{214}}
		\arrow[bend right=10,swap,pos=0.75]{dd}{q_3}
		&&&&
		b^{\red 4-n}
		\arrow[bend left=10,leftarrow,pos=0.35]{dd}{q_{214}}
		\arrow[bend right=10,swap,pos=0.75]{dd}{q_3}
		&&
		\!\!\!\cdots\!\!\!
		&&
		b^{\red n}
		\arrow[bend left=10,leftarrow,pos=0.65]{dd}{q_{214}}
		\arrow[bend right=10,swap,pos=0.75]{dd}{q_3}
		&&&&
		d^{\red n}
		\\
		a^{\red 1-n}
		\arrow[bend left=10,leftarrow]{ru}{p_2}
		\arrow[bend right=10,swap,pos=0.15]{ru}{p_{341}}
		\arrow[crossing over,in=180,out=0,swap,pos=0.1]{rrrrrd}{p_{41}}
		&&&&
		a^{\red 3-n}
		\arrow[bend left=10,leftarrow]{ru}{p_2}
		\arrow[bend right=10,swap,pos=0.15]{ru}{p_{341}}
		\arrow[dotted,crossing over,in=180,out=0,swap,pos=0.1]{rrrrrd}{p_{41}}
		&&
		\!\!\!\cdots\!\!\!
		&&
		a^{\red n-1}
		\arrow[bend left=10,leftarrow]{ru}{p_2}
		\arrow[bend right=10,swap,pos=0.15]{ru}{p_{341}}
		\arrow[crossing over,pos=0.9,in=180,out=0]{urrrrr}{p_1}
		\\
		&
		c^{\red 1-n}
		\arrow[bend left=10,leftarrow]{ld}{p_{412}}
		\arrow[bend right=10,swap,pos=0.3]{ld}{p_3}
		\arrow[in=180,out=0,pos=0.03,swap]{rrru}{q_{14}}
		&&&&
		c^{\red 3-n}
		\arrow[bend left=10,leftarrow]{ld}{p_{412}}
		\arrow[bend right=10,swap,pos=0.3]{ld}{p_3}
		\arrow[dotted,in=180,out=0,pos=0.03,swap]{rrru}{q_{14}}
		&&
		\!\!\!\cdots\!\!\!
		&&
		c^{\red n-1}
		\arrow[bend left=10,leftarrow]{ld}{p_{412}}
		\arrow[bend right=10,swap,pos=0.4]{ld}{p_3}
		\arrow[pos=0.65,in=-150,out=0,pos=0.6]{uurrrr}{q_4}
		\\
		b^{\red -n}
		\arrow[crossing over,bend left=10,leftarrow,pos=0.5]{uu}{q_2}
		\arrow[crossing over,bend right=10,swap,pos=0.55]{uu}{q_{143}}
		&&&&
		b^{\red 2-n}
		\arrow[crossing over,bend left=10,leftarrow,pos=0.65]{uu}{q_2}
		\arrow[crossing over,bend right=10,swap,pos=0.65]{uu}{q_{143}}
		\arrow[swap,in=0,out=180,pos=0.8,leftarrow]{llluuu}{1}
		&&
		\!\!\!\cdots\!\!\!
		&&
		b^{\red n-2}
		\arrow[crossing over,bend left=10,leftarrow,pos=0.65]{uu}{q_2}
		\arrow[crossing over,bend right=10,swap,pos=0.65]{uu}{q_{143}}
		\arrow[swap,dotted,in=0,out=180,pos=0.8,leftarrow]{llluuu}{1}
		&&&&
		b^{\red n}
		\arrow[swap,in=0,out=180,pos=0.8,leftarrow]{llluuu}{1}
		\arrow[bend left=10,leftarrow,pos=0.2]{uuur}{p_{34}+q_{21}}
		\arrow[bend right=10,swap,pos=0.3]{uuur}{p_{12}+q_{43}}
		\end{tikzcd}$
	}
	\caption{A complex of peculiar modules homotopic to $\CFTd(T_n)$.}\label{fig:ntwistascomplex}
\end{figure} 

\begin{remark}\label{rem:singularcrossings}
It is interesting to compare the ``figure-8'' curve to the local Heegaard diagram for a singular crossing $\!\!\raisebox{-5pt}{
\psset{unit=0.2}
\psset{arrowsize=1.5pt 2}
\begin{pspicture}(-1.05,-1.45)(1.05,1.05)
\psline{->}(0.9,-0.9)(-0.9,0.9)
\psline{->}(-0.9,-0.9)(0.9,0.9)
\psdot(0,0)
\end{pspicture}
}\!\!$ in \cite{OSSHFS}, as the number of generators agree for the second loop, up to an additional tensor factor. Also note that the proposition above gives rise to an exact triangle similar to the one in \cite{OSrescube}. Moreover, we can write the $n$-twist tangle $T_{n}$ from Figure~\ref{fig:OrientedSkeinRelationTn}, with both strands oriented upwards, as a complex in the objects $\!\!\raisebox{-5pt}{
\psset{unit=0.2}
\psset{arrowsize=1.5pt 2}
\begin{pspicture}(-1.05,-1.45)(1.05,1.05)
\psline{->}(0.9,-0.9)(-0.9,0.9)
\psline{->}(-0.9,-0.9)(0.9,0.9)
\psdot(0,0)
\end{pspicture}
}\!\!$ 
and
$
\!\!\raisebox{-5pt}{
\psset{unit=0.2}
\psset{arrowsize=1.5pt 2}
\begin{pspicture}(-1.05,-1.45)(1.05,1.05)
\psline{->}(-0.9,-0.9)(0.9,0.9)
\psline{->}(0.9,-0.9)(-0.9,0.9)
\pscircle*[linecolor=white](0,0){0.5}
\psarc(0.7071,0){0.5}{135}{225}
\psarc(-0.7071,0){0.5}{-45}{45}
\end{pspicture}
}\!\!
$. Indeed, cancelling the identity components in the complex from Figure~\ref{fig:ntwistascomplex} gives us a loop representing $T_{n}$.
\pagebreak[3]

Similarly, we can obtain a complex for $T_{-n}$ by applying the mirror operation. Furthermore, it is also easy to find such complexes for other orientations of $T_{-n}$ and $T_n$. Note that these complexes look very much like the ones we get in Bar-Natan's Khovanov homology of tangles \cite{BarNatanKhT}. We assume that every tangle can be written as a complex in the two objects $\!\!\raisebox{-5pt}{
\psset{unit=0.2}
\psset{arrowsize=1.5pt 2}
\begin{pspicture}(-1.05,-1.45)(1.05,1.05)
\psline{->}(0.9,-0.9)(-0.9,0.9)
\psline{->}(-0.9,-0.9)(0.9,0.9)
\psdot(0,0)
\end{pspicture}
}\!\!$ 
and
$
\!\!\raisebox{-5pt}{
\psset{unit=0.2}
\psset{arrowsize=1.5pt 2}
\begin{pspicture}(-1.05,-1.45)(1.05,1.05)
\psline{->}(-0.9,-0.9)(0.9,0.9)
\psline{->}(0.9,-0.9)(-0.9,0.9)
\pscircle*[linecolor=white](0,0){0.5}
\psarc(0.7071,0){0.5}{135}{225}
\psarc(-0.7071,0){0.5}{-45}{45}
\end{pspicture}
}\!\!
$, or the two objects 
$
\!\!\raisebox{-5pt}{
\psset{unit=0.2}
\psset{arrowsize=1.5pt 2}
\begin{pspicture}(-1.05,-1.45)(1.05,1.05)
\psline{-}(-0.9,-0.9)(0.9,0.9)
\psline{<->}(0.9,-0.9)(-0.9,0.9)
\pscircle*[linecolor=white](0,0){0.5}
\psarc(0.7071,0){0.5}{135}{225}
\psarc(-0.7071,0){0.5}{-45}{45}
\end{pspicture}
}\!\!
$
and 
$
\!\!\raisebox{-5pt}{
\psset{unit=0.2}
\psset{arrowsize=1.5pt 2}
\begin{pspicture}(-1.05,-1.45)(1.05,1.05)
\psrotate(0,0){90}{
\psline{<->}(-0.9,-0.9)(0.9,0.9)
\psline{-}(0.9,-0.9)(-0.9,0.9)
\pscircle*[linecolor=white](0,0){0.5}
\psarc(0.7071,0){0.5}{135}{225}
\psarc(-0.7071,0){0.5}{-45}{45}
}
\end{pspicture}
}\!\!
$, depending on the orientation. In fact, we can iteratively use the type AA glueing structure from Theorem~\ref{thm:CFTdGeneralGlueing} together with the skein exact sequence from Theorem~\ref{thm:nTwistSkeinRelation} to locally modify tangles until we obtain a complex of peculiar modules of trivial and 1-crossing tangles, up to a large number of tensor factors from glueing. Then, one (only) needs to get rid of these extra factors. For example, in the case of the $(2,-3)$-pretzel tangle, this is indeed possible. 

It is also interesting to compare our ``figure-8'' curve to the curve  
that Hedden, Herald, Kirk associate with a trivial tangle in \cite[Figure~10]{HHK13} in the context of instanton tangle Floer homology in the pillowcase.
\end{remark}

\begin{wrapfigure}{r}{0.3333\textwidth}
	\centering
	\psset{unit=0.3}
	\begin{pspicture}(-5.6,-5.6)(5.6,5.6)
		\psset{linewidth=\stringwidth}
		\pscircle[linestyle=dotted](0,0){5.5}
		
		\rput(0,0){$\left\{\textcolor{white}{\rule[-0.8cm]{1.9cm}{1.8cm}}\right\}$}
		\rput{-90}(4.45,0){$2m+1$}
		\rput{90}(-4.45,0){$2n$}
		
		\psecurve(-1,1)(-3,-1)(0,-2.4)(3,-1)(1,1)
		
		\pscustom{
			\psline{<-}(-3,4.5)(-3,3)
			\psecurve(-1,5)(-3,3)(-1,1)(-3,-1)
		}
		\rput(-2,0.3){$\vdots$}
		\psline[linestyle=dotted,dotsep=0.4](-2,-0.6)(-2,0.6)

		\psecurve[linecolor=white,linewidth=\stringwhite](-1,-5)(-3,-3)(-1,-1)(-3,1)
		\pscustom{
			\psline(-3,-4.5)(-3,-3)
			\psecurve(-1,-5)(-3,-3)(-1,-1)(-3,1)
		}
		
		\pscustom{
			\psline{<-}(3,4.5)(3,3)
			\psecurve(1,5)(3,3)(1,1)(3,-1)
		}
		\psline[linestyle=dotted,dotsep=0.4](2,-0.6)(2,0.6)
		\psecurve[linecolor=white,linewidth=\stringwhite](1,-5)(3,-3)(1,-1)(3,1)
		\pscustom{
			\psline(3,-4.5)(3,-3)
			\psecurve(1,-5)(3,-3)(1,-1)(3,1)
		}
		
		\psecurve[linecolor=white,linewidth=\stringwhite](-1,-1)(-3,1)(0,2.4)(3,1)(1,-1)
		\psecurve(-1,-1)(-3,1)(0,2.4)(3,1)(1,-1)
		
		\rput(-2,4){$t_1$}
		\rput(2,4){$t_2$}
		\rput(-2.2,-4){$t_1$}
		\rput(2.2,-4){$t_2$}
	
	\end{pspicture}
	\caption{The pretzel tangle from Theorem~\ref{thm:pretzeltangleCalc}.}\label{fig:pretzeltangle2nm2mp1}\vspace*{-20pt}
\end{wrapfigure}
\myfixwrapfig
\subsection{\texorpdfstring{Peculiar modules of $(2n,-(2m+1))$-pretzel tangles}{Peculiar modules of (2n,-(2m+1))-pretzel tangles}}\label{subsec:pretzels}

\begin{theorem}\label{thm:pretzeltangleCalc}
	The peculiar modules of \((2n,-(2m+1))\)-pretzel tangles for \(n,m>0\), oriented as in Figure~\ref{fig:pretzeltangle2nm2mp1}, are equal to those shown in Figure~\ref{fig:ResultPretzels}.
\end{theorem}

\begin{remark} 
	For $n=m=1$, this calculation was already done in Example~\ref{exa:HFTdpretzeltangle} directly from the definition of peculiar modules. The general case uses the combinatorial algorithm for computing peculiar modules from Corollary~\ref{cor:PeculiarModulesFromNiceDiagrams}.
	The Mathematica package \cite{PQM.m} allows us to easily confirm Theorem~\ref{thm:pretzeltangleCalc} for fixed pairs $(n,m)$. The cases $(n,m)=(3,4)$ and $(5,2)$ are included as examples in the manual for~\cite{PQM.m}. 
\end{remark}

\begin{theorem}\label{thm:GeneralMutationInvariance}
	Let \(T\) be a tangle in the closed 3-ball \(B^3\) and \(T'\) the tangle obtained by relabelling the sites such that \(a\) and \(c\), and \(b\) and \(d\) are interchanged. If \(T\) is oriented, orient \(T'\) such that the orientation at the first tangle ends of \(T\) and \(T'\) (and hence any others) agree, by either changing the orientation of all strands or leaving them all the same. Then, if \(\CFTd(T)\) and \(\CFTd(T')\) are (graded) chain homotopic, mutation of these tangles preserves (graded) link Floer homology. 
\end{theorem}
\begin{proof}
	This follows directly from the definition of mutation and the Glueing Theorem.
\end{proof}

\begin{figure}[p]
	\centering
	\psset{unit=0.5}
		\begin{subfigure}[b]{0.99\textwidth}\centering
\begin{pspicture}(-13,-5.5)(13,5.5)

\rput(-12.5,-4){
	
	{\psset{linecolor=lightgray}
		\psframe*(-0.5,-0.5)(0.5,0.5)
		\psframe*(1.5,-0.5)(2.5,0.5)
		\psframe*(3.5,-0.5)(4.5,0.5)
		\psframe*(5.5,-0.5)(6.5,0.5)
		\psframe*(7.5,-0.5)(8.5,0.5)
		
		\psframe*(0.5,0.5)(1.5,1.5)
		\psframe*(2.5,0.5)(3.5,1.5)
		\psframe*(4.5,0.5)(5.5,1.5)
		\psframe*(6.5,0.5)(7.5,1.5)
		
		\psframe*(-0.5,1.5)(0.5,2.5)
		\psframe*(1.5,1.5)(2.5,2.5)
		\psframe*(3.5,1.5)(4.5,2.5)
		\psframe*(5.5,1.5)(6.5,2.5)
		
		\psframe*(0.5,2.5)(1.5,3.5)
		\psframe*(2.5,2.5)(3.5,3.5)
		\psframe*(4.5,2.5)(5.5,3.5)
		
		\psframe*(1.5,3.5)(2.5,4.5)
		\psframe*(3.5,3.5)(4.5,4.5)
		
		\psframe*(2.5,4.5)(3.5,5.5)
	}
	
	\psline(1,0)(0,1)
	
	\psline(3,0)(2.25,1)
	\psline(5,0)(4.25,1)
	\psline(7,0)(6.25,1)
	\psline(9,0)(8.25,1)
	
	\psline(2,0)(1.75,1)
	\psline(4,0)(3.75,1)
	\psline(6,0)(5.75,1)
	\psline(8,0)(7.75,1)
	
	\psline(0,2)(1,2)
	\psline(2,2)(2.75,2)
	\psline(3.25,2)(2,3)
	\psline(4.75,2)(3.75,3)
	\psline(5.25,2)(4.25,3)
	\psline(6.75,2)(5.75,3)
	\psline(7.25,2)(6.25,3)
	
	\psline(2,4)(3,4)
	\psline(4,4)(4.75,4)
	\psline(5.25,4)(4,5)
	
	{\psset{linestyle=dashed,dash=2pt 1pt}
		\psline(1,0)(2,0)
		\psline(3,0)(4,0)
		\psline(5,0)(6,0)
		\psline(7,0)(8,0)
		
		\psline(0,1)(0,2)
		\psline(1.75,1)(1,2)
		\psline(2.25,1)(2,2)
		\psline(3.75,1)(2.75,2)
		\psline(4.25,1)(3.25,2)
		\psline(5.75,1)(4.75,2)
		\psline(6.25,1)(5.25,2)
		\psline(7.75,1)(6.75,2)
		\psline(8.25,1)(7.25,2)
		
		\psline(2,3)(2,4)
		\psline(3.75,3)(3,4)
		\psline(4.25,3)(4,4)
		\psline(5.75,3)(4.75,4)
		\psline(6.25,3)(5.25,4)
		
		\psline(4,5)(4,5.5)
		\psline(9,0)(9.5,0)
	}

	\psdots[linecolor=red]%
	(0,1)%
	(1.75,1)(2.25,1)%
	(3.75,1)(4.25,1)%
	(5.75,1)(6.25,1)%
	(7.75,1)(8.25,1)%
	(2,3)%
	(3.75,3)(4.25,3)%
	(5.75,3)(6.25,3)%
	(4,5)
	
	\psdots[linecolor=darkgreen]
	(1,0)%
	(3,0)%
	(5,0)%
	(7,0)%
	(9,0)%
	(1,2)%
	(2.75,2)(3.25,2)%
	(4.75,2)(5.25,2)%
	(6.75,2)(7.25,2)%
	(3,4)%
	(4.75,4)(5.25,4)%
	
	\psdots[linecolor=gold]
	(2,0)(4,0)(6,0)(8,0)%
	(0,2)(2,2)%
	(2,4)(4,4)%
}

\rput(6,-4){
	
	{\psset{linecolor=lightgray}
		\psframe*(-0.5,-0.5)(0.5,0.5)
		\psframe*(-2.5,-0.5)(-1.5,0.5)
		\psframe*(-4.5,-0.5)(-3.5,0.5)
		\psframe*(-6.5,-0.5)(-5.5,0.5)
		
		\psframe*(0.5,0.5)(1.5,1.5)
		\psframe*(-1.5,0.5)(-0.5,1.5)
		\psframe*(-3.5,0.5)(-2.5,1.5)
		\psframe*(-5.5,0.5)(-4.5,1.5)
		\psframe*(-7.5,0.5)(-6.5,1.5)
		
		\psframe*(-0.5,1.5)(0.5,2.5)
		\psframe*(-2.5,1.5)(-1.5,2.5)
		\psframe*(-4.5,1.5)(-3.5,2.5)
		\psframe*(-6.5,1.5)(-5.5,2.5)
		
		\psframe*(-1.5,2.5)(-0.5,3.5)
		\psframe*(-3.5,2.5)(-2.5,3.5)
		\psframe*(-5.5,2.5)(-4.5,3.5)
		
		\psframe*(-2.5,3.5)(-1.5,4.5)
		\psframe*(-4.5,3.5)(-3.5,4.5)
		
		\psframe*(-3.5,4.5)(-2.5,5.5)
	}

	\psline(0,0)(0,1)
	
	\psline(-1,0)(-1.75,1)
	\psline(-2,0)(-2.25,1)
	\psline(-3,0)(-3.75,1)
	\psline(-4,0)(-4.25,1)
	\psline(-5,0)(-5.75,1)
	\psline(-6,0)(-6.25,1)
	
	\psline(0,2)(-0.25,3)
	\psline(1,2)(0.25,3)
	\psline(-0.75,2)(-1.75,3)
	\psline(-1.25,2)(-2.25,3)
	\psline(-2.75,2)(-3.75,3)
	\psline(-3.25,2)(-4.25,3)
	\psline(-0.75,4)(-1.75,5)
	\psline(-1.25,4)(-2.25,5)
	
	\psline(-7,0)(-7.5,0.5)
	\psline(-4.75,2)(-5.25,2.5)
	\psline(-5.25,2)(-5.75,2.5)
	\psline(-2.75,4)(-3.25,4.5)
	\psline(-3.25,4)(-3.75,4.5)
	
	{\psset{linestyle=dashed,dash=2pt 1pt}
		\psline(1.5,2)(1,2)
		\psline(0,0)(-1,0)
		\psline(-2,0)(-3,0)
		
		\psline(-1.75,1)(-2.75,2)
		\psline(-2.25,1)(-3.25,2)
		\psline(-3.75,1)(-4.75,2)
		\psline(-4.25,1)(-5.25,2)
		
		\psline(-1.75,3)(-2.75,4)
		\psline(-2.25,3)(-3.25,4)
		
		\psline(0.25,3)(-0.75,4)
		\psline(-0.25,3)(-1.25,4)
		
		\psline(0,2)(-0.75,2)
		\psline(0,1)(-1.25,2)
		
		\psline(-4,0)(-5,0)
		\psline(-6,0)(-7,0)
		
		\psline(-1.75,5)(-2.25,5.5)
		\psline(-2.25,5)(-2.75,5.5)
		\psline(-3.75,3)(-4.25,3.5)
		\psline(-4.25,3)(-4.75,3.5)
		\psline(-5.75,1)(-6.25,1.5)
		\psline(-6.25,1)(-6.75,1.5)
	}
	
	\psdots[linecolor=blue]
	(0,0)(0,2)

	\psdots[linecolor=red]%
	(0,1)%
	(-1.75,1)(-2.25,1)%
	(-3.75,1)(-4.25,1)%
	(-5.75,1)(-6.25,1)%
	(0.25,3)(-0.25,3)%
	(-1.75,3)(-2.25,3)%
	(-3.75,3)(-4.25,3)%
	(-1.75,5)(-2.25,5)%
	
	\psdots[linecolor=darkgreen]
	(1,2)%
	(-1,0)(-3,0)(-5,0)(-7,0)%
	(-0.75,2)(-1.25,2)%
	(-2.75,2)(-3.25,2)%
	(-4.75,2)(-5.25,2)%
	(-0.75,4)(-1.25,4)%
	(-2.75,4)(-3.25,4)%
	
	\psdots[linecolor=gold]
	(-2,0)(-4,0)(-6,0)
}

\rput(-6,4){\psrotate(0,0){180}{

		{\psset{linecolor=lightgray}
			\psframe*(-0.5,-0.5)(0.5,0.5)
			\psframe*(-2.5,-0.5)(-1.5,0.5)
			\psframe*(-4.5,-0.5)(-3.5,0.5)
			\psframe*(-6.5,-0.5)(-5.5,0.5)
			
			\psframe*(0.5,0.5)(1.5,1.5)
			\psframe*(-1.5,0.5)(-0.5,1.5)
			\psframe*(-3.5,0.5)(-2.5,1.5)
			\psframe*(-5.5,0.5)(-4.5,1.5)
			\psframe*(-7.5,0.5)(-6.5,1.5)
			
			\psframe*(-0.5,1.5)(0.5,2.5)
			\psframe*(-2.5,1.5)(-1.5,2.5)
			\psframe*(-4.5,1.5)(-3.5,2.5)
			\psframe*(-6.5,1.5)(-5.5,2.5)
			
			\psframe*(-1.5,2.5)(-0.5,3.5)
			\psframe*(-3.5,2.5)(-2.5,3.5)
			\psframe*(-5.5,2.5)(-4.5,3.5)
			
			\psframe*(-2.5,3.5)(-1.5,4.5)
			\psframe*(-4.5,3.5)(-3.5,4.5)
			
			\psframe*(-3.5,4.5)(-2.5,5.5)
		}
		
		{\psset{linestyle=dashed,dash=2pt 1pt}
			\psline(0,0)(0,1)
			
			\psline(-1,0)(-1.75,1)
			\psline(-2,0)(-2.25,1)
			\psline(-3,0)(-3.75,1)
			\psline(-4,0)(-4.25,1)
			\psline(-5,0)(-5.75,1)
			\psline(-6,0)(-6.25,1)
			
			\psline(0,2)(-0.25,3)
			\psline(1,2)(0.25,3)
			\psline(-0.75,2)(-1.75,3)
			\psline(-1.25,2)(-2.25,3)
			\psline(-2.75,2)(-3.75,3)
			\psline(-3.25,2)(-4.25,3)
			\psline(-0.75,4)(-1.75,5)
			\psline(-1.25,4)(-2.25,5)
			
			\psline(-7,0)(-7.5,0.5)
			\psline(-4.75,2)(-5.25,2.5)
			\psline(-5.25,2)(-5.75,2.5)
			\psline(-2.75,4)(-3.25,4.5)
			\psline(-3.25,4)(-3.75,4.5)
		}

		\psline(1.5,2)(1,2)
		\psline(0,0)(-1,0)
		\psline(-2,0)(-3,0)
		
		\psline(-1.75,1)(-2.75,2)
		\psline(-2.25,1)(-3.25,2)
		\psline(-3.75,1)(-4.75,2)
		\psline(-4.25,1)(-5.25,2)
		
		\psline(-1.75,3)(-2.75,4)
		\psline(-2.25,3)(-3.25,4)
		
		\psline(0.25,3)(-0.75,4)
		\psline(-0.25,3)(-1.25,4)
		
		\psline(0,2)(-0.75,2)
		\psline(0,1)(-1.25,2)
		
		\psline(-4,0)(-5,0)
		\psline(-6,0)(-7,0)
		
		\psline(-1.75,5)(-2.25,5.5)
		\psline(-2.25,5)(-2.75,5.5)
		\psline(-3.75,3)(-4.25,3.5)
		\psline(-4.25,3)(-4.75,3.5)
		\psline(-5.75,1)(-6.25,1.5)
		\psline(-6.25,1)(-6.75,1.5)
		
		\psdots[linecolor=gold]
		(0,0)(0,2)

		\psdots[linecolor=red]%
		(0,1)%
		(-1.75,1)(-2.25,1)%
		(-3.75,1)(-4.25,1)%
		(-5.75,1)(-6.25,1)%
		(0.25,3)(-0.25,3)%
		(-1.75,3)(-2.25,3)%
		(-3.75,3)(-4.25,3)%
		(-1.75,5)(-2.25,5)%
		
		\psdots[linecolor=darkgreen]
		(1,2)%
		(-1,0)(-3,0)(-5,0)(-7,0)%
		(-0.75,2)(-1.25,2)%
		(-2.75,2)(-3.25,2)%
		(-4.75,2)(-5.25,2)%
		(-0.75,4)(-1.25,4)%
		(-2.75,4)(-3.25,4)%
		
		\psdots[linecolor=blue]
		(-2,0)(-4,0)(-6,0)
	}}
		
		\rput(12.5,4){\psrotate(0,0){180}{
				
				{\psset{linecolor=lightgray}
					\psframe*(-0.5,-0.5)(0.5,0.5)
					\psframe*(1.5,-0.5)(2.5,0.5)
					\psframe*(3.5,-0.5)(4.5,0.5)
					\psframe*(5.5,-0.5)(6.5,0.5)
					\psframe*(7.5,-0.5)(8.5,0.5)
					
					\psframe*(0.5,0.5)(1.5,1.5)
					\psframe*(2.5,0.5)(3.5,1.5)
					\psframe*(4.5,0.5)(5.5,1.5)
					\psframe*(6.5,0.5)(7.5,1.5)
					
					\psframe*(-0.5,1.5)(0.5,2.5)
					\psframe*(1.5,1.5)(2.5,2.5)
					\psframe*(3.5,1.5)(4.5,2.5)
					\psframe*(5.5,1.5)(6.5,2.5)
					
					\psframe*(0.5,2.5)(1.5,3.5)
					\psframe*(2.5,2.5)(3.5,3.5)
					\psframe*(4.5,2.5)(5.5,3.5)
					
					\psframe*(1.5,3.5)(2.5,4.5)
					\psframe*(3.5,3.5)(4.5,4.5)
					
					\psframe*(2.5,4.5)(3.5,5.5)
				}
				
				{\psset{linestyle=dashed,dash=2pt 1pt}
					\psline(1,0)(0,1)
					
					\psline(3,0)(2.25,1)
					\psline(5,0)(4.25,1)
					\psline(7,0)(6.25,1)
					\psline(9,0)(8.25,1)
					
					\psline(2,0)(1.75,1)
					\psline(4,0)(3.75,1)
					\psline(6,0)(5.75,1)
					\psline(8,0)(7.75,1)
					
					\psline(0,2)(1,2)
					\psline(2,2)(2.75,2)
					\psline(3.25,2)(2,3)
					\psline(4.75,2)(3.75,3)
					\psline(5.25,2)(4.25,3)
					\psline(6.75,2)(5.75,3)
					\psline(7.25,2)(6.25,3)
					
					\psline(2,4)(3,4)
					\psline(4,4)(4.75,4)
					\psline(5.25,4)(4,5)
				}
				\psline(1,0)(2,0)
				\psline(3,0)(4,0)
				\psline(5,0)(6,0)
				\psline(7,0)(8,0)
				
				\psline(0,1)(0,2)
				\psline(1.75,1)(1,2)
				\psline(2.25,1)(2,2)
				\psline(3.75,1)(2.75,2)
				\psline(4.25,1)(3.25,2)
				\psline(5.75,1)(4.75,2)
				\psline(6.25,1)(5.25,2)
				\psline(7.75,1)(6.75,2)
				\psline(8.25,1)(7.25,2)
				
				\psline(2,3)(2,4)
				\psline(3.75,3)(3,4)
				\psline(4.25,3)(4,4)
				\psline(5.75,3)(4.75,4)
				\psline(6.25,3)(5.25,4)
				
				\psline(4,5)(4,5.5)
				\psline(9,0)(9.5,0)

				\psdots[linecolor=red]%
				(0,1)%
				(1.75,1)(2.25,1)%
				(3.75,1)(4.25,1)%
				(5.75,1)(6.25,1)%
				(7.75,1)(8.25,1)%
				(2,3)%
				(3.75,3)(4.25,3)%
				(5.75,3)(6.25,3)%
				(4,5)
				
				\psdots[linecolor=darkgreen]
				(1,0)%
				(3,0)%
				(5,0)%
				(7,0)%
				(9,0)%
				(1,2)%
				(2.75,2)(3.25,2)%
				(4.75,2)(5.25,2)%
				(6.75,2)(7.25,2)%
				(3,4)%
				(4.75,4)(5.25,4)%
				
				\psdots[linecolor=blue]
				(2,0)(4,0)(6,0)(8,0)%
				(0,2)(2,2)%
				(2,4)(4,4)%
			}}
			
%
%
%
%
%
%
%
			
			\rput[r](11.5,5){\small$t_1^{-n}t_2^{n+2m}$}
			\psecurve{->}(10.5,4)(11.5,5)(12.5,4)(12.5,3.25)
			\rput[l](-11.5,-5){\small$\,t_1^{n}t_2^{-n-2m}$}
			\psecurve{->}(-10.5,-4)(-11.5,-5)(-12.5,-4)(-12.5,-3.25)
			
			\rput[l](-5,5){\small$\,t_1^{-n}t_2^{n-2m-2}$}
			\psecurve{->}(-4,4)(-5,5)(-6,4.25)(-6,3.25)
			\rput[r](5,-5){\small$t_1^{n}t_2^{-n+2m+2}$}
			\psecurve{->}(4,-4)(5,-5)(6,-4.25)(6,-3.25)
			
			\end{pspicture}
			\caption{$n\leq m+1$}\label{fig:ResultPretzelsCase1}
		\end{subfigure}
		\\
		\begin{subfigure}[b]{0.99\textwidth}\centering
\begin{pspicture}(-13,-8.5)(13,8.5)

\rput(-12.5,-7){
	
	{\psset{linecolor=lightgray}
		\psframe*(-0.5,-0.5)(0.5,0.5)
		\psframe*(1.5,-0.5)(2.5,0.5)
		\psframe*(3.5,-0.5)(4.5,0.5)
		\psframe*(5.5,-0.5)(6.5,0.5)
		\psframe*(7.5,-0.5)(8.5,0.5)
		
		\psframe*(0.5,0.5)(1.5,1.5)
		\psframe*(2.5,0.5)(3.5,1.5)
		\psframe*(4.5,0.5)(5.5,1.5)
		\psframe*(6.5,0.5)(7.5,1.5)
		
		\psframe*(-0.5,1.5)(0.5,2.5)
		\psframe*(1.5,1.5)(2.5,2.5)
		\psframe*(3.5,1.5)(4.5,2.5)
		\psframe*(5.5,1.5)(6.5,2.5)
		
		\psframe*(0.5,2.5)(1.5,3.5)
		\psframe*(2.5,2.5)(3.5,3.5)
		\psframe*(4.5,2.5)(5.5,3.5)
		
		\psframe*(1.5,3.5)(2.5,4.5)
		\psframe*(3.5,3.5)(4.5,4.5)
		
		\psframe*(2.5,4.5)(3.5,5.5)
	}
	
	\psline(1,0)(0,1)
	
	\psline(3,0)(2.25,1)
	\psline(5,0)(4.25,1)
	\psline(7,0)(6.25,1)
	\psline(9,0)(8.25,1)
	
	\psline(2,0)(1.75,1)
	\psline(4,0)(3.75,1)
	\psline(6,0)(5.75,1)
	\psline(8,0)(7.75,1)
	
	\psline(0,2)(1,2)
	\psline(2,2)(2.75,2)
	\psline(3.25,2)(2,3)
	\psline(4.75,2)(3.75,3)
	\psline(5.25,2)(4.25,3)
	\psline(6.75,2)(5.75,3)
	\psline(7.25,2)(6.25,3)
	
	\psline(2,4)(3,4)
	\psline(4,4)(4.75,4)
	\psline(5.25,4)(4,5)
	
	{\psset{linestyle=dashed,dash=2pt 1pt}
		\psline(1,0)(2,0)
		\psline(3,0)(4,0)
		\psline(5,0)(6,0)
		\psline(7,0)(8,0)
		
		\psline(0,1)(0,2)
		\psline(1.75,1)(1,2)
		\psline(2.25,1)(2,2)
		\psline(3.75,1)(2.75,2)
		\psline(4.25,1)(3.25,2)
		\psline(5.75,1)(4.75,2)
		\psline(6.25,1)(5.25,2)
		\psline(7.75,1)(6.75,2)
		\psline(8.25,1)(7.25,2)
		
		\psline(2,3)(2,4)
		\psline(3.75,3)(3,4)
		\psline(4.25,3)(4,4)
		\psline(5.75,3)(4.75,4)
		\psline(6.25,3)(5.25,4)
		
		\psline(4,5)(4,5.5)
		\psline(9,0)(9.5,0)
	}

	\psdots[linecolor=red]%
	(0,1)%
	(1.75,1)(2.25,1)%
	(3.75,1)(4.25,1)%
	(5.75,1)(6.25,1)%
	(7.75,1)(8.25,1)%
	(2,3)%
	(3.75,3)(4.25,3)%
	(5.75,3)(6.25,3)%
	(4,5)
	
	\psdots[linecolor=darkgreen]
	(1,0)%
	(3,0)%
	(5,0)%
	(7,0)%
	(9,0)%
	(1,2)%
	(2.75,2)(3.25,2)%
	(4.75,2)(5.25,2)%
	(6.75,2)(7.25,2)%
	(3,4)%
	(4.75,4)(5.25,4)%
	
	\psdots[linecolor=gold]
	(2,0)(4,0)(6,0)(8,0)%
	(0,2)(2,2)%
	(2,4)(4,4)%
}

\rput(2.5,-7){
	
	{\psset{linecolor=lightgray}
		\psframe*(-0.5,-0.5)(0.5,0.5)
		\psframe*(-2.5,-0.5)(-1.5,0.5)
		\psframe*(-4.5,-0.5)(-3.5,0.5)
		
		\psframe*(0.5,0.5)(1.5,1.5)
		\psframe*(-1.5,0.5)(-0.5,1.5)
		\psframe*(-3.5,0.5)(-2.5,1.5)
		\psframe*(-5.5,0.5)(-4.5,1.5)
		
		\psframe*(-0.5,1.5)(0.5,2.5)
		\psframe*(-2.5,1.5)(-1.5,2.5)
		\psframe*(-4.5,1.5)(-3.5,2.5)
		
		\psframe*(-1.5,2.5)(-0.5,3.5)
		\psframe*(-3.5,2.5)(-2.5,3.5)
		
		\psframe*(-2.5,3.5)(-1.5,4.5)
	}

	\psline(0,0)(0,1)
	
	\psline(-1,0)(-1.75,1)
	\psline(-2,0)(-2.25,1)
	\psline(-3,0)(-3.75,1)
	\psline(-4,0)(-4.25,1)
	
	\psline(0,2)(-0.25,3)
	\psline(1,2)(0.25,3)
	\psline(-0.75,2)(-1.75,3)
	\psline(-1.25,2)(-2.25,3)
	\psline(-2.75,2)(-3.75,3)
	\psline(-3.25,2)(-4.25,3)
	\psline(-0.75,4)(-1.75,5)
	\psline(-1.25,4)(-2.25,5)
	
	\psline(-4.75,2)(-5.25,2.5)
	\psline(-5.25,2)(-5.75,2.5)
	\psline(-2.75,4)(-3.25,4.5)
	\psline(-3.25,4)(-3.75,4.5)
	
	{\psset{linestyle=dashed,dash=2pt 1pt}
		\psline(1.5,2)(1,2)
		\psline(0,0)(-1,0)
		\psline(-2,0)(-3,0)
		
		\psline(-1.75,1)(-2.75,2)
		\psline(-2.25,1)(-3.25,2)
		\psline(-3.75,1)(-4.75,2)
		\psline(-4.25,1)(-5.25,2)
		
		\psline(-1.75,3)(-2.75,4)
		\psline(-2.25,3)(-3.25,4)
		
		\psline(0.25,3)(-0.75,4)
		\psline(-0.25,3)(-1.25,4)
		
		\psline(0,1)(-0.75,2)
		\pscurve(0,2)(-0.25,1.8)(-1,1.8)(-1.25,2)
		
		\psline(-4,0)(-4.5,0)
		\psline(-1.75,5)(-2.25,5.5)
		\psline(-2.25,5)(-2.75,5.5)
		\psline(-3.75,3)(-4.25,3.5)
		\psline(-4.25,3)(-4.75,3.5)
	}
	
	\psdots[linecolor=blue]
	(0,0)(0,2)

	\psdots[linecolor=red]%
	(0,1)%
	(-1.75,1)(-2.25,1)%
	(-3.75,1)(-4.25,1)%
	(0.25,3)(-0.25,3)%
	(-1.75,3)(-2.25,3)%
	(-3.75,3)(-4.25,3)%
	(-1.75,5)(-2.25,5)%
	
	\psdots[linecolor=darkgreen]
	(1,2)%
	(-1,0)(-3,0)%
	(-0.75,2)(-1.25,2)%
	(-2.75,2)(-3.25,2)%
	(-4.75,2)(-5.25,2)%
	(-0.75,4)(-1.25,4)%
	(-2.75,4)(-3.25,4)%
	
	\psdots[linecolor=gold]
	(-2,0)(-4,0)
}

\rput(-8.5,-2){
	
	{\psset{linecolor=lightgray}
		\psframe*(3.5,-2.5)(4.5,-1.5)
		
		\psframe*(2.5,-1.5)(3.5,-0.5)
		\psframe*(4.5,-1.5)(5.5,-0.5)
		
		\psframe*(1.5,-0.5)(2.5,0.5)
		\psframe*(3.5,-0.5)(4.5,0.5)
		\psframe*(5.5,-0.5)(6.5,0.5)
		
		\psframe*(0.5,0.5)(1.5,1.5)
		\psframe*(2.5,0.5)(3.5,1.5)
		\psframe*(4.5,0.5)(5.5,1.5)
		\psframe*(6.5,0.5)(7.5,1.5)
		
		\psframe*(1.5,1.5)(2.5,2.5)
		\psframe*(3.5,1.5)(4.5,2.5)
		\psframe*(5.5,1.5)(6.5,2.5)
		
		\psframe*(0.5,2.5)(1.5,3.5)
		\psframe*(2.5,2.5)(3.5,3.5)
		\psframe*(4.5,2.5)(5.5,3.5)
		\psframe*(6.5,2.5)(7.5,3.5)
		
		\psframe*(1.5,3.5)(2.5,4.5)
		\psframe*(3.5,3.5)(4.5,4.5)
		\psframe*(5.5,3.5)(6.5,4.5)
		
		\psframe*(2.5,4.5)(3.5,5.5)
		\psframe*(4.5,4.5)(5.5,5.5)
		
	}
	
	\psline(0.5,2)(1,2)
	\psline(2,2)(2.75,2)
	\psline(3.25,2)(2,3)
	\psline(4.75,2)(3.75,3)
	\psline(5.25,2)(4.25,3)
	\psline(6.75,2)(5.75,3)
	\psline(7.25,2)(6.25,3)
	
	\psline(2.75,0)(1.75,1)
	\psline(3.25,0)(2.25,1)
	\psline(4.75,0)(3.75,1)
	\psline(5.25,0)(4.25,1)
	
	\psline(2,4)(3,4)
	\pscurve(4,4)(4.25,4.2)(5,4.2)(5.25,4)
	\psline(4.75,4)(4,5)
	
	\psline(4.25,-1.5)(3.75,-1)
	\psline(4.75,-1.5)(4.25,-1)
	\psline(6.25,0.5)(5.75,1)
	\psline(6.75,0.5)(6.25,1)
	
	{\psset{linestyle=dashed,dash=2pt 1pt}
		
		\psline(1.75,1)(1,2)
		\psline(2.25,1)(2,2)
		\psline(3.75,1)(2.75,2)
		\psline(4.25,1)(3.25,2)
		\psline(5.75,1)(4.75,2)
		\psline(6.25,1)(5.25,2)
		
		\psline(3.75,-1)(2.75,0)
		\psline(4.25,-1)(3.25,0)
		
		\psline(5.25,-0.5)(4.75,0)
		\psline(5.75,-0.5)(5.25,0)
		\psline(7.25,1.5)(6.75,2)
		\psline(7.75,1.5)(7.25,2)
		
		\psline(2,3)(2,4)
		\psline(3.75,3)(3,4)
		\psline(4.25,3)(4,4)
		\psline(5.75,3)(4.75,4)
		\psline(6.25,3)(5.25,4)
		
		\psline(4,5)(4,5.5)
	}

	\psdots[linecolor=red]%
	(1.75,1)(2.25,1)%
	(3.75,1)(4.25,1)%
	(5.75,1)(6.25,1)%
	(2,3)%
	(3.75,3)(4.25,3)%
	(5.75,3)(6.25,3)%
	(4,5)
	(3.75,-1)(4.25,-1)%
	
	\psdots[linecolor=darkgreen]
	(1,2)%
	(2.75,2)(3.25,2)%
	(4.75,2)(5.25,2)%
	(6.75,2)(7.25,2)%
	(3,4)%
	(4.75,4)(5.25,4)%
	(2.75,0)(3.25,0)%
	(4.75,0)(5.25,0)%
	
	\psdots[linecolor=gold]
	(2,2)%
	(2,4)(4,4)%
}

\rput(-2.5,7){\psrotate(0,0){180}{
		
		{\psset{linecolor=lightgray}
			\psframe*(-0.5,-0.5)(0.5,0.5)
			\psframe*(-2.5,-0.5)(-1.5,0.5)
			\psframe*(-4.5,-0.5)(-3.5,0.5)
			
			\psframe*(0.5,0.5)(1.5,1.5)
			\psframe*(-1.5,0.5)(-0.5,1.5)
			\psframe*(-3.5,0.5)(-2.5,1.5)
			\psframe*(-5.5,0.5)(-4.5,1.5)
			
			\psframe*(-0.5,1.5)(0.5,2.5)
			\psframe*(-2.5,1.5)(-1.5,2.5)
			\psframe*(-4.5,1.5)(-3.5,2.5)
			
			\psframe*(-1.5,2.5)(-0.5,3.5)
			\psframe*(-3.5,2.5)(-2.5,3.5)
			
			\psframe*(-2.5,3.5)(-1.5,4.5)
		}
		
		{\psset{linestyle=dashed,dash=2pt 1pt}
			\psline(0,0)(0,1)
			
			\psline(-1,0)(-1.75,1)
			\psline(-2,0)(-2.25,1)
			\psline(-3,0)(-3.75,1)
			\psline(-4,0)(-4.25,1)
			
			\psline(0,2)(-0.25,3)
			\psline(1,2)(0.25,3)
			\psline(-0.75,2)(-1.75,3)
			\psline(-1.25,2)(-2.25,3)
			\psline(-2.75,2)(-3.75,3)
			\psline(-3.25,2)(-4.25,3)
			\psline(-0.75,4)(-1.75,5)
			\psline(-1.25,4)(-2.25,5)
			
			\psline(-4.75,2)(-5.25,2.5)
			\psline(-5.25,2)(-5.75,2.5)
			\psline(-2.75,4)(-3.25,4.5)
			\psline(-3.25,4)(-3.75,4.5)
		}
		\psline(1.5,2)(1,2)
		\psline(0,0)(-1,0)
		\psline(-2,0)(-3,0)
		
		\psline(-1.75,1)(-2.75,2)
		\psline(-2.25,1)(-3.25,2)
		\psline(-3.75,1)(-4.75,2)
		\psline(-4.25,1)(-5.25,2)
		
		\psline(-1.75,3)(-2.75,4)
		\psline(-2.25,3)(-3.25,4)
		
		\psline(0.25,3)(-0.75,4)
		\psline(-0.25,3)(-1.25,4)
		
		\psline(0,1)(-0.75,2)
		\pscurve(0,2)(-0.25,1.8)(-1,1.8)(-1.25,2)
		
		\psline(-4,0)(-4.5,0)
		\psline(-1.75,5)(-2.25,5.5)
		\psline(-2.25,5)(-2.75,5.5)
		\psline(-3.75,3)(-4.25,3.5)
		\psline(-4.25,3)(-4.75,3.5)

		\psdots[linecolor=gold]
		(0,0)(0,2)

		\psdots[linecolor=red]%
		(0,1)%
		(-1.75,1)(-2.25,1)%
		(-3.75,1)(-4.25,1)%
		(0.25,3)(-0.25,3)%
		(-1.75,3)(-2.25,3)%
		(-3.75,3)(-4.25,3)%
		(-1.75,5)(-2.25,5)%
		
		\psdots[linecolor=darkgreen]
		(1,2)%
		(-1,0)(-3,0)%
		(-0.75,2)(-1.25,2)%
		(-2.75,2)(-3.25,2)%
		(-4.75,2)(-5.25,2)%
		(-0.75,4)(-1.25,4)%
		(-2.75,4)(-3.25,4)%
		
		\psdots[linecolor=blue]
		(-2,0)(-4,0)
	}}
	
	\rput(8.5,2){\psrotate(0,0){180}{
			
			{\psset{linecolor=lightgray}
				\psframe*(3.5,-2.5)(4.5,-1.5)
				
				\psframe*(2.5,-1.5)(3.5,-0.5)
				\psframe*(4.5,-1.5)(5.5,-0.5)
				
				\psframe*(1.5,-0.5)(2.5,0.5)
				\psframe*(3.5,-0.5)(4.5,0.5)
				\psframe*(5.5,-0.5)(6.5,0.5)
				
				\psframe*(0.5,0.5)(1.5,1.5)
				\psframe*(2.5,0.5)(3.5,1.5)
				\psframe*(4.5,0.5)(5.5,1.5)
				\psframe*(6.5,0.5)(7.5,1.5)
				
				\psframe*(1.5,1.5)(2.5,2.5)
				\psframe*(3.5,1.5)(4.5,2.5)
				\psframe*(5.5,1.5)(6.5,2.5)
				
				\psframe*(0.5,2.5)(1.5,3.5)
				\psframe*(2.5,2.5)(3.5,3.5)
				\psframe*(4.5,2.5)(5.5,3.5)
				\psframe*(6.5,2.5)(7.5,3.5)
				
				\psframe*(1.5,3.5)(2.5,4.5)
				\psframe*(3.5,3.5)(4.5,4.5)
				\psframe*(5.5,3.5)(6.5,4.5)
				
				\psframe*(2.5,4.5)(3.5,5.5)
				\psframe*(4.5,4.5)(5.5,5.5)
				
			}
			
			{\psset{linestyle=dashed,dash=2pt 1pt}
				\psline(0.5,2)(1,2)
				\psline(2,2)(2.75,2)
				\psline(3.25,2)(2,3)
				\psline(4.75,2)(3.75,3)
				\psline(5.25,2)(4.25,3)
				\psline(6.75,2)(5.75,3)
				\psline(7.25,2)(6.25,3)
				
				\psline(2.75,0)(1.75,1)
				\psline(3.25,0)(2.25,1)
				\psline(4.75,0)(3.75,1)
				\psline(5.25,0)(4.25,1)
				
				\psline(2,4)(3,4)
				\pscurve(4,4)(4.25,4.2)(5,4.2)(5.25,4)
				\psline(4.75,4)(4,5)
				
				\psline(4.25,-1.5)(3.75,-1)
				\psline(4.75,-1.5)(4.25,-1)
				\psline(6.25,0.5)(5.75,1)
				\psline(6.75,0.5)(6.25,1)
				
			}
			
			\psline(1.75,1)(1,2)
			\psline(2.25,1)(2,2)
			\psline(3.75,1)(2.75,2)
			\psline(4.25,1)(3.25,2)
			\psline(5.75,1)(4.75,2)
			\psline(6.25,1)(5.25,2)
			
			\psline(3.75,-1)(2.75,0)
			\psline(4.25,-1)(3.25,0)
			
			\psline(5.25,-0.5)(4.75,0)
			\psline(5.75,-0.5)(5.25,0)
			\psline(7.25,1.5)(6.75,2)
			\psline(7.75,1.5)(7.25,2)
			
			\psline(2,3)(2,4)
			\psline(3.75,3)(3,4)
			\psline(4.25,3)(4,4)
			\psline(5.75,3)(4.75,4)
			\psline(6.25,3)(5.25,4)
			
			\psline(4,5)(4,5.5)

			\psdots[linecolor=red]%
			(1.75,1)(2.25,1)%
			(3.75,1)(4.25,1)%
			(5.75,1)(6.25,1)%
			(2,3)%
			(3.75,3)(4.25,3)%
			(5.75,3)(6.25,3)%
			(4,5)
			(3.75,-1)(4.25,-1)%
			
			\psdots[linecolor=darkgreen]
			(1,2)%
			(2.75,2)(3.25,2)%
			(4.75,2)(5.25,2)%
			(6.75,2)(7.25,2)%
			(3,4)%
			(4.75,4)(5.25,4)%
			(2.75,0)(3.25,0)%
			(4.75,0)(5.25,0)%
			
			\psdots[linecolor=blue]
			(2,2)%
			(2,4)(4,4)%
		}}
		
		\rput(12.5,7){\psrotate(0,0){180}{
				
				{\psset{linecolor=lightgray}
					\psframe*(-0.5,-0.5)(0.5,0.5)
					\psframe*(1.5,-0.5)(2.5,0.5)
					\psframe*(3.5,-0.5)(4.5,0.5)
					\psframe*(5.5,-0.5)(6.5,0.5)
					\psframe*(7.5,-0.5)(8.5,0.5)
					
					\psframe*(0.5,0.5)(1.5,1.5)
					\psframe*(2.5,0.5)(3.5,1.5)
					\psframe*(4.5,0.5)(5.5,1.5)
					\psframe*(6.5,0.5)(7.5,1.5)
					
					\psframe*(-0.5,1.5)(0.5,2.5)
					\psframe*(1.5,1.5)(2.5,2.5)
					\psframe*(3.5,1.5)(4.5,2.5)
					\psframe*(5.5,1.5)(6.5,2.5)
					
					\psframe*(0.5,2.5)(1.5,3.5)
					\psframe*(2.5,2.5)(3.5,3.5)
					\psframe*(4.5,2.5)(5.5,3.5)
					
					\psframe*(1.5,3.5)(2.5,4.5)
					\psframe*(3.5,3.5)(4.5,4.5)
					
					\psframe*(2.5,4.5)(3.5,5.5)
				}
				
				{\psset{linestyle=dashed,dash=2pt 1pt}
					\psline(1,0)(0,1)
					
					\psline(3,0)(2.25,1)
					\psline(5,0)(4.25,1)
					\psline(7,0)(6.25,1)
					\psline(9,0)(8.25,1)
					
					\psline(2,0)(1.75,1)
					\psline(4,0)(3.75,1)
					\psline(6,0)(5.75,1)
					\psline(8,0)(7.75,1)
					
					\psline(0,2)(1,2)
					\psline(2,2)(2.75,2)
					\psline(3.25,2)(2,3)
					\psline(4.75,2)(3.75,3)
					\psline(5.25,2)(4.25,3)
					\psline(6.75,2)(5.75,3)
					\psline(7.25,2)(6.25,3)
					
					\psline(2,4)(3,4)
					\psline(4,4)(4.75,4)
					\psline(5.25,4)(4,5)
				}
				\psline(1,0)(2,0)
				\psline(3,0)(4,0)
				\psline(5,0)(6,0)
				\psline(7,0)(8,0)
				
				\psline(0,1)(0,2)
				\psline(1.75,1)(1,2)
				\psline(2.25,1)(2,2)
				\psline(3.75,1)(2.75,2)
				\psline(4.25,1)(3.25,2)
				\psline(5.75,1)(4.75,2)
				\psline(6.25,1)(5.25,2)
				\psline(7.75,1)(6.75,2)
				\psline(8.25,1)(7.25,2)
				
				\psline(2,3)(2,4)
				\psline(3.75,3)(3,4)
				\psline(4.25,3)(4,4)
				\psline(5.75,3)(4.75,4)
				\psline(6.25,3)(5.25,4)
				
				\psline(4,5)(4,5.5)
				\psline(9,0)(9.5,0)

				\psdots[linecolor=red]%
				(0,1)%
				(1.75,1)(2.25,1)%
				(3.75,1)(4.25,1)%
				(5.75,1)(6.25,1)%
				(7.75,1)(8.25,1)%
				(2,3)%
				(3.75,3)(4.25,3)%
				(5.75,3)(6.25,3)%
				(4,5)
				
				\psdots[linecolor=darkgreen]
				(1,0)%
				(3,0)%
				(5,0)%
				(7,0)%
				(9,0)%
				(1,2)%
				(2.75,2)(3.25,2)%
				(4.75,2)(5.25,2)%
				(6.75,2)(7.25,2)%
				(3,4)%
				(4.75,4)(5.25,4)%
				
				\psdots[linecolor=blue]
				(2,0)(4,0)(6,0)(8,0)%
				(0,2)(2,2)%
				(2,4)(4,4)%
			}}
			
%
%
%
%
%
%
			
			\rput[r](11.5,8){\small$t_1^{-n}t_2^{n+2m}$}
			\psecurve{->}(10.5,7)(11.5,8)(12.5,7)(12.5,6.25)
			\rput[l](-11.5,-8){\small$\,t_1^{n}t_2^{-n-2m}$}
			\psecurve{->}(-10.5,-7)(-11.5,-8)(-12.5,-7)(-12.5,-6.25)
			
			\rput[l](-1.5,8){\small$\,t_1^{-n}t_2^{n-2m-2}$}
			\psecurve{->}(-0.5,7)(-1.5,8)(-2.5,7.25)(-2.5,6.25)
			\rput[r](1.5,-8){\small$t_1^{n}t_2^{-n+2m+2}$}
			\psecurve{->}(0.5,-7)(1.5,-8)(2.5,-7.25)(2.5,-6.25)
			
			\rput[r](-7.75,2){\small$t_1^{n-2m-2}t_2^{-n}$}
			\psline{->}(-7.75,2)(-6.75,2)
			\rput[l](7.75,-2){\small$\,t_1^{-n+2m+2}t_2^{n}$}
			\psline{->}(7.75,-2)(6.75,-2)
			
			\end{pspicture}
			\caption{$n>m+1$}\label{fig:ResultPretzelsCase2}
		\end{subfigure}
		
		\caption{The peculiar module of the $(2n,-(2m+1))$-pretzel tangle, shown in Figure~\ref{fig:pretzeltangle2nm2mp1}. We use the same conventions as in Example~\ref{exa:HFTdpretzeltangle}, see also Figures~\ref{fig:firstcomplex} and~\ref{fig:examplesimplifiedgraph}. All generators for site a and c, ie the red and green vertices, are in the same $\delta$-grading. The diagonals connecting pairs of red and green generators of the same Alexander gradings should be continued in such a way that they do not intersect each other. }\label{fig:ResultPretzels}
	\end{figure}
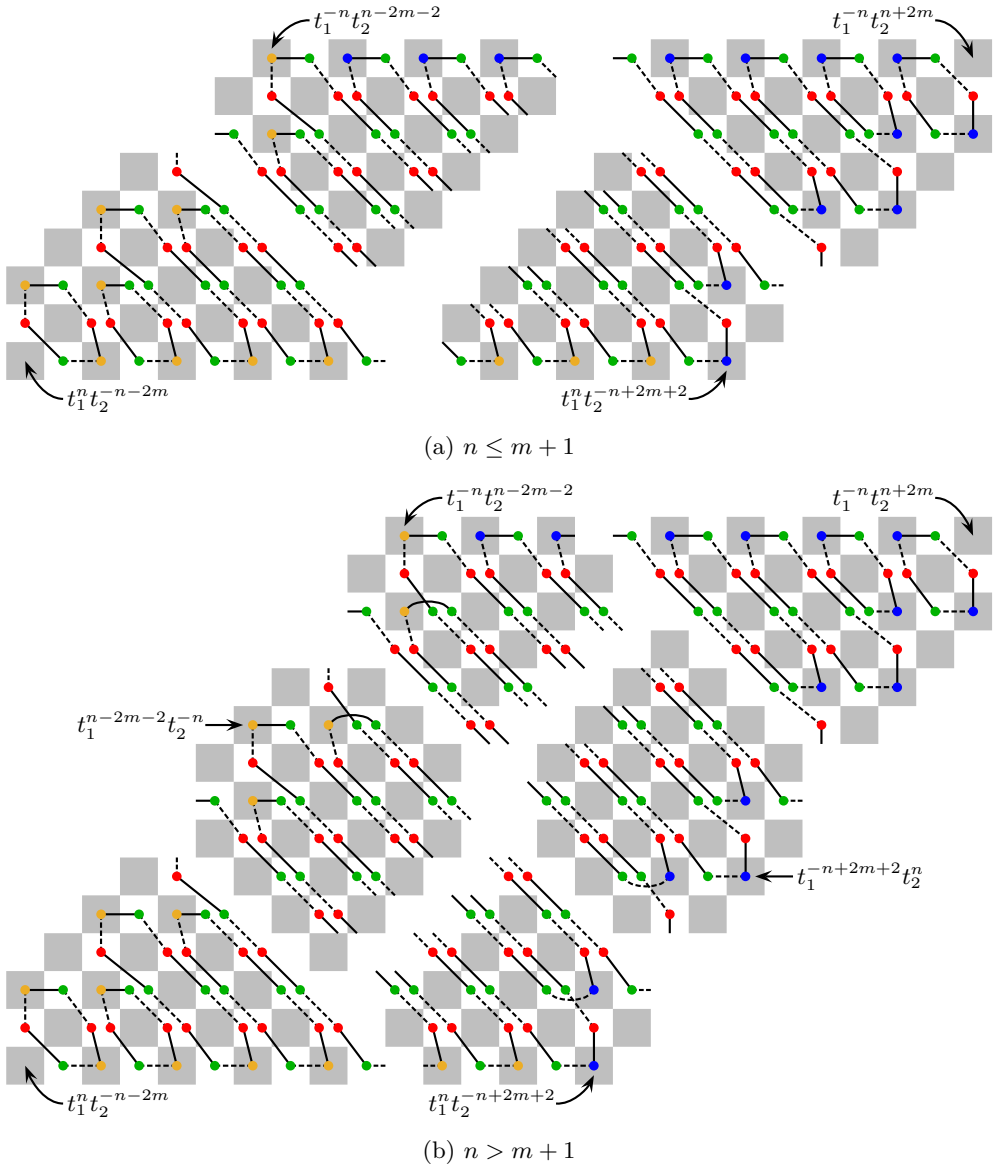

\begin{figure}[p]
	\begin{subfigure}[b]{\textwidth}\centering
		\psset{unit=0.45,linearc=0.25}
		\begin{pspicture}(-19.5,-17.5)(9.5,17.5)
		
		\rput(0,10){
			\pscustom*[linecolor=lightgray,linewidth=0pt]{
				\psline[liftpen=1](-3,0)(-3,-4)(-4,-4)
				\psline[liftpen=1](-4,-4)(-4,-5)
				\psline[liftpen=1](-4,-5)(3,-5)
				\psline[liftpen=1](3,-5)(3,-4)
				\psline[linecolor=violet,liftpen=1](3,-4)(-2,-4)(-2,0)(-3,0)
			}
			
			\psline[linecolor=blue](4,-5)(-6,-5)(-6,3)(0,3)(0,-2)(5,-2)(5,6)(-9,6)(-9,-8)(-3,-8)
			\psline[linecolor=violet](-3,0)(-3,-4)(-5,-4)(-5,2)(-1,2)(-1,-3)(6,-3)(6,7)(-10,7)(-10,-9)(-3,-9)
			\psline[linecolor=violet](-3,0)(-2,0)(-2,-4)(4,-4)
			\psline[linecolor=violet](4,-6)(-7,-6)(-7,4)(1,4)(1,-1)(4,-1)(4,5)(-8,5)(-8,-7)(-3,-7)
			\psline[linecolor=violet,linestyle=dotted,dotsep=1pt](4.5,-4)(4,-4)
			\psline[linecolor=blue,linestyle=dotted,dotsep=1pt](4.5,-5)(4,-5)
			\psline[linecolor=violet,linestyle=dotted,dotsep=1pt](4.5,-6)(4,-6)
			\psline[linecolor=violet,linestyle=dotted,dotsep=1pt](-2.5,-7)(-3,-7)
			\psline[linecolor=blue,linestyle=dotted,dotsep=1pt](-2.5,-8)(-3,-8)
			\psline[linecolor=violet,linestyle=dotted,dotsep=1pt](-2.5,-9)(-3,-9)
			\psline[linecolor=red](-3,0)(-3,1)(3,1)(3,0)
			\psline[linecolor=red](-3,0)(-4,0)(-4,-10)
			\psline[linecolor=red](3,0)(3,-7)
			\pscircle[fillstyle=solid,fillcolor=white](-3,0){0.5}
			\pscircle[fillstyle=solid,fillcolor=white,linecolor=darkgreen](3,0){0.5}
			{\psset{dotstyle=|}
				\psdot[dotangle=-45](-1,1)\uput{0.2}[45](-1,1){$\overline{d}$}
				\psdot[dotangle=-45](0,1)\uput{0.2}[45](0,1){$d$}
				\psdot[dotangle=-45](1,1)\uput{0.2}[45](1,1){$\underline{d}$}
				
				\psdot[dotangle=45](3,-1)\uput{0.2}[-45](3,-1){$\underline{c}_1$}
				\psdot[dotangle=45](3,-2)\uput{0.2}[-45](3,-2){$c_1$}
				\psdot[dotangle=45](3,-3)\uput{0.2}[-45](3,-3){$\overline{c}_1$}
				\psdot[dotangle=45](3,-4)\uput{0.2}[-45](3,-4){$\overline{c}_2$}
				\psdot[dotangle=45](3,-5)\uput{0.2}[-45](3,-5){$c_2$}
				\psdot[dotangle=45](3,-6)\uput{0.2}[-45](3,-6){$\underline{c}_2$}
				
				\psdot[dotangle=-45](-4,-4)\uput{0.2}[-135](-4,-4){$\overline{y}_1$}
				\psdot[dotangle=-45](-4,-5)\uput{0.2}[-135](-4,-5){$y_1$}
				\psdot[dotangle=-45](-4,-6)\uput{0.2}[-135](-4,-6){$\underline{y}_1$}
				\psdot[dotangle=-45](-4,-7)\uput{0.2}[-135](-4,-7){$\underline{y}_2$}
				\psdot[dotangle=-45](-4,-8)\uput{0.2}[-135](-4,-8){$y_2$}
				\psdot[dotangle=-45](-4,-9)\uput{0.2}[-135](-4,-9){$\overline{y}_2$}
				
				\uput{0.7}[65](3,0){$q_4$}
				\uput{0.7}[180](3,0){$p_4$}
			}
		}

		\rput(0,-10){
			\pscustom*[linecolor=lightgray,linewidth=0pt]{
				\psline[liftpen=1](3,7)(3,8)
				\psline[liftpen=1](3,8)(8,8)(8,-5)(-6,-5)(-6,2)(-4,2)
				\psline[liftpen=1](-4,2)(-4,1)
				\psline[liftpen=1](-4,1)(-5,1)(-5,0)(-4,0)
				\psline[liftpen=1](-4,0)(-4,-4)(7,-4)(7,7)(3,7)
			}

			\psline[linecolor=blue](2,8)(8,8)(8,-5)(-6,-5)(-6,2)(0,2)(0,-2)(5,-2)(5,5)(-5,5)
			\psline[linecolor=violet](-4,0)(-5,0)(-5,1)(-1,1)(-1,-3)(6,-3)(6,6)(-5,6)
			\psline[linecolor=violet](-4,0)(-4,-4)(7,-4)(7,7)(2,7)
			\psline[linecolor=violet](2,9)(9,9)(9,-6)(-7,-6)(-7,3)(1,3)(1,-1)(4,-1)(4,4)(-5,4)
			\psline[linecolor=violet,linestyle=dotted,dotsep=1pt](-5.5,4)(-5,4)
			\psline[linecolor=blue,linestyle=dotted,dotsep=1pt](-5.5,5)(-5,5)
			\psline[linecolor=violet,linestyle=dotted,dotsep=1pt](-5.5,6)(-5,6)
			\psline[linecolor=violet,linestyle=dotted,dotsep=1pt](1.5,7)(2,7)
			\psline[linecolor=blue,linestyle=dotted,dotsep=1pt](1.5,8)(2,8)
			\psline[linecolor=violet,linestyle=dotted,dotsep=1pt](1.5,9)(2,9)
			\psline[linecolor=red](-4,0)(3,0)
			\psline[linecolor=red](-4,0)(-4,7)
			\psline[linecolor=red](3,0)(3,10)
			\psline[linestyle=dotted,linecolor=red](-4,10)(-4,7)
			\psline[linestyle=dotted,linecolor=red](3,13)(3,10)
			\pscircle[fillstyle=solid,fillcolor=white](-4,0){0.5}
			\pscircle[fillstyle=solid,fillcolor=white,linecolor=darkgreen](3,0){0.5}
			
			{\psset{dotstyle=|}
				\psdot[dotangle=-45](-1,0)\uput{0.2}[-135](-1,0){$\underline{b}$}
				\psdot[dotangle=-45](0,0)\uput{0.2}[-135](0,0){$b$}
				\psdot[dotangle=-45](1,0)\uput{0.2}[-135](1,0){$\overline{b}$}
				
				\psdot[dotangle=-45](3,4)\uput{0.2}[-135](3,4){$\overline{c}_{2m+1}$}
				\psdot[dotangle=-45](3,5)\uput{0.2}[-135](3,5){$c_{2m+1}$}
				\psdot[dotangle=-45](3,6)\uput{0.2}[-135](3,6){$\underline{c}_{2m+1}$}
				\psdot[dotangle=-45](3,7)\uput{0.2}[-135](3,7){$\underline{c}_{2m}$}
				\psdot[dotangle=-45](3,8)\uput{0.2}[-135](3,8){$c_{2m}$}
				\psdot[dotangle=-45](3,9)\uput{0.2}[-135](3,9){$\overline{c}_{2m}$}
				
				\psdot[dotangle=-45](-4,1)\uput{0.2}[45](-4,1){$\underline{y}_{2m+1}$}
				\psdot[dotangle=-45](-4,2)\uput{0.2}[45](-4,2){$y_{2m+1}$}
				\psdot[dotangle=-45](-4,3)\uput{0.2}[45](-4,3){$\overline{y}_{2m+1}$}
				\psdot[dotangle=-45](-4,4)\uput{0.2}[45](-4,4){$\overline{y}_{2m}$}
				\psdot[dotangle=-45](-4,5)\uput{0.2}[45](-4,5){$y_{2m}$}
				\psdot[dotangle=-45](-4,6)\uput{0.2}[45](-4,6){$\underline{y}_{2m}$}
				
				\uput{0.7}[65](3,0){$q_3$}
				\uput{0.7}[135](3,0){$p_3$}
			}
		}
		
		\rput(-16,10){
			\pscustom*[linecolor=lightgray,linewidth=0pt]{
				\psline[liftpen=1](-1,-1)(-1,1)(0,1)
				\psline[liftpen=1](0,1)(0,0)(1,0)
				\psline[liftpen=1](1,0)(1,-1)(-1,-1)
			}
			
			\psline[linecolor=violet](1,0)(0,0)(0,2)(3,2)(3,-2)(-2,-2)
			\psline[linecolor=violet](1,0)(1,-1)(-2,-1)(-2,3)(4,3)(4,-3)(1,-3)
			\psline[linecolor=violet,linestyle=dotted,dotsep=1pt](0.5,-3)(1,-3)
			\psline[linecolor=violet,linestyle=dotted,dotsep=1pt](-2.5,-2)(-2,-2)
			\psline[linecolor=red](-1,0)(-1,1)(1,1)(1,0)
			\psline[linecolor=red](-1,0)(-1,-3)
			\psline[linecolor=red](1,0)(2,0)(2,-4)
			\pscircle[fillstyle=solid,fillcolor=white,linecolor=darkgreen](-1,0){0.5}
			\pscircle[fillstyle=solid,fillcolor=white](1,0){0.5}
			\psline[linecolor=darkgreen](-0.75,0)(-0.25,0)
			{\psset{dotstyle=|}
				\psdot[dotangle=45](0,1)\uput{0.2}[135](0,1){$d'$}
				
				\psdot[dotangle=45](2,-2)\uput{0.2}[-45](2,-2){$x_1$}
				\psdot[dotangle=45](2,-3)\uput{0.2}[-45](2,-3){$x_2$}
				
				\psdot[dotangle=-45](-1,-1)\uput{0.2}[-135](-1,-1){$a_1$}
				\psdot[dotangle=-45](-1,-2)\uput{0.2}[-135](-1,-2){$a_2$}
				
				\uput{0.7}[115](-1,0){$q_1$}
			}
		}
		
		\rput(-16,-10){
			\pscustom*[linecolor=lightgray,linewidth=0pt]{
				\psline[liftpen=1](-1,2)(-1,0)(0,0)
				\psline[liftpen=1](0,0)(0,-1)(2,-1)(2,0)(3,0)(3,-2)(-2,-2)(-2,2)(-1,2)
			}
			
			\psline[linecolor=violet](2,0)(2,-1)(0,-1)(0,1)(4,1)(4,-3)(-3,-3)(-3,3)(0,3)
			\psline[linecolor=violet](2,0)(3,0)(3,-2)(-2,-2)(-2,2)(3,2)
			\psline[linecolor=violet,linestyle=dotted,dotsep=1pt](0.5,3)(0,3)
			\psline[linecolor=violet,linestyle=dotted,dotsep=1pt](3.5,2)(3,2)
			\psline[linecolor=red](-1,0)(2,0)
			\psline[linecolor=red](-1,0)(-1,4)
			\psline[linecolor=red](2,0)(2,3)
			\psline[linecolor=red,linestyle=dotted](-1,17)(-1,4)
			\psline[linecolor=red,linestyle=dotted](2,16)(2,3)
			\pscircle[fillstyle=solid,fillcolor=white,linecolor=darkgreen](-1,0){0.5}
			\pscircle[fillstyle=solid,fillcolor=white](2,0){0.5}
			\rput(-1,0){\psline[linecolor=darkgreen](0.75;-135)(0.25;-135)}
			{\psset{dotstyle=|}
				\psdot[dotangle=-45](0,0)\uput{0.2}[45](0,0){$b'$}
				
				\psdot[dotangle=-45](2,1)\uput{0.2}[45](2,1){$x_{2n}$}
				\psdot[dotangle=-45](2,2)\uput{0.2}[45](2,2){$x_{2n-1}$}
				
				\psdot[dotangle=-45](-1,2)\uput{0.2}[45](-1,2){$a_{2n}$}
				\psdot[dotangle=-45](-1,3)\uput{0.2}[45](-1,3){$a_{2n-1}$}
				
				\uput{0.7}[65](-1,0){$p_2$}
				
			}
		}
		
		\end{pspicture}
		\caption{A niceified Heegaard diagram for a $(2n,-(2m+1))$-pretzel tangle with $n,m>0$.}\label{fig:CalculationPretzelsStep1HD}
	\end{subfigure}
	%
	%
	%
	%
	\begin{subfigure}[b]{0.97\textwidth}\centering
		\bigskip
		\begin{tabular}{cccccccc}
			&
			$\textcolor{red}{a_iy_j}$
			&
			$\textcolor{blue}{b'y_j}$
			&
			$\textcolor{blue}{x_ib}$
			&
			$\textcolor{darkgreen}{x_ic_j}$
			&&
			$\textcolor{gold}{d'y_j}$
			&
			$\textcolor{gold}{x_id}$
			\\ 
			&
			&
			$\textcolor{blue}{\underline{b}y_j}$
			&
			$\textcolor{blue}{\underline{y}_jb}$
			&
			$\textcolor{darkgreen}{\underline{c}_jy_{k}}$
			&
			$\textcolor{darkgreen}{\underline{y}_kc_{j}}$
			&
			$\textcolor{gold}{\underline{d}y_j}$
			&
			$\textcolor{gold}{\underline{y}_jd}$
			\\ 
			&
			&
			$\textcolor{blue}{\overline{b}y_j}$
			&
			$\textcolor{blue}{\overline{y}_jb}$
			&
			$\textcolor{darkgreen}{\overline{c}_jy_{k}}$
			&
			$\textcolor{darkgreen}{\overline{y}_kc_{j}}$
			&
			$\textcolor{gold}{\overline{d}y_j}$
			&
			$\textcolor{gold}{\overline{y}_jd}$
		\end{tabular}
		\medskip 
		\caption{Generators of the Heegaard diagram above, where $1\leq i\leq 2n$ and $1\leq j,k\leq 2m+1$. The generators of  the second and third row can be cancelled. }\label{fig:CalculationPretzelsStep1Gens}
	\end{subfigure}
	\caption{The first step of the calculation of the peculiar modules of $(2n,-(2m+1))$-pretzel tangles.}\label{fig:CalculationPretzelsStep1}
\end{figure}

\begin{figure}[h]
	\centering
	\begin{subfigure}[b]{0.37\textwidth}
		{
			\[
			\begin{tikzcd}[row sep=0.6cm, column sep=0.4cm]
			\textcolor{darkgreen}{x_{i+1}c_{j-1}}
			\arrow{dr}{q_{14}}	
			\\
			&
			\textcolor{red}{\ovalbox{$a_iy_j$}}
			\\
			&&
			\textcolor{darkgreen}{x_{i-1}c_{j+1}}
			\arrow[swap]{ul}{p_{23}}	
			\end{tikzcd}
			\]
		}
		\caption{$1<i<2n$ and $1<j<2m+1$}\label{fig:CalculationPretzelsStep2Redij}
	\end{subfigure}
	\begin{subfigure}[b]{0.28\textwidth}
		{
			\[
			\begin{tikzcd}[row sep=0.6cm, column sep=0.4cm]
			\textcolor{gold}{x_{i+1}d}
			\arrow{d}{q_{1}}	
			\\
			\textcolor{red}{\ovalbox{$a_iy_1$}}
			\\
			&
			\textcolor{darkgreen}{x_{i-1}c_{2}}
			\arrow[swap]{ul}{p_{23}}	
			\end{tikzcd}
			\]
		}
		\caption{$1<i<2n$ and $j=1$}\label{fig:CalculationPretzelsStep2Redi1}
	\end{subfigure}
	\begin{subfigure}[b]{0.33\textwidth}
		{
			\[
			\begin{tikzcd}[row sep=0.6cm, column sep=0.4cm]
			\textcolor{darkgreen}{x_{i+1}c_{2m}}
			\arrow{rd}{q_{14}}	
			\\
			&
			\textcolor{red}{\ovalbox{$a_{i}y_{2m+1}$}}
			\\
			&
			\textcolor{blue}{x_{i-1}b}
			\arrow[swap]{u}{p_{2}}	
			\end{tikzcd}
			\]
		}
		\caption{$1<i<2n$ and $j=2m+1$}\label{fig:CalculationPretzelsStep2Redi2mp1}
	\end{subfigure}
	\\
	\begin{subfigure}[b]{0.37\textwidth}
		{
			\[
			\begin{tikzcd}[row sep=0.6cm, column sep=0.4cm]
			\textcolor{red}{a_{i+1}y_{j-1}}	
			\\
			&
			\textcolor{darkgreen}{\ovalbox{$x_{i}c_{j}$}}
			\arrow[swap]{ul}{p_{23}}
			\arrow{dr}{q_{14}}
			\\
			&&
			\textcolor{red}{a_{i-1}y_{j+1}}	
			\end{tikzcd}
			\]
		}
		\caption{$1<i<2n$ and $1<j<2m+1$}\label{fig:CalculationPretzelsStep2Greenij}
	\end{subfigure}
	\begin{subfigure}[b]{0.28\textwidth}
		{
			\[
			\begin{tikzcd}[row sep=0.6cm, column sep=0.4cm]
			\textcolor{gold}{x_{i}d}
			\arrow{r}{p_4}	
			&
			\textcolor{darkgreen}{\ovalbox{$x_{i}c_{1}$}}
			\arrow{dr}{q_{14}}
			\\
			&&
			\textcolor{red}{a_{i-1}y_2}	
			\end{tikzcd}
			\]
		}
		\caption{$1<i<2n$ and $j=1$}\label{fig:CalculationPretzelsStep2Greeni1}
	\end{subfigure}
	\begin{subfigure}[b]{0.33\textwidth}
		{
			\[
			\begin{tikzcd}[row sep=0.6cm, column sep=0.4cm]
			\textcolor{red}{a_{i+1}y_{2m}}	
			\\
			&
			\textcolor{darkgreen}{\ovalbox{$x_{i}c_{2m+1}$}}
			\arrow[swap]{lu}{p_{23}}	
			&
			\textcolor{blue}{x_{i}b}
			\arrow[swap]{l}{q_{3}}\\
			\phantom{x_1}
			\end{tikzcd}
			\]
		}
		\caption{$1<i<2n$ and $j=2m+1$}\label{fig:CalculationPretzelsStep2Greeni2mp1}
	\end{subfigure}
	\\
	\begin{subfigure}[b]{0.27\textwidth}
		{
			\[
			\begin{tikzcd}[row sep=0.6cm, column sep=0.4cm]
			&
			\textcolor{red}{a_{i+1}y_{2m+1}}
			\\
			\textcolor{darkgreen}{x_{i}c_{2m+1}}
			&
			\textcolor{blue}{\ovalbox{$x_{i}b$}}
			\arrow[swap]{l}{q_{3}}	
			\arrow{u}{p_{2}}
			\end{tikzcd}
			\]
		}
		\caption{$1\leq i<2n$}\label{fig:CalculationPretzelsStep2Bluei}
	\end{subfigure}
	\begin{subfigure}[b]{0.27\textwidth}
		{
			\[
			\begin{tikzcd}[row sep=0.6cm, column sep=0.4cm]
			\textcolor{gold}{\ovalbox{$x_{i}d$}}
			\arrow[swap]{d}{q_{1}}	
			\arrow{r}{p_{4}}
			&
			\textcolor{darkgreen}{x_{i}c_{1}}
			\\
			\textcolor{red}{a_{i-1}y_{1}}
			\end{tikzcd}
			\]
		}
		\caption{$1< i\leq 2n$}\label{fig:CalculationPretzelsStep2Goldi}
	\end{subfigure}		
	\caption{Some differentials for the computation of the $(2n,-(2m+1))$-pretzel tangle in non-extremal $t_1$-Alexander grading.}\label{fig:CalculationPretzelsAStep2generic}
\end{figure}
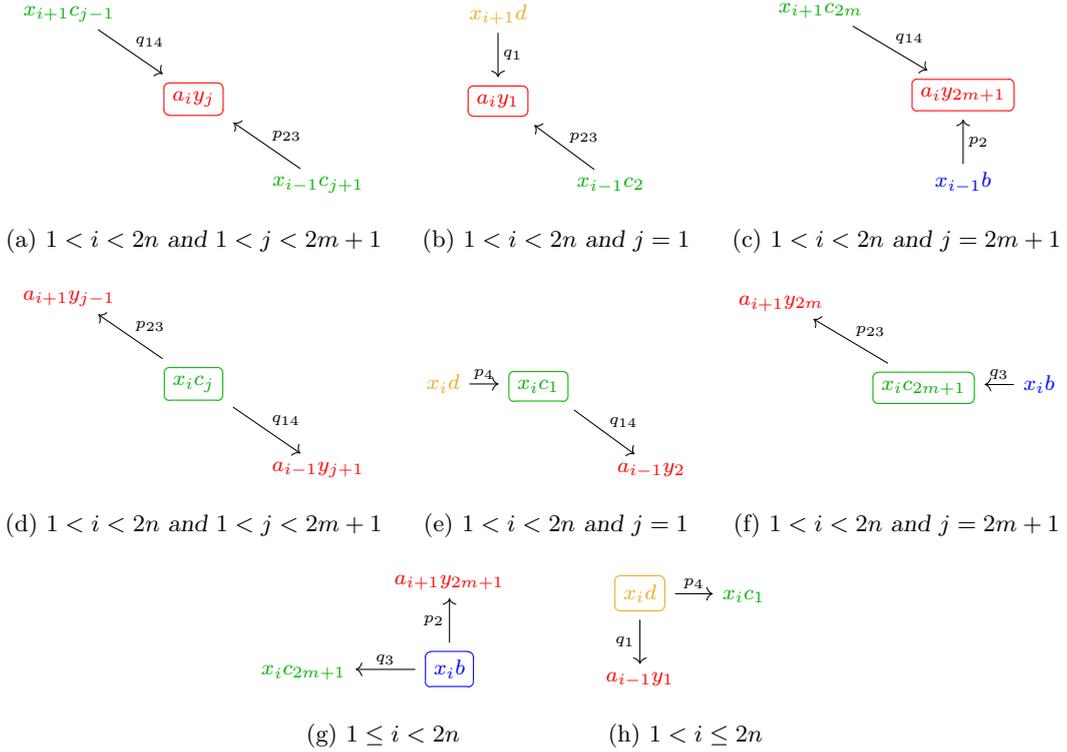

\begin{figure}[p]
	\begin{subfigure}[b]{0.95\textwidth}
		{
		\[
		\begin{tikzcd}[row sep=0.6cm, column sep=0.4cm,%
		execute at end picture={
			\begin{pgfonlayer}{background}
			\foreach \Nombre in  {A,B,...,F}
			{\coordinate (\Nombre) at (\Nombre.center);}
			\fill[lightgray!40] 
			([xshift=-27pt,yshift=9pt]A) [rounded corners=5pt] %
			-- ([xshift=-27pt,yshift=-12pt]A) %
			-- ([xshift=21pt,yshift=-12pt]B) %
			-- ([xshift=21pt,yshift=9pt]B) %
			-- cycle;
			\fill[lightgray!40] 
			([xshift=-21pt,yshift=12pt]C) [rounded corners=5pt] %
			-- ([xshift=-21pt,yshift=-12pt]C) %
			-- ([xshift=10pt,yshift=-12pt]D) %
			-- ([xshift=10pt,yshift=12pt]D) %
			-- cycle;
			\fill[lightgray!40] 
			([xshift=-18pt,yshift=9pt]E) [rounded corners=5pt] %
			-- ([xshift=-18pt,yshift=-9pt]E) %
			-- ([xshift=30pt,yshift=-9pt]F) %
			-- ([xshift=30pt,yshift=9pt]F) %
			-- cycle;
			\end{pgfonlayer}
		}
		]
		|[alias=E]|
		\textcolor{darkgreen}{x_2c_{2m}}
		\arrow[swap]{drr}{q_{14}}
		&
		\textcolor{darkgreen}{x_3c_{2m-1}}
		\arrow[dashed,swap,pos=0.25]{drr}{q_{14}}
		&
		|[alias=F]|
		\textcolor{darkgreen}{\underline{y}_{2m+1}c_{2m-1}}
		\arrow[dotted]{dr}{q_{14}}
		\\
		&&
		|[alias=A]|
		\textcolor{red}{\ovalbox{$a_1y_{2m+1}$}}
		&
		|[alias=B]|
		\textcolor{red}{\ovalbox{$a_2y_{2m}$}}
		\\
		\textcolor{darkgreen}{x_{1}c_{2m}}
		\arrow{r}{q_4}
		&
		|[alias=C]|
		\textcolor{gold}{d'y_{2m+1}}
		&
		\textcolor{blue}{\ovalbox{$\overline{y}_{1}b$}}
		\arrow{r}{1}
		\arrow{u}{p_2}
		&
		|[alias=D]|
		\textcolor{blue}{\overline{b}y_{1}}
		&
		\textcolor{darkgreen}{\ovalbox{$x_{1}c_{2m+1}$}}
		\arrow[swap]{l}{p_3}
		\arrow[swap]{ul}{p_{23}}
		&
		\textcolor{blue}{x_1b}
		\arrow[swap]{l}{q_3}
		\end{tikzcd}
		\]
		}
		\caption{}\label{fig:CalculationPretzelsStep2BottomEnd}
	\end{subfigure}
	\\
	\begin{subfigure}[b]{0.95\textwidth}
		{
			\[
			\begin{tikzcd}[row sep=0.6cm, column sep=0.4cm,%
			execute at end picture={
				\begin{pgfonlayer}{background}
				\foreach \Nombre in  {A,B,...,H}
				{\coordinate (\Nombre) at (\Nombre.center);}
				\fill[lightgray!40] 
				([xshift=-25pt,yshift=9pt]A) [rounded corners=5pt] %
				-- ([xshift=-25pt,yshift=-13pt]A) %
				-- ([xshift=19pt,yshift=-13pt]B) %
				-- ([xshift=19pt,yshift=9pt]B) %
				-- cycle;
				\fill[lightgray!40] 
				([xshift=-16pt,yshift=9pt]G) [rounded corners=5pt] %
				-- ([xshift=-16pt,yshift=-9pt]G) %
				-- ([xshift=29pt,yshift=-9pt]H) %
				-- ([xshift=29pt,yshift=9pt]H) %
				-- cycle;
				\fill[lightgray!40] 
				([xshift=-25pt,yshift=11pt]C) [rounded corners=5pt] %
				-- ([xshift=-25pt,yshift=-13pt]C) %
				-- ([xshift=24pt,yshift=-13pt]D) %
				-- ([xshift=24pt,yshift=11pt]D) %
				-- cycle;
				\fill[lightgray!40] 
				([shift=(145:23pt)]E) %
				[rounded corners=5pt]-- ([shift=(35:23pt)]E) %
				-- ([shift=(-35:23pt)]F) %
				-- ([shift=(-145:23pt)]F) %
				-- cycle;
				\end{pgfonlayer}
			}
			]
			|[alias=G]|
			\textcolor{darkgreen}{x_2c_{2j}}
			\arrow[swap]{drr}{q_{14}}
			&
			\textcolor{darkgreen}{x_3c_{2j-1}}
			\arrow[dashed,swap,pos=0.25]{drr}{q_{14}}
			&
			|[alias=H]|
			\textcolor{darkgreen}{\underline{y}_{2m+1}c_{2j-1}}
			\arrow[dotted]{dr}{q_{14}}
			\\
			&&
			|[alias=A]|
			\textcolor{red}{\ovalbox{$a_1y_{2j+1}$}}
			&
			|[alias=B]|
			\textcolor{red}{\ovalbox{$a_2y_{2j}$}}
			&&&
			|[alias=E]|
			\textcolor{gold}{\overline{d}y_{2j+3}}
			\\
			&
			\textcolor{darkgreen}{x_{1}c_{2j}}
			\arrow{r}{q_4}
			&
			\textcolor{gold}{d'y_{2j+1}}
			&
			|[alias=C]|
			\textcolor{darkgreen}{\ovalbox{$\overline{y}_{1}c_{2j+2}$}}
			\arrow{r}{1}
			\arrow{ul}{p_{23}}
			\arrow[in=180,out=90,looseness=0.3,pos=.7]{rrru}{q_4}
			&
			\textcolor{darkgreen}{\overline{c}_{2j+2}y_{1}}
			&
			|[alias=D]|
			\textcolor{darkgreen}{\ovalbox{$x_{1}c_{2j+1}$}}
			\arrow[swap]{l}{1}
			\arrow{ull}{p_{23}}
			\arrow{r}{q_4}
			&
			|[alias=F]|
			\textcolor{gold}{d'y_{2j+2}}
			\end{tikzcd}
			\]
		}
		\caption{$1\leq j< m$}\label{fig:CalculationPretzelsStep2BottomOdd}
	\end{subfigure}
	\\
	\begin{subfigure}[b]{0.95\textwidth}
		{
		\[
		\begin{tikzcd}[row sep=0.6cm, column sep=0.4cm,%
		execute at end picture={
			\begin{pgfonlayer}{background}
			\foreach \Nombre in  {A,B,...,D}
			{\coordinate (\Nombre) at (\Nombre.center);}
			\fill[lightgray!40] 
			([xshift=-16pt,yshift=9pt]A) [rounded corners=5pt] %
			-- ([xshift=-16pt,yshift=-13pt]A) %
			-- ([xshift=25pt,yshift=-13pt]B) %
			-- ([xshift=25pt,yshift=9pt]B) %
			-- cycle;
			\fill[lightgray!40] 
			([xshift=-21pt,yshift=9pt]C) [rounded corners=5pt] %
			-- ([xshift=-21pt,yshift=-9pt]C) %
			-- ([xshift=29pt,yshift=-9pt]D) %
			-- ([xshift=29pt,yshift=9pt]D) %
			-- cycle;
			\end{pgfonlayer}
		}]
		|[alias=C]|
		\textcolor{darkgreen}{x_3c_{2j-2}}
		\arrow[dashed,swap,pos=0.25]{drr}{q_{14}}
		&
		|[alias=D]|
		\textcolor{darkgreen}{\underline{y}_{2m+1}c_{2j-2}}
		\arrow[dotted]{dr}{q_{14}}
		\\
		&
		|[alias=A]|
		\textcolor{red}{a_1y_{2j}}
		&
		|[alias=B]|
		\textcolor{red}{\ovalbox{$a_2y_{2j-1}$}}
		\\
		\textcolor{darkgreen}{x_{1}c_{2j-1}}
		\arrow{r}{q_4}
		&
		\textcolor{gold}{d'y_{2j}}
		&
		&
		\textcolor{darkgreen}{\ovalbox{$x_{1}c_{2j}$}}
		\arrow{ul}{p_{23}}
		\arrow{r}{q_4}
		&
		\textcolor{gold}{d'y_{2j+1}}
		\end{tikzcd}
		\]
		}
		\caption{$1\leq j\leq m$}\label{fig:CalculationPretzelsStep2BottomEven}
	\end{subfigure}
	\\
	\begin{subfigure}[b]{0.3\textwidth}
		{
			\[
			\begin{tikzcd}[row sep=0.6cm, column sep=0.4cm,%
			execute at end picture={
				\begin{pgfonlayer}{background}
				\foreach \Nombre in  {A,B,...,D}
				{\coordinate (\Nombre) at (\Nombre.center);}
				\fill[lightgray!40] 
				([xshift=-17pt,yshift=11pt]A) [rounded corners=5pt]%
				-- ([xshift=17pt,yshift=11pt]A) %
				-- ([xshift=17pt,yshift=-11pt]B) %
				-- ([xshift=-17pt,yshift=-11pt]B) %
				-- cycle;
				\fill[lightgray!40] 
				([xshift=-16pt,yshift=13pt]C) [rounded corners=5pt]%
				-- ([xshift=16pt,yshift=13pt]C) %
				-- ([xshift=16pt,yshift=-11pt]D) %
				-- ([xshift=-16pt,yshift=-11pt]D) %
				-- cycle;
				\end{pgfonlayer}
			}
			]
			\textcolor{gold}{x_{2}d}
			\arrow[swap]{d}{q_{1}}	
			\\
			\textcolor{red}{\ovalbox{$a_{1}y_1$}}
			\\
			&
			|[alias=A]|
			\textcolor{darkgreen}{\ovalbox{$\overline{y}_1c_2$}}
			\arrow[swap]{d}{1}
			\arrow[swap]{lu}{p_{23}}
			\\
			|[alias=C]|
			\textcolor{gold}{\ovalbox{$d'y_1$}}
			\arrow{r}{p_{4}}
			&
			\textcolor{darkgreen}{\ovalbox{$\overline{c}_2y_1$}}		
			\\
			|[alias=D]|
			\textcolor{gold}{\ovalbox{$x_1d$}}
			\arrow{r}{p_{4}}
			\arrow{u}{1}
			&
			|[alias=B]|
			\textcolor{darkgreen}{\ovalbox{$x_1c_1$}}
			\arrow{u}{1}
			\arrow{r}{q_4}
			&
			\textcolor{gold}{d'y_2}
			\end{tikzcd}
			\]
		}
		\caption{}\label{fig:CalculationPretzelsStep2GoldCancel}
	\end{subfigure}
	
	\caption{Some differentials for the computation of the $(2n,-(2m+1))$-pretzel tangle at generators in maximal $t_1$-Alexander grading before cancellation. The dotted arrows appear iff $n=1$, the dashed arrows iff $n>1$.}\label{fig:CalculationPretzelsAStep2MaximumP}
\end{figure}
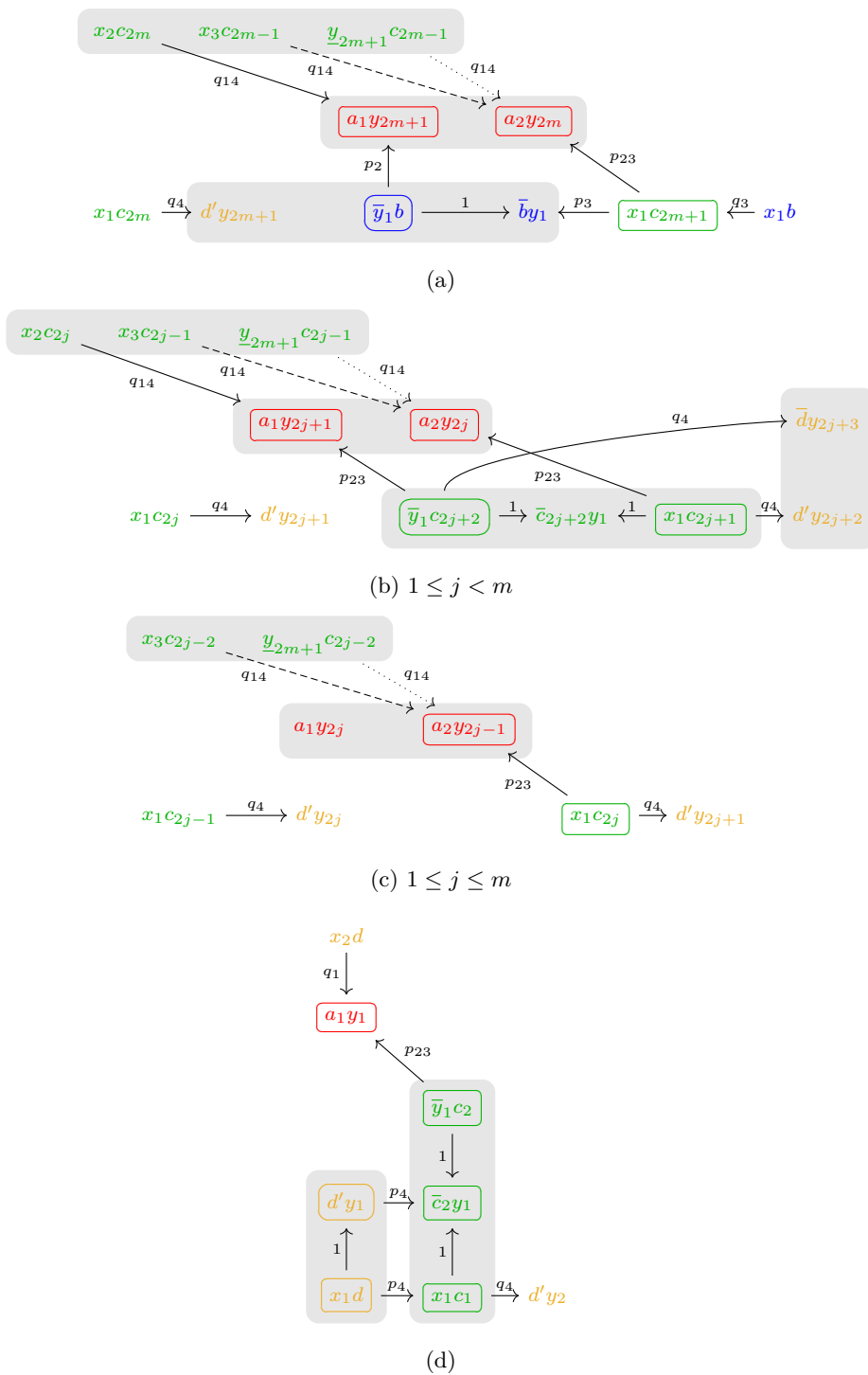 

\begin{figure}[b]
	\begin{subfigure}[b]{0.95\textwidth}
		{
			\[
			\begin{tikzcd}[row sep=0.6cm, column sep=0.4cm,%
			execute at end picture={
				\begin{pgfonlayer}{background}
				\foreach \Nombre in  {A,B,...,D}
				{\coordinate (\Nombre) at (\Nombre.center);}
				\fill[lightgray!40] 
				([xshift=-23pt,yshift=8pt]C) [rounded corners=5pt] %
				-- ([xshift=-23pt,yshift=-10pt]C) %
				-- ([xshift=21pt,yshift=-10pt]D) %
				-- ([xshift=21pt,yshift=8pt]D) %
				-- cycle;
				\fill[lightgray!40] 
				([xshift=-25pt,yshift=9pt]A) [rounded corners=5pt] %
				-- ([xshift=-25pt,yshift=-12pt]A) %
				-- ([xshift=20pt,yshift=-12pt]B) %
				-- ([xshift=20pt,yshift=9pt]B) %
				-- cycle;
				\end{pgfonlayer}
			}
			]
			\textcolor{gold}{x_{2n}d}
			\arrow{r}{p_4}
			&
			\textcolor{darkgreen}{x_{2n}c_{1}}
			\arrow{rrd}{q_{14}}
			\arrow[swap]{rd}{q_{14}}
			&
			&
			\textcolor{blue}{b'y_1}
			\arrow{d}{q_{143}}
			&
			\textcolor{darkgreen}{x_{2n}c_{2}}
			\arrow[swap]{l}{p_3}
			\\
			&&
			|[alias=A]|
			\textcolor{red}{a_{2n-1}y_{2}}
			&
			|[alias=B]|
			\textcolor{red}{a_{2n}y_1}
			\\
			&&&&
			|[alias=C]|
			\textcolor{darkgreen}{x_{2n-2}c_{3}}
			\arrow[dashed,swap,pos=0.25]{llu}{p_{23}}
			&
			|[alias=D]|
			\textcolor{darkgreen}{x_{2n-1}c_{2}}
			\arrow[swap]{llu}{p_{23}}
			\end{tikzcd}
			\]
		}
		\caption{}\label{fig:CalculationPretzelsStep2TopEnd}
	\end{subfigure}
	\\
	\begin{subfigure}[b]{\textwidth}
		{
			\[
			\begin{tikzcd}[row sep=0.6cm, column sep=0.35cm,%
			execute at end picture={
				\begin{pgfonlayer}{background}
				\foreach \Nombre in  {A,B,...,D}
				{\coordinate (\Nombre) at (\Nombre.center);}
				\fill[lightgray!40] 
				([xshift=-33pt,yshift=9pt]A) [rounded corners=5pt] %
				-- ([xshift=-33pt,yshift=-13pt]A) %
				-- ([xshift=27pt,yshift=-13pt]B) %
				-- ([xshift=27pt,yshift=9pt]B) %
				-- cycle;
				\fill[lightgray!40] 
				([xshift=-30pt,yshift=9pt]C) [rounded corners=5pt] %
				-- ([xshift=-30pt,yshift=-9pt]C) %
				-- ([xshift=30pt,yshift=-9pt]D) %
				-- ([xshift=30pt,yshift=9pt]D) %
				-- cycle;
				\end{pgfonlayer}
			}
			]
			\textcolor{blue}{b'y_{2j}}
			&
			\textcolor{darkgreen}{x_{2n}c_{2j+1}}
			\arrow[swap]{l}{p_3}
			\arrow[swap]{rd}{q_{14}}
			\arrow{rrd}{q_{14}}
			&
			&
			\textcolor{blue}{b'y_{2j+1}}
			\arrow{d}{q_{143}}
			&
			\textcolor{darkgreen}{x_{2n}c_{2j+2}}
			\arrow[swap]{l}{p_3}
			\\
			&&
			|[alias=A]|
			\textcolor{red}{a_{2n-1}y_{2j+2}}
			&
			|[alias=B]|
			\textcolor{red}{a_{2n}y_{2j+1}}
			\\
			&&&&
			|[alias=C]|
			\textcolor{darkgreen}{x_{2n-2}c_{2j+3}}
			\arrow[dashed,swap,pos=0.25]{llu}{p_{23}}
			&
			|[alias=D]|
			\textcolor{darkgreen}{x_{2n-1}c_{2j+2}}
			\arrow[swap]{llu}{p_{23}}
			\end{tikzcd}
			\]
		}
		\caption{$1\leq j<m$}\label{fig:CalculationPretzelsStep2TopOdd}
	\end{subfigure}
	\\
	\begin{subfigure}[b]{0.95\textwidth}
		{
			\[
			\begin{tikzcd}[row sep=0.6cm, column sep=0.4cm,%
			execute at end picture={
				\begin{pgfonlayer}{background}
				\foreach \Nombre in  {A,B,C,D}
				{\coordinate (\Nombre) at (\Nombre.center);}
				\fill[lightgray!40] 
				([xshift=-33pt,yshift=9pt]A) [rounded corners=5pt] %
				-- ([xshift=-33pt,yshift=-13pt]A) %
				-- ([xshift=18pt,yshift=-13pt]B) %
				-- ([xshift=18pt,yshift=9pt]B) %
				-- cycle;
				\end{pgfonlayer}
			}
			]
			\textcolor{blue}{b'y_{2j-1}}
			&
			\textcolor{darkgreen}{x_{2n}c_{2j}}
			\arrow[swap]{l}{p_3}
			\arrow{rd}{q_{14}}
			&
			&
			\textcolor{blue}{b'y_{2j}}
			\arrow{d}{q_{143}}
			&
			\textcolor{darkgreen}{x_{2n}c_{2j+1}}
			\arrow[swap]{l}{p_3}
			\\
			&&
			|[alias=A]|
			\textcolor{red}{a_{2n-1}y_{2j+1}}
			&
			|[alias=B]|
			\textcolor{red}{a_{2n}y_{2j}}
			\\
			&&&&
			\textcolor{darkgreen}{x_{2n-2}c_{2j+2}}
			\arrow[dashed,swap,pos=0.25]{llu}{p_{23}}
			\end{tikzcd}
			\]
		}
		\caption{$1\leq j\leq m$}\label{fig:CalculationPretzelsStep2TopEven}
	\end{subfigure}
	\\
	\begin{subfigure}[b]{0.4\textwidth}
		{
			\[
			\begin{tikzcd}[row sep=0.6cm, column sep=0.4cm	]
			\textcolor{blue}{b'y_{2m}}
			&
			\textcolor{darkgreen}{x_{2n}c_{2m+1}}
			\arrow[swap]{l}{p_{3}}
			\arrow{rd}{q_{14}}
			\\
			&&
			\textcolor{red}{a_{2n}y_{2m+1}}
			\\
			&&
			\textcolor{blue}{x_{2n-1}b}
			\arrow[swap]{u}{p_{2}}
			\end{tikzcd}
			\]
		}
		\caption{}\label{fig:CalculationPretzelsStep2BlueCancel}
	\end{subfigure}
	\caption{Some differentials for the computation of the $(2n,-(2m+1))$-pretzel tangle at generators in minimal $t_1$-Alexander grading after cancellation. The dashed arrows appear iff $n>1$.}\label{fig:CalculationPretzelsAStep2MinimumP}
\end{figure}
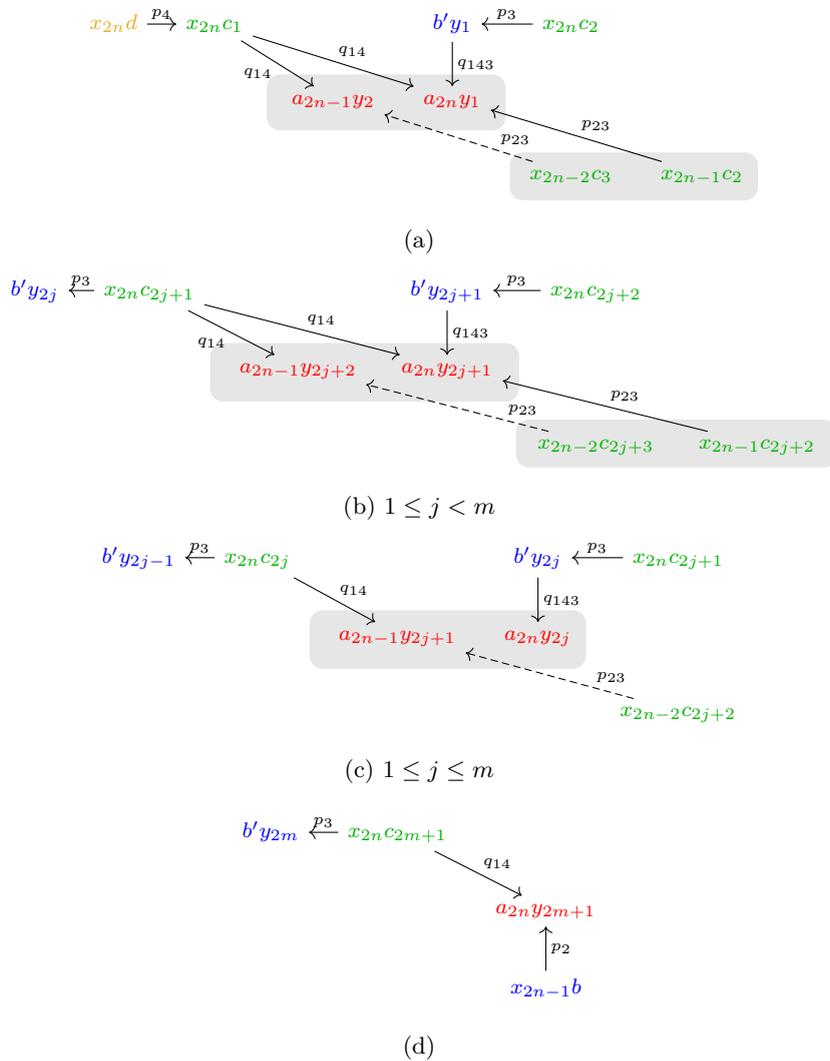

\begin{figure}[t]
	\centering
	\psset{unit=0.5}
	\begin{pspicture}(-13,-7.5)(13,7.5)
	
	\rput(-12.5,-7){
		
		{\psset{linecolor=lightgray}
			\psframe*(-0.5,-0.5)(0.5,0.5)
			\psframe*(1.5,-0.5)(2.5,0.5)
			\psframe*(3.5,-0.5)(4.5,0.5)
			\psframe*(5.5,-0.5)(6.5,0.5)
			\psframe*(7.5,-0.5)(8.5,0.5)
			
			\psframe*(0.5,0.5)(1.5,1.5)
			\psframe*(2.5,0.5)(3.5,1.5)
			\psframe*(4.5,0.5)(5.5,1.5)
			\psframe*(6.5,0.5)(7.5,1.5)
			
			\psframe*(-0.5,1.5)(0.5,2.5)
			\psframe*(1.5,1.5)(2.5,2.5)
			\psframe*(3.5,1.5)(4.5,2.5)
			\psframe*(5.5,1.5)(6.5,2.5)
			
			\psframe*(0.5,2.5)(1.5,3.5)
			\psframe*(2.5,2.5)(3.5,3.5)
			\psframe*(4.5,2.5)(5.5,3.5)
			
			\psframe*(1.5,3.5)(2.5,4.5)
			\psframe*(3.5,3.5)(4.5,4.5)
			
			\psframe*(2.5,4.5)(3.5,5.5)
		}
		
		\psline(1,0)(0,1)
		
		\psline(3,0)(2.25,1)
		\psline(5,0)(4.25,1)
		\psline(7,0)(6.25,1)
		\psline(9,0)(8.25,1)
		
		\psline(2,0)(1.75,1)
		\psline(4,0)(3.75,1)
		\psline(6,0)(5.75,1)
		\psline(8,0)(7.75,1)
		
		\psline(0,2)(1,2)
		\pscurve(2,2)(2.25,2.2)(3,2.2)(3.25,2)
		\psline(2.75,2)(2,3)
		\psline(4.75,2)(3.75,3)
		\psline(5.25,2)(4.25,3)
		\psline(6.75,2)(5.75,3)
		\psline(7.25,2)(6.25,3)
		
		\psline(2,4)(3,4)
		\pscurve(4,4)(4.25,4.2)(5,4.2)(5.25,4)
		\psline(4.75,4)(4,5)
		
		{\psset{linestyle=dashed,dash=2pt 1pt}
			\psline(1,0)(2,0)
			\psline(3,0)(4,0)
			\psline(5,0)(6,0)
			\psline(7,0)(8,0)
			
			\psline(0,1)(0,2)
			\psline(1.75,1)(1,2)
			\psline(2.25,1)(2,2)
			\psline(3.75,1)(2.75,2)
			\psline(4.25,1)(3.25,2)
			\psline(5.75,1)(4.75,2)
			\psline(6.25,1)(5.25,2)
			\psline(7.75,1)(6.75,2)
			\psline(8.25,1)(7.25,2)
			
			\psline(2,3)(2,4)
			\psline(3.75,3)(3,4)
			\psline(4.25,3)(4,4)
			\psline(5.75,3)(4.75,4)
			\psline(6.25,3)(5.25,4)
			
			\psline(4,5)(4,5.5)
			\psline(9,0)(9.5,0)
		}
		
		\pscurve{->}(5,0)(4.75,0.1)(4,0.8)(3.75,1)
		\pscurve{->}(9,0)(8.75,0.1)(8,0.8)(7.75,1)
		\pscurve{->}(4.25,1)(4.2,0.5)(4,0)
		\pscurve{->}(8.25,1)(8.2,0.5)(8,0)
		
		\psdots[linecolor=red]%
		(0,1)%
		(1.75,1)(2.25,1)%
		(3.75,1)(4.25,1)%
		(5.75,1)(6.25,1)%
		(7.75,1)(8.25,1)%
		(2,3)%
		(3.75,3)(4.25,3)%
		(5.75,3)(6.25,3)%
		(4,5)
		
		{\tiny
			\uput{0.15}[180](1.75,1){oe}
			\uput{0.15}[0](2.25,1){eo}
			\uput{0.15}[180](3.75,1){oo}
			\uput{0.15}[0](4.25,1){ee}
			\uput{0.15}[180](5.75,1){oe}
			\uput{0.15}[0](6.25,1){eo}
			\uput{0.15}[180](7.75,1){oo}
			\uput{0.15}[0](8.25,1){ee}
			
			\uput{0.15}[180](2,3){oo}
			\uput{0.15}[180](4,5){oo}
		}
		
		\psdots[linecolor=darkgreen]
		(1,0)%
		(3,0)%
		(5,0)%
		(7,0)%
		(9,0)%
		(1,2)%
		(2.75,2)(3.25,2)%
		(4.75,2)(5.25,2)%
		(6.75,2)(7.25,2)%
		(3,4)%
		(4.75,4)(5.25,4)%
		
		{\tiny
			\uput{0.15}[90](1,2){eo}
			\uput{0.15}[90](3,4){eo}
		}
		
		\psdots[linecolor=gold]
		(2,0)(4,0)(6,0)(8,0)%
		(0,2)(2,2)%
		(2,4)(4,4)%
		
		{\tiny
			\uput{0.15}[-90](2,0){e}
			\uput{0.15}[-90](4,0){o}
			\uput{0.15}[-90](6,0){e}
			\uput{0.15}[-90](8,0){o}
			
			\uput{0.15}[135](0,2){e}
			\uput{0.15}[135](2,2){o}
			\uput{0.15}[135](2,4){e}
			\uput{0.15}[135](4,4){o}
		}
		
	}
	
	\rput(2.5,-7){
		{\psset{linecolor=lightgray}
			\psframe*(-0.5,-0.5)(0.5,0.5)
			\psframe*(-2.5,-0.5)(-1.5,0.5)
			\psframe*(-4.5,-0.5)(-3.5,0.5)
			
			\psframe*(0.5,0.5)(1.5,1.5)
			\psframe*(-1.5,0.5)(-0.5,1.5)
			\psframe*(-3.5,0.5)(-2.5,1.5)
			\psframe*(-5.5,0.5)(-4.5,1.5)
			
			\psframe*(-0.5,1.5)(0.5,2.5)
			\psframe*(-2.5,1.5)(-1.5,2.5)
			\psframe*(-4.5,1.5)(-3.5,2.5)
			
			\psframe*(-1.5,2.5)(-0.5,3.5)
			\psframe*(-3.5,2.5)(-2.5,3.5)
			
			\psframe*(-2.5,3.5)(-1.5,4.5)
		}

		\psline(0,0)(0,1)
		
		\psline(-1,0)(-1.75,1)
		\psline(-2,0)(-2.25,1)
		\psline(-3,0)(-3.75,1)
		\psline(-4,0)(-4.25,1)
		
		\psline(0,2)(-0.25,3)
		\psline(1,2)(0.25,3)
		\psline(-0.75,2)(-1.75,3)
		\psline(-1.25,2)(-2.25,3)
		\psline(-2.75,2)(-3.75,3)
		\psline(-3.25,2)(-4.25,3)
		\psline(-0.75,4)(-1.75,5)
		\psline(-1.25,4)(-2.25,5)
		
		\psline(-4.75,2)(-5.25,2.5)
		\psline(-5.25,2)(-5.75,2.5)
		\psline(-2.75,4)(-3.25,4.5)
		\psline(-3.25,4)(-3.75,4.5)
		
		{\psset{linestyle=dashed,dash=2pt 1pt}
			\psline(1.5,2)(1,2)
			\psline(0,0)(-1,0)
			\psline(-2,0)(-3,0)
			
			\psline(-1.75,1)(-2.75,2)
			\psline(-2.25,1)(-3.25,2)
			\psline(-3.75,1)(-4.75,2)
			\psline(-4.25,1)(-5.25,2)
			
			\psline(-1.75,3)(-2.75,4)
			\psline(-2.25,3)(-3.25,4)
			
			\psline(0.25,3)(-0.75,4)
			\psline(-0.25,3)(-1.25,4)
			
			\psline(0,1)(-0.75,2)
			\pscurve(0,2)(-0.25,1.8)(-1,1.8)(-1.25,2)
			
			\psline(-4,0)(-4.5,0)
			\psline(-1.75,5)(-2.25,5.5)
			\psline(-2.25,5)(-2.75,5.5)
			\psline(-3.75,3)(-4.25,3.5)
			\psline(-4.25,3)(-4.75,3.5)
		}
		
		\pscurve{->}(-1,0)(-1.25,0.1)(-2,0.8)(-2.25,1)
		\pscurve{->}(-1.75,1)(-1.8,0.5)(-2,0)
		
		\psdots[linecolor=blue]
		(0,0)(0,2)
		
		{\tiny
			\uput{0.15}[-45](0,0){o}
			\uput{0.15}[-45](0,2){e}
		}

		\psdots[linecolor=red]%
		(0,1)%
		(-1.75,1)(-2.25,1)%
		(-3.75,1)(-4.25,1)%
		(0.25,3)(-0.25,3)%
		(-1.75,3)(-2.25,3)%
		(-3.75,3)(-4.25,3)%
		(-1.75,5)(-2.25,5)%
		
		{\tiny
			\uput{0.15}[0](-1.75,1){ee}
			\uput{0.15}[180](-2.25,1){oo}
			\uput{0.15}[0](-3.75,1){eo}
			\uput{0.15}[180](-4.25,1){oe}
			
			\uput{0.15}[0](0,1){eo}
		}
		
		\psdots[linecolor=darkgreen]
		(1,2)%
		(-1,0)(-3,0)%
		(-0.75,2)(-1.25,2)%
		(-2.75,2)(-3.25,2)%
		(-4.75,2)(-5.25,2)%
		(-0.75,4)(-1.25,4)%
		(-2.75,4)(-3.25,4)%
		
		{\tiny
			\uput{0.15}[-90](1,2){oo}
		}
		
		\psdots[linecolor=gold]
		(-2,0)(-4,0)
		
		{\tiny
			\uput{0.15}[-90](-2,0){o}
			\uput{0.15}[-90](-4,0){e}
		}
		
	}
	
	\rput(-8.5,-2){
		
		{\psset{linecolor=lightgray}
			\psframe*(3.5,-2.5)(4.5,-1.5)
			
			\psframe*(2.5,-1.5)(3.5,-0.5)
			\psframe*(4.5,-1.5)(5.5,-0.5)
			
			\psframe*(1.5,-0.5)(2.5,0.5)
			\psframe*(3.5,-0.5)(4.5,0.5)
			\psframe*(5.5,-0.5)(6.5,0.5)
			
			\psframe*(0.5,0.5)(1.5,1.5)
			\psframe*(2.5,0.5)(3.5,1.5)
			\psframe*(4.5,0.5)(5.5,1.5)
			\psframe*(6.5,0.5)(7.5,1.5)
			
			\psframe*(1.5,1.5)(2.5,2.5)
			\psframe*(3.5,1.5)(4.5,2.5)
			\psframe*(5.5,1.5)(6.5,2.5)
			
			\psframe*(0.5,2.5)(1.5,3.5)
			\psframe*(2.5,2.5)(3.5,3.5)
			\psframe*(4.5,2.5)(5.5,3.5)
			\psframe*(6.5,2.5)(7.5,3.5)
			
			\psframe*(1.5,3.5)(2.5,4.5)
			\psframe*(3.5,3.5)(4.5,4.5)
			\psframe*(5.5,3.5)(6.5,4.5)
			
			\psframe*(2.5,4.5)(3.5,5.5)
			\psframe*(4.5,4.5)(5.5,5.5)
			
		}
		
		\psline(0.5,2)(1,2)
		\pscurve(2,2)(2.25,2.2)(3,2.2)(3.25,2)
		\psline(2.75,2)(2,3)
		\psline(4.75,2)(3.75,3)
		\psline(5.25,2)(4.25,3)
		\psline(6.75,2)(5.75,3)
		\psline(7.25,2)(6.25,3)
		
		\psline(2.75,0)(1.75,1)
		\psline(3.25,0)(2.25,1)
		\psline(4.75,0)(3.75,1)
		\psline(5.25,0)(4.25,1)
		
		\psline(2,4)(3,4)
		\pscurve(4,4)(4.25,4.2)(5,4.2)(5.25,4)
		\psline(4.75,4)(4,5)
		
		\psline(4.25,-1.5)(3.75,-1)
		\psline(4.75,-1.5)(4.25,-1)
		\psline(6.25,0.5)(5.75,1)
		\psline(6.75,0.5)(6.25,1)
		
		{\psset{linestyle=dashed,dash=2pt 1pt}
			
			\psline(1.75,1)(1,2)
			\psline(2.25,1)(2,2)
			\psline(3.75,1)(2.75,2)
			\psline(4.25,1)(3.25,2)
			\psline(5.75,1)(4.75,2)
			\psline(6.25,1)(5.25,2)
			
			\psline(3.75,-1)(2.75,0)
			\psline(4.25,-1)(3.25,0)
			
			\psline(5.25,-0.5)(4.75,0)
			\psline(5.75,-0.5)(5.25,0)
			\psline(7.25,1.5)(6.75,2)
			\psline(7.75,1.5)(7.25,2)
			
			\psline(2,3)(2,4)
			\psline(3.75,3)(3,4)
			\psline(4.25,3)(4,4)
			\psline(5.75,3)(4.75,4)
			\psline(6.25,3)(5.25,4)
			
			\psline(4,5)(4,5.5)
		}

		\psdots[linecolor=red]%
		(1.75,1)(2.25,1)%
		(3.75,1)(4.25,1)%
		(5.75,1)(6.25,1)%
		(2,3)%
		(3.75,3)(4.25,3)%
		(5.75,3)(6.25,3)%
		(4,5)
		(3.75,-1)(4.25,-1)%
		
		{\tiny
			\uput{0.15}[180](2,3){oo}
			\uput{0.15}[180](4,5){oo}
		}
		
		\psdots[linecolor=darkgreen]
		(1,2)%
		(2.75,2)(3.25,2)%
		(4.75,2)(5.25,2)%
		(6.75,2)(7.25,2)%
		(3,4)%
		(4.75,4)(5.25,4)%
		(2.75,0)(3.25,0)%
		(4.75,0)(5.25,0)%
		
		{\tiny
			\uput{0.15}[90](1,2){eo}
			\uput{0.15}[90](3,4){eo}
		}
		
		\psdots[linecolor=gold]
		(2,2)%
		(2,4)(4,4)%
		
		{\tiny
			\uput{0.15}[135](2,2){o}
			\uput{0.15}[135](2,4){e}
			\uput{0.15}[135](4,4){o}
		}

	}
	
	\rput(-2.5,7){\psrotate(0,0){180}{
			
			{\psset{linecolor=lightgray}
				\psframe*(-0.5,-0.5)(0.5,0.5)
				\psframe*(-2.5,-0.5)(-1.5,0.5)
				\psframe*(-4.5,-0.5)(-3.5,0.5)
				
				\psframe*(0.5,0.5)(1.5,1.5)
				\psframe*(-1.5,0.5)(-0.5,1.5)
				\psframe*(-3.5,0.5)(-2.5,1.5)
				\psframe*(-5.5,0.5)(-4.5,1.5)
				
				\psframe*(-0.5,1.5)(0.5,2.5)
				\psframe*(-2.5,1.5)(-1.5,2.5)
				\psframe*(-4.5,1.5)(-3.5,2.5)
				
				\psframe*(-1.5,2.5)(-0.5,3.5)
				\psframe*(-3.5,2.5)(-2.5,3.5)
				
				\psframe*(-2.5,3.5)(-1.5,4.5)
			}
			
			{\psset{linestyle=dashed,dash=2pt 1pt}
				\psline(0,0)(0,1)
				
				\psline(-1,0)(-1.75,1)
				\psline(-2,0)(-2.25,1)
				\psline(-3,0)(-3.75,1)
				\psline(-4,0)(-4.25,1)
				
				\psline(0,2)(-0.25,3)
				\psline(1,2)(0.25,3)
				\psline(-0.75,2)(-1.75,3)
				\psline(-1.25,2)(-2.25,3)
				\psline(-2.75,2)(-3.75,3)
				\psline(-3.25,2)(-4.25,3)
				\psline(-0.75,4)(-1.75,5)
				\psline(-1.25,4)(-2.25,5)
				
				\psline(-4.75,2)(-5.25,2.5)
				\psline(-5.25,2)(-5.75,2.5)
				\psline(-2.75,4)(-3.25,4.5)
				\psline(-3.25,4)(-3.75,4.5)
				
				\pscurve{->}(-1,0)(-1.25,0.1)(-2,0.8)(-2.25,1)
				\pscurve{->}(-1.75,1)(-1.8,0.5)(-2,0)
			}
			\psline(1.5,2)(1,2)
			\psline(0,0)(-1,0)
			\psline(-2,0)(-3,0)
			
			\psline(-1.75,1)(-2.75,2)
			\psline(-2.25,1)(-3.25,2)
			\psline(-3.75,1)(-4.75,2)
			\psline(-4.25,1)(-5.25,2)
			
			\psline(-1.75,3)(-2.75,4)
			\psline(-2.25,3)(-3.25,4)
			
			\psline(0.25,3)(-0.75,4)
			\psline(-0.25,3)(-1.25,4)
			
			\psline(0,1)(-0.75,2)
			\pscurve(0,2)(-0.25,1.8)(-1,1.8)(-1.25,2)
			
			\psline(-4,0)(-4.5,0)
			\psline(-1.75,5)(-2.25,5.5)
			\psline(-2.25,5)(-2.75,5.5)
			\psline(-3.75,3)(-4.25,3.5)
			\psline(-4.25,3)(-4.75,3.5)

			\psdots[linecolor=gold]
			(0,0)(0,2)
			{\tiny
				\uput{0.15}[-45]{180}(0,0){e}
				\uput{0.15}[-45]{180}(0,2){o}
			}
			
			\psdots[linecolor=red]%
			(0,1)%
			(-1.75,1)(-2.25,1)%
			(-3.75,1)(-4.25,1)%
			(0.25,3)(-0.25,3)%
			(-1.75,3)(-2.25,3)%
			(-3.75,3)(-4.25,3)%
			(-1.75,5)(-2.25,5)%
			
			{\tiny
				\uput{0.15}[0]{180}(-1.75,1){oe}
				\uput{0.15}[180]{180}(-2.25,1){eo}
				\uput{0.15}[0]{180}(-3.75,1){oo}
				\uput{0.15}[180]{180}(-4.25,1){ee}
				
				\uput{0.15}[0]{180}(0,1){oo}
			}
			
			\psdots[linecolor=darkgreen]
			(1,2)%
			(-1,0)(-3,0)%
			(-0.75,2)(-1.25,2)%
			(-2.75,2)(-3.25,2)%
			(-4.75,2)(-5.25,2)%
			(-0.75,4)(-1.25,4)%
			(-2.75,4)(-3.25,4)%
			
			{\tiny
				\uput{0.15}[-90]{180}(1,2){eo}
			}
			
			\psdots[linecolor=blue]
			(-2,0)(-4,0)
			
			{\tiny
				\uput{0.15}[-90]{180}(-2,0){o}
				\uput{0.15}[-90]{180}(-4,0){e}
			}
		}}
		
		\rput(8.5,2){\psrotate(0,0){180}{
				
				{\psset{linecolor=lightgray}
					\psframe*(3.5,-2.5)(4.5,-1.5)
					
					\psframe*(2.5,-1.5)(3.5,-0.5)
					\psframe*(4.5,-1.5)(5.5,-0.5)
					
					\psframe*(1.5,-0.5)(2.5,0.5)
					\psframe*(3.5,-0.5)(4.5,0.5)
					\psframe*(5.5,-0.5)(6.5,0.5)
					
					\psframe*(0.5,0.5)(1.5,1.5)
					\psframe*(2.5,0.5)(3.5,1.5)
					\psframe*(4.5,0.5)(5.5,1.5)
					\psframe*(6.5,0.5)(7.5,1.5)
					
					\psframe*(1.5,1.5)(2.5,2.5)
					\psframe*(3.5,1.5)(4.5,2.5)
					\psframe*(5.5,1.5)(6.5,2.5)
					
					\psframe*(0.5,2.5)(1.5,3.5)
					\psframe*(2.5,2.5)(3.5,3.5)
					\psframe*(4.5,2.5)(5.5,3.5)
					\psframe*(6.5,2.5)(7.5,3.5)
					
					\psframe*(1.5,3.5)(2.5,4.5)
					\psframe*(3.5,3.5)(4.5,4.5)
					\psframe*(5.5,3.5)(6.5,4.5)
					
					\psframe*(2.5,4.5)(3.5,5.5)
					\psframe*(4.5,4.5)(5.5,5.5)
					
				}
				
				{\psset{linestyle=dashed,dash=2pt 1pt}
					\psline(0.5,2)(1,2)
					\pscurve(2,2)(2.25,2.2)(3,2.2)(3.25,2)
					\psline(2.75,2)(2,3)
					\psline(4.75,2)(3.75,3)
					\psline(5.25,2)(4.25,3)
					\psline(6.75,2)(5.75,3)
					\psline(7.25,2)(6.25,3)
					
					\psline(2.75,0)(1.75,1)
					\psline(3.25,0)(2.25,1)
					\psline(4.75,0)(3.75,1)
					\psline(5.25,0)(4.25,1)
					
					\psline(2,4)(3,4)
					\pscurve(4,4)(4.25,4.2)(5,4.2)(5.25,4)
					\psline(4.75,4)(4,5)
					
					\psline(4.25,-1.5)(3.75,-1)
					\psline(4.75,-1.5)(4.25,-1)
					\psline(6.25,0.5)(5.75,1)
					\psline(6.75,0.5)(6.25,1)
					
				}
				
				\psline(1.75,1)(1,2)
				\psline(2.25,1)(2,2)
				\psline(3.75,1)(2.75,2)
				\psline(4.25,1)(3.25,2)
				\psline(5.75,1)(4.75,2)
				\psline(6.25,1)(5.25,2)
				
				\psline(3.75,-1)(2.75,0)
				\psline(4.25,-1)(3.25,0)
				
				\psline(5.25,-0.5)(4.75,0)
				\psline(5.75,-0.5)(5.25,0)
				\psline(7.25,1.5)(6.75,2)
				\psline(7.75,1.5)(7.25,2)
				
				\psline(2,3)(2,4)
				\psline(3.75,3)(3,4)
				\psline(4.25,3)(4,4)
				\psline(5.75,3)(4.75,4)
				\psline(6.25,3)(5.25,4)
				
				\psline(4,5)(4,5.5)

				\psdots[linecolor=red]%
				(1.75,1)(2.25,1)%
				(3.75,1)(4.25,1)%
				(5.75,1)(6.25,1)%
				(2,3)%
				(3.75,3)(4.25,3)%
				(5.75,3)(6.25,3)%
				(4,5)
				(3.75,-1)(4.25,-1)%
				
				{\tiny
					\uput{0.15}[180]{180}(2,3){eo}
					\uput{0.15}[180]{180}(4,5){eo}
				}
				
				\psdots[linecolor=darkgreen]
				(1,2)%
				(2.75,2)(3.25,2)%
				(4.75,2)(5.25,2)%
				(6.75,2)(7.25,2)%
				(3,4)%
				(4.75,4)(5.25,4)%
				(2.75,0)(3.25,0)%
				(4.75,0)(5.25,0)%
				
				{\tiny
					\uput{0.15}[90]{180}(1,2){oo}
					\uput{0.15}[90]{180}(3,4){oo}
				}
				
				\psdots[linecolor=blue]
				(2,2)%
				(2,4)(4,4)%
				
				{\tiny
					\uput{0.15}[135]{180}(2,2){e}
					\uput{0.15}[135]{180}(2,4){o}
					\uput{0.15}[135]{180}(4,4){e}
				}
			}}
			
			\rput(12.5,7){\psrotate(0,0){180}{
					
					{\psset{linecolor=lightgray}
						\psframe*(-0.5,-0.5)(0.5,0.5)
						\psframe*(1.5,-0.5)(2.5,0.5)
						\psframe*(3.5,-0.5)(4.5,0.5)
						\psframe*(5.5,-0.5)(6.5,0.5)
						\psframe*(7.5,-0.5)(8.5,0.5)
						
						\psframe*(0.5,0.5)(1.5,1.5)
						\psframe*(2.5,0.5)(3.5,1.5)
						\psframe*(4.5,0.5)(5.5,1.5)
						\psframe*(6.5,0.5)(7.5,1.5)
						
						\psframe*(-0.5,1.5)(0.5,2.5)
						\psframe*(1.5,1.5)(2.5,2.5)
						\psframe*(3.5,1.5)(4.5,2.5)
						\psframe*(5.5,1.5)(6.5,2.5)
						
						\psframe*(0.5,2.5)(1.5,3.5)
						\psframe*(2.5,2.5)(3.5,3.5)
						\psframe*(4.5,2.5)(5.5,3.5)
						
						\psframe*(1.5,3.5)(2.5,4.5)
						\psframe*(3.5,3.5)(4.5,4.5)
						
						\psframe*(2.5,4.5)(3.5,5.5)
					}
					
					{\psset{linestyle=dashed,dash=2pt 1pt}

						\psline(1,0)(0,1)
						
						\psline(3,0)(2.25,1)
						\psline(5,0)(4.25,1)
						\psline(7,0)(6.25,1)
						\psline(9,0)(8.25,1)
						
						\psline(2,0)(1.75,1)
						\psline(4,0)(3.75,1)
						\psline(6,0)(5.75,1)
						\psline(8,0)(7.75,1)
						
						\psline(0,2)(1,2)
						\pscurve(2,2)(2.25,2.2)(3,2.2)(3.25,2)
						\psline(2.75,2)(2,3)
						\psline(4.75,2)(3.75,3)
						\psline(5.25,2)(4.25,3)
						\psline(6.75,2)(5.75,3)
						\psline(7.25,2)(6.25,3)
						
						\psline(2,4)(3,4)
						\pscurve(4,4)(4.25,4.2)(5,4.2)(5.25,4)
						\psline(4.75,4)(4,5)
						
						\pscurve{->}(5,0)(4.75,0.1)(4,0.8)(3.75,1)
						\pscurve{->}(9,0)(8.75,0.1)(8,0.8)(7.75,1)
						\pscurve{->}(4.25,1)(4.2,0.5)(4,0)
						\pscurve{->}(8.25,1)(8.2,0.5)(8,0)
					}
					\psline(1,0)(2,0)
					\psline(3,0)(4,0)
					\psline(5,0)(6,0)
					\psline(7,0)(8,0)
					
					\psline(0,1)(0,2)
					\psline(1.75,1)(1,2)
					\psline(2.25,1)(2,2)
					\psline(3.75,1)(2.75,2)
					\psline(4.25,1)(3.25,2)
					\psline(5.75,1)(4.75,2)
					\psline(6.25,1)(5.25,2)
					\psline(7.75,1)(6.75,2)
					\psline(8.25,1)(7.25,2)
					
					\psline(2,3)(2,4)
					\psline(3.75,3)(3,4)
					\psline(4.25,3)(4,4)
					\psline(5.75,3)(4.75,4)
					\psline(6.25,3)(5.25,4)
					
					\psline(4,5)(4,5.5)
					\psline(9,0)(9.5,0)

					\psdots[linecolor=red]%
					(0,1)%
					(1.75,1)(2.25,1)%
					(3.75,1)(4.25,1)%
					(5.75,1)(6.25,1)%
					(7.75,1)(8.25,1)%
					(2,3)%
					(3.75,3)(4.25,3)%
					(5.75,3)(6.25,3)%
					(4,5)
					
					{\tiny
						\uput{0.15}[180]{180}(1.75,1){ee}
						\uput{0.15}[0]{180}(2.25,1){oo}
						\uput{0.15}[180]{180}(3.75,1){eo}
						\uput{0.15}[0]{180}(4.25,1){oe}
						\uput{0.15}[180]{180}(5.75,1){ee}
						\uput{0.15}[0]{180}(6.25,1){oo}
						\uput{0.15}[180]{180}(7.75,1){eo}
						\uput{0.15}[0]{180}(8.25,1){oe}
						
						\uput{0.15}[180]{180}(2,3){eo}
						\uput{0.15}[180]{180}(4,5){eo}
					}
					
					\psdots[linecolor=darkgreen]
					(1,0)%
					(3,0)%
					(5,0)%
					(7,0)%
					(9,0)%
					(1,2)%
					(2.75,2)(3.25,2)%
					(4.75,2)(5.25,2)%
					(6.75,2)(7.25,2)%
					(3,4)%
					(4.75,4)(5.25,4)%
					
					{\tiny
						\uput{0.15}[90]{180}(1,2){oo}
						\uput{0.15}[90]{180}(3,4){oo}
					}
					
					\psdots[linecolor=blue]
					(2,0)(4,0)(6,0)(8,0)%
					(0,2)(2,2)%
					(2,4)(4,4)%
					
					{\tiny
						\uput{0.15}[-90]{180}(2,0){e}
						\uput{0.15}[-90]{180}(4,0){o}
						\uput{0.15}[-90]{180}(6,0){e}
						\uput{0.15}[-90]{180}(8,0){o}
						
						\uput{0.15}[135]{180}(0,2){o}
						\uput{0.15}[135]{180}(2,2){e}
						\uput{0.15}[135]{180}(2,4){o}
						\uput{0.15}[135]{180}(4,4){e}
					}
				}}

				\end{pspicture}
				\caption{The last step of the calculation of the peculiar modules of $(2n,-(2m+1))$-pretzel tangles. Some of the vertices are labelled according to the parity of the indices of the generators they correspond to, where ``e'' stands for ``even'' and ``o'' for ``odd''. }\label{fig:PreResultPretzels}
			\end{figure}
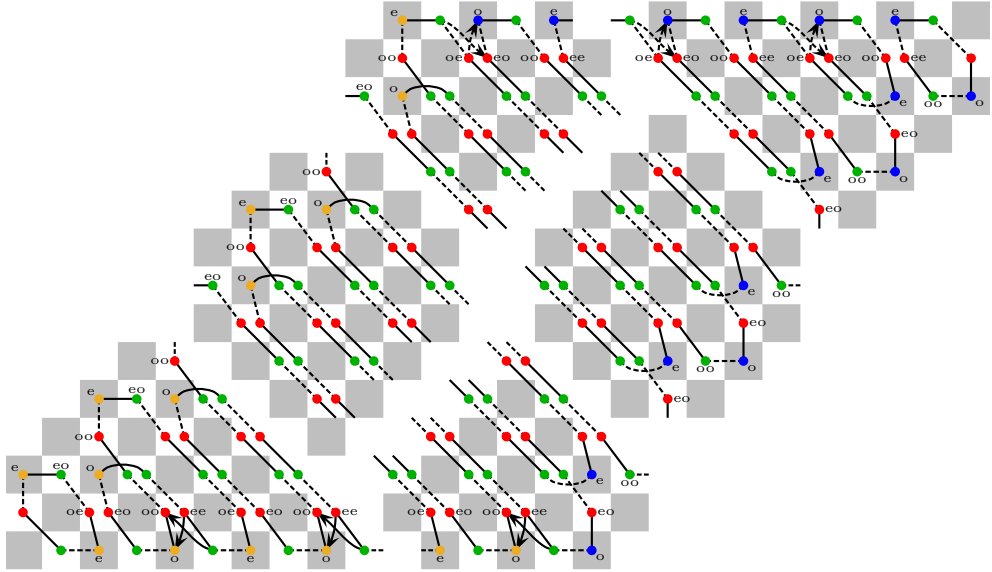

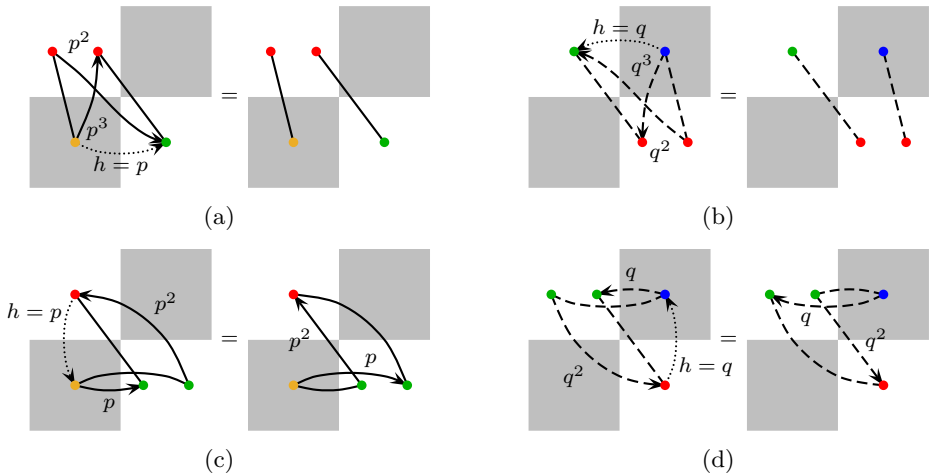
\begin{figure}
	\centering
	\psset{unit=1.2}
	\begin{subfigure}[b]{0.45\textwidth}\centering
	\begin{pspicture}(-2.8,-1)(2.8,1)
		\rput(-1.2,0){
			{\psset{linecolor=lightgray}
				\psframe*(0,0)(1,1)
				\psframe*(0,0)(-1,-1)
			}
			
			\psline(-0.5,-0.5)(-0.75,0.5)
			\psline(0.5,-0.5)(-0.25,0.5)
			
			\pscurve{<-}(0.5,-0.5)(0.25,-0.4)(-0.5,0.3)(-0.75,0.5)
			\pscurve{<-}(-0.25,0.5)(-0.3,0)(-0.5,-0.5)
	
			\pscurve[linestyle=dotted,dotsep=1pt]{->}(-0.5,-0.5)(-0.25,-0.6)(0.25,-0.6)(0.5,-0.5)
	
			\psdots[linecolor=red](-0.75,0.5)(-0.25,0.5)
			\psdots[linecolor=darkgreen](0.5,-0.5)
			\psdots[linecolor=gold](-0.5,-0.5)
			
			\rput[t](0,-0.65){$h=p$}
			\rput[c](-0.25,-0.35){$p^3$}
			\rput[c](-0.45,0.6){$p^2$}
		}
		\rput(0,0){$=$}
		\rput(1.2,0){
			{\psset{linecolor=lightgray}
				\psframe*(0,0)(1,1)
				\psframe*(0,0)(-1,-1)
			}
			
			\psline(-0.5,-0.5)(-0.75,0.5)
			\psline(0.5,-0.5)(-0.25,0.5)
			
			\psdots[linecolor=red](-0.75,0.5)(-0.25,0.5)
			\psdots[linecolor=darkgreen](0.5,-0.5)
			\psdots[linecolor=gold](-0.5,-0.5)
		}
	\end{pspicture}
	\caption{}\label{fig:CalculationPretzelsBasicHomotopies1}
	\end{subfigure}
	\begin{subfigure}[b]{0.45\textwidth}\centering
		\begin{pspicture}(-2.8,-1)(2.8,1)
		\psset{linestyle=dashed,dash=4pt 2pt}
		\rput{180}(-1.2,0){
			{\psset{linecolor=lightgray}
				\psframe*(0,0)(1,1)
				\psframe*(0,0)(-1,-1)
			}
			
			\psline(-0.5,-0.5)(-0.75,0.5)
			\psline(0.5,-0.5)(-0.25,0.5)
			
			\pscurve{<-}(0.5,-0.5)(0.25,-0.4)(-0.5,0.3)(-0.75,0.5)
			\pscurve{<-}(-0.25,0.5)(-0.3,0)(-0.5,-0.5)
			
			\pscurve[linestyle=dotted,dotsep=1pt]{->}(-0.5,-0.5)(-0.25,-0.6)(0.25,-0.6)(0.5,-0.5)
			
			\psdots[linecolor=red](-0.75,0.5)(-0.25,0.5)
			\psdots[linecolor=darkgreen](0.5,-0.5)
			\psdots[linecolor=blue](-0.5,-0.5)
			
			\rput[b]{180}(0,-0.65){$h=q$}
			\rput[c]{180}(-0.25,-0.35){$q^3$}
			\rput[c]{180}(-0.45,0.6){$q^2$}
		}
		\rput(0,0){$=$}
		\rput{180}(1.2,0){
			{\psset{linecolor=lightgray}
				\psframe*(0,0)(1,1)
				\psframe*(0,0)(-1,-1)
			}
			
			\psline(-0.5,-0.5)(-0.75,0.5)
			\psline(0.5,-0.5)(-0.25,0.5)
			
			\psdots[linecolor=red](-0.75,0.5)(-0.25,0.5)
			\psdots[linecolor=darkgreen](0.5,-0.5)
			\psdots[linecolor=blue](-0.5,-0.5)
		}
		\end{pspicture}
		\caption{}\label{fig:CalculationPretzelsBasicHomotopies2}
	\end{subfigure}	
	\\
	\begin{subfigure}[b]{0.45\textwidth}\centering
		\begin{pspicture}(-2.8,-1)(2.8,1)
		\rput(-1.2,0){
			{\psset{linecolor=lightgray}
				\psframe*(0,0)(1,1)
				\psframe*(0,0)(-1,-1)
			}
			
			\psline(0.25,-0.5)(-0.5,0.5)
			\pscurve(-0.5,-0.5)(-0.25,-0.4)(0.5,-0.4)(0.75,-0.5)
			
			\pscurve{<-}(0.25,-0.5)(0,-0.55)(-0.25,-0.55)(-0.5,-0.5)
			\pscurve{->}(0.75,-0.5)(0.5,0)(0,0.4)(-0.5,0.5)
			
			\pscurve[linestyle=dotted,dotsep=1pt]{<-}(-0.5,-0.5)(-0.6,-0.25)(-0.6,0.25)(-0.5,0.5)
			
			\psdot[linecolor=red](-0.5,0.5)
			\psdot[linecolor=darkgreen](0.25,-0.5)
			\psdot[linecolor=darkgreen](0.75,-0.5)
			\psdot[linecolor=gold](-0.5,-0.5)
			
			\rput[r](-0.65,0.3){$h=p$}
			\rput[c](0.5,0.4){$p^2$}
			\rput[t](-0.125,-0.65){$p$}
		}
		\rput(0,0){$=$}
		\rput(1.2,0){
			{\psset{linecolor=lightgray}
				\psframe*(0,0)(1,1)
				\psframe*(0,0)(-1,-1)
			}
			
			\psline{->}(0.25,-0.5)(-0.5,0.5)
			\pscurve{->}(-0.5,-0.5)(-0.25,-0.4)(0.5,-0.4)(0.75,-0.5)
			
			\pscurve(0.25,-0.5)(0,-0.55)(-0.25,-0.55)(-0.5,-0.5)
			\pscurve(0.75,-0.5)(0.5,0)(0,0.4)(-0.5,0.5)
			
			
			\psdot[linecolor=red](-0.5,0.5)
			\psdot[linecolor=darkgreen](0.25,-0.5)
			\psdot[linecolor=darkgreen](0.75,-0.5)
			\psdot[linecolor=gold](-0.5,-0.5)
			
			\rput[r](-0.3,0){$p^2$}
			\rput[b](0.35,-0.3){$p$}
		}
		\end{pspicture}
		\caption{}\label{fig:CalculationPretzelsBasicHomotopies3}
	\end{subfigure}
	\begin{subfigure}[b]{0.45\textwidth}\centering
		\begin{pspicture}(-2.8,-1)(2.8,1)
		\psset{linestyle=dashed,dash=4pt 2pt}
		\rput(0,0){$=$}
		\rput{180}(1.2,0){
			{\psset{linecolor=lightgray}
				\psframe*(0,0)(1,1)
				\psframe*(0,0)(-1,-1)
			}
			
			\psline{->}(0.25,-0.5)(-0.5,0.5)
			\pscurve{->}(-0.5,-0.5)(-0.25,-0.4)(0.5,-0.4)(0.75,-0.5)
			
			\pscurve(0.25,-0.5)(0,-0.55)(-0.25,-0.55)(-0.5,-0.5)
			\pscurve(0.75,-0.5)(0.5,0)(0,0.4)(-0.5,0.5)
			
			
			\psdot[linecolor=red](-0.5,0.5)
			\psdot[linecolor=darkgreen](0.25,-0.5)
			\psdot[linecolor=darkgreen](0.75,-0.5)
			\psdot[linecolor=blue](-0.5,-0.5)
			
			\rput[l]{180}(-0.3,0){$q^2$}
			\rput[t]{180}(0.35,-0.3){$q$}
		}
		\rput{180}(-1.2,0){
			{\psset{linecolor=lightgray}
				\psframe*(0,0)(1,1)
				\psframe*(0,0)(-1,-1)
			}
			
			\psline(0.25,-0.5)(-0.5,0.5)
			\pscurve(-0.5,-0.5)(-0.25,-0.4)(0.5,-0.4)(0.75,-0.5)
			
			\pscurve{<-}(0.25,-0.5)(0,-0.55)(-0.25,-0.55)(-0.5,-0.5)
			\pscurve{->}(0.75,-0.5)(0.5,0)(0,0.4)(-0.5,0.5)
			
			\pscurve[linestyle=dotted,dotsep=1pt]{<-}(-0.5,-0.5)(-0.6,-0.25)(-0.6,0.25)(-0.5,0.5)
			
			\psdot[linecolor=red](-0.5,0.5)
			\psdot[linecolor=darkgreen](0.25,-0.5)
			\psdot[linecolor=darkgreen](0.75,-0.5)
			\psdot[linecolor=blue](-0.5,-0.5)
			
			\rput[l]{180}(-0.65,0.3){$h=q$}
			\rput[c]{180}(0.5,0.4){$q^2$}
			\rput[b]{180}(-0.125,-0.65){$q$}
		}
		\end{pspicture}
		\caption{}\label{fig:CalculationPretzelsBasicHomotopies4}
	\end{subfigure}
	\caption{The morphisms $h$ for the final step of the proof of Theorem~\ref{thm:pretzeltangleCalc}.}\label{fig:CalculationPretzelsBasicHomotopies}
	\end{figure}

\begin{corollary}\label{cor:pretzeltangleMutation}
	Mutation about \((2n,-(2m+1))\)-pretzel tangles for \(n,m>0\), oriented as in Figure~\ref{fig:pretzeltangle2nm2mp1}, preserves bigraded link Floer homology, after identifying the Alexander gradings of the two open strands. If we reverse the orientation of one of the two strands, mutation in general only preserves \(\delta\)-graded link Floer homology.
\end{corollary}

\begin{proof}[of Corollary~\ref{cor:pretzeltangleMutation}]
	The invariants of the \((2n,-(2m+1))\)-pretzel tangles simply have the desired symmetry. This can be seen as follows: 
	in terms of the loops on our infinite chessboard, a relabelling of the sites as in Theorem~\ref{thm:GeneralMutationInvariance} corresponds to a recolouring of the vertices: $\textcolor{red}{\bullet}\leftrightarrow\textcolor{darkgreen}{\bullet}$ and $\textcolor{blue}{\bullet}\leftrightarrow\textcolor{gold}{\bullet}$. An orientation reversal of both tangle strands corresponds to a rotation of the chessboard by $\pi$. After identifying the Alexander gradings of the two tangle strands, all generators on the diagonals from bottom-left to top-right have the same Alexander grading. Finally, if we reverse the orientation of one strand, we do not need to rotate the curves, but the generators in the same Alexander gradings now sit on the diagonals that go from top-left to the bottom-right.
\end{proof}

\begin{proof}[of Theorem~\ref{thm:pretzeltangleCalc}]
	The generators of the peculiar module can already be determined from the decategorified invariants and from the observation of two obvious differentials that can be cancelled as in Example~\ref{exa:HFTdpretzeltangle}. Thus, the vertices of the graphs in Figure~\ref{fig:ResultPretzels} are fixed. What we need to decide is how they are connected. Because of the restrictions given by the gradings, for most cases, there is only one way to connect them such that the result is a peculiar module. The only question is how the diagonal strings of red and green generators connect the generators on the top left to the generators on the bottom right of each of the subfigures of Figure~\ref{fig:ResultPretzels}. 
	For this, we are going to apply the algorithm from Corollary~\ref{cor:PeculiarModulesFromNiceDiagrams}, setting $p_1=0$ and $q_2=0$. 
	
	We start with a Heegaard diagram obtained by glueing two Heegaard diagrams for the rational tangles with $2n$~twist, respectively $-(2m+1)$~twist together. We can niceify the diagram by doing two handleslides of the $\beta$-curve for the $2n$-twist rational tangle across the other $\beta$-curve. The result is shown in Figure~\ref{fig:CalculationPretzelsStep1HD}. The generators of this diagram are shown in Figure~\ref{fig:CalculationPretzelsStep1Gens}. The generators of the second and third row are the ones that were created during the first and second handleslide, respectively, and thus can be cancelled along the identity arrows connecting those generators of the same site and with the same indices. 
	
	Note that the nice Heegaard diagram has the same symmetry as the tangles, and thus the complex inherits this symmetry. More precisely, the Heegaard diagram remains invariant under the following operation: in the names of the generators, exchange the letters $d$ and $b$, exchange underlining and overlining, replace $i$ by $2n+1-i$, and $j$ and $k$ by $2m+2-j$ and $2m+2-k$, respectively. Finally, in the algebra, exchange $p_2$ and $q_1$, as well as $p_3$ and $q_4$, and $q_3$ and $p_4$. This symmetry corresponds to mutation about the horizontal axis, which leaves the tangle invariant up to exchanging the two sites $b$ and $d$. We will use this symmetry in the following to simplify some parts of the computation. 
	
	In Figures~\ref{fig:CalculationPretzelsAStep2generic} and \ref{fig:CalculationPretzelsAStep2MaximumP}, we compute all differentials that start/end at some selected generators, which are enclosed in those figures by boxes. Note that in all figures, generators in the same shaded regions share the same Alexander bigrading. Since the Heegaard diagram is nice, the only contributing domains are bigons and squares, so the computation is purely combinatorial and straightforward. We therefore ask the reader to check for themselves that indeed, the differentials starting/ending at all marked generators in those figures are included. (There are two observations that one might find useful when determining the contributing differentials: firstly, the only bigons in the diagram contribute arrows labelled by the elementary algebra elements $q_3$ and $p_4$. Secondly, all other contributions come from squares, which necessarily have boundary on both $\beta$-curves.)

	Next, we consider the effect of cancelling generators, first those corresponding to undoing the handleslides and then any remaining identity arrows. Obviously, the pictures in Figure~\ref{fig:CalculationPretzelsAStep2generic} do not change. 
	In Figure~\ref{fig:CalculationPretzelsStep2BottomEnd}, cancellation only contributes an arrow $p_{23}\co\textcolor{darkgreen}{x_{1}c_{2m+1}}\rightarrow\textcolor{red}{a_{1}y_{2m+1}}$. The only possible arrow labelled by a power of $p$ leaving $\textcolor{gold}{d'y_{2m+1}}$ can go to $\textcolor{red}{a_{1}y_{2m+1}}$. Because of the $\partial^2$-relation in the peculiar module, this arrow has to be there. 
	Similarly, we can argue for Figure~\ref{fig:CalculationPretzelsStep2BottomOdd}. Cancellation contributes an arrow $p_{23}\co\textcolor{darkgreen}{x_{1}c_{2j+1}}\rightarrow\textcolor{red}{a_{1}y_{2j+1}}$. It might also contribute another arrow, $p_{3}\co\textcolor{darkgreen}{x_{1}c_{2j+1}}\rightarrow\textcolor{gold}{d'y_{2j+2}}$, stemming from the arrow $q_{4}\co\textcolor{darkgreen}{\overline{y}_{1}c_{2j+2}}\rightarrow\textcolor{gold}{\overline{d}y_{2j+3}}$; however, the $\partial^2$-relation at $\textcolor{darkgreen}{x_{1}c_{2j+1}}$ in the peculiar module tells us that this does not happen. Again, there has to be a contribution $p_{234}\co\textcolor{gold}{d'y_{2j+1}}\rightarrow\textcolor{red}{a_{1}y_{2j+1}}$. 
	In Figure~\ref{fig:CalculationPretzelsStep2TopEven}, there is no arrow $\textcolor{darkgreen}{x_{1}c_{2j}}\rightarrow\textcolor{red}{a_{1}y_{2j}}$, but again, the cancellation must contribute an arrow $p_{234}\co\textcolor{gold}{d'y_{2j}}\rightarrow\textcolor{red}{a_{1}y_{2j}}$. 
	Finally, in Figure~\ref{fig:CalculationPretzelsStep2GoldCancel}, we cancel two arrows, namely $\textcolor{gold}{x_1d}\rightarrow\textcolor{gold}{d'y_1}$ and $\textcolor{darkgreen}{\overline{y}_1c_2}\rightarrow\textcolor{darkgreen}{\overline{c}_2y_1}$. This only contributes one arrow, namely  $p_{23}\co\textcolor{darkgreen}{x_{1}c_1}\rightarrow\textcolor{red}{a_1y_{1}}$.
	
	Similarly, we argue for those generators in minimal Alexander grading corresponding to the colour $t_1$. Alternatively, we may apply the symmetry of the Heegaard diagram. The corresponding subcomplexes after cancellation are shown in Figure~\ref{fig:CalculationPretzelsAStep2MinimumP}.

	We have now obtained a reduced complex. To this, we add some arrows labelled by basic algebra elements that lie in the kernel of $\Ad\rightarrow\Ad/(p_1=0=q_2)$ to turn the complex into the peculiar module shown in Figure~\ref{fig:PreResultPretzels}.
	To see that this is indeed the result of putting all pieces together, one might for example start at the ends of Figure~\ref{fig:CalculationPretzelsStep2GoldCancel} and~\ref{fig:CalculationPretzelsStep2BlueCancel} which sit at the bottom left and top right corners of Figure~\ref{fig:PreResultPretzels}, respectively, and then connect these subcomplexes. To see how the diagonal strings of red and green generators connect the generators on the top left to the generators on the bottom right of Figure~\ref{fig:PreResultPretzels}, 
	one can consider the parity of the generator indices, some of which are marked in the figure. 
	
	We can now apply the Clean-Up Lemma (\ref{lem:AbstractCleanUp}) using the morphisms $h$ shown in Figure~\ref{fig:CalculationPretzelsBasicHomotopies} to obtain the desired result. 
\end{proof}